\documentclass[12pt, reqno]{amsart}
\usepackage{mathrsfs}
\usepackage{amssymb,amsthm,amsmath}
\usepackage[numbers,sort&compress]{natbib}
\usepackage{amssymb,amsmath}
\usepackage{amsfonts}
\usepackage{mathrsfs}
\usepackage{latexsym}
\usepackage{amssymb}
\usepackage{amsthm}
\usepackage{color}
\usepackage{pdfsync}
\usepackage{indentfirst}
\usepackage{graphics}
\usepackage{subfigure}
\usepackage{epsfig}
\usepackage{float}
\usepackage{cases}
\date{today}
\allowdisplaybreaks

\usepackage{hyperref}
\hypersetup{
colorlinks=true,
linkcolor=blue,
anchorcolor=g,
citecolor=red}

\usepackage{amsmath}
\usepackage{amsthm}

\newtheorem{remark}{Remark}[section]

\newtheorem{theorem}{Theorem}[section]
\newtheorem{proposition}{Proposition}[section]
\newtheorem{lemma}{Lemma}[section]
\newtheorem{corollary}{Corollary}[section]

\newcommand{\beq}{\begin{equation}}
\newcommand{\eeq}{\end{equation}}
\newcommand{\ben}{\begin{eqnarray}}
\newcommand{\een}{\end{eqnarray}}
\newcommand{\beno}{\begin{eqnarray*}}
\newcommand{\eeno}{\end{eqnarray*}}

\numberwithin{equation}{section}

\begin{document}
\title[\vspace{-0.2cm}Suppression of blow-up for PKS--NS system]{ On the sharp critical mass threshold for the 3D Patlak--Keller--Segel--Navier--Stokes system via  Couette flow}
\author{Shikun~Cui}
\address[Shikun~Cui]{School of Mathematical Sciences, Dalian University of Technology, Dalian, 116024,  China}
\email{cskmath@163.com}
\author{Lili~Wang}
\address[Lili~Wang]{School of Mathematical Sciences, Dalian University of Technology, Dalian, 116024,  China}
\email{wllmath@163.com}
\author{Wendong~Wang}
\address[Wendong~Wang]{School of Mathematical Sciences, Dalian University of Technology, Dalian, 116024,  China}
\email{wendong@dlut.edu.cn}
\author{Juncheng~Wei}
\address[Juncheng~Wei]{Department of Mathematics, Chinese University of Hong Kong, Shatin, NT, Hong Kong}
\email{wei@math.cuhk.edu.hk}
\date{\today}


\vspace{-0.4cm}

\begin{abstract}
 As is well-known, the solution of the Patlak--Keller--Segel system in 3D may blow up in finite time regardless of any initial cell mass. 
 In this paper, we are interested in the suppression of blow-up and the critical mass threshold for the 3D Patlak--Keller--Segel--Navier--Stokes system via the Couette flow $(Ay, 0, 0)$.
 It is proved that if the Couette flow is sufficiently strong ($A$ is large enough), the initial cell mass is less than $16\pi^{2}$, and the zero modes of the initial velocity are sufficiently small as in \eqref{conditions:u20 u30}, then the solutions for the system are global in time.
 The threshold $16\pi^2$ seems to be optimal: for every finite
 \(A>0\) and every \(M>16\pi^2\), we show that there exist smooth positive
 initial densities of mass \(M\), with zero initial velocity
 perturbation, whose corresponding solutions blow up in finite
 time. 
A key ingredient is a time-integrable
 decay estimate for the nonconstant part of the 
 zero-mode velocity, which allows the logarithmic
 Hardy--Littlewood--Sobolev inequality to be used throughout the  range $\frac{M}{2\pi}<8\pi$.
Moreover, the global existence is obtained  by combining quasi-linear space-time   
 estimates proposed by Wei--Zhang (Comm. Pure Appl. Math., 2021) and a free-energy argument as in Bedrossian--He (SIAM J. Math. Anal., 2017)  for the  zero mode. 
 The blow-up result follows from an exact
 invariant zero-mode reduction to the two-dimensional
Patlak--Keller--Segel system and a localized virial argument.
\end{abstract}

\maketitle

{\small {\bf Keywords:} 	Blow-up suppression;	Patlak--Keller--Segel--Navier--Stokes system; stability; 
Couette flow}
\tableofcontents

\parskip5pt
\parindent=1.5em
\section{Introduction}

Consider the following parabolic-elliptic Patlak--Keller--Segel (PKS) system
\begin{equation}\label{eq:ks}
\left\{
\begin{array}{lr}
	\partial_tn=\Delta n-\nabla\cdot(n\nabla c), \\
	\Delta c+n=0,
\end{array}
\right.
\end{equation}
which is designed to depict the diffusion and chemotactic motion of chemical substances within a population of cells or microorganisms.
Patlak made significant contributions in \cite{Patlak1}, and later, Keller and Segel further advanced it in \cite{Keller1}. Extensive applications have been found across diverse scientific fields, including biology, ecology, and medicine \cite{HP1}.
In the realm of biology, this system provides crucial insights into the complex behavior of cells, such as their migration, aggregation, and diffusion \cite{HT1}.

{
The behavior of the PKS system depends strongly on dimension and on the
spatial setting. For the parabolic--elliptic system on $\mathbb R^2$,
the mass $8\pi$ separates the subcritical and supercritical regimes.
Subcritical global existence was established by
Blanchet--Dolbeault--Perthame \cite{BDP2006}; see also Wei \cite{wei11}
for global well-posedness and blow-up results under the corresponding
assumptions on the initial data. At critical mass, infinite-time
concentration is possible; see \cite{BCM2008,DDDMW}.
Constructions and stability results for finite-time blow-up in the
parabolic--elliptic system are given in \cite{BDDM2023,CGMN2022}.
For the parabolic--parabolic system on $\mathbb R^2$, Calvez--Corrias
\cite{Calvez1} proved global existence below the mass $8\pi$ under
their hypotheses on the initial data. Schweyer \cite{Schweyer1}
constructed stable finite-time blow-up dynamics for suitable radial
data. These blow-up constructions do not assert finite-time blow-up
for every supercritical initial datum.
}

When the spatial dimension is higher than two, the PKS system \eqref{eq:ks} becomes supercritical.
In this condition, regardless of  parabolic-elliptic form or parabolic-parabolic form, the finite-time blow-up may occur for arbitrarily
small values of the initial mass.
In this time, the solution with any initial mass may blow up in finite time. This behavior has been established in various settings: for the parabolic-elliptic case by Nagai \cite{Na2000} and Souplet--Winkler \cite{SW2019}, and for the parabolic-parabolic case by Winkler \cite{winkler1}. Additional results and further developments on this topic can be found in \cite{BCM2008,DDDMW,TW2017} and the related therein.

Generally, the processes involving chemical attraction take place in fluids. As said in \cite{Kiselev1}:
{\it A natural question is whether the presence of fluid flow
can affect singularity formation by mixing the bacteria thus making concentration
harder to achieve.}

In this paper, we investigate the suppression of blow-up and the critical mass threshold of the following three-dimensional parabolic-elliptic Patlak--Keller--Segel (PKS) system coupled with Navier--Stokes (NS) equations in $\mathbb{T}^{3}$ with $ \mathbb{T}=[0,2\pi] $:
\begin{equation}\label{ini}
\left\{
\begin{array}{lr}
	\partial_tn+v\cdot\nabla n=\Delta n-\nabla\cdot(n\nabla c), \\
	\Delta c+n-\bar{n}=0, \\
	\partial_tv+v\cdot\nabla v+\nabla P=\Delta v+n\nabla \phi, \\
	\nabla\cdot v=0,
\end{array}
\right.
\end{equation}
along with initial conditions
$$(n,v)\big|_{t=0}=(n_{\rm in},v_{\rm in}),$$
where $n$ represents the cell density, $c$ denotes the chemoattractant density, and $v$ denotes the velocity of fluid. In addition, $\bar{n}=\frac{1}{|\mathbb{T}|^{3}}\int_{\mathbb{T}^{3}}ndxdydz$ denotes the average of $ n $, $P$ is the pressure and $\phi$ is the given potential function.

Let us by briefly review some results obtained in the 2D case, which implies the $8\pi$ critical mass threshold vanishes by the mixing effect of the fluid.
For the parabolic-elliptic PKS system of $(\ref{ini})_1$--$(\ref{ini})_2$, Kiselev--Xu \cite{Kiselev1} showed no critical mass threshold by suppressing the  blow-up by stationary relaxation
enhancing flows and time-dependent Yao--Zlatos near-optimal mixing flows in $\mathbb{T}^d$ with $d=2,3$. Bedrossian--He \cite{Bedro2} studied the suppression of blow-up by non-degenerate shear flows in
$\mathbb{T}^2$ for the 2D parabolic-elliptic case. He \cite{he0} investigated the suppression of blow-up for the parabolic-parabolic PKS model near the large strictly monotone shear flow in $\mathbb{T}\times\mathbb{R}$. For the coupled PKS--NS system, Zeng--Zhang--Zi \cite{zeng} first considered the 2D PKS--NS system near the Couette flow in $\mathbb{T}\times\mathbb{R}$, and
proved that if the Couette flow is sufficiently strong, the solution stays globally regular. He \cite{he05} considered the blow-up suppression for the parabolic-elliptic PKS--NS system in $\mathbb{T}\times\mathbb{R}$ with the coupling of buoyancy effects for a class of initial data with small vorticity.
Wang--Wang--Zhang \cite{Wanglili}  studied the blow-up suppression of 2D supercritical parabolic-elliptic PKS--NS system near the Couette flow in $\mathbb{T}\times\mathbb{R}.$
\textcolor[rgb]{0,0,0}{Li--Xiang--Xu \cite{Li0} suppressed the blow-up for the PKS--NS system via the Poiseuille flow  in $\mathbb{T}\times\mathbb{R},$
and proved that the solution is global as long as the  Poiseuille flow is strong enough.
Furthermore, Cui--Wang \cite{cui1} considered the blow-up suppression for the PKS--NS system in $\mathbb{T}\times \mathbb{I}$ with Navier--slip boundary condition.
In addition, Hu \cite{Hu2023} proved that sufficiently large buoyancy can also suppress the blow-up of the PKS system, see also the recent results by Hu--Kiselev--Yao \cite{Hu0} and Hu--Kiselev \cite{Hu1}.} The  blow-up also can be suppressed by adding some logistic terms, see \cite{TW2016} and the references therein.

In 3D, some results have  been made in the study of the corresponding problem.
Bedrossian--He \cite{Bedro2} studied the suppression of blow-up for the parabolic-elliptic PKS system by shear flows in $\mathbb{T}^3$ and $\mathbb{T}\times\mathbb{R}^2$.
Feng--Shi--Wang \cite{Feng1}  suppressed the blow-up for the advective Kuramoto--Sivashinsky and
the Keller--Segel equations via the planar helical flows. Shi--Wang \cite{wangweike2} considered the suppression effect of the Couette--Poiseuille flow $(z,z^2,0)$ in $\mathbb{T}^2\times\mathbb{R}$, and Deng--Shi--Wang \cite{wangweike1} proved the Couette flow with a sufficiently large amplitude prevents the blow-up of solutions in the whole space for exponential decay data. For the parabolic-parabolic PKS system, He \cite{he24-1} introduced a family of time-dependent alternating shear flows in the domain $\mathbb{T}^3$, and proved the solution remains globally regular as long as the flow is sufficiently strong. For a time-dependent shear flow, He \cite{he24-2} demonstrated that when the total
mass of the cell density is below a specific threshold ($8\pi|\mathbb{T}|$), the solution remains globally regular  in $\mathbb{T}^3$ as long as the
flow is sufficiently strong.

For the 3D PKS--NS system, fewer results are obtained. The first three authors \cite{CWW1} considered the PKS system coupled with
the linearized NS equations near the Couette flow $(Ay,0,0)$ with $A>0$ in  $\mathbb{T}\times\mathbb{I}\times\mathbb{T},$ and showed that
when $A$ is big enough the solutions are global in time as long as  $M< \frac{8\pi}{9}$ and $A(\|u_{2,\rm in}\|_{L^2}+\|u_{3,\rm in}\|_{L^2})\leq C.$
Then, the first three authors \cite{CWW2025} studied the blow-up suppression of the PKS--NS system for the Couette flow in $\mathbb{T}\times\mathbb{R}\times\mathbb{T},$ and proved that
the solutions are global in time as long as the initial cell mass $M< \frac{24}{5}\pi^2$ and initial velocity $A^{\frac13+}\|u_{\rm in}\|_{H^2}\leq C$.

%

Our main goal is to investigate the suppression of blow-up and obtain the optimal critical mass for the PKS--NS system (\ref{ini}) near the 3D Couette flow $( Ay,0,0 )$. Introduce a perturbation $u=(u_1,u_2,u_3)$ around the Couette flow $( Ay,0,0 )$, which $u(t,x,y,z)=v(t,x,y,z)-( Ay,0,0 )$ satisfying $u\big|_{t=0}=u_{\rm in}=(u_{1,\rm in}, u_{2,\rm in}, u_{3,\rm in})$. We prescribe the body force $ne_1$, where $e_1=(1,0,0)^T$ (the notation $\nabla\phi=e_1$ denotes this prescribed vector field),  then we rewrite the system (\ref{ini}) into
\begin{equation}\label{ini1}
\left\{
\begin{array}{lr}
	\partial_tn+Ay\partial_x n+u\cdot\nabla n-\Delta n=-\nabla\cdot(n\nabla c), \\
	\Delta c+n-\bar{n}=0, \\
	\partial_tu+Ay\partial_x u+Au_2 e_1
	-\Delta u+u\cdot\nabla u+\nabla P^{N_1}+\nabla P^{N_2}=ne_1, \\
	\nabla \cdot u=0,
\end{array}
\right.
\end{equation}
where the pressure $P^{N_1}$ and $P^{N_2}$ are determined by
\begin{equation}\label{pressure_1}
	\Delta P^{N_1}=-2A\partial_xu_2+\partial_{x}n,\quad
	\Delta P^{N_2}=-{\rm div}~(u\cdot\nabla u),\quad
    \int_{\mathbb{T}^3} P^{N_i}dxdydz=0,\quad i\in\{1,2\}.
\end{equation}
After the time rescaling $t\mapsto\frac{t}{A}$, we get
\begin{equation}\label{ini1'}
\left\{
\begin{array}{lr}
	\partial_tn+y\partial_x n+\frac{1}{A}u\cdot\nabla n-\frac{1}{A}\Delta n=-\frac{1}{A}\nabla\cdot(n\nabla c), \\
	\Delta c+n-\bar{n}=0, \\
	\partial_tu+y\partial_x u+u_{2}e_1
	-\frac{1}{A}\Delta u+\frac{1}{A}u\cdot\nabla u+\frac{1}{A}\nabla P^{N_1}+\frac{1}{A}\nabla P^{N_2}=\frac{n}{A}e_1, \\
	\nabla \cdot u=0.
\end{array}
\right.
\end{equation}
For the non-zero mode, the Couette flow exhibits an important effect known as enhanced dissipation \cite{BGM}. To  estimate non-zero norms conveniently, we introduce the Kelvin coordinates 
\begin{equation}\label{eq:kelvin-pullback-main}
X=x-ty,\qquad Y=y,\qquad Z=z,\qquad
F(t,X,Y,Z)=f(t,x,y,z).
\end{equation}
We impose $2\pi$-periodicity on the Kelvin coordinates. Equivalently, in the physical coordinates,
\[
f(t,x+2\pi,y,z)=f(t,x,y,z),\qquad
f(t,x+2\pi t,y+2\pi,z)=f(t,x,y,z),
\]
with ordinary periodicity in $z$. 

Introducing $\omega_2=\partial_z u_1-\partial_x u_3$ and $\Delta u_2$,
we infer from
\eqref{ini1'} that:
\begin{equation}\label{ini11}
\left\{
\begin{array}{lr}
	\partial_tn+y\partial_x n+\frac{1}{A}u\cdot\nabla n-\frac{1}{A}\Delta n=-\frac{1}{A}\nabla\cdot(n\nabla c), \\
	\Delta c+n-\bar{n}=0, \\
	\partial_{t}\omega_2+y\partial_{x}\omega_2-\frac{1}{A}\Delta\omega_2+\partial_{z}u_{2}=-\frac{1}{A}\partial_{z}(u\cdot\nabla u_{1})+\frac{1}{A}\partial_{x}(u\cdot\nabla u_{3})+\frac{1}{A}\partial_{z}n, \\
	\partial_{t}\Delta u_{2}+y\partial_{x}\Delta u_{2}-\frac{1}{A}\Delta(\Delta u_{2})=-\frac{1}{A}\partial_{y}\partial_{x}n-\frac{1}{A}(\partial_{x}^{2}+\partial_{z}^{2})(u\cdot\nabla u_{2})\\
	\qquad\qquad\qquad\qquad\qquad\qquad\qquad+\frac{1}{A}\partial_{y}\left[\partial_{x}(u\cdot\nabla u_{1})+\partial_{z}(u\cdot\nabla u_{3}) \right],\\
	{\nabla \cdot u=0.}
\end{array}
\right.
\end{equation}

Before stating the result, we define
\begin{equation*}\label{define:f0 fneq}
P_{0}f=f_{0}=\frac{1}{|\mathbb{T}|}\int_{\mathbb{T}}f(t,x,y,z)dx.
\end{equation*}

Our main results are stated as follows.
\begin{theorem}\label{result0}
	Assume that $ 0 \leq n_{\rm in}\in H^{2}(\mathbb{T}^{3})$ and $u_{\rm in}\in H^{2}(\mathbb{T}^{3})$ with $\nabla\cdot u_{\rm in}=0$. There exist a sufficiently small positive constant $\epsilon$ depending on
$\|n_{\rm in}\|_{H^2}$ and the subcritical mass gap $1-M/(16\pi^2)$,
and a positive constant $A_{1}$ depending on $\|n_{\rm in}\|_{ H^{2}(\mathbb{T}^{3})}$ and $\|u_{\rm in}\|_{H^{2}(\mathbb{T}^{3})}$, as well as the same mass gap,
such that if $A\geq A_{1}$,
	\begin{equation}\label{conditions:u20 u30}
		\|(u_{2,\rm in})_{0}\|_{H^{2}}
		+
		\|(u_{3,\rm in})_{0}\|_{H^{1}}
		\leq\epsilon,
	\end{equation}
	and
	\begin{equation}
		\label{assumption_mass}
		M
		=
		\int_{\mathbb{T}^{3}}
		n_{\rm in}dxdydz
		<16\pi^{2},
	\end{equation}
	then the solution of \eqref{ini1'}
	exists globally in time.
{

Moreover, let $I_{\rm in}=1+\|n_{\rm in}\|_{H^2}+\|u_{\rm in}\|_{H^2}$.
There exist constants $C_M>0$ and $C_{\rm in}>0$, independent of $A$ and $t$,
where $C_M$ depends only on $M$ and $C_{\rm in}$ may additionally depend
on the initial $H^2$ norms, such that
\[
\|n(t)\|_{L^\infty}\le C_{\rm in},\qquad
\|u(t)\|_{L^\infty}
\le C_Mt+C_{\rm in}+C_M A^{\frac16}I_{\rm in}^2
\quad (t>0).
\]
The dependence of $C_{\rm in}$ includes the entropy truncation constant
in Lemma \ref{lem:n0 L2}.

}

\end{theorem}

\begin{theorem}[Sharpness of the critical mass threshold]
	\label{thm:sharpness}
	Let $A>0$ and $M>16\pi^{2}$. There exist
	\[
	0<n_{\rm in}\in C^{\infty}(\mathbb{T}^{3}),
	\qquad
	u_{\rm in}=0,
	\qquad
	\int_{\mathbb{T}^{3}}n_{\rm in} dxdydz=M,
	\]
	such that the corresponding solution  of system \eqref{ini1} blows up in finite time.
	More precisely, if $T^{*}$ denotes its maximal classical existence
	time, then
	\[
	T^{*}<\infty,
	\qquad
	\limsup_{t\to T^{*}}
	\|n(t)\|_{L^{\infty}(\mathbb{T}^{3})}
	=+\infty.
	\]
	Moreover, the initial density can be chosen to be independent of the
	   variable $x$, and the corresponding solution remains in the
	   zero-mode subspace
	\[
	n(t,x,y,z)=n_{0}(t,y,z),\qquad
	c(t,x,y,z)=c_{0}(t,y,z),\qquad
	u(t,x,y,z)=\bigl(u_{1,0}(t,y,z),0,0\bigr).
	\]
\end{theorem}

\begin{remark}[Origin and sharpness of the mass threshold]
    \label{rem:sharpness}
    Recall that $\mathbb{T}=[0,2\pi]$. Due to $n_0$ is the
     zero mode of the density, its two-dimensional mass is
    \[
        m
        :=
        \int_{\mathbb{T}^2} n_0(t,y,z)dydz
        =
        \frac{M}{|\mathbb{T}|}
        =
        \frac{M}{2\pi}.
    \]
    Consequently, the condition $M<16\pi^2$ is equivalent to
    $m<8\pi$, which is precisely the subcritical mass condition for
    the two-dimensional parabolic-elliptic Patlak--Keller--Segel system.
    Combining Theorems \ref{result0} and \ref{thm:sharpness}, the
    threshold $16\pi^2$ is sharp. More precisely, under the
    assumptions of Theorem \ref{result0}, a sufficiently strong Couette
    flow yields global regularity whenever $M<16\pi^2$. On the other
    hand, for every finite $A>0$ and every $M>16\pi^2$, there exist
    smooth positive initial densities of mass $M$, with zero initial
    velocity perturbation, whose corresponding solutions blow up in
    finite time.
    The blow-up solutions constructed in Theorem
    \ref{thm:sharpness} remain in an invariant  zero-mode
    subspace, on which the density equation reduces exactly to the
    two-dimensional parabolic-elliptic Patlak--Keller--Segel system. In this reduction, the factor $A^{-1}$ changes only the time scale and does
    not alter the critical mass. 
\end{remark}

\begin{remark}[Dissipation of the zero-mode velocity]
    \label{rem:transverse-decay}
    A key new ingredient in the proof of Theorem \ref{result0} is the
    dissipative decay of the nonconstant part of the transverse
    zero-mode velocity. More precisely, writing
    \[
        u_{j,0}(t,y,z)
        =
        \overline{u}_{j,0}(t)
        +
        \widetilde{u}_{j,0}(t,y,z),
        \qquad
        \int_{\mathbb{T}^2}
        \widetilde{u}_{j,0}(t,y,z)dydz=0,
        \qquad j\in\{2,3\},
    \]
    the nonconstant components
    $(\widetilde{u}_{2,0},\widetilde{u}_{3,0})$ satisfy a
    heat-dissipative estimate which implies
    \[
        \frac{1}{A}
        \int_0^\infty
        \left(
        \|\widetilde{u}_{2,0}(t)\|_{L^\infty(\mathbb{T}^2)}^2
        +
        \|\widetilde{u}_{3,0}(t)\|_{L^\infty(\mathbb{T}^2)}^2
        \right)dt
        <\infty.
    \]
    The constant parts
    $(\overline{u}_{2,0},\overline{u}_{3,0})$ make no contribution
    to the free-energy identity because of periodicity, while the
    contribution of
    $(\widetilde{u}_{2,0},\widetilde{u}_{3,0})$ becomes integrable in
    time by the estimate above. This permits the logarithmic
    Hardy--Littlewood--Sobolev argument to give a uniform
    $L\log L$ bound for $n_0$ throughout the full subcritical range
    $m<8\pi$. Combined with the truncation argument in Section \ref{sec4},
    this yields a uniform $L^2$ bound for $n_0$. It is precisely this
    mechanism that allows the mass condition to reach
    $M<16\pi^2$, rather than stopping at a smaller threshold.
\end{remark}

\begin{remark}[Lift-up effect and the growth of the velocity]
    \label{rem:velocity-growth}
    The  zero mode $u_{1,0}$ is affected by the
    three-dimensional lift-up effect and by the non-decaying mean
    forcing generated by the density. To separate these effects, we
    decompose
        $u_{1,0}=\mathbf{U}_1+\mathbf{U}_2,$
    where $\mathbf{U}_1$ is controlled through the  zero mode of the density, while $\mathbf{U}_2$ contains the potentially growing part of the
     velocity. Although $\mathbf{U}_2$ itself may grow linearly in
    time, its spatial and temporal variations satisfy
    \[
        \frac{1}{A}\|\Delta \mathbf{U}_2\|_{L^\infty_tH^2}
        +
        \frac{1}{A^{\frac32}}\|\nabla\Delta \mathbf{U}_2\|_{L^2_tH^2}
        +
        \|\partial_t\mathbf{U}_2\|_{L^\infty_tH^2}
        \leq C\epsilon.
    \]
    Consequently, for the modified shear profile
        $V(t,y,z)=y+\frac{\mathbf{U}_2(t,y,z)}{A},$
    one has
    $\partial_yV=1+O(\epsilon)$ and
    $\partial_zV=O(\epsilon)$ in the norms required by the
    space-time estimates. Thus, even though a spatially constant part
    of $\mathbf{U}_2$ may grow, the shear geometry remains a small
    time-dependent perturbation of the Couette geometry. This permits
    the use of the quasilinear framework and of the good derivative
    $\partial_z-\kappa\partial_y$, where
    $\kappa=\frac{\partial_zV}{\partial_yV}$.

\end{remark}

\begin{remark}
In our pervious work \cite{CWW2025}, the first three authors investigated the suppression of blow-up and the nonlinear stability of the PKS--NS system via the 3D Couette flow
in a different domain of $\mathbb{T}\times\mathbb{R}\times\mathbb{T}$ and $\phi=y$. When $A$ is big enough, as long as
$$\left\{
\begin{array}{lr}
	A^{\frac{1}{3}+}(\|u_{\rm in}\|_{H^2(\mathbb{T}\times\mathbb{R}\times\mathbb{T})}
	+\|(n_{\rm in})_{(0,\neq)}\|_{L^2(\mathbb{T}\times\mathbb{R}\times\mathbb{T})})\leq C, \\
	M = \int_{\mathbb{T}\times\mathbb{R}\times\mathbb{T}}n_{\rm in}dxdydz < \frac{24}{5}\pi^2,
\end{array}
\right.$$
the solutions for the PKS--NS system are global in time, where a new energy estimate for  $\frac{\partial_z}{\sqrt{1-\Delta}}n_{(0,\neq)}$ was introduced to control the density. In the present compact transverse geometry, the heat dissipation of $(\widetilde u_{2,0},\widetilde u_{3,0})$ allows the zero-mode free-energy estimate to be used throughout the full subcritical mass range. The domains and forcing directions in the two results are different.
\end{remark}
\begin{remark}\label{rem:local-well-posedness}
Local well-posedness for the full system \eqref{ini1'} holds for the
divergence-free $H^2$ initial data in Theorem~\ref{result0}. In Kelvin
coordinates the operator is uniformly parabolic on every bounded time
interval, and standard parabolic estimates apply;
see, for example, \cite{CWW2025,Hu1,winkler1}. The solution is continuous
in $H^2$ at $t=0$ and smooth for positive times. A finite maximal time
$T_*$ can occur only if
\[
\limsup_{t\rightarrow T_*}
\bigl(\|n(t)\|_{L^\infty}+\|u(t)\|_{L^\infty}\bigr)=\infty.
\]
The equations \eqref{ini11} are used for a priori estimates, together
with the zero-mode equations derived from \eqref{ini1'}. We omit the
standard local construction and continuation argument.
\end{remark}
Here are some notations used in this paper.

\noindent\textbf{Notations}:
\begin{itemize}
\item The Kelvin Fourier expansion and its physical-coordinate form are
\begin{align*}
F(t,X,Y,Z)
&=\sum_{k_1,k_2,k_3\in\mathbb Z}
\widehat F_{k_1,k_2,k_3}(t){\rm e}^{i(k_1 X+k_2 Y+k_3 Z)},\\
f(t,x,y,z)
&=\sum_{k_1,k_2,k_3\in\mathbb Z}
\widehat F_{k_1,k_2,k_3}(t)
 {\rm e}^{i(k_1 x+(k_2-k_1 t)y+k_3 z)}.
\end{align*}
\item The zero mode and non-zero mode are
\[
P_0f=f_0=\frac{1}{2\pi}\int_0^{2\pi} f(t,x,y,z)dx,
\qquad
P_{\neq}f=f_{\neq}=f-f_0.
\]
They correspond to $k_1=0$ and $k_1\ne0$, respectively, in the Kelvin Fourier series.  Especially, we use $u_{j,0}$, and $u_{j,\neq}$ to represent the zero mode and non-zero mode of the velocity $u_{j} (j\in\{1,2,3\})$, respectively. Similarly, we use $\omega_{2,0}$ and $\omega_{2,\neq}$ to represent the zero mode and non-zero mode of the vorticity $\omega_2$, respectively.	
	
\item The norm of the $L^p$ space and the time-space norm $\|f\|_{L^{q}L^{p}}$ are defined as	
$\|f\|_{L^p(\mathbb{T}^{3})}=\left(\int_{\mathbb{T}^{3}}|f|^p dxdydz\right)^{\frac{1}{p}},$
and
$\|f\|_{L^qL^p}=\left\|  \|f\|_{L^p(\mathbb{T}^{3})}\ \right\|_{L^q(0,t)}.$
Moreover, $\langle\cdot,\cdot\rangle$ denotes the standard $L^2$ scalar product. For simplicity, we write $\|f\|_{L^p(\mathbb{T}^{3})}$ as $\|f\|_{L^p}.$

\item For $a>0,$ we define the norms
\begin{equation*}
	\begin{aligned}
		&	\|f\|_{X_{a}}^2
		=\|{\rm e}^{aA^{-\frac{1}{3}}t}f\|^2_{L^{\infty}L^{2}}
		+\|{\rm e}^{aA^{-\frac{1}{3}}t}\nabla\Delta^{-1}\partial_{x} f\|^2_{L^{2}L^{2}}
		+\frac{\|{\rm e}^{aA^{-\frac{1}{3}}t}f\|^2_{L^{2}L^{2}}}{A^{\frac{1}{3}}}
		+\frac{\|{\rm e}^{aA^{-\frac{1}{3}}t}\nabla f\|^2_{L^{2}L^{2}}}{A},
		\\
		&\|f\|_{Y_{0}}^2
		=\|f\|^2_{L^{\infty}L^{2}}
		+\frac{1}{A}\|\nabla f\|^2_{L^{2}L^{2}}.
	\end{aligned}
\end{equation*}		
\item  We sometimes denote the partial derivatives $\partial_x,$ $\partial_y$ and $\partial_z$ by
$\partial_1,$ $\partial_2$ and $\partial_3,$  respectively.
\item In this paper, unless otherwise specified, we use the Einstein summation convention.
\item The total mass $ \|n(t)\|_{L^{1}}$ is denoted by $ M $, and let $m:=\|n_{0}\|_{L^{1}}.$ Clearly,
\begin{equation}\label{m M}
	\begin{aligned}
		&M:=\|n(t)\|_{L^{1}}=\|n_{\rm in}\|_{L^{1}},\\
		&m:=\|(n_{\rm in})_{0}\|_{L^{1}}=\frac{\|n\|_{L^{1}}}{|\mathbb{T}|}=\frac{M}{|\mathbb{T}|}.
	\end{aligned}
\end{equation}


\item Throughout this paper, we denote $C$ by a positive constant independent of $A$, $t$ and the initial data, and it may be different from line to line.
\end{itemize}

\section{Key ideas and proofs of the main theorems}\label{sec2}
We first explain the main difficulties and the estimates used
in the proof, and then give the bootstrap argument for global existence.
\subsection{Main difficulties and proof strategy}
The Couette flow gives enhanced dissipation for the non-zero modes,
but it does not act on the zero-mode density $n_0$.
Thus, a separate estimate for $n_0$ is needed to prevent blow-up.
Another difficulty is that the 3D lift-up effect
and the forcing from the density can cause $u_{1,0}$ to grow in time.
We explain below how to control these terms without imposing
any $A$-dependent smallness condition on the initial non-zero modes.
\paragraph{$\bullet$ \bf Free-energy estimates for the zero-mode density.}
To control $n_0$, we use the two-dimensional free energy.
The difficulty is to estimate the additional terms involving
the zero-mode velocity and the non-zero modes.
For the velocity terms, we write
$u_{j,0}=\overline{u}_{j,0}+\widetilde u_{j,0}$, $j\in\{2,3\}$.
The spatial averages $\overline{u}_{2,0}$ and $\overline{u}_{3,0}$
do not decay, but their contributions to the free-energy identity
vanish by periodicity.
For $\widetilde u_{2,0}$ and $\widetilde u_{3,0}$,
we prove heat dissipation estimates.
The heat dissipation and enhanced dissipation estimates control the
free energy. The logarithmic Hardy--Littlewood--Sobolev inequality
then gives a uniform $L\log L$ bound for $n_0$ when $M<16\pi^2$,
and a truncation argument yields the uniform $L^2$ bound.
The velocity and density estimates are given in
Sections~\ref{sec3} and~\ref{sec4}.

\paragraph{$\bullet$ \bf The lift-up effect and the velocity estimates.}
The equation for $u_{1,0}$ contains both the lift-up term $u_{2,0}$
and the density forcing $\frac{n_0}{A}$.
These terms may lead to growth in time, so we cannot require
$u_{1,0}$ to remain uniformly small.
As described in Section~\ref{decomposition}, we decompose
$u_{1,0}=\mathbf{U}_1+\mathbf{U}_2$.
The $H^1$ norm of $\mathbf{U}_1$ is controlled by the initial data
and the $L^2$ bound for $n_0$.
The part $\mathbf{U}_2$ contains the lift-up contribution
and the forcing term.
Although $\mathbf{U}_2$ may grow in time, we control the spatial
derivatives of $\frac{\mathbf{U}_2}{A}$ and the $H^2$ norm of
$\partial_t\mathbf{U}_2$. These bounds enter the quasilinear
space-time analysis based on the method of Wei--Zhang \cite{wei2}.

\paragraph{$\bullet$ \bf Sharpness of the mass threshold.}
To prove sharpness, we choose an initial density independent of $x$
and take $u_{\rm in}=0$.
The solution then remains independent of $x$.
Although the density generates a nonzero $u_{1,0}$,
neither this velocity nor the Couette flow contributes
to the density equation, since $\partial_xn_0=0$.
Thus, the equations for $n_0$ and $c_0$ reduce exactly
to the two-dimensional parabolic--elliptic PKS system on $\mathbb T^2$.
For $M>16\pi^2$, we choose smooth positive initial data
concentrated near one point and use a localized virial argument
to prove finite-time blow-up.
This gives a blow-up solution of the full three-dimensional system
for every finite $A>0$ and every $M>16\pi^2$;
see Section~\ref{sharp1}.
Together with the global existence result for $M<16\pi^2$,
this proves the sharpness of the mass threshold.
\subsection{Modal dynamics and dissipation mechanisms}
Let $F(t,X,Y,Z)=f(t,x,y,z)$.
Its Fourier series is
\[
F(t,X,Y,Z)=\sum_{k_1,k_2,k_3\in\mathbb Z}
\widehat F_{k_1,k_2,k_3}(t){\rm e}^{i(k_1 X+k_2 Y+k_3 Z)},
\]
and therefore
\[
f(t,x,y,z)=\sum_{k_1,k_2,k_3\in\mathbb Z}
\widehat F_{k_1,k_2,k_3}(t){\rm e}^{i(k_1 x+(k_2-k_1 t)y+k_3 z)}.
\]
The decomposition $f=f_{\neq}+f_0$ is the decomposition into $k_1\ne0$ and
$k_1=0$.  In particular,
\begin{align*}
f_{\neq}(t,x,y,z)
=\sum_{k_1 \ne0,k_2,k_3}
\widehat F_{k_1,k_2,k_3}(t){\rm e}^{i(k_1x+(k_2-k_1 t)y+k_3 z)},\quad 
f_0(t,y,z)
=\sum_{k_2,k_3}\widehat F_{0,k_2,k_3}(t){\rm e}^{i(k_2 y+k_3 z)}.
\end{align*}
The zero mode is an ordinary periodic function of $(y,z)$.  We further write
$f_0=\overline f_0+\widetilde f_0$, where
\begin{equation}\label{def:check{f_0}}
\overline f_0(t)=\widehat F_{0,0,0}(t)
=\frac1{|\mathbb T|^3}\int_{\mathbb{T}^{3}}f(t,x,y,z)\,dxdydz,
\qquad
\widetilde f_0=\sum_{k_2^2+k_3^2\ne0}
\widehat F_{0,k_2,k_3}{\rm e}^{i(k_2 y+k_3 z)}.
\end{equation}
For the linear Couette equation
\[
\partial_tf+y\partial_xf-A^{-1}\Delta f=0,
\qquad f|_{t=0}=f_{\rm in},
\]
each non-zero mode satisfies
\[
\partial_t\widehat F_{k_1,k_2,k_3}
+A^{-1}(k_1^2+(k_2-k_1 t)^{2}+k_3^2)\widehat F_{k_1,k_2,k_3}=0.
\]
The non-zero mode $f_{\neq}$ experiences enhanced dissipation:
\begin{equation*}
\|f_{\neq}(t)\|_{L^2}\leq C{\rm e}^{-aA^{-\frac{1}{3}}t}
\|(f_{\rm in})_{\neq}\|_{L^2}.
\end{equation*}
The non-average part $\widetilde f_0$ experiences heat dissipation
with
\begin{equation*}
\|\widetilde{f}_{0}(t)\|_{L^2}\leq {\rm e}^{-\frac{t}{2A}}
\|\widetilde{(f_{\rm in})_0}\|_{L^2}.
\end{equation*}
The average part $\overline{f}_{0}(t)$ is a constant satisfying
$$\overline{f}_{0}(t)=\overline{(f_{\rm in})_0}.$$


{
\noindent
\textbf{Key observation 1. Different parts of the  zero-mode
velocity play different roles in the free-energy estimate.}
For \(j\in\{2,3\}\), the zero-mode velocity can be decomposed as
\[
u_{j,0}(t,y,z)=\overline{u}_{j,0}(t)+
\widetilde{u}_{j,0}(t,y,z).
\]
The spatially constant part \(\overline{u}_{j,0}\) does not decay, but its
contribution to the free-energy identity vanishes by periodicity. In contrast, the nonconstant part \(\widetilde{u}_{j,0}\) enjoys heat dissipation. Thus, there is no need to obtain a uniform decay estimate for the whole transverse
zero-mode velocity; it is enough to exploit the cancellation of its constant
part and the dissipation of its nonconstant part.
This observation will be used in Section \ref{sec4} to obtain the free-energy estimate for \(n_0\).}

\subsection{Decomposition of the streamwise zero-mode velocity}\label{decomposition}
The first component $u_{1,0}$ is affected by the 3D lift-up effect and is the worst part (see Section 2.1 in \cite{Chen1}). Therefore, we must handle it more carefully and meticulously.

Recall that $u_{1,0}$ satisfies
$$\partial_tu_{1,0}-\frac{1}{A}\Delta u_{1,0}
+u_{2,0}=-\frac{u_{2,0}\partial_yu_{1,0}
+u_{3,0}\partial_zu_{1,0}}{A}
-\frac{(u_{\neq}\cdot\nabla u_{1,\neq})_0}{A}+\frac{n_0}{A}.$$

{
\noindent
\textbf{Key observation 2. The potentially growing part of \(u_{1,0}\)
should be incorporated into the background shear rather than estimated as a small perturbation.}
Instead of requiring the whole  zero-mode velocity to remain
uniformly bounded, we separate \(u_{1,0}\) into a uniformly controlled part and a potentially growing part. The latter will be absorbed into a modified
time-dependent shear profile. Therefore, what needs to remain small is not
this component itself, but its spatial and temporal variations.}
Motivated by this observation, we first decompose $u_{1,0}$ into
$u_{1,0}(t,y,z)=\mathbf{G}_1(t,y,z)+\mathbf{B}_1(t,y,z)+\mathbf{B}_2(t,y,z),$ satisfying
\begin{equation}\label{u_decom_1}
\begin{aligned}
	&\partial_t\mathbf{G}_1-\frac{1}{A}\Delta \mathbf{G}_1=
	-\frac{u_{2,0}\partial_y\mathbf{G}_1
		+u_{3,0}\partial_z\mathbf{G}_1}{A}-
	\frac{(u_{\neq}\cdot\nabla u_{1,\neq})_0}{A},\\
	&\partial_t\mathbf{B}_1-\frac{1}{A}\Delta \mathbf{B}_1
	=-\frac{u_{2,0}\partial_y\mathbf{B}_1
		+u_{3,0}\partial_z\mathbf{B}_1}{A}+\frac{n_0}{A},\\
	&\partial_t\mathbf{B}_2-\frac{1}{A}\Delta \mathbf{B}_2
	+u_{2,0}=-\frac{u_{2,0}\partial_y\mathbf{B}_2
		+u_{3,0}\partial_z\mathbf{B}_2}{A},
\end{aligned}
\end{equation}
along with the initial conditions
\begin{equation*}
\mathbf{G}_1|_{t=0}=(u_{1,\rm in})_0,\qquad\mathbf{B}_1|_{t=0}=0,\qquad\mathbf{B}_2|_{t=0}=0.
\end{equation*}
In this way, $\mathbf{G}_1(t,y,z)$ is a good term
satisfying the following energy estimate
\begin{equation*}
\|\mathbf{G}_1\|^2_{L^{\infty}H^1}+\frac{1}{A}\|\nabla \mathbf{G}_1\|_{L^2H^1}^2\leq C(\|(u_{1,\rm in})_0\|_{H^1}^2+\epsilon^2).
\end{equation*}
We further decompose $\mathbf{B}_1(t,y,z)$ and $\mathbf{B}_2(t,y,z)$ into their spatial averages on $\mathbb T^2$ and mean-zero remainders
\begin{equation*}\label{u1:decom}
\begin{aligned}
	\mathbf{B}_1(t,y,z)=\overline{\mathbf{B}}_{1}(t)+\widetilde{\mathbf{B}}_{1}(t,y,z),\\
	\mathbf{B}_2(t,y,z)=\overline{\mathbf{B}}_{2}(t)+\widetilde{\mathbf{B}}_{2}(t,y,z),\\
\end{aligned}
\end{equation*}
satisfying
\begin{equation}\label{u1:decom1}
\left\{
\begin{array}{lr}
	\partial_t\widetilde{\mathbf{B}}_{1}(t,y,z)-\frac{1}{A}
	\Delta \widetilde{\mathbf{B}}_{1}(t,y,z)
	=-\frac{1}{A}(u_{2,0}\partial_y\widetilde{\mathbf{B}}_{1}
	+u_{3,0}\partial_z\widetilde{\mathbf{B}}_{1})+\frac{1}{A}\widetilde{n}_{0}(t,y,z),\\
	\partial_t\overline{\mathbf{B}}_{1}(t)=\frac{1}{A}\overline{n}_{0}=\frac{M}{A|\mathbb{T}|^3},
\end{array}
\right.
\end{equation}
and
\begin{equation}\label{u1:decom2}
\left\{
\begin{array}{lr}
	\partial_t\widetilde{\mathbf{B}}_{2}(t,y,z)-\frac{1}{A}
	\Delta \widetilde{\mathbf{B}}_{2}(t,y,z)+\widetilde{u}_{2,0}(t,y,z)=-
	\frac{1}{A}(u_{2,0}\partial_y\widetilde{\mathbf{B}}_{2}
	+u_{3,0}\partial_z\widetilde{\mathbf{B}}_{2}),\\
	\partial_t\overline{\mathbf{B}}_{2}(t)+\overline{u}_{2,0}(t)=0.
\end{array}
\right.
\end{equation}
Lastly, we rewrite above decompositions into
\begin{equation*}
\begin{aligned}
	&\mathbf{U}_1(t,y,z)=\mathbf{G}_1(t,y,z)+\widetilde{\mathbf{B}}_{1}(t,y,z),\\
	&\mathbf{U}_2(t,y,z)=\widetilde{\mathbf{B}}_{2}(t,y,z)+\overline{\mathbf{B}}_{2}(t)
	+\overline{\mathbf{B}}_{1}(t),
\end{aligned}
\end{equation*}
satisfying $\mathbf{U}_1(t,y,z)+\mathbf{U}_2(t,y,z)=u_{1,0}(t,y,z).$

In this way, we decompose the velocity $u_{1,0}$ into the good part $\mathbf{U}_1(t,y,z)$ and the  bad part $\mathbf{U}_2(t,y,z)$ satisfying
\begin{equation*}
\left\{
\begin{array}{lr}
	\partial_t\mathbf{U}_1-\frac{1}{A}
	\Delta \mathbf{U}_1
	=-\frac{1}{A}(u_{2,0}\partial_y\mathbf{U}_1
	+u_{3,0}\partial_z\mathbf{U}_1)+\frac{1}{A}\widetilde{n}_{0}(t,y,z)
	-\frac{(u_{\neq}\cdot\nabla u_{1,\neq})_0}{A},\\
	\mathbf{U}_1|_{t=0}=(u_{1,\rm in})_0,
\end{array}
\right.
\end{equation*}
and
\begin{equation*}
\begin{cases}
	\partial_t\mathbf{U}_2-\frac{1}{A}
	\Delta \mathbf{U}_2+{u}_{2,0}=-
	\frac{1}{A}(u_{2,0}\partial_y\mathbf{U}_2
	+u_{3,0}\partial_z\mathbf{U}_2)+\frac{M}{A|\mathbb{T}|^{3}},\\
	\mathbf{U}_2|_{t=0}=0.
\end{cases}
\end{equation*}
For the good velocity $\mathbf{U}_1,$
the $H^1$ estimate suffices for the main nonlinear bootstrap; an additional transverse derivative estimate is used for the final $L^\infty$ bound,
and it has a good relation with the cell density $n_0$
by $$\|\mathbf{U}_1\|_{L^{\infty}H^1}\leq C(\|(u_{1,\rm in})_0\|_{H^1}
+\|n_0\|_{L^{\infty}L^2}+\epsilon).$$
For the bad velocity $\mathbf{U}_2$, when it satisfies
$\frac{\|\Delta \mathbf{U}_2\|_{L^{\infty}H^2}}{A}
+\|\partial_t \mathbf{U}_2\|_{L^{\infty}L^{\infty}}<c$
for some small $c$, Proposition~\ref{Lvf 0} gives the scalar space-time
estimate for the nonzero streamwise modes. The coupled and higher-order
estimates use \eqref{condition U2 high}, which follows from
Lemma~\ref{u1_hat2}. For the coupled estimate, the additional fact
$\mathbf U_2(0)=0$ and the weight gap $b>a$ control a remainder by
$C\epsilon^2A^{-1/3}E_4(t)^2$; see Proposition~\ref{timespace1} and
\eqref{omega:1}.
Here,  $\mathcal{L}_V$ and  $\mathcal{L}$
can be found in \eqref{ope1}.
\subsection{The construction of energy functional}\label{2.4}
Firstly, we introduce the energy functional
\begin{align*}
	&E_{1,1}(t)=\|u_{2,0}\|_{Y_{0}}+\|u_{3,0}\|_{Y_{0}}+\|\nabla u_{2,0}\|_{Y_{0}}+\|\nabla u_{3,0}\|_{Y_{0}}+\|\Delta u_{2,0}\|_{Y_{0}}\\&\qquad\qquad+\|\min\{(A^{-\frac23}+A^{-1}t)^{\frac12}, 1\}\Delta u_{3,0}\|_{Y_{0}},
	\\&E_{1,2}(t)= A^{-1}(\|\Delta\mathbf{U}_2\|_{L^{\infty}H^2}
	+A^{-\frac12}{\|\nabla\Delta\mathbf{U}_2\|_{L^{2}H^2}})
	+\|\partial_t\mathbf{U}_2\|_{L^{\infty}H^2},
	\\&E_{2,1}(t)=\|\partial_{x}^{2}n_{\neq}\|_{X_{b}},
	\\&E_{2,2}(t)=\|\Delta u_{2,\neq}\|_{X_{a}}+\|\partial_{x}\omega_{2,\neq}\|_{X_{a}}+A^{-\frac13}\left(\|\partial_{y}\omega_{2,\neq}\|_{X_{a}}+\|\partial_{z}\omega_{2,\neq}\|_{X_{a}} \right),
	\\&E_{3}(t)=\|n\|_{L^{\infty}L^{\infty}},
	\\&E_{4}(t)=\|\partial_{x}^{2}u_{2,\neq}\|_{X_{b}}+\|\partial_{x}^{2}u_{3,\neq}\|_{X_{b}},
\end{align*}
where $a$ and $b$ are fixed positive constants with
$0<a<b<\min\{2a,a_*\}$, and $a_*$ is the absolute constant in
Proposition~\ref{timespace1}. Moreover, we denote
\begin{equation*}
E_{1}(t)=E_{1,1}(t)+E_{1,2}(t),\quad E_{2}(t)=E_{2,1}(t)+E_{2,2}(t).
\end{equation*}
To close the energy, when estimating the nonlinear interactions between the non-zero modes of the velocity of $\|e^{2aA^{-\frac13}t}\nabla(u_{\neq}\cdot\nabla u_{3,\neq})\|_{L^{2}L^{2}}$, it is required that the power of $A$ must be strictly less than $\frac53$ (see Lemma \ref{lemma_neq1}). To this end, we introduce the first auxiliary energy
\begin{equation*}
E_{5,1}(t)=A^{-\frac23}\|\Delta u_{3,\neq}\|_{X_{b}}.
\end{equation*}
Besides, to estimate $E_{4}$ and handle the nonlinear interaction between the bad component $\mathbf{U}_{2}$ of $u_{1,0}$ and the non-zero mode $u_{\neq},$ we also need the second auxiliary energy
\begin{equation*}
\begin{aligned}
	E_{5,2}(t)=\sum_{j=2}^{3}\left(\|\partial_{x}^{2}u_{j,\neq}\|_{X_{b}}+\|\partial_{x}(\partial_{z}-\kappa\partial_{y})u_{j,\neq}\|_{X_{b}} \right)+\|\partial_{x}\nabla W\|_{X_{b}},		
\end{aligned}
\end{equation*}
where the definition of $\kappa$ and the expression of $W$ can be found in Section \ref{sec 8}.

Let us briefly describe the roles of the above-mentioned energy functionals.
\begin{itemize}
\item $E_1(t)$ is introduced to control the zero modes of the velocity $u_0.$
$E_{1,1}(t)$ is used to control the good velocities $u_{2,0}$ and  $u_{3,0}.$
$E_{1,2}(t)$ is used to control the bad part $\mathbf{U}_2$ of  the velocity $u_{1,0}.$
\item   $E_2(t)$ is used to  deal with the non-zero mode of the cell density $n_{\neq}$ and the velocity $u_{\neq}.$
\item   $E_3(t)$ is introduced to  estimate $\|n_0(t)\|_{L^2}$
and deal with the nonlinear interaction items in the density $n$.
\item   $E_4(t)$ is introduced to deal with the nonlinear interaction items with 3D lift-up effect in $u_{\neq}$.
\end{itemize}

\subsection{Bootstrap argument}
\begin{proof}[Proof of Theorem \ref{result0}]
We prove Theorem \ref{result0} in the following three steps.

\textbf{Step~1:} Let $T$ be the terminal point of the largest range $[0, T]$ such that the following hypothesis holds
\begin{equation}\label{assumption}
	E_{1}(t)\leq 2E_{1},\quad E_{2}(t)\leq 2E_{2},\quad E_{3}(t)\leq 2E_{3},\quad E_{4}(t)\leq 2E_{4},\quad E_{5,1}(t)\leq 2E_{5}
\end{equation}
for any $t\in[0,T],$ where $E_{1}, E_{2}, E_{3}, E_{4}$ and $E_{5}$ are constants independent of $t$ and $A$, and are fixed before choosing $A$.
{Take $E_4=E_5$ to be a sufficiently large multiple of
$1+\|(u_{\rm in})_{\neq}\|_{H^2}$, as permitted by
Lemma~\ref{E5 E6}. The bounds $E_2$ and $E_3$ are then chosen as in
Corollary~\ref{cor:E2} and \eqref{eq:E3}, with $H_1$ as in
Lemma~\ref{lem:n0 L2}. Initial values are strictly below the bootstrap
levels.}

\textbf{Step~2:} We need to prove the following propositions:
\begin{proposition}\label{prop:E0}
	Assume that the initial data $(n_{\rm in}, u_{\rm in})$ satisfy the assumptions of Theorem \ref{result0} and  \eqref{assumption}, there exists a positive constant $A_{2}$ independent of $A$ and $t,$ such that if $A\geq A_{2},$ for all $t\in[0,T],$ there holds
	\begin{equation*}
		E_{1}(t)\leq E_{1}.
	\end{equation*}	
\end{proposition}
\begin{proposition}\label{prop:E3}
	Assume that the initial data $(n_{\rm in}, u_{\rm in})$ satisfy the assumptions of Theorem \ref{result0} and  \eqref{assumption}, there exists a positive constant $A_{5}$ independent of $A$ and $t,$ such that if $A\geq A_{5},$ for all $t\in[0,T],$ there holds
	\begin{equation*}
		E_{3}(t)\leq E_{3}.
	\end{equation*}
	
\end{proposition}
\begin{proposition}\label{prop:E5}
	Assume that the initial data $(n_{\rm in}, u_{\rm in})$ satisfy the assumptions of Theorem \ref{result0} and  \eqref{assumption}, there exists a positive constant $A_{7}$ independent of $A$ and $t,$ such that if $A\geq A_{7},$ for all $t\in[0,T],$ there holds
	\begin{equation*}
		\begin{aligned}
			E_{4}(t)\leq E_{4},\quad E_{5,1}(t)\leq E_{5},\quad E_{2}(t)\leq E_{2},
		\end{aligned}
	\end{equation*}
\end{proposition}

\textbf{Step 3:} Choose the initial-data-dependent constants above
and $\epsilon$ first, and then take
$A_1\ge\max\{A_2,A_5,A_7,A_8,1\}$. The propositions strictly improve
all the bootstrap bounds. The finite-time velocity bound in
Corollary~\ref{corollary_velocity_Linfty} and the continuation criterion
in Remark~\ref{rem:local-well-posedness} rule out a finite maximal time.
The density bound and the velocity bound in Theorem~\ref{result0}
follow from \eqref{eq:E3} and
Corollary~\ref{corollary_velocity_Linfty}, with the dependence on $H_1$ retained in $C_{\rm in}$.
\end{proof}
Theorem~\ref{thm:sharpness} is proved in Section~\ref{sharp1}.

\section{A priori estimates}
At each fixed time, all functions with non-zero mode below admit the Kelvin Fourier expansion
$\sum_{k_1\ne0,k_2,k_3}\widehat F_{k_1,k_2,k_3}(t)
 {\rm e}^{i(k_1x+(k_2-k_1 t)y+k_3 z)}.$
The divergence-vorticity identities are algebraic identities for
the symbol $(k_1,k_2-k_1t,k_3)$. For Sobolev estimates we use the periodic
inequalities in Appendix~\ref{sec7}, including their lower-order terms.
Uniformity in time may also be seen by writing $t=N+\tau$, with
$N\in\mathbb Z$ and $0\le\tau<1$, and relabelling $k_2$ as $k_2-Nk_1$.
The remaining shear has bounded coefficients. Mixed norms are obtained
by the corresponding one-dimensional Hilbert-valued embeddings; the
$x$-translations at the $y$ boundary preserve the relevant $L^2_x$ norms.
On $k_1\ne0$, lower-order terms are controlled by streamwise Poincare
inequalities whenever an $x$ derivative is available.

Firstly, we give some relationships between velocity $u_{\neq}$ and the new vorticity $\omega_{2,\neq}$, which will be frequently used in the later calculations.
Since ${\rm div}~u_{\neq}=0$, the result follows immediately from Fourier series (see Lemma 3.13 in \cite{CWW2025}),  we omit the proof.
\begin{lemma}[]\label{lemma_u}
There holds
\begin{equation}\label{eq:velocity embed}
	\begin{aligned}
		&\left\|(\partial_x,
		\partial_z)\partial_xu_{\neq}\right\|_{L^2}\leq C(\|\partial_x\omega_{2,\neq}\|_{L^2}
		+\|\Delta u_{2,\neq}\|_{L^2}),\\
		&\left\|(\partial_x,
		\partial_z)\partial_yu_{\neq}\right\|_{L^2}\leq C(\|\partial_y\omega_{2,\neq}\|_{L^2}
		+\|\Delta u_{2,\neq}\|_{L^2}),\\
		&\left\|(\partial_x,
		\partial_z)\partial_zu_{\neq}\right\|_{L^2}\leq C(\|\partial_z\omega_{2,\neq}\|_{L^2}
		+\|\Delta u_{2,\neq}\|_{L^2}),\\
		&\left\|(\partial_x^2,
		\partial_z^2)u_{3,\neq}\right\|_{L^2}
		\leq C(\|\partial_x\omega_{2,\neq}\|_{L^2}
		+\|\Delta u_{2,\neq}\|_{L^2}),\\
		&\left\|(\partial_x,
		\partial_z)\partial_x\nabla u_{\neq}\right\|_{L^2}\leq C(\|\partial_x\nabla\omega_{2,\neq}\|_{L^2}
		+\|\nabla\Delta u_{2,\neq}\|_{L^2}),\\
		&\left\|(\partial_x,
		\partial_z)\partial_y\nabla u_{\neq}\right\|_{L^2}\leq C(\|\partial_y\nabla\omega_{2,\neq}\|_{L^2}
		+\|\nabla\Delta u_{2,\neq}\|_{L^2}).
	\end{aligned}
\end{equation}
\end{lemma}

Secondly, we present the nonlinear interactions between the non-zero modes of the velocity, which will often be used to estimate $E_1(t)$, $E_2(t)$ and $E_4(t)$.
\begin{lemma}\label{lemma_neq1}
It holds that
{\small
		\begin{align}
			\|{\rm e}^{2aA^{-\frac{1}{3}}t}|u_{\neq}|^2\|_{L^2L^2}^2&\leq
			CA^{\frac{2}{3}}E_2^4,\nonumber\\
			\|{\rm e}^{2aA^{-\frac{1}{3}}t}u_{\neq}\cdot\nabla u_{\neq}\|_{L^2L^2}^2&\leq
			CA^{\frac23}E_2^4,\nonumber\\
			\|{\rm e}^{2aA^{-\frac{1}{3}}t}\partial_x(u_{\neq}\cdot\nabla u_{\neq})\|^2_{L^2L^2}
			&\leq CA^{\frac16+\alpha}E_2^4,\nonumber\\		
			\|{\rm e}^{2aA^{-\frac{1}{3}}t}\nabla(u_{\neq}\cdot\nabla u_{2,\neq})\|^2_{L^2L^2}
			&\leq CA^{\frac{1}{2}+\frac{2}{3}\alpha}E_2^4, \label{eq:non-neq0}\\	
			\|{\rm e}^{2aA^{-\frac{1}{3}}t}\partial_z(u_{\neq}\cdot\nabla u_{3,\neq})\|_{L^2L^2}^2	
			&\leq CA^{\frac{1}{2}+\frac{2}{3}\alpha}E_2^4,\nonumber\\		
			\|{\rm e}^{2aA^{-\frac{1}{3}}t}\partial_z(u_{\neq}\cdot\nabla u_{1,\neq})\|^2_{L^2L^2}
			&\leq  CA^{\frac43}E_2^4,\nonumber\\
			\|{\rm e}^{2aA^{-\frac{1}{3}}t}\nabla(u_{\neq}\cdot\nabla u_{3,\neq})\|^2_{L^2L^2}
			&\leq   C(A^{\frac76+\frac{1}{3}\alpha }E_2^2E_{5}^2+A^{\frac43} E_2^4),\nonumber
		\end{align}}
where $\alpha$ is a constant with $\alpha\in(\frac12, \frac34).$
\end{lemma}
\begin{proof}
For convenience, we denote by
\begin{equation*}
	\begin{aligned}
		&\Gamma_1=\|\partial_x\omega_{2,\neq}\|_{L^2}+\|\Delta u_{2,\neq}\|_{L^2},\quad 
		\Gamma_2=\|\nabla\omega_{2,\neq}\|_{L^2}+\|\Delta u_{2,\neq}\|_{L^2},\\
		&\Gamma_3=\|\partial_x\nabla\omega_{2,\neq}\|_{L^2}+\|\nabla\Delta u_{2,\neq}\|_{L^2},
        \quad\Gamma_4=\|\partial_x\nabla u_{2,\neq}\|_{L^2}.
	\end{aligned}
\end{equation*}
According to the definition of  $E_2(t)$ and assumptions (\ref{assumption}), there holds
\begin{equation*}
	\|{\rm e}^{aA^{-\frac{1}{3}}t}\Gamma_1\|_{L^{\infty}}+\frac{\|{\rm e}^{aA^{-\frac{1}{3}}t}\Gamma_1\|_{L^{2}}}{A^{\frac16}}
	+\frac{\|{\rm e}^{aA^{-\frac{1}{3}}t}(\Gamma_2,\Gamma_3)\|_{L^{2}}}{A^{\frac12}}+\frac{\|{\rm e}^{aA^{-\frac{1}{3}}t}\Gamma_2\|_{L^{\infty}}}{A^{\frac13}}
	+\|{\rm e}^{aA^{-\frac{1}{3}}t}\Gamma_4\|_{L^{2}}\leq CE_2,
\end{equation*}
which will be frequently used in later calculations.

{\bf Estimate of $(\ref{eq:non-neq0})_{1}.$}
Using $\eqref{sob_result_2}_5$ in Lemma \ref{sob_inf_2} and $\eqref{eq:velocity embed}_1$ in Lemma \ref{lemma_u}, we have
\begin{align*}
	\|u_{\neq}\|_{L^{\infty}_{x,z}L^2_y}^2
	\leq C\big(\|\partial_xu_{\neq}\|^{2\alpha}_{L^2}
	\|u_{\neq}\|^{2-2\alpha}_{L^2}
	+\|\partial_x\partial_zu_{\neq}\|^{2\alpha}_{L^2}
	\|u_{\neq}\|^{2-2\alpha}_{L^2}\big)
	\leq C\Gamma_1^2.
\end{align*}
Therefore, by $\eqref{sob_result_2}_8$ and $\eqref{eq:velocity embed}_2,$ we arrive
\begin{equation*}
	\begin{aligned}
		\||u_{\neq}|^2\|_{L^2}^2\leq
		C\big(\|\partial_xu_{\neq}\|^{2\alpha}_{L^2}
		\|u_{\neq}\|^{2-2\alpha}_{L^2}
		+\|\partial_x\partial_zu_{\neq}\|^{2\alpha}_{L^2}
		\|u_{\neq}\|^{2-2\alpha}_{L^2}\big)\|\nabla u_{\neq}\|_{L^2}\|u_{\neq}\|_{L^2}
		\leq C\Gamma_1^3\Gamma_2,
	\end{aligned}
\end{equation*}
which indicates that
$$\|{\rm e}^{2aA^{-\frac{1}{3}}t}|u_{\neq}|^2\|_{L^2L^2}^2\leq CA^{\frac{2}{3}}(\|\partial_x\omega_{2,\neq}\|^4_{X_a}+\|\Delta u_{2,\neq}\|^4_{X_a})\leq CA^{\frac23} E_2^4. $$	

{\bf Estimate of $(\ref{eq:non-neq0})_{2}.$}
Using Lemma \ref{sob_inf_2} and Lemma \ref{lemma_u}  again, there hold
\begin{equation*}\label{appa_1}
	\begin{aligned}
		&\|u_{2,\neq}\|_{L^{\infty}_{x,y}L^2_z}^2
		\leq C\|\partial_x u_{2,\neq}\|_{L^2}\|\partial_x\nabla u_{2,\neq}\|_{L^2}\leq C\Gamma_4^2,\\
		&\|u_{\neq}\|_{L^{\infty}_{x,z}L^2_y}^2
		\leq C\big(\|\partial_xu_{\neq}\|^{2}_{L^2}
		+\|(\partial_x,\partial_z)\partial_xu_{\neq}\|_{L^2}^{2}
		\big)\leq C\Gamma_1^2,\\
		&\|\partial_y u_{\neq}\|_{L^2_{y,z}L^{\infty}_{x}}^2+\|\partial_y u_{\neq}\|_{L^2_{x,y}L^{\infty}_{z}}^2
		\leq C\Gamma_2^2,\\
		&\|(\partial_x,\partial_z)u_{\neq}\|_{L^2_{x,z}L^{\infty}_{y}}^2
		\leq C\Gamma_1\Gamma_2.	
	\end{aligned}
\end{equation*}
Therefore, one gets
	\begin{align*}
		\|u_{\neq}\cdot\nabla u_{\neq}\|^2_{L^2}
		\leq& 3\big(\|u_{1,\neq}\partial_x u_{\neq}\|_{L^2}^2+
		\|u_{2,\neq}\partial_y u_{\neq}\|_{L^2}^2+
		\|u_{3,\neq}\partial_z u_{\neq}\|_{L^2}^2\big)\\
		\leq& C\big(\|u_{\neq}\|_{L^{\infty}_{x,z}L^2_y}^2\|(\partial_x,\partial_z)u_{\neq}\|_{L^2_{x,z}L^{\infty}_{y}}^2+
		\|u_{2,\neq}\|_{L^{\infty}_{x,y}L^2_z}^2
		\|\partial_y u_{\neq}\|_{L^2_{x,y} L^{\infty}_{z}}^2\big)
		\leq C(\Gamma_1^3\Gamma_2+\Gamma_2^2\Gamma_4^2),
	\end{align*}
which implies that
$$\|{\rm e}^{2aA^{-\frac{1}{3}}t}u_{\neq}\cdot\nabla u_{\neq}\|^2_{L^2L^2}\leq CA^{\frac23} E_2^4.$$

{\bf Estimate of $(\ref{eq:non-neq0})_{3}.$}
For $j\in\{1,3\},$ by Lemma \ref{sob_inf_2} and Lemma \ref{lemma_u}, we have
\begin{equation*}
	\begin{aligned}
		\|u_{j,\neq}\|^2_{L^{\infty}_{x,z}L^2_y}+\|\partial_xu_{j,\neq}\|^2_{L^2_{x,y}L^{\infty}_{z}}&\leq C\|\partial_x(\partial_x,\partial_z)u_{j,\neq}\|^2_{L^2}
		\leq C\Gamma_1^2
	\end{aligned}
\end{equation*}
and
\begin{equation*}
	\begin{aligned}
		&\quad\|\partial_x\partial_ju_{\neq}\|^2_{L^2_{x,z}L^{\infty}_{y}}+
		\|\partial_ju_{\neq}\|^2_{L^{\infty}_{x,y}L^2_{z}}\leq C\|\partial_x\partial_j\nabla u_{\neq}\|_{L^2}\|\partial_x\partial_ju_{\neq}\|_{L^2}
		\leq C\Gamma_1\Gamma_3,
	\end{aligned}
\end{equation*}
which indicate that
\begin{equation}\label{ap:u1}
	\begin{aligned}
		\|\partial_x(u_{j,\neq}\partial_j u_{\neq})\|_{L^2}^2\leq
		C\Gamma_1^3\Gamma_3.
	\end{aligned}
\end{equation}
Using $\eqref{sob_result_2}_1$, $\eqref{sob_result_2}_4$ and $\eqref{sob_result_2}_7$, there hold
\begin{equation*}
	\begin{aligned}
		\|\partial_xu_{2,\neq}\|^2_{L^{\infty}_{x,y}L^2_z}+\|u_{2,\neq}\|^2_{L^{\infty}}
		\leq C\|\partial_x\nabla u_{2,\neq}\|_{L^2}^{3-2\alpha}
		\|\nabla\Delta u_{2,\neq}\|_{L^2}^{2\alpha-1}
		\leq C\Gamma_4^{3-2\alpha}
		\Gamma_3^{2\alpha-1}
	\end{aligned}
\end{equation*}
and
\begin{equation*}
	\begin{aligned}
		\|\partial_yu_{\neq}\|^2_{L^2_{x,y}L^{\infty}_{z}}
		\leq C\|(\partial_x,\partial_z)\partial_yu_{\neq}\|^2_{L^2}
		\leq C\Gamma_2^2,
	\end{aligned}
\end{equation*}
which show that
\begin{equation}\label{ap:u2}
	\begin{aligned}
		\|\partial_x(u_{2,\neq}\partial_y u_{\neq})\|_{L^2}^2\leq
		C\Gamma_2^2\Gamma_3^{2\alpha-1}\Gamma_4^{3-2\alpha}.
	\end{aligned}
\end{equation}

Combining \eqref{ap:u1} with \eqref{ap:u2}, we conclude that
$$\|{\rm e}^{2aA^{-\frac{1}{3}}t}\partial_x(u_{\neq}\cdot\nabla u_{\neq})\|^2_{L^2L^2}
\leq CA^{\frac16+\alpha}E_2^4.$$

{\bf Estimate of $(\ref{eq:non-neq0})_{4}.$}
By Lemma \ref{sob_inf_2} and Lemma \ref{lemma_u}, direct calculations show that
	\begin{align}
		&\| u_{\neq}\|^2_{L^{\infty}}\leq C(\|\partial_x\omega_{2,\neq}\|_{L^2}+\|\Delta u_{2,\neq}\|_{L^2})
		(\|\nabla\omega_{2,\neq}\|_{L^2}+\|\Delta u_{2,\neq}\|_{L^2})\leq C\Gamma_1\Gamma_2,\nonumber\\
		&\|\nabla u_{\neq}\|^2_{L^2_{x,y}L^{\infty}_{z}}+\|\nabla u_{\neq}\|^2_{L^2_{y,z}L^{\infty}_{x}}\leq C\Gamma_2^2, \label{ap:u11}\\
		&\|\nabla u_{2,\neq}\|^2_{L^{\infty}_{x,y}L^{2}_{z}}\leq C\|\partial_x\nabla u_{2,\neq}\|_{L^{2}}\|\nabla\Delta u_{2,\neq}\|^{2\alpha-1}_{L^{2}}
		\|\Delta u_{2,\neq}\|^{2-2\alpha}_{L^{2}}\leq C\Gamma_4\Gamma_3^{2\alpha-1}\Gamma_1^{2-2\alpha}.\nonumber
	\end{align}
Combining above results with
\begin{equation*}
	\begin{aligned}
		\|\nabla(u_{\neq}\cdot\nabla u_{2,\neq})\|^2_{L^2}&\leq C(\|\nabla u_{\neq}\|^2_{L^2_{x,y}L^{\infty}_{z}}\|\nabla u_{2,\neq}\|^2_{L^{\infty}_{x,y}L^{2}_{z}}+\|u_{\neq}\|^2_{L^{\infty}}\|\Delta u_{2,\neq}\|^2_{L^2})\\
		&\leq C(\Gamma_4\Gamma_3^{2\alpha-1}
		\Gamma_1^{2-2\alpha}\Gamma_2^2+\Gamma_1^3\Gamma_2),
	\end{aligned}
\end{equation*}
there holds
\begin{equation*}
	\begin{aligned}
		\|{\rm e}^{2aA^{-\frac{1}{3}}t}\nabla(u_{\neq}\cdot\nabla u_{2,\neq})\|^2_{L^2L^2}\leq CA^{\frac{1}{2}+\frac{2}{3}\alpha}E_2^4.
	\end{aligned}
\end{equation*}

{\bf Estimate of $(\ref{eq:non-neq0})_{5}.$}
For $j\in\{1,3\}$, using Lemma \ref{sob_inf_2} and Lemma \ref{lemma_u}, we have
\begin{equation*}
	\begin{aligned}
		\|\partial_zu_{j,\neq}\|_{L^2_{y,z}L^{\infty}_{x}}^{2}\leq C\Gamma_1^2,\quad
		\|\partial_ju_{3,\neq}\|_{L^{\infty}_{y,z}L^2_{x}}^{2}\leq C\Gamma_1\Gamma_3,
	\end{aligned}
\end{equation*}
which along with $\eqref{ap:u11}_1$ indicates that
\begin{equation}\label{temp:u1}
	\begin{aligned}
		\|\partial_z(u_{j,\neq}\partial_ju_{3,\neq})\|_{L^2}^2
		\leq C\Gamma_1^3\Gamma_3.
	\end{aligned}
\end{equation}
For $j=2,$ by Lemma \ref{lemma_u} and Lemma \ref{sob_inf_2}, there holds
\begin{equation}\label{temp:u111}
	\begin{aligned}
		&\|u_{2,\neq}\|^2_{L^{\infty}}\leq C\|\partial_x\nabla u_{2,\neq}\|_{L^2}^{3-2\alpha}
		\|\Delta u_{2,\neq}\|_{L^2}^{2\alpha-1}
		\leq C\Gamma_4^{3-2\alpha}\Gamma_1^{2\alpha-1},\\
		&\|\partial_y u_{3,\neq}\|^2_{L^2_{x,y}L^{\infty}_{z}}\leq C\|(\partial_x,\partial_z)\nabla u_{3,\neq}\|^2_{L^2}\leq C\Gamma_2^2,
	\end{aligned}
\end{equation}
which along with  $\eqref{eq:velocity embed}_2$ and $\eqref{ap:u11}_3$ shows that
\begin{equation}\label{temp:u2}
	\begin{aligned}
		\|\partial_z(u_{2,\neq}\partial_yu_{3,\neq})\|_{L^2}^2
		\leq C(\Gamma_4\Gamma_3^{2\alpha-1}
		\Gamma_1^{2-2\alpha}\Gamma_2^2+\Gamma_4^{3-2\alpha}\Gamma_1^{2\alpha-1}\Gamma_2^2).
	\end{aligned}
\end{equation}
Combining \eqref{temp:u1} and \eqref{temp:u2},  we obtain that
\begin{equation*}
	\begin{aligned}
		\|{\rm e}^{2aA^{-\frac{1}{3}}t}\partial_z(u_{\neq}\cdot\nabla u_{3,\neq})\|_{L^2L^2}^2	
		\leq CA^{\frac{1}{2}+\frac{2}{3}\alpha}E_2^4.
	\end{aligned}
\end{equation*}

{\bf Estimate of $(\ref{eq:non-neq0})_{6}.$}
According to \eqref{temp:u111} and  \eqref{temp:u2}, after replacing $u_{3,\neq}$ with $u_{1,\neq}$ , we can prove that
\begin{equation}\label{temp:u41}
	\begin{aligned}
		\|\partial_z(u_{2,\neq}\partial_yu_{1,\neq})\|_{L^2}^2
		\leq C(\Gamma_4\Gamma_3^{2\alpha-1}
		\Gamma_1^{2-2\alpha}\Gamma_2^2+\Gamma_4^{3-2\alpha}\Gamma_1^{2\alpha-1}\Gamma_2^2).
	\end{aligned}
\end{equation}
Similar to \eqref{temp:u1}, for $j\in\{1,3\},$ by using $(\ref{eq:velocity embed})_3,$ one deduces
\begin{equation}\label{temp:u42}
	\begin{aligned}
		\|\partial_z(u_{j,\neq}\partial_ju_{1,\neq})\|_{L^2}^2
		\leq C\Gamma_1\Gamma_2 ^2\Gamma_3.
	\end{aligned}
\end{equation}
Collecting \eqref{temp:u41} and  \eqref{temp:u42}, there holds
\begin{equation*}
	\|{\rm e}^{2aA^{-\frac{1}{3}}t}\partial_z(u_{\neq}\cdot\nabla u_{1,\neq})\|^2_{L^2L^2}
	\leq  CA^{\frac43}E_2^4.
\end{equation*}

{\bf Estimate of $(\ref{eq:non-neq0})_{7}.$}
For $j\in\{1,3\}$, using Lemma \ref{sob_inf_2} and Lemma \ref{lemma_u} again, we have
\begin{equation*}
	\begin{aligned}
		\|\partial_ju_{3,\neq}\|_{L^{\infty}_{y,z}L^2_{x}}^2\leq C\| (\partial_x,\partial_z)\partial_j\nabla u_{3,\neq}\|_{L^2}\| (\partial_x,\partial_z)\partial_j u_{3,\neq}\|_{L^2}
		\leq C\Gamma_1\Gamma_3,
	\end{aligned}
\end{equation*}
which along with $\eqref{ap:u11}_1$ and $\eqref{ap:u11}_2$ indicates that
\begin{equation}\label{temp:u3}
	\begin{aligned}
		\|\nabla(u_{j,\neq}\partial_ju_{3,\neq})\|_{L^2}^2
		&\leq C(\|u_{j,\neq}\|_{L^{\infty}}^2\|\nabla\partial_ju_{3,\neq}\|_{L^2}^2+
		\|\nabla u_{j,\neq}\|_{L^2_{y,z}L^{\infty}_x}^2\|\partial_ju_{3,\neq}\|_{L^{\infty}_{y,z}L^2_{x}}^2)\\
		&\leq C(\Gamma_1\Gamma_2^3+\Gamma_2^2\Gamma_1\Gamma_3).
	\end{aligned}
\end{equation}
According to $\eqref{ap:u11}_3$ and \eqref{temp:u111}, one obtains
\begin{equation}\label{temp:u4}
	\begin{aligned}
		\|\nabla(u_{2,\neq}\partial_y u_{3,\neq})\|_{L^2}^2
		&\leq C(\|u_{2,\neq}\|_{L^{\infty}}^2\|\partial_y \nabla u_{3,\neq}\|_{L^2}^2
		+\|\nabla u_{2,\neq}\|_{L^{\infty}_{x,y}L^2_z}^2\|\partial_y u_{3,\neq}\|_{L^2_{x,y}L^{\infty}_z}^2)\\
		&\leq C(\Gamma_4^{3-2\alpha}\Gamma_1^{2\alpha-1}\|\Delta u_{3,\neq}\|_{L^2}^2
		+\Gamma_4\Gamma_3^{2\alpha-1}\Gamma_1^{2-2\alpha}\Gamma_2^2).
	\end{aligned}
\end{equation}
Therefore, we infer from \eqref{temp:u3} and \eqref{temp:u4} that
\begin{equation*}
	\begin{aligned}
		\|{\rm e}^{2aA^{-\frac{1}{3}}t}\nabla(u_{\neq}\cdot\nabla u_{3,\neq})\|^2_{L^2L^2}
		\leq C(A^{\frac76+\frac{1}{3}\alpha }E_2^2E_{5}^2+A^{\frac43} E_2^4).
	\end{aligned}
\end{equation*}

We finish the proof.
\end{proof}

Next, we give the nonlinear interaction between the zero mode $u_{1,0}$ and the non-zero mode $u_{\neq}$.
Recall that the zero mode $u_{1,0}$  is decomposed into two components with $u_{1,0}=\mathbf{U}_1+\mathbf{U}_2,$
where $\mathbf{U}_1$ is the good component and $\mathbf{U}_2$ is the bad component.

The following result will be used in estimating the energies  $\{\|\Delta u_{2,\neq}\|_{X_a}^2, \|\nabla \omega_{2,\neq}\|_{X_a}^2\}$.
Consequently, to close the energy estimates, the degree of $A$ must be strictly less than 1.
\begin{lemma}\label{lem_uneq2}
It holds that
\begin{equation}\label{lemma_neq2_2}
	\begin{aligned}
		&\|{\rm e}^{aA^{-\frac{1}{3}}t}
		\mathbf{U}_1\partial_x\nabla u_{2,\neq}\|_{L^2L^2}^2
		\leq CA^{\frac12}
		\|\mathbf{U}_1\|_{L^{\infty}H^1}^2
		\|\Delta u_{2,\neq}\|_{X_a}^2,\\
		&\|{\rm e}^{aA^{-\frac{1}{3}}t}
		\mathbf{U}_1\partial_x(\partial_x,\partial_z) u_{3,\neq}\|_{L^2L^2}^2\leq
		CA^{\frac23}\|\mathbf{U}_1\|_{L^{\infty}H^1}^2
		(\|\partial_x\omega_{2,\neq}\|^2_{X_a}
		+\|\Delta u_{2,\neq}\|^2_{X_a}),\\
		&\|{\rm e}^{aA^{-\frac{1}{3}}t}\mathbf{U}_1
		\partial_x^{2} u_{1,\neq}\|_{L^2L^2}^2\leq
		CA^{\frac23}\|\mathbf{U}_1\|_{L^{\infty}H^1}^2
		(\|\partial_x\omega_{2,\neq}\|^2_{X_a}
		+\|\Delta u_{2,\neq}\|^2_{X_a}),\\
		&\|{\rm e}^{aA^{-\frac13}t}\mathbf{U}_1\partial_{x}\partial_{z}u_{1,\neq}\|_{L^{2}L^{2}}^{2}\leq CA^{\frac23}\|\mathbf{U}_1\|_{L^{\infty}H^1}^2\left(\|\partial_{x}\omega_{2,\neq}\|^2_{X_{a}}+\|\Delta u_{2,\neq}\|^2_{X_{a}} \right),
		\\&\|{\rm e}^{aA^{-\frac{1}{3}}t}
		(\partial_y,\partial_z) \mathbf{U}_1\partial_x u_{\neq}\|_{L^2L^2}^2\leq
		CA^{\frac23}\|\mathbf{U}_1\|_{L^{\infty}H^1}^2(\|\partial_x\omega_{2,\neq}\|^2_{X_a}+\|\Delta u_{2,\neq}\|^2_{X_a}),
		\\&\|{\rm e}^{aA^{-\frac{1}{3}}t}
		\partial_y \mathbf{U}_1\partial_y u_{2,\neq}\|_{L^2L^2}^2\leq
		CA^{\frac23}\|\mathbf{U}_1\|_{L^{\infty}H^1}^2\|\Delta u_{2,\neq}\|^2_{X_a}.
	\end{aligned}
\end{equation}
\end{lemma}
\begin{proof}
Using Lemma \ref{sob_inf_1} and Lemma \ref{sob_inf_2}, we have
\begin{equation*}
	\begin{aligned}
		\|\mathbf{U}_1\partial_x\nabla u_{2,\neq}\|_{L^2}^2&\leq
		\|\mathbf{U}_1\|_{L^2_yL^{\infty}_{z}}^2
		\|\partial_x\nabla u_{2,\neq}\|_{L^{\infty}_{y}L^2_{x,z}}^2\\
		&\leq C\|\mathbf{U}_1\|_{H^1}^2
		(\|\partial_x\nabla u_{2,\neq}\|_{L^2}
		\|\partial_x\partial_y\nabla u_{2,\neq}\|_{L^2}
		+\|\partial_x\nabla u_{2,\neq}\|_{L^2}^2),
	\end{aligned}
\end{equation*}
which implies that
\begin{equation*}
	\|{\rm e}^{aA^{-\frac{1}{3}}t}
	\mathbf{U}_1\partial_x\nabla u_{2,\neq}\|_{L^2L^2}^2
	\leq CA^{\frac12}
	\|\mathbf{U}_1\|_{L^{\infty}H^1}^2
	\|\Delta u_{2,\neq}\|_{X_a}^2.
\end{equation*}
For $j\in\{1,3\},$ by Lemma \ref{sob_inf_1}, Lemma \ref{sob_inf_2} and Lemma \ref{lemma_u}, there holds
\begin{equation*}
	\begin{aligned}
		&\quad \|\mathbf{U}_1\partial_x(\partial_x,\partial_z) u_{j,\neq}\|_{L^2}^2\\
		&\leq C\|\mathbf{U}_1\|_{H^1}^2
		\|\partial_x(\partial_x,\partial_z) u_{j,\neq}\|_{L^2}
		\|\partial_x(\partial_x,\partial_z)\nabla u_{j,\neq}\|_{L^2}\\
		&\leq C\|\mathbf{U}_1\|_{H^1}^2
		(\|\partial_x\omega_{2,\neq}\|_{L^2}
		+\|\Delta u_{2,\neq}\|_{L^2})
		(\|\partial_x\nabla\omega_{2,\neq}\|_{L^2}
		+\|\nabla\Delta u_{2,\neq}\|_{L^2})
	\end{aligned}
\end{equation*}
and
\begin{equation*}
	\|{\rm e}^{aA^{-\frac{1}{3}}t}
	\mathbf{U}_1\partial_x(\partial_x,\partial_z) u_{j,\neq}\|_{L^2L^2}^2\leq
	CA^{\frac23}\|\mathbf{U}_1\|_{L^{\infty}H^1}^2
	(\|\partial_x\omega_{2,\neq}\|^2_{X_a}
	+\|\Delta u_{2,\neq}\|^2_{X_a}),
\end{equation*}
which give $\eqref{lemma_neq2_2}_2$, $\eqref{lemma_neq2_2}_3$ and $\eqref{lemma_neq2_2}_4$.

Using Lemma \ref{sob_inf_2}, we obtain that
	\begin{align*}
		\|(\partial_y,\partial_z) \mathbf{U}_1\partial_x u_{\neq}\|_{L^2}^2&\leq
		\|\nabla\mathbf{U}_1\|_{L^2}^2\|\partial_xu_{\neq}\|_{L^{\infty}_{y,z}L^2_x}^{2}\leq C\|\nabla\mathbf{U}_1\|_{L^2}^2\|\partial_x(\partial_x,\partial_z)u_{\neq}\|_{L^2}\|\partial_x(\partial_x,\partial_z)\nabla u_{\neq}\|_{L^2}\\
		&\leq C\|\nabla\mathbf{U}_1\|_{L^2}^2(\|\partial_x\omega_{2,\neq}\|_{L^2}
		+\|\Delta u_{2,\neq}\|_{L^2})
		(\|\partial_x\nabla\omega_{2,\neq}\|_{L^2}
		+\|\nabla\Delta u_{2,\neq}\|_{L^2}),
	\end{align*}
which indicates  that
\begin{equation*}
	\|{\rm e}^{aA^{-\frac{1}{3}}t}
	\nabla \mathbf{U}_1\partial_x u_{\neq}\|_{L^2L^2}^2\leq
	CA^{\frac23}\|\mathbf{U}_1\|_{L^{\infty}H^1}^2(\|\partial_x\omega_{2,\neq}\|^2_{X_a}+\|\Delta u_{2,\neq}\|^2_{X_a}).
\end{equation*}
Using Lemma \ref{sob_inf_2}, we obtain that
\begin{equation*}
	\begin{aligned}
		\|\partial_y \mathbf{U}_1\partial_y u_{2,\neq}\|_{L^2}^2&\leq
		\|\nabla\mathbf{U}_1\|_{L^2}^2\|\partial_yu_{2,\neq}\|_{L^{\infty}_{y,z}L^2_x}^{2}
		\leq C\|\nabla\mathbf{U}_1\|_{L^2}^2
		\|\Delta u_{2,\neq}\|_{L^2}\|\nabla\Delta u_{2,\neq}\|_{L^2}
	\end{aligned}
\end{equation*}
and
\begin{equation*}
	\|{\rm e}^{aA^{-\frac{1}{3}}t}
	\partial_y \mathbf{U}_1\partial_y u_{2,\neq}\|_{L^2L^2}^2\leq
	CA^{\frac23}\|\mathbf{U}_1\|_{L^{\infty}H^1}^2\|\Delta u_{2,\neq}\|^2_{X_a}.
\end{equation*}
\end{proof}

The following result is only used to estimate  $\|(\partial_y,\partial_z) \omega_{2,\neq}\|_{X_a}^2$.
Hence, as long as the degree of $A$ is less than $\frac{5}{3}$, the energy estimates become achievable.
Sometimes, we need to use the results of Lemma \ref{lemma_u23_1}, and we will prove them in next section.
\begin{lemma}\label{lemma_neq2}
Under the conditions of Theorem \ref{result0} and the assumptions \eqref{assumption}, there exists a constant $A_3$ independent of $A$ and $t$, such that if $A>A_3,$
it holds that
\begin{equation}\label{lem:uneq3}
	\begin{aligned}
		&\|{\rm e}^{aA^{-\frac{1}{3}}t}(\partial_y,\partial_z)
		(u_{2,\neq}\nabla \mathbf{U}_1 )\|^2_{L^2L^2}\leq
		C\big(A^{\frac23}\|\mathbf{U}_1\|^2_{L^{\infty}H^1}E_2^2+AE_2^{2\alpha}E_4^{2-2\alpha}\big),\\
		&\|{\rm e}^{aA^{-\frac{1}{3}}t}\partial_z
		(u_{3,\neq}\nabla \mathbf{U}_1 )\|^2_{L^2L^2}\leq
		C\big(A^{\frac23}\|\mathbf{U}_1\|^2_{L^{\infty}H^1}E_2^2+A^{\frac43}E_2^{2\alpha}E_4^{2-2\alpha}\big),
	\end{aligned}
\end{equation}
where $\alpha$ is a constant with $\alpha\in(\frac12, \frac34).$
\end{lemma}
\begin{proof}
First, for $j\in\{2,3\},$ direct calculations show that
\begin{equation}\label{equ:r1}
	\begin{aligned}
		&\|\partial_j(u_{2,\neq}\nabla \mathbf{U}_1 )\|^2_{L^2}
		\leq C\left(\|\partial_j u_{2,\neq}\nabla \mathbf{U}_1\|^2_{L^2}+\| u_{2,\neq}\partial_j\nabla \mathbf{U}_1\|^2_{L^2}\right)
		,\\
		&\|\partial_z(u_{3,\neq}\nabla \mathbf{U}_1 )\|^2_{L^2}
		\leq C\left(\|\partial_z u_{3,\neq}\nabla \mathbf{U}_1\|^2_{L^2}+\| u_{3,\neq}\partial_z\nabla \mathbf{U}_1\|^2_{L^2}\right).
	\end{aligned}
\end{equation}
By Lemma \ref{sob_inf_2}, there hold
\begin{equation*}\label{equ:u11_1}
	\begin{aligned}
		\|\partial_j u_{2,\neq}\nabla \mathbf{U}_1\|^2_{L^2}\leq
		C\|\nabla\mathbf{U}_1\|^2_{L^2}\|\Delta u_{2,\neq}\|_{L^2}\|\nabla\Delta u_{2,\neq}\|_{L^2}
	\end{aligned}
\end{equation*}
For the second term, set $D=(\partial_x^2,\partial_z^2)$.
The one-dimensional Sobolev inequalities in $y$ and $z$, followed
by streamwise Poincar\'e and Lemma~\ref{lemma_u}, give
\begin{equation*}\label{equ:u11_2}
	\begin{aligned}
		\|\partial_z u_{3,\neq}\nabla \mathbf{U}_1\|^2_{L^2}&\leq
		C\|\nabla\mathbf{U}_1\|^2_{L^2}\|D u_{3,\neq}\|_{L^2}
		\big(\|D u_{3,\neq}\|_{L^2}
		+\|\partial_yD u_{3,\neq}\|_{L^2}\big)\\
		&\leq C\|\mathbf{U}_1\|_{H^1}^2(\|\partial_x\omega_{2,\neq}\|_{L^2}
		+\|\Delta u_{2,\neq}\|_{L^2})
		(\|\partial_x\nabla\omega_{2,\neq}\|_{L^2}
		+\|\nabla\Delta u_{2,\neq}\|_{L^2}),
	\end{aligned}
\end{equation*}
which indicate that
\begin{equation}\label{equ:r2}
	\begin{aligned}
		&\|{\rm e}^{aA^{-\frac{1}{3}}t}\partial_j u_{2,\neq}\nabla \mathbf{U}_1\|^2_{L^2L^2}\leq
		CA^{\frac23}\|\mathbf{U}_1\|^2_{L^{\infty}H^1}\|\Delta u_{2,\neq}\|_{X_a}^2
		\leq CA^{\frac23}\|\mathbf{U}_1\|^2_{L^{\infty}H^1}E_2^2,\\
		&\|{\rm e}^{aA^{-\frac{1}{3}}t}\partial_z u_{3,\neq}\nabla \mathbf{U}_1\|^2_{L^2L^2}\leq
		CA^{\frac23}\|\mathbf{U}_1\|^2_{L^{\infty}H^1}(\|\partial_x \omega_{2,\neq}\|_{X_a}^2+\|\Delta u_{2,\neq}\|_{X_a}^2)
		\leq CA^{\frac23}\|\mathbf{U}_1\|^2_{L^{\infty}H^1}E_2^2.
	\end{aligned}
\end{equation}

Then, we need to deal with $\|{\rm e}^{aA^{-\frac{1}{3}}t} u_{2,\neq}\partial_j\nabla \mathbf{U}_1\|^2_{L^2L^2}$ and
$\|{\rm e}^{aA^{-\frac{1}{3}}t} u_{3,\neq}\partial_z\nabla \mathbf{U}_1\|^2_{L^2L^2}.$
Since it is difficult to estimate $\|\partial_j\nabla \mathbf{U}_1\|_{L^{\infty}L^2}$,
the traditional Sobolev embedding is insufficient to handle $\|{\rm e}^{aA^{-\frac{1}{3}}t} u_{2,\neq}\partial_j\nabla \mathbf{U}_1\|^2_{L^2L^2}$ and
$\|{\rm e}^{aA^{-\frac{1}{3}}t} u_{3,\neq}\partial_z\nabla \mathbf{U}_1\|^2_{L^2L^2}.$

According to \eqref{u_decom_1} and \eqref{u1:decom1}, $\mathbf{U}_1$ satisfies
\begin{equation}\label{u1:1}
	\left\{
	\begin{array}{lr}
		\partial_t\mathbf{U}_1-\frac{1}{A}
		\Delta\mathbf{U}_1
		=\frac{1}{A}\widetilde{n}_{0}-\frac{1}{A}(u_{2,0}\partial_y\mathbf{U}_1
		+u_{3,0}\partial_z\mathbf{U}_1)-\frac{1}{A}(u_{\neq}\cdot \nabla u_{1,\neq})_0,\\
		\mathbf{U}_1|_{t=0}=(u_{1,\rm in})_0.
	\end{array}
	\right.
\end{equation}
Therefore, for the given positive constant $\epsilon_1,$ we have
	\begin{align*}
		&\partial_t({\rm e}^{-\epsilon_1A^{-\frac13}t}\mathbf{U}_1)
		+\frac{\epsilon_1{\rm e}^{-\epsilon_1A^{-\frac13}t}\mathbf{U}_1}{A^{\frac13}}-\frac{{\rm e}^{-\epsilon_1A^{-\frac13}t}\Delta\mathbf{U}_1}{A}\\
		=&\frac{{\rm e}^{-\epsilon_1A^{-\frac13}t}\widetilde{n}_{0}}{A}-\frac{{\rm e}^{-\epsilon_1A^{-\frac13}t}u_{j,0}\partial_j\mathbf{U}_1
		}{A}-\frac{{\rm e}^{-\epsilon_1A^{-\frac13}t}(u_{\neq}\cdot \nabla u_{1,\neq})_0}{A},
	\end{align*}
where $j\in\{2,3\}.$
Multiplying $-\frac{{\rm e}^{-\epsilon_1A^{-\frac13}t}\Delta\mathbf{U}_1}{2}$ on both sides of the above equation, the energy estimate shows that
\begin{equation}\label{U1:T1}
	\begin{aligned}
		&\|{\rm e}^{-\epsilon_1A^{-\frac13}t}\nabla\mathbf{U}_1\|_{L^{\infty}L^2}^2
		+\frac{\|{\rm e}^{-\epsilon_1A^{-\frac13}t}\Delta\mathbf{U}_1\|^2_{L^2L^2}}{A}\\
		\leq~& \|u_{1,\rm in}\|^2_{H^1}
		+\frac{C\big(\|{\rm e}^{-\epsilon_1A^{-\frac13}t}\widetilde{n}_{0}\|^2_{L^2L^2}
			+\|{\rm e}^{-\epsilon_1A^{-\frac13}t}u_{j,0}\partial_j\mathbf{U}_1\|^2_{L^2L^2}
			+\|{\rm e}^{-\epsilon_1A^{-\frac13}t}u_{\neq}\cdot\nabla u_{1,\neq}\|^2_{L^2L^2}\big)}{A}.
	\end{aligned}
\end{equation}
When $A>A_2,$ by Lemma \ref{sob_inf_1} and Lemma \ref{lemma_u23_1}, we get
\begin{equation}\label{u20 u30 infty}
	\begin{aligned}
		\|u_{2,0}\|_{L^{\infty}L^{\infty}}^2+\|u_{3,0}\|_{L^{\infty}L^{\infty}}^2
		\leq C(\|u_{2,0}\|_{L^{\infty}H^{2}}^2+\|u_{3,0}\|_{L^{\infty}H^{1}}^2)\leq C\epsilon^2,
	\end{aligned}
\end{equation}
which implies that
\begin{equation}\label{U1:T2}
	\frac{\|{\rm e}^{-\epsilon_1A^{-\frac13}t}u_{j,0}\partial_j\mathbf{U}_1\|^2_{L^2L^2}}{A}
	\leq \frac{C\epsilon^2\|{\rm e}^{-\epsilon_1A^{-\frac13}t}\Delta\mathbf{U}_1\|^2_{L^2L^2}}{A}.
\end{equation}
Using  $\eqref{eq:non-neq0}_2$ and \eqref{U1:T2}, we infer from \eqref{U1:T1} that
\begin{equation*}\label{U1:T3}
	\begin{aligned}
		\|{\rm e}^{-\epsilon_1A^{-\frac13}t}\nabla\mathbf{U}_1\|_{L^{\infty}L^2}^2
		+\frac{\|{\rm e}^{-\epsilon_1A^{-\frac13}t}\Delta\mathbf{U}_1\|^2_{L^2L^2}}{A}
		\leq C\Big(\|u_{1,\rm in}\|^2_{H^1}
		+\frac{\|{\rm e}^{-\epsilon_1A^{-\frac13}t}\widetilde{n}_{0}\|^2_{L^2L^2}}{A}+\frac{E_2^4}{A^{\frac13}}\Big).
	\end{aligned}
\end{equation*}
Due to 
$\|{\rm e}^{-\epsilon_1A^{-\frac13}t}\|^2_{L^2(0,T)}\leq \frac{A^{\frac13}}{2\epsilon_1}$
and
$$\|\widetilde{n}_{0}\|_{L^2}^2\leq \|n_{0}\|_{L^2}^2\leq C\|n\|_{L^{\infty}L^{\infty}}\|n_{0}\|_{L^1}\leq CE_3M,$$
there holds
\begin{equation}\label{u1:result1}
	\begin{aligned}
		\|{\rm e}^{-\epsilon_1A^{-\frac13}t}\nabla\mathbf{U}_1\|_{L^{\infty}L^2}^2
		+\frac{\|{\rm e}^{-\epsilon_1A^{-\frac13}t}\Delta\mathbf{U}_1\|^2_{L^2L^2}}{A}
		\leq C\Big(\|u_{1,\rm in}\|^2_{H^1}
		+\frac{E_3M}{\epsilon_1A^{\frac23}}+\frac{E_2^4}{A^{\frac13}}\Big).
	\end{aligned}
\end{equation}
When $A\geq\max\{A_2, {\epsilon_1}^{-\frac32}M^{\frac32}E_3^{\frac32},E_2^{12}\}=:A_3,$ by $\eqref{sob_result_2}_3$, $\eqref{u1:result1}$ and Lemma \ref{lemma_u}, we have
\begin{equation}\label{equ:r3}
	\begin{aligned}
		\|{\rm e}^{aA^{-\frac{1}{3}}t} u_{2,\neq}\partial_j\nabla \mathbf{U}_1\|^2_{L^2L^2}
		&\leq C\|\Delta u_{2,\neq}\|^{2\alpha}_{X_a}\|u_{2,\neq}\|_{X_b}^{2-2\alpha}
		\|{\rm e}^{-(1-\alpha)(b-a)A^{-\frac{1}{3}}t}\Delta \mathbf{U}_1\|_{L^2L^2}^2\\
		&\leq CAE_2^{2\alpha}E_4^{2-2\alpha},
	\end{aligned}
\end{equation}
where $\alpha$ is a constant satisfying  $\alpha\in\big(\frac12,\frac34\big).$
By  $\eqref{sob_result_2}_3$ and Lemma \ref{lemma_u}, there holds
	\begin{align*}
		\|u_{3,\neq}\|^2_{L^{\infty}_{y,z}L^{2}_{x}}&
		\leq C\|\nabla(\partial_x, \partial_z) u_{3,\neq}\|_{L^2}
		\|(\partial_x, \partial_z) u_{3,\neq}\|^{2\alpha-1}_{L^2}\|u_{3,\neq}\|_{L^2}^{2-2\alpha}\\
		&\leq C(\|\partial_y\omega_{2,\neq}\|_{L^2}+\|\Delta u_{2,\neq}\|_{L^2})
		(\|\partial_x\omega_{2,\neq}\|_{L^2}
		+\|\Delta u_{2,\neq}\|_{L^2})^{2\alpha-1}\|u_{3,\neq}\|_{L^2}^{2-2\alpha},
	\end{align*}
which along with $\eqref{u1:result1}$ implies that
\begin{equation}\label{equ:r4}
	\begin{aligned}
		\|{\rm e}^{aA^{-\frac{1}{3}}t} u_{3,\neq}\partial_z\nabla \mathbf{U}_1\|^2_{L^2L^2}
		&\leq C(\|\partial_y\omega_{2,\neq}\|_{X_a}+\|\Delta u_{2,\neq}\|_{X_a})
		(\|\partial_x\omega_{2,\neq}\|_{X_a}
		+\|\Delta u_{2,\neq}\|_{X_a})^{2\alpha-1}\\
		&\quad\times\|\partial_x^2u_{3,\neq}\|_{X_b}^{2-2\alpha}
		\|{\rm e}^{-(1-\alpha)(b-a)A^{-\frac{1}{3}}t}\Delta \mathbf{U}_1\|_{L^2L^2}^2\\
		&\leq CA^{\frac43}E_2^{2\alpha}E_4^{2-2\alpha}.
	\end{aligned}
\end{equation}
Collecting \eqref{equ:r1}, \eqref{equ:r2}, \eqref{equ:r3} and \eqref{equ:r4}, we finish the proof.

\end{proof}

Furthermore, we show the  the nonlinear interactions between the bad component $\mathbf{U}_2$ of $u_{1,0}$ and the non-zero mode $u_{\neq}$.
\begin{lemma}\label{lemma_neq22}
For $j\in\{2,3\}$, there holds
\begin{equation}\label{lemneq2_2}
	\begin{aligned}
		&\|{\rm e}^{aA^{-\frac{1}{3}}t}
		\nabla \mathbf{U}_2\partial_x u_{j,\neq}\|_{L^2L^2}^2\leq
		CAE_{1}^2E_4^{\frac12}(t)(\|\partial_x \omega_{2,\neq}\|_{X_a}+\|\Delta u_{2,\neq}\|_{X_a})^{\frac32}
		,\\
		&\|{\rm e}^{aA^{-\frac13}t}\partial_{z}\mathbf{U}_2\partial_{x}^2u_{1,\neq}\|_{L^{2}L^{2}}^{2}\leq CA^{\frac{17}{12}}
		E_{1}^{2}E_4^{\frac14}(\|\partial_x \omega_{2,\neq}\|_{X_a}+\|\Delta u_{2,\neq}\|_{X_a})^{\frac74},
		\\
		&\|{\rm e}^{aA^{-\frac13}t}\partial_{j}\mathbf{U}_2
		\partial_{x}\partial_ju_{2,\neq}\|_{L^{2}L^{2}}^{2}\leq CA^{\frac43}E_{1}^{2}E_4^{\frac12}\|\Delta u_{2,\neq}\|_{X_a}^{\frac32},
		\\
		&\|{\rm e}^{aA^{-\frac{1}{3}}t}\partial_z
		(u_{2,\neq}\nabla \mathbf{U}_2 )\|^2_{L^2L^2}\leq
		CAE_{1}^2E_4^{\frac12}\|\Delta u_{2,\neq}\|_{X_a}^{\frac32},\\
		&\|{\rm e}^{aA^{-\frac{1}{3}}t}\partial_z
		(u_{3,\neq}\nabla \mathbf{U}_2 )\|^2_{L^2L^2}\leq
		CE_{1}^2A^{\frac43}E_4(\|\partial_x \omega_{2,\neq}\|_{X_a}+\|\Delta u_{2,\neq}\|_{X_a}).
	\end{aligned}
\end{equation}
\end{lemma}
\begin{proof}
First of all, there holds
\begin{equation*}
	\begin{aligned}
		\|\partial_j\mathbf{U}_2\|_{H^1}
		\leq \int_0^t \|\partial_s\partial_j\mathbf{U}_2(s)\|_{H^1}ds
		\leq CE_{1}t.
	\end{aligned}
\end{equation*}
For the given positive constant $\epsilon_2,$ since
$\lim_{t\rightarrow\infty}A^{-\frac13}t{\rm e}^{-\epsilon_2A^{-\frac13}t}=0,$
there holds
\begin{equation}\label{u221}
	\big\|{\rm e}^{-\epsilon_2A^{-\frac13}t}\|\partial_j\mathbf{U}_2\|_{H^1}\big\|_{L^{\infty}_t}
	\leq CA^{\frac13}E_{1}.
\end{equation}
By Lemma \ref{sob_inf_1} and Lemma \ref{sob_inf_2}, we have
\begin{equation*}
	\begin{aligned}
		\|\nabla \mathbf{U}_2\partial_x u_{j,\neq}\|_{L^2}^2
		\leq \|\nabla \mathbf{U}_2\|_{L^2_{z}L^{\infty}_{y}}^2
		\|\partial_x u_{j,\neq}\|_{L^{\infty}_{z}L^2_{x,y}}^2
		\leq  C\|\nabla \mathbf{U}_2\|_{H^1}^2
		\|\partial_x u_{j,\neq}\|_{L^2}^{\frac12}
		\|(\partial_x,\partial_z)\partial_x u_{j,\neq}\|_{L^2}^{\frac32},
	\end{aligned}
\end{equation*}
which along with \eqref{u221} and Lemma \ref{lemma_u} indicates that
\begin{equation*}
	\begin{aligned}
		\|{\rm e}^{aA^{-\frac{1}{3}}t}
		\nabla \mathbf{U}_2\partial_x u_{j,\neq}\|_{L^2L^2}^2
		&\leq C\int_0^t{\rm 	e}^{\frac{a-b}{2}A^{-\frac13}s}\|\nabla \mathbf{U}_2\|^2_{H^1}
		{\rm e}^{\frac{b+3a}{2}A^{-\frac13}s}
		\|\partial_x^2 u_{j,\neq}\|^{\frac12}_{L^2}
		\|(\partial_x,\partial_z)\partial_x u_{j,\neq}\|^{\frac32}_{L^2}ds\\
		&\leq CAE_{1}^2E_4^{\frac12}(t)(\|\partial_x \omega_{2,\neq}\|_{X_a}+\|\Delta u_{2,\neq}\|_{X_a})^{\frac32}.
	\end{aligned}
\end{equation*}
Similarly, one obtains that
\begin{equation*}
	\begin{aligned}
		&\quad\|\partial_z \mathbf{U}_2\partial_x^2 u_{1,\neq}\|_{L^2}^2
		\leq  C\|\partial_z \mathbf{U}_2\|_{H^1}^2
		\|\partial_x^2 u_{1,\neq}\|_{L^2}
		\|\partial_x^2\nabla u_{1,\neq}\|_{L^2}\\
		&\leq C\|\partial_z \mathbf{U}_2\|_{H^1}^2
		\|\partial_x^2 u_{1,\neq}\|_{L^2}^{\frac14}
		(\|\partial_x\omega_{2,\neq}\|_{L^2}+\|\Delta u_{2,\neq}\|_{L^2})^{\frac34}
		(\|\partial_x\nabla\omega_{2,\neq}\|_{L^2}
		+\|\nabla\Delta u_{2,\neq}\|_{L^2}).
	\end{aligned}
\end{equation*}
By the divergence-free property
	$\partial_{x}^2u_{1,\neq}=-\partial_{x}(\partial_{y}u_{2,\neq}+\partial_{z}u_{3,\neq} ),$
we can prove the second result.

Using Lemma \ref{sob_inf_1} and Lemma \ref{sob_inf_2} again, by H\"{o}lder's inequality,  there holds
\begin{equation*}
	\begin{aligned}
		\|\partial_{j}\mathbf{U}_2\partial_{x}\partial_ju_{2,\neq}\|_{L^2}^{2}
		&\leq C\|\partial_{j}\mathbf{U}_2\|_{H^1}^{2}
		\|\partial_{x}\partial_ju_{2,\neq}\|_{L^2}
		\|\partial_{x}\partial_j\nabla u_{2,\neq}\|_{L^2}\\
		&\leq C\|\partial_{j}\mathbf{U}_2\|_{H^1}^{2}
		\|\partial_{x}^2u_{2,\neq}\|_{L^2}^{\frac12}
		\|\Delta u_{2,\neq}\|_{L^2}^{\frac12}
		\|\partial_{x}\partial_j\nabla u_{2,\neq}\|_{L^2},
	\end{aligned}
\end{equation*}
which implies that
$$\|{\rm e}^{aA^{-\frac13}t}\partial_{j}\mathbf{U}_2
\partial_{x}\partial_ju_{2,\neq}\|_{L^{2}L^{2}}^{2}\leq CA^{\frac43}E_{1}^{2}E_4^{\frac12}
\|\Delta u_{2,\neq}\|_{X_a}^{\frac32}.$$
According to  Lemma \ref{sob_inf_1} and Lemma \ref{sob_inf_2},  by H\"{o}lder's inequality, we get
\begin{equation*}
	\begin{aligned}
		\|\partial_z
		(u_{2,\neq}\nabla \mathbf{U}_2 )\|^2_{L^2}
		\leq C\|\nabla \mathbf{U}_2\|^2_{H^1}
		\|\nabla u_{2,\neq}\|_{L^2}
		\|\Delta	u_{2,\neq}\|_{L^2}
		\leq C\|\nabla \mathbf{U}_2\|^2_{H^1}
		\| u_{2,\neq}\|_{L^2}^{\frac12}
		\|\Delta	u_{2,\neq}\|_{L^2}^{\frac32},
	\end{aligned}
\end{equation*}
which  indicates that
\begin{equation*}
	\begin{aligned}
		&\quad\|{\rm e}^{aA^{-\frac{1}{3}}t}\partial_z
		(u_{2,\neq}\nabla \mathbf{U}_2 )\|^2_{L^2L^2}\\
		&\leq C\big\|{\rm 	e}^{\frac{a-b}{4}A^{-\frac13}t}\|\nabla \mathbf{U}_2\|_{H^1}\big\|_{L^{\infty}_t}^2
		\|{\rm e}^{bA^{-\frac{1}{3}}t}
		u_{2,\neq}\|_{L^2L^2}^{\frac12}
		\|{\rm e}^{aA^{-\frac{1}{3}}t}\Delta u_{2,\neq}\|_{L^2L^2}^{\frac32}
		\leq CE_{1}^2E_4^{\frac12}A\|\Delta u_{2,\neq}\|_{X_a}^{\frac32}.
	\end{aligned}
\end{equation*}
Using Lemma \ref{lemma_u}, we get
\begin{equation*}
	\begin{aligned}
		\quad\|{\rm e}^{aA^{-\frac{1}{3}}t}\partial_z
		(u_{3,\neq}\nabla \mathbf{U}_2 )\|^2_{L^2L^2}
		\leq CE_{1}^2A^{\frac43}E_4(\|\partial_x \omega_{2,\neq}\|_{X_a}+\|\Delta u_{2,\neq}\|_{X_a}).
	\end{aligned}
\end{equation*}
\end{proof}
Lastly, we show the  the nonlinear interaction between the zero modes $\{u_{2,0},u_{3,0}\}$ and the non-zero mode $u_{\neq}$.
The proof can be found in Lemma 3.16 of \cite{CWW2025}, and we omit it.
\begin{lemma}\label{lemma_neq3}
For $j\in\{2,3\}$, it holds that
	\begin{align*}
		&\|{\rm e}^{aA^{-\frac{1}{3}}t}u_{j,0}(\partial_x,\partial_z)
		\nabla u_{\neq}\|_{L^2L^2}^2\leq
		CA(\|u_{2,0}\|^2_{L^{\infty}H^2}+\|u_{3,0}\|^2_{L^{\infty}H^1})(\|\partial_x\omega_{2,\neq}\|^2_{X_a}
		+\|\Delta u_{2,\neq}\|^2_{X_a}),\\
		&\|{\rm e}^{aA^{-\frac{1}{3}}t}\partial_z\nabla u_{j,0}
		\cdot u_{\neq}\|_{L^2L^2}^2\leq CA^{\frac23}(\|u_{2,0}\|^2_{L^{\infty}H^2}+\|u_{3,0}\|^2_{L^{\infty}H^1})(\|\partial_x\omega_{2,\neq}\|^2_{X_a}
		+\|\Delta u_{2,\neq}\|^2_{X_a}),\\
		&\|{\rm e}^{aA^{-\frac{1}{3}}t}\partial_z u_{j,0}
		\nabla u_{\neq}\|_{L^2L^2}^2\leq CA(\|u_{2,0}\|^2_{L^{\infty}H^2}+\|u_{3,0}\|^2_{L^{\infty}H^1})(\|\partial_x\omega_{2,\neq}\|^2_{X_a}
		+\|\Delta u_{2,\neq}\|^2_{X_a}),\\
		&\|{\rm e}^{aA^{-\frac13}t}\partial_{y}u_{3,0}(\partial_{x},\partial_{z})u_{2,\neq}\|_{L^{2}L^{2}}^{2}\leq CA^{\frac13}\left(\|u_{2,0}\|_{L^{\infty}H^{2}}^{2}+\|u_{3,0}\|_{L^{\infty}H^{1}}^{2} \right)\|\Delta u_{2,\neq}\|_{X_{a}}^{2}.\quad\quad\quad\quad
	\end{align*}
\end{lemma}

\section{Estimates for the zero modes of velocity $E_1(t)$: Proof of Proposition \ref{prop:E0}}\label{sec3}
In this section, we give some estimates for the zero modes of the velocity, which will be used in estimating the zero mode of the density and the non-zero modes.
\subsection{Energy estimates for $u_{2,0}$ and $u_{3,0}$}
We recall that
\begin{equation}\label{u_zero_1}
\left\{
\begin{array}{lr}
	\partial_tu_{2,0}-\frac{1}{A}\Delta u_{2,0}
	+\frac{1}{A}(u\cdot\nabla u_{2})_0
	+\frac{1}{A}\partial_yP^{N_1}_0
	+\frac{1}{A}\partial_y P^{N_2}_0=0, \\
	\partial_tu_{3,0}-\frac{1}{A}\Delta u_{3,0}
	+\frac{1}{A}(u\cdot\nabla u_3)_0
	+\frac{1}{A}\partial_zP^{N_1}_0
	+\frac{1}{A}\partial_z P^{N_2}_0=0,\\
	\partial_yu_{2,0}+\partial_zu_{3,0}=0,
\end{array}
\right.
\end{equation}
where
\begin{equation*}
\left\{
\begin{array}{lr}
	\Delta P^{N_1}=-2A\partial_xu_2+\partial_{x}n,\\
	\Delta P^{N_2}=-{\rm div}~(u\cdot\nabla u).
\end{array}
\right.
\end{equation*}

\begin{lemma}\label{lemma_u23_1}
Under the conditions of Theorem \ref{result0} and the assumptions \eqref{conditions:u20 u30} and \eqref{assumption}, there exists a positive constant $A_{2,3}$ independent of $A$ and $t$, such that if $A\geq A_{2,3},$
it holds that
\begin{equation}\label{eq:u20u30}
	\begin{aligned}
&\|u_{2,0}\|_{Y_0}+\|u_{3,0}\|_{Y_0}\leq C\epsilon,\\
		&\|\nabla u_{2,0}\|_{Y_0}+\|\nabla u_{3,0}\|_{Y_0}\leq C\epsilon,
		\\
		&\|\Delta u_{2,0}\|_{Y_0}\leq C\epsilon,
		\\& \|\min\{(A^{-\frac23}+A^{-1}t)^{\frac12}, 1\}\Delta u_{3,0}\|_{Y_{0}}\leq C\epsilon.
	\end{aligned}
\end{equation}
\end{lemma}
\begin{proof}
{\bf Estimate of $(\ref{eq:u20u30})_{1}.$}	Due to $\nabla\cdot u_{0}=0$, we have
\begin{equation*}
	<\partial_{y}P^{N_{k}}_{0}, u_{2,0}>+<\partial_{z}P^{N_{k}}_{0}, u_{3,0}>=-<P^{N_{k}}_{0}, \partial_{y}u_{2,0}+\partial_{z}u_{3,0}>=0,\quad{\rm for}~~k\in\{1,2\}
\end{equation*}
\begin{equation*}
	<u_0\cdot\nabla u_{2,0},u_{2,0}>+<u_0\cdot\nabla u_{3,0},u_{3,0}>=0.
\end{equation*}
When $A\geq\max\{1,E_2^{12}\epsilon^{-6}\}=:A_{2,1}$,
using Lemma \ref{lemma_neq1} and assumption (\ref{conditions:u20 u30}),
energy estimates of \eqref{u_zero_1} show that
\begin{equation*}\label{u0_ans11}
	\begin{aligned}
		&\|u_{2,0}\|_{L^{\infty}L^2}^2+\|u_{3,0}\|_{L^{\infty}L^2}^2
		+\frac{1}{A}\left(\|\nabla u_{2,0}\|_{L^2L^2}^2
		+\|\nabla u_{3,0}\|_{L^2L^2}^2\right)\\
		\leq~ &\|(u_{2,{\rm in}})_0\|_{L^2}^2+\|(u_{3,{\rm in}})_0\|_{L^2}^2+\frac{C\| |u_{\neq}|^2 \|^2_{L^2L^2}}{A}\leq \|(u_{2,{\rm in}})_0\|_{L^2}^2+\|(u_{3,{\rm in}})_0\|_{L^2}^2+\frac{CE_2^4}{A^{\frac{1}{3}}}\leq C\epsilon^{2}.
	\end{aligned}
\end{equation*}

{\bf Estimate of $(\ref{eq:u20u30})_{2}.$}
Multiplying $-2\Delta u_{2,0}$ on $(\ref{u_zero_1})_1$ and
$-2\Delta u_{3,0}$ on $(\ref{u_zero_1})_2$, respectively, and
integrating over $\mathbb{T}^2$, we obtain
\begin{equation}\label{u23_new_1}
	\begin{aligned}
			&\frac{d}{dt}
		\left(
		\|\nabla u_{2,0}\|_{L^2}^2
		+\|\nabla u_{3,0}\|_{L^2}^2
		\right)
		+\frac{2}{A}
		\left(
		\|\Delta u_{2,0}\|_{L^2}^2
		+\|\Delta u_{3,0}\|_{L^2}^2
		\right)
	\\
		={}&
		-\frac{2}{A}
		\sum_{j=2}^3
		\left\langle
		u_0\cdot\nabla u_{j,0},
		-\Delta u_{j,0}
		\right\rangle
		-\frac{2}{A}
		\sum_{j=2}^3
		\left\langle
		(u_{\neq}\cdot\nabla u_{j,\neq})_0,
		-\Delta u_{j,0}
		\right\rangle,
	\end{aligned}
\end{equation}
where we use
$<\partial_y P_0^{N_k}, \Delta u_{2,0}>+<\partial_z P_0^{N_k}, \Delta u_{3,0}>=0$ for $k\in\{1,2\}$.
By the divergence-free condition
$\partial_yu_{2,0}+\partial_zu_{3,0}=0$,
we have
\[
-\Delta u_{2,0}
=
\partial_z
\left(
\partial_yu_{3,0}-\partial_zu_{2,0}
\right),\quad
-\Delta u_{3,0}
=
-\partial_y
\left(
\partial_yu_{3,0}-\partial_zu_{2,0}
\right).
\]
Moreover,
\begin{align*}
	&\partial_y
	\left(
	u_0\cdot\nabla u_{3,0}
	\right)
	-
	\partial_z
	\left(
	u_0\cdot\nabla u_{2,0}
	\right)=
	u_{2,0}\partial_y
	\left(
	\partial_yu_{3,0}-\partial_zu_{2,0}
	\right)
	+
	u_{3,0}\partial_z
	\left(
	\partial_yu_{3,0}-\partial_zu_{2,0}
	\right).
\end{align*}
Consequently, integration by parts gives
	\begin{align*}
			&\sum_{j=2}^3
		\left\langle
		u_0\cdot\nabla u_{j,0},
		-\Delta u_{j,0}
		\right\rangle
		\\
		={}&
		\int_{\mathbb{T}^2}
		\left[
		u_{2,0}\partial_y
		\left(
		\partial_yu_{3,0}-\partial_zu_{2,0}
		\right)
		+
		u_{3,0}\partial_z
		\left(
		\partial_yu_{3,0}-\partial_zu_{2,0}
		\right)
		\right]
		\left(
		\partial_yu_{3,0}-\partial_zu_{2,0}
		\right)dydz
	\\
		=~&
		\frac12
		\int_{\mathbb{T}^2}
		\left(
		u_{2,0}\partial_y+u_{3,0}\partial_z
		\right)
		\left|
		\partial_yu_{3,0}-\partial_zu_{2,0}
		\right|^2dydz=0.
	\end{align*}
Therefore, \eqref{u23_new_1} becomes
\begin{equation}\label{u23_new_2}
	\begin{aligned}
		\frac{d}{dt}
		\left(
		\|\nabla u_{2,0}\|_{L^2}^2
		+\|\nabla u_{3,0}\|_{L^2}^2
		\right)
		+\frac{2}{A}
		\left(
		\|\Delta u_{2,0}\|_{L^2}^2
		+\|\Delta u_{3,0}\|_{L^2}^2
		\right)=
		-\frac{2}{A}
		\sum_{j=2}^3
		\left\langle
		(u_{\neq}\cdot\nabla u_{j,\neq})_0,
		-\Delta u_{j,0}
		\right\rangle.
	\end{aligned}
\end{equation}
By Young's inequality, we have
\begin{align*}
	\frac{2}{A}
	\left|
	\sum_{j=2}^3
	\left\langle
	(u_{\neq}\cdot\nabla u_{j,\neq})_0,
	-\Delta u_{j,0}
	\right\rangle
	\right|
	\leq
	\frac{1}{A}
	\left(
	\|\Delta u_{2,0}\|_{L^2}^2
	+\|\Delta u_{3,0}\|_{L^2}^2
	\right)
	+
	\frac{C}{A}
	\|u_{\neq}\cdot\nabla u_{\neq}\|_{L^2}^2.
\end{align*}
It follows from \eqref{u23_new_2} that
\begin{align}
	\frac{d}{dt}
	\left(
	\|\nabla u_{2,0}\|_{L^2}^2
	+\|\nabla u_{3,0}\|_{L^2}^2
	\right)
	+\frac{1}{A}
	\left(
	\|\Delta u_{2,0}\|_{L^2}^2
	+\|\Delta u_{3,0}\|_{L^2}^2
	\right)
	\leq
	\frac{C}{A}
	\|u_{\neq}\cdot\nabla u_{\neq}\|_{L^2}^2.
	\label{u23_new_3}
\end{align}
Integrating \eqref{u23_new_3} over $[0,t]$ and using Lemma \ref{lemma_neq1}, we obtain
\begin{align*}
	\|\nabla u_{2,0}\|_{L^\infty L^2}^2
	+\|\nabla u_{3,0}\|_{L^\infty L^2}^2
	+\frac{
	\|\Delta u_{2,0}\|_{L^2L^2}^2
	+\|\Delta u_{3,0}\|_{L^2L^2}^2
	}{A}
	\leq
	\|\nabla (u_{2,\mathrm{in}},u_{3,\mathrm{in}})_0\|_{L^2}^2
	+\frac{CE_{2}^{4}}{A^{\frac13}}\leq C\epsilon^{2}
\end{align*}
provided that $A\geq A_{2,1}$.

{\bf Estimate of $(\ref{eq:u20u30})_{3}.$}
The $H^{2}$ energy estimate for $(\ref{u_zero_1})_{1}$ shows that
\begin{equation}\label{energy u20''}
	\begin{aligned}
		&\|\Delta u_{2,0}\|_{L^{\infty}L^{2}}^{2}+\frac{\|\nabla\Delta u_{2,0}\|_{L^{2}L^{2}}^{2}}{A}\\\leq~&\|(u_{2,\rm in})_{0}\|_{H^{2}}^{2}+\frac{C}{A}\left(\|\nabla(u_{0}\cdot\nabla u_{2,0})\|_{L^{2}L^{2}}^{2}+\|\nabla(u_{\neq}\cdot\nabla u_{2,\neq})\|_{L^{2}L^{2}}^{2}+\|\Delta P_{0}^{N_{2}}\|_{L^{2}L^{2}}^{2}\right).
	\end{aligned}
\end{equation}
Using Lemma \ref{sob_inf_1}, $(\ref{eq:u20u30})_{1,2}$ and $\partial_{z}u_{3,0}=-\partial_{y}u_{2,0}$, for $j\in\{2,3\}$, there holds
\begin{equation}\label{u0 uj0}
	\begin{aligned}
		\frac{\|\nabla(u_{0}\cdot\nabla u_{j,0})\|_{L^{2}L^{2}}^{2}}{A}\leq~&\frac{C}{A}\left(\|u_{2,0}\|_{L^{\infty}H^{2}}^{2}+\|u_{3,0}\|_{L^{\infty}H^{1}}^{2} \right)\left(\|\nabla u_{2,0}\|_{L^{2}H^{1}}^{2}+\|\nabla u_{3,0}\|_{L^{2}H^{1}}^{2} \right)\\\leq~& C\epsilon^{4}+C\epsilon^{2}\|\Delta u_{2,0}\|_{L^{\infty}L^{2}}^{2}.
	\end{aligned}
\end{equation}
Moreover, due to ${\rm div}~(u\cdot\nabla u)=\partial_{x}(u\cdot\nabla u_{1})+\partial_{y}(u\cdot\nabla u_{2})+\partial_{z}(u\cdot\nabla u_{3}),$ there holds
\begin{equation*}
	\begin{aligned}
		&\|{\rm div}~(u\cdot\nabla u)_{0}\|_{L^{2}}^{2}\leq\|\partial_{y}(u\cdot\nabla u_{2})_{0}\|_{L^{2}}^{2}+\|\partial_{z}(u\cdot\nabla u_{3})_{0}\|_{L^{2}}^{2}\\\leq~&\|\partial_{y}(u_{0}\cdot\nabla u_{2,0})\|_{L^{2}}^{2}+\|\partial_{z}(u_{0}\cdot\nabla u_{3,0})\|_{L^{2}}^{2}+\|\partial_{y}(u_{\neq}\cdot\nabla u_{2,\neq})\|_{L^{2}}^{2}+\|\partial_{z}(u_{\neq}\cdot\nabla u_{3,\neq})\|_{L^{2}}^{2}.
	\end{aligned}
\end{equation*}
From this, along with Lemma \ref{eq:non-neq0} and (\ref{u0 uj0}), we have
\begin{equation}\label{PN2''}
	\begin{aligned}
		\frac{\|\Delta P_{0}^{N_{2}}\|_{L^{2}L^{2}}^{2}}{A}=\frac{\|{\rm div}~(u\cdot\nabla u)_{0}\|_{L^{2}L^{2}}^{2}}{A}\leq C\left(\epsilon^{4}+\frac{E_{2}^{4}}{A^{\frac12-\frac23\alpha}} \right)+C\epsilon^{2}\|\Delta u_{2,0}\|_{L^{\infty}L^{2}}^{2},
	\end{aligned}
\end{equation}
where $\alpha\in(\frac12, \frac{3}{4})$ is a constant.

Then using assumption (\ref{conditions:u20 u30}), (\ref{u0 uj0}), (\ref{PN2''}) and Lemma \ref{eq:non-neq0}, we get from (\ref{energy u20''}) that
\begin{equation*}
	\begin{aligned}
		\|\Delta u_{2,0}\|_{L^{\infty}L^{2}}^{2}+\frac{\|\nabla\Delta u_{2,0}\|_{L^{2}L^{2}}^{2}}{A}\leq \epsilon^{2}+C\left(\epsilon^{4}+\frac{E_{2}^{4}}{A^{\frac12-\frac23\alpha}} \right)+C\epsilon^{2}\|\Delta u_{2,0}\|_{L^{\infty}L^{2}}^{2}.
	\end{aligned}
\end{equation*}
Taking $\epsilon$ sufficiently small gives
$\|\Delta u_{2,0}\|_{Y_0}^2\le C\epsilon^2$, provided that
\[
A\ge\max\left\{A_{2,1},(E_2^2\epsilon^{-1})^{12/(3-4\alpha)}\right\}
=:A_{2,2}.
\]

{\bf Estimate of $(\ref{eq:u20u30})_{4}.$}
Taking $H^{2}$ energy estimate for $(\ref{u_zero_1})_{2}$, there holds
\begin{equation*}
	\begin{aligned}
		\frac{d}{dt}\|\Delta u_{3,0}\|_{L^{2}}^{2}+\frac{\|\nabla\Delta u_{3,0}\|_{L^{2}}^{2}}{A}\leq\frac{C}{A}\left(\|\nabla(u_{0}\cdot\nabla u_{3,0})\|_{L^{2}}^{2}+\|\nabla(u_{\neq}\cdot\nabla u_{3,\neq})\|_{L^{2}}^{2}+\|\Delta P_{0}^{N_{2}}\|_{L^{2}}^{2} \right),
	\end{aligned}
\end{equation*}
which follows that
\begin{equation*}\label{u30''}
	\begin{aligned}
		&\frac{d}{dt}\left(\min\{A^{-\frac23}+A^{-1}t, 1\}\|\Delta u_{3,0}\|_{L^{2}}^{2} \right)+\frac{\min\{A^{-\frac23}+A^{-1}t, 1 \}}{A}\|\nabla\Delta u_{3,0}\|_{L^{2}}^{2}\\\leq~&\frac{C}{A}\left(\|\nabla(u_{0}\cdot\nabla u_{3,0})\|_{L^{2}}^{2}+(A^{-\frac23}+A^{-1}t)\|\nabla(u_{\neq}\cdot\nabla u_{3,\neq})\|_{L^{2}}^{2}+\|\Delta P_{0}^{N_{2}}\|_{L^{2}}^{2}+\|\Delta u_{3,0}\|_{L^{2}}^{2} \right).
	\end{aligned}
\end{equation*}
Thanks to $(\ref{eq:u20u30})_{3}$, (\ref{u0 uj0}) and (\ref{PN2''}), we get
\begin{equation}\label{u0 uj0 PN2''}
	\begin{aligned}
		\frac{\|\nabla(u_{0}\cdot\nabla u_{j,0})\|_{L^{2}L^{2}}^{2}}{A}\leq C\epsilon^{4},\quad 	\frac{\|\Delta P_{0}^{N_{2}}\|_{L^{2}L^{2}}^{2}}{A}\leq C\left(\epsilon^{4}+\frac{E_{2}^{4}}{A^{\frac12-\frac23\alpha}} \right).
	\end{aligned}
\end{equation}
Notice that $\|(A^{-\frac23}+A^{-1}t){\rm e}^{-aA^{-\frac13}t}\|_{L^{\infty}_{t}}\leq CA^{-\frac23},$
using $(\ref{eq:u20u30})_{2},$ (\ref{u0 uj0 PN2''}) and Lemma \ref{eq:non-neq0}, one obtains
\begin{equation*}
	\begin{aligned}
		&\|\min\{(A^{-\frac23}+A^{-1}t)^{\frac12}, 1\}\Delta u_{3,0}\|_{Y_{0}}^{2}\\\leq~&\frac{\|(u_{3,\rm in})_{0}\|_{H^{2}}^{2}}{A^{\frac23}}+C\left(\epsilon^{2}+\frac{E_{2}^{4}}{A^{\frac12-\frac23\alpha}} \right)+\frac{C}{A}\|(A^{-\frac23}+A^{-1}t){\rm e}^{-aA^{-\frac13}t}\|_{L^{\infty}_{t}}\|{\rm e}^{2aA^{-\frac13}t}\nabla(u_{\neq}\cdot\nabla u_{3,\neq})\|_{L^{2}L^{2}}^{2}\\\leq~&\frac{\|(u_{3,\rm in})_{0}\|_{H^{2}}^{2}}{A^{\frac23}}+C\left(\epsilon^{2}+\frac{E_{2}^{4}}{A^{\frac12-\frac23\alpha}}+\frac{E_{2}^{2}E_{5}^{2}}{A^{\frac12-\frac13\alpha}} \right),\quad\alpha\in\Big(\frac12, \frac34\Big).
	\end{aligned}
\end{equation*}
When $$A\geq \max\{ A_{2,2}, (\|(u_{3,\rm in})_{0}\|_{H^{2}}\epsilon^{-1})^{3}, (E_{2}E_{5}\epsilon^{-1})^{\frac{12}{3-2\alpha}} \}=:A_{2,3},$$ we infer from the above inequality that
\begin{equation*}
	\|\min\{(A^{-\frac23}+A^{-1}t)^{\frac12}, 1\}\Delta u_{3,0}\|_{Y_{0}}^{2}\leq C\epsilon^{2}.
\end{equation*}

The proof is complete.
\end{proof}

\subsection{Heat dissipation estimates for $\widetilde{u}_{2,0}$ and $\widetilde{u}_{3,0}$}\

We decompose the velocity $u_{2, 0}$ and $u_{3, 0}$ in the same way of (\ref{def:check{f_0}}), satisfying
\begin{equation}\label{u23:decompose}
\begin{aligned}
	u_{2,0}(t,y,z)=\overline{u}_{2,0}(t)+\widetilde{u}_{2,0}(t,y,z),\\
	u_{3,0}(t,y,z)=\overline{u}_{3,0}(t)+\widetilde{u}_{3,0}(t,y,z).
\end{aligned}
\end{equation}
In this way, we can prove a typical heat dissipation estimates for $\widetilde{u}_{2,0}(t,y,z)$ and $\widetilde{u}_{3,0}(t,y,z)$.
Furthermore, for $j\in\{2,3\},$ $\overline{u}_{j,0}(t)$ is
a constant satisfying
\begin{equation*}
\begin{aligned}
	\overline{u}_{j,0}(t)\equiv\frac{1}{|\mathbb{T}|^2}
	\int_{|\mathbb{T}|^2}(u_{j,\rm in})_0dydz.
\end{aligned}
\end{equation*}

\begin{lemma}
	\label{lem:H1-decay-zero-mode}
	Under the assumptions of Lemma \ref{lemma_u23_1}, there exists a
	sufficiently large constant $A_{2,4}>0$, independent of $t$ and $A$,
	such that if $A\geq A_{2,4}$, then
			\begin{align*}
			\|{\rm e}^{\frac{t}{2A}}
			\widetilde{u}_{2,0}\|_{L^\infty H^1}^2
			+\|{\rm e}^{\frac{t}{2A}}
			\widetilde{u}_{3,0}\|_{L^\infty H^1}^2
			+\frac{\|{\rm e}^{\frac{t}{2A}}
			\nabla\widetilde{u}_{2,0}\|_{L^2H^1}^2}{A}
			+\frac{\|{\rm e}^{\frac{t}{2A}}
			\nabla\widetilde{u}_{3,0}\|_{L^2H^1}^2}{A}
			\leq
			C(\|\widetilde{(u_{2,\rm in},u_{3,\rm in})}_0\|_{H^1}^2+1).
		\end{align*}
\end{lemma}

\begin{proof}
	Taking the non-average parts with respect to $(y,z)$ in
	\eqref{u_zero_1}, we obtain
	\begin{equation}
		\left\{
		\begin{aligned}
			\partial_t\widetilde{u}_{2,0}
			-\frac{1}{A}\Delta\widetilde{u}_{2,0}
			+\frac{1}{A}
			\widetilde{(u\cdot\nabla u_2)}_0
			+\frac{1}{A}\partial_y\widetilde P^{N_1}_0
			+\frac{1}{A}\partial_y\widetilde P^{N_2}_0
			&=0,\\
			\partial_t\widetilde{u}_{3,0}
			-\frac{1}{A}\Delta\widetilde{u}_{3,0}
			+\frac{1}{A}
			\widetilde{(u\cdot\nabla u_3)}_0
			+\frac{1}{A}\partial_z\widetilde P^{N_1}_0
			+\frac{1}{A}\partial_z\widetilde P^{N_2}_0
			&=0.
		\end{aligned}
		\right.
		\label{eq:tilde-u23-equation}
	\end{equation}
	Recall that
	$\Delta P^{N_1}=-2A\partial_xu_2+\partial_xn.$
	Taking the zero mode in the $x$-direction gives
	$\Delta P^{N_1}_0=0$. Hence, by periodicity,
	\begin{equation}
		\nabla P^{N_1}_0=0.
		\label{eq:PN1-zero-mode}
	\end{equation}
	Multiplying 
	$\eqref{eq:tilde-u23-equation}_1$ by
$	2{\rm e}^{\frac{t}{A}}
	\left(\widetilde u_{2,0}-\Delta\widetilde u_{2,0}\right)$,
	the second equation by
	$2{\rm e}^{\frac{t}{A}}
	\left(\widetilde u_{3,0}-\Delta\widetilde u_{3,0}\right)$,
	and then integrating over $\mathbb T^2$, we obtain
	\begin{align}\label{eq:direct-H1-energy-new}
		&\frac{d}{dt}
		\left[
		{\rm e}^{\frac{t}{A}}
		\left(
		\|\widetilde u_{2,0}\|_{H^1}^2
		+\|\widetilde u_{3,0}\|_{H^1}^2
		\right)
		\right]
		-\frac{1}{A}{\rm e}^{\frac{t}{A}}
		\left(
		\|\widetilde u_{2,0}\|_{H^1}^2
		+\|\widetilde u_{3,0}\|_{H^1}^2
		\right)
	\nonumber\\
		&
		+\frac{2}{A}{\rm e}^{\frac{t}{A}}
		\left(
		\|\nabla\widetilde u_{2,0}\|_{L^2}^2
		+\|\Delta\widetilde u_{2,0}\|_{L^2}^2
		+\|\nabla\widetilde u_{3,0}\|_{L^2}^2
		+\|\Delta\widetilde u_{3,0}\|_{L^2}^2
		\right)
		\\&=
		-\frac{2}{A}{\rm e}^{\frac{t}{A}}
		\sum_{j=2}^3
		\left\langle
		\widetilde{(u\cdot\nabla u_j)}_0,
		\widetilde u_{j,0}-\Delta\widetilde u_{j,0}
		\right\rangle.\nonumber
	\end{align}
	Here, the pressure terms vanish because of
	the diverge-free condition.
	Due to $\widetilde u_{2,0}$ and $\widetilde u_{3,0}$ have zero
	average on $\mathbb T^2$, the Poincar\'e inequality yields
	\[
	\|\widetilde u_{j,0}\|_{L^2}
	\leq \|\nabla\widetilde u_{j,0}\|_{L^2},
	\qquad
	\|\nabla\widetilde u_{j,0}\|_{L^2}
	\leq \|\Delta\widetilde u_{j,0}\|_{L^2},
	\quad j\in\{2,3\}.
	\]
	Consequently,
	\begin{equation}\label{eq:absorb-weight-term-new}
			\begin{aligned}
			&-\frac{1}{A}
			\left(
			\|\widetilde u_{2,0}\|_{H^1}^2
			+\|\widetilde u_{3,0}\|_{H^1}^2
			\right)
			+\frac{2}{A}
			\left(
			\|\nabla\widetilde u_{2,0}\|_{L^2}^2
			+\|\Delta\widetilde u_{2,0}\|_{L^2}^2
			+\|\nabla\widetilde u_{3,0}\|_{L^2}^2
			+\|\Delta\widetilde u_{3,0}\|_{L^2}^2
			\right)\\&\geq
			\frac{1}{A}
			\left(
			\|\nabla\widetilde u_{2,0}\|_{L^2}^2
			+\|\Delta\widetilde u_{2,0}\|_{L^2}^2
			+\|\nabla\widetilde u_{3,0}\|_{L^2}^2
			+\|\Delta\widetilde u_{3,0}\|_{L^2}^2
			\right).
		\end{aligned}
	\end{equation}
	For $j\in\{2,3\}$, we have
	\begin{equation}
		\widetilde{(u\cdot\nabla u_j)}_0
		=
		u_{2,0}\partial_y\widetilde u_{j,0}
		+u_{3,0}\partial_z\widetilde u_{j,0}
		+\widetilde{(u_{\neq}\cdot\nabla u_{j,\neq})}_0.
		\label{eq:nonlinear-decomposition-H1-new}
	\end{equation}
	We now use the two-dimensional incompressible structure to treat
	the zero-mode transport terms. First, direct calculation shows that
	\begin{align}
		&\sum_{j=2}^3
		\left\langle
		u_{2,0}\partial_y\widetilde u_{j,0}
		+u_{3,0}\partial_z\widetilde u_{j,0},
		\widetilde u_{j,0}
		\right\rangle=
		\frac12\int_{\mathbb T^2}
		\left(u_{2,0}\partial_y+u_{3,0}\partial_z\right)
		\left(
		|\widetilde u_{2,0}|^2
		+|\widetilde u_{3,0}|^2
		\right)dydz
		=0.
		\label{eq:L2-transport-cancel}
	\end{align}
Set
$\omega_{1,0}
:=\partial_y\widetilde{u}_{3,0}
-\partial_z\widetilde{u}_{2,0}
=\partial_yu_{3,0}-\partial_zu_{2,0}.$
Since
$\partial_yu_{2,0}+\partial_zu_{3,0}=0$,
we have
$-\Delta\widetilde{u}_{2,0}
=\partial_z\omega_{1,0}$ and 
$-\Delta\widetilde{u}_{3,0}
=-\partial_y\omega_{1,0}.$
	Hence, integration by parts gives
	\begin{equation}\label{eq:H1-transport-cancel}
	\begin{aligned}
		&\sum_{j=2}^3
		\left\langle
		u_{2,0}\partial_y\widetilde u_{j,0}
		+u_{3,0}\partial_z\widetilde u_{j,0},
		-\Delta\widetilde u_{j,0}
		\right\rangle=
		\int_{\mathbb T^2}
		\left(
		u_{2,0}\partial_y\omega_{1,0}
		+u_{3,0}\partial_z\omega_{1,0}
		\right)\omega_{1,0} dydz\\&=
		\frac12\int_{\mathbb T^2}
		\left(
		u_{2,0}\partial_y+u_{3,0}\partial_z
		\right)|\omega_{1,0}|^2 dydz
		=0.
	\end{aligned}
	\end{equation}
	For the interactions between the non-zero modes, the Poincar\'e inequality
	and Young's inequality imply that
	\begin{align}\label{eq:nonzero-nonlinear-H1-new}
		&\left|
		\sum_{j=2}^3
		\left\langle
		\widetilde{(u_{\neq}\cdot\nabla u_{j,\neq})}_0,
		\widetilde u_{j,0}-\Delta\widetilde u_{j,0}
		\right\rangle
		\right|
	       \leq
		C\|u_{\neq}\cdot\nabla u_{\neq}\|_{L^2}
		\sum_{j=2}^3
		\left(
		\|\widetilde u_{j,0}\|_{L^2}
		+\|\Delta\widetilde u_{j,0}\|_{L^2}
		\right)\\
		&\leq
		\frac{1}{4}
		\left(
		\|\nabla\widetilde u_{2,0}\|_{L^2}^2
		+\|\Delta\widetilde u_{2,0}\|_{L^2}^2
		+\|\nabla\widetilde u_{3,0}\|_{L^2}^2
		+\|\Delta\widetilde u_{3,0}\|_{L^2}^2
		\right)
		+C\|u_{\neq}\cdot\nabla u_{\neq}\|_{L^2}^2.\nonumber
	\end{align}
	Combining
	\eqref{eq:direct-H1-energy-new}--\eqref{eq:nonzero-nonlinear-H1-new}, 
	after integrating over $[0,t]$, we obtain 
	\begin{align}\label{eq:weighted-H1-integrated-new}
		&{\rm e}^{\frac{t}{A}}
		\left(
		\|\widetilde u_{2,0}(t)\|_{H^1}^2
		+\|\widetilde u_{3,0}(t)\|_{H^1}^2
		\right)
		+\frac{1}{2A}
		\int_0^t{\rm e}^{\frac{s}{A}}
		\left(
		\|\nabla\widetilde u_{2,0}\|_{H^1}^2		
        +\|\nabla\widetilde u_{3,0}\|_{H^1}^2
		\right)ds\\&\leq
		\|\widetilde{(u_{2,\rm in})}_0\|_{H^1}^2
		+\|\widetilde{(u_{3,\rm in})}_0\|_{H^1}^2
		+\frac{C}{A}
		\int_0^t{\rm e}^{\frac{s}{A}}
		\|u_{\neq}\cdot\nabla u_{\neq}\|_{L^2}^2 ds.\nonumber
	\end{align}
	By Lemma \ref{lemma_neq1}, we arrive at
	\[
	\|
	{\rm e}^{2aA^{-\frac13}t}
	u_{\neq}\cdot\nabla u_{\neq}
	\|_{L^2L^2}^2
	\leq CA^{\frac23}E_2^4.
	\]
	When $	A\geq \left(\frac{1}{4a}\right)^{\frac32}$,
	we have
$	{\rm e}^{\frac{t}{2A}}
	\leq {\rm e}^{2aA^{-\frac13}t}.$
	It follows that
	\begin{equation*}
		\frac{1}{A}
		\int_0^t{\rm e}^{\frac{s}{A}}
		\|u_{\neq}\cdot\nabla u_{\neq}\|_{L^2}^2\,ds
		\leq \frac{CE_2^4}{A^{\frac13}}.
		\label{eq:nonzero-weighted-bound-new}
	\end{equation*}
	Substituting it into
	\eqref{eq:weighted-H1-integrated-new} and letting
	$$A\geq\max\Big\{A_{2,1},  \Big(\frac{1}{4a}\Big)^{\frac32}\Big\}=:A_{2,4},$$
	and  the proof is complete.
\end{proof}

\subsection{Energy estimates for $u_{1,0}$}

\begin{lemma}\label{lemma:u1:1}
Under the assumptions of Lemma \ref{lemma_u23_1}, it holds that
\begin{equation*}
	\begin{aligned}
		\|\mathbf{U}_1\|_{L^{\infty}H^1}\leq C\left(\|(u_{1,\rm in})_0\|_{H^1}+\|n_0\|_{L^{\infty}L^2}+\epsilon\right).
	\end{aligned}
\end{equation*}
\end{lemma}
\begin{proof}
Multiplying $\Delta\widetilde{\mathbf{B}}_{1}$ on both sides of $(\ref{u1:decom1})_1 ,$ and applying H\"{o}lder's inequality, the energy estimate shows that
\begin{equation}\label{n_01}
	\frac{d}{dt}\|\nabla\widetilde{\mathbf{B}}_{1}\|^2_{L^2}
	+\frac{1}{A}\|\Delta\widetilde{\mathbf{B}}_{1}\|^2_{L^2}\leq \frac{C}{A}\left(\|u_{2,0}\|_{L^{\infty}}^{2}+\|u_{3,0}\|_{L^{\infty}}^{2} \right)\|\nabla\widetilde{\mathbf{B}}_{1}\|_{L^{2}}^{2}+\frac{C}{A}\|\widetilde{n}_{0}\|_{L^{2}}^{2}.
\end{equation}
Using (\ref{u20 u30 infty}) and Poincar\'{e} inequality
\begin{equation*}
	\|\nabla\widetilde{\mathbf{B}}_{1}\|^2_{L^2}\leq C\|\Delta\widetilde{\mathbf{B}}_{1}\|^2_{L^2},
\end{equation*}
when choosing $\epsilon$ is small enough,	we infer from (\ref{n_01}) that
\begin{equation}\label{n_02}
	\frac{d}{dt}\|\nabla\widetilde{\mathbf{B}}_{1}\|^2_{L^2}\leq -\frac{\|\nabla \widetilde{\mathbf{B}}_{1}\|^2_{L^2}}{2AC}+\frac{C\|\widetilde{n}_{0}\|_{L^2}^2}{A}
	\leq-\frac{\|\nabla \widetilde{\mathbf{B}}_{1}\|^2_{L^2}
		-C\|n_{0}\|_{L^2}^2}{2AC},
\end{equation}	
where we use $$\|\widetilde{n}_{0}\|_{L^2}=\|n_0-\overline{n}_{0}\|_{L^2}\leq C\|n_{0}\|_{L^2}.$$
According to (\ref{n_02}), there holds
\begin{equation*}
	\|\nabla\widetilde{\mathbf{B}}_{1}\|^2_{L^{\infty}L^2}\leq C\|n_0\|^2_{L^{\infty}L^2}.
\end{equation*}
Otherwise, there must be  a time $t=t^*,$ such that
\begin{equation}\label{n_021}
		\|\nabla\widetilde{\mathbf{B}}_{1}\|^2_{L^2}|_{t=t^*}=C\|n_0\|^2_{L^{\infty}L^2},\quad 
		\frac{d}{dt}\|\nabla\widetilde{\mathbf{B}}_{1}\|^2_{L^2}\Big|_{t=t^*}>0.	
\end{equation}
Let $t=t^*,$ using $(\ref{n_021})_1$, we get from (\ref{n_02}) that
\begin{equation}\label{n_03}
	\begin{aligned}
		\frac{d}{dt}\|\nabla\widetilde{\mathbf{B}}_{1}\|^2_{L^2}\Big|_{t=t^*}\leq -\frac{\|\nabla \widetilde{\mathbf{B}}_{1}\|^2_{L^2}
			-C\|n_{0}\|_{L^2}^2}{2AC}\Big|_{t=t^*}
		=-\frac{\|n_0\|^2_{L^{\infty}L^2}-\|n_0\|_{L^2}^2|_{t=t^*}}{2A}\leq0.
	\end{aligned}
\end{equation}	
A contradiction has arisen between $(\ref{n_021})$ and (\ref{n_03}). This shows that
\begin{equation}\label{result_1}
	\|\nabla\widetilde{\mathbf{B}}_{1}\|^2_{L^{\infty}L^2}\leq C\|n_0\|^2_{L^{\infty}L^2}.
\end{equation}
By Poincar\'{e} inequality, there holds
\begin{equation*}
	\begin{aligned}
		\|\widetilde{\mathbf{B}}_{1}\|_{L^2}\leq C\|\nabla\widetilde{\mathbf{B}}_{1}\|_{L^2}.
	\end{aligned}
\end{equation*}
Then it follows from (\ref{result_1}) that
\begin{equation}\label{result_2}
	\begin{aligned}
		\|\widetilde{\mathbf{B}}_{1}\|_{L^{\infty}H^1}\leq C\|\nabla\widetilde{\mathbf{B}}_{1}\|_{L^{\infty}L^2}\leq  C\|n_0\|_{L^{\infty}L^2}.
	\end{aligned}
\end{equation}

Note that the estimate of $\mathbf{G}_{1}$ is similar to $(\ref{eq:u20u30})_{1,2},$ so we omit it. For $(\ref{u_decom_1})_1$, energy estimates show that
\begin{equation}\label{result_3}
	\|\mathbf{G}_1\|^2_{L^{\infty}H^1}+\frac{1}{A}\|\nabla \mathbf{G}_1\|_{L^2H^1}^2\leq C\|(u_{1,\rm in})_{0}\|_{H^{1}}^{2}
	+C\epsilon^2.
\end{equation}
Combining (\ref{result_2}) with (\ref{result_3}), we conclude that
\begin{equation*}
	\|\mathbf{U}_1\|_{L^{\infty}H^1}
	\leq C(\|\widetilde{\mathbf{B}}_{1}\|_{L^{\infty}H^1}
	+\|\mathbf{G}_1\|_{L^{\infty}H^1})
	\leq C\left(\|(u_{1,\rm in})_{0}\|_{H^{1}}+\|n_0\|_{L^{\infty}L^2}+\epsilon\right).
\end{equation*}

The proof is complete.
\end{proof}

{
\begin{lemma}\label{lemma:partial-z-U1}
	Under the assumptions of Lemma \ref{lemma_u23_1}, it holds that
	\begin{equation*}
		\begin{aligned}
			\|\partial_z\mathbf{U}_1\|_{L^\infty H^1}
			\leq C\left(
			\|(u_{1,\rm in})_0\|_{H^2}
			+\|\partial_z n_0\|_{L^\infty L^2}
			+\|n_0\|_{L^\infty L^2}+A^{\frac16}E_{2}^{2}+\epsilon^{2}
			\right).
		\end{aligned}
	\end{equation*}
\end{lemma}
\begin{proof}
	Applying $\partial_z$ to \eqref{u1:1}, we obtain
		\begin{align}\label{equation:partial-z-U1}
			\partial_t\partial_z\mathbf{U}_1
			-\frac{1}{A}\Delta\partial_z\mathbf{U}_1
			={}&
			\frac{1}{A}\partial_z n_0
			-\frac{1}{A}
			\left(
			u_{2,0}\partial_y\partial_z\mathbf{U}_1
			+u_{3,0}\partial_z^2\mathbf{U}_1
			\right)\nonumber
			\\
			&-\frac{1}{A}
			\left(
			\partial_zu_{2,0}\partial_y\mathbf{U}_1
			+\partial_zu_{3,0}\partial_z\mathbf{U}_1
			\right)
			-\frac{1}{A}
			\partial_z
			\left(u_{\neq}\cdot\nabla u_{1,\neq}\right)_0,
		\end{align}
	with the initial condition
	$\partial_z\mathbf{U}_1\big|_{t=0}
		=\partial_z(u_{1,\rm in})_0$.
	Multiplying \eqref{equation:partial-z-U1} by $\left(\partial_z\mathbf{U}_1-\Delta\partial_z\mathbf{U}_1\right)$,
	and integrating over $\mathbb T^2$, we get
	\begin{equation}\label{energy:partial-z-U1}
		\begin{aligned}
			&\frac{1}{2}\frac{d}{dt}
			\|\partial_z\mathbf{U}_1\|_{H^1}^2
			+\frac{1}{A}
			\left(
			\|\nabla\partial_z\mathbf{U}_1\|_{L^2}^2
			+\|\Delta\partial_z\mathbf{U}_1\|_{L^2}^2
			\right)
			\\
			={}&
			-\frac{1}{A}
			\left\langle
			u_{2,0}\partial_y\partial_z\mathbf{U}_1
			+u_{3,0}\partial_z^2\mathbf{U}_1,
			\partial_z\mathbf{U}_1
			-\Delta\partial_z\mathbf{U}_1
			\right\rangle
			\\
			&-\frac{1}{A}
			\left\langle
			\partial_zu_{2,0}\partial_y\mathbf{U}_1
			+\partial_zu_{3,0}\partial_z\mathbf{U}_1,
			\partial_z\mathbf{U}_1
			-\Delta\partial_z\mathbf{U}_1
			\right\rangle
			\\
			&+\frac{1}{A}
			\left\langle
			\partial_z n_0,
			\partial_z\mathbf{U}_1
			-\Delta\partial_z\mathbf{U}_1
			\right\rangle-\frac{1}{A}
			\left\langle
			\partial_z
			\left(u_{\neq}\cdot\nabla u_{1,\neq}\right)_0,
			\partial_z\mathbf{U}_1
			-\Delta\partial_z\mathbf{U}_1
			\right\rangle.
		\end{aligned}
	\end{equation}

We estimate the terms on the right-hand side separately. First, since
	$\partial_yu_{2,0}+\partial_zu_{3,0}=0$,
	integration by parts gives
	\begin{equation*}\label{cancel:transport-L2}
		\left\langle
		u_{2,0}\partial_y\partial_z\mathbf{U}_1
		+u_{3,0}\partial_z^2\mathbf{U}_1,
		\partial_z\mathbf{U}_1
		\right\rangle=0.
	\end{equation*}
	For the gradient part,  the
	two-dimensional Gagliardo-Nirenberg inequality imply that
	\begin{align*}
		&\left|
		\left\langle
		u_{2,0}\partial_y\partial_z\mathbf{U}_1
		+u_{3,0}\partial_z^2\mathbf{U}_1,
		-\Delta\partial_z\mathbf{U}_1
		\right\rangle
		\right|
		\\
		\leq{}&
		C\|\nabla(u_{2,0},u_{3,0})\|_{L^2}
		\|\nabla\partial_z\mathbf{U}_1\|_{L^4}^2\leq
		C\|\nabla(u_{2,0},u_{3,0})\|_{L^2}
		\|\nabla\partial_z\mathbf{U}_1\|_{L^2}
		\|\Delta\partial_z\mathbf{U}_1\|_{L^2}.
	\end{align*}
	Using $(\ref{eq:u20u30})_2$, we obtain
	\begin{equation}\label{estimate:transport-partial-z-U1}
		\begin{aligned}
			&\frac{1}{A}
			\left|
			\left\langle
			u_{2,0}\partial_y\partial_z\mathbf{U}_1
			+u_{3,0}\partial_z^2\mathbf{U}_1,
			\partial_z\mathbf{U}_1
			-\Delta\partial_z\mathbf{U}_1
			\right\rangle
			\right|\leq
			\frac{C\epsilon}{A}
			\left(
			\|\nabla\partial_z\mathbf{U}_1\|_{L^2}^2
			+\|\Delta\partial_z\mathbf{U}_1\|_{L^2}^2
			\right).
		\end{aligned}
	\end{equation}
Next, noting that  the divergence-free condition $\partial_zu_{3,0}=-\partial_yu_{2,0}$ and $(\ref{sob_result_1})_{3,4}$, we have
	\begin{align*}
		&\left\|
		\partial_zu_{2,0}\partial_y\mathbf{U}_1
		\right\|_{L^2}
		\leq
		\|\partial_zu_{2,0}\|_{L_y^\infty L_z^2}
		\|\partial_y\mathbf{U}_1\|_{L_y^2L_z^\infty}\leq
		C\|u_{2,0}\|_{H^2}
		\left(
		\|\mathbf{U}_1\|_{H^1}
		+\|\partial_z\mathbf{U}_1\|_{H^1}
		\right),\\
		&\left\|
		\partial_zu_{3,0}\partial_z\mathbf{U}_1
		\right\|_{L^2}
		=
		\left\|
		\partial_yu_{2,0}\partial_z\mathbf{U}_1
		\right\|_{L^2}\leq
		C\|u_{2,0}\|_{H^2}
		\|\partial_z\mathbf{U}_1\|_{H^1}.
	\end{align*}
	These along with Lemma \ref{lemma_u23_1} and Young's inequality yield
	\begin{equation}\label{estimate:commutator-partial-z-U1}
		\begin{aligned}
			&\frac{1}{A}
			\left|
			\left\langle
			\partial_zu_{2,0}\partial_y\mathbf{U}_1
			+\partial_zu_{3,0}\partial_z\mathbf{U}_1,
			\partial_z\mathbf{U}_1
			-\Delta\partial_z\mathbf{U}_1
			\right\rangle
			\right|
			\\
			\leq{}&
			\frac{1}{16A}
			\left(
			\|\nabla\partial_z\mathbf{U}_1\|_{L^2}^2
			+\|\Delta\partial_z\mathbf{U}_1\|_{L^2}^2
			\right)+\frac{C\epsilon^{2}}{A}
			\left(
			\|\mathbf{U}_1\|_{H^1}^2
			+\|\partial_z\mathbf{U}_1\|_{H^1}^2
			\right).
		\end{aligned}
	\end{equation}
For the density term, the Poincar\'e inequality and Young's
	inequality give
	\begin{equation}\label{estimate:density-partial-z-U1}
		\begin{aligned}
			&\frac{1}{A}
			\left|
			\left\langle
			\partial_z n_0,
			\partial_z\mathbf{U}_1
			-\Delta\partial_z\mathbf{U}_1
			\right\rangle
			\right|\leq
			\frac{1}{16A}
			\left(
			\|\nabla\partial_z\mathbf{U}_1\|_{L^2}^2
			+\|\Delta\partial_z\mathbf{U}_1\|_{L^2}^2
			\right)
			+\frac{C}{A}
			\|\partial_z n_0\|_{L^2}^2.
		\end{aligned}
	\end{equation}
	Similarly,
	\begin{equation}\label{estimate:nonzero-partial-z-U1}
		\begin{aligned}
			&\frac{1}{A}
			\left|
			\left\langle
			\partial_z
			\left(u_{\neq}\cdot\nabla u_{1,\neq}\right)_0,
			\partial_z\mathbf{U}_1
			-\Delta\partial_z\mathbf{U}_1
			\right\rangle
			\right|
			\\
			\leq{}&
			\frac{1}{16A}
			\left(
			\|\nabla\partial_z\mathbf{U}_1\|_{L^2}^2
			+\|\Delta\partial_z\mathbf{U}_1\|_{L^2}^2
			\right)+\frac{C}{A}
			\left\|
			\partial_z
			\left(u_{\neq}\cdot\nabla u_{1,\neq}\right)_0
			\right\|_{L^2}^2.
		\end{aligned}
	\end{equation} 
	Since $\partial_z\mathbf{U}_1$ has zero average with respect to
	the $z$-variable, the Poincar\'e inequality implies
	\begin{equation}\label{poincare:partial-z-U1}
		\|\partial_z\mathbf{U}_1\|_{H^1}^2
		\leq
		C\left(
		\|\nabla\partial_z\mathbf{U}_1\|_{L^2}^2
		+\|\Delta\partial_z\mathbf{U}_1\|_{L^2}^2
		\right).
	\end{equation}
	Therefore, by taking $\epsilon>0$ sufficiently small, the terms
	containing $\epsilon^2\|\partial_z\mathbf{U}_1\|_{H^1}^2$
	can be absorbed by the dissipation on the left-hand side of
	\eqref{energy:partial-z-U1}. Consequently, it follows from 
    \eqref{energy:partial-z-U1}--\eqref{poincare:partial-z-U1} that
	\begin{equation*}
		\begin{aligned}
			\frac{d}{dt}
			\|\partial_z\mathbf{U}_1\|_{H^1}^2
			+\frac{c}{A}
			\|\partial_z\mathbf{U}_1\|_{H^1}^2
			\leq\frac{C}{A}\Big(
			\|\partial_z n_0\|_{L^2}^2
			+\epsilon^{2}\|\mathbf{U}_1\|_{H^1}^2+
			\left\|
			\partial_z
			\left(u_{\neq}\cdot\nabla u_{1,\neq}\right)_0
			\right\|_{L^2}^2
			\Big),
		\end{aligned}
	\end{equation*}
	where $c>0$ is independent of $A$ and $t$. Multiplying the above inequality by
	${\rm e}^{ct/A}$ and integrating over $(0,t)$, we arrive at
	\begin{align*}
		\|\partial_z\mathbf{U}_1(t)\|_{H^1}^2
		\leq{}&
		{\rm e}^{-\frac{ct}{A}}
		\|\partial_z(u_{1,\rm in})_0\|_{H^1}^2+\frac{C}{A}
		\int_0^t
		{\rm e}^{-\frac{c(t-s)}{A}}
		\|\partial_z n_0(s)\|_{L^2}^2 ds
		\\
		&+\frac{C\epsilon^{2}}{A}
		\int_0^t
		{\rm e}^{-\frac{c(t-s)}{A}}
		\|\mathbf{U}_1(s)\|_{H^1}^2 ds+\frac{C}{A}
		\int_0^t
		{\rm e}^{-\frac{c(t-s)}{A}}
		\left\|
		\partial_z
		\left(u_{\neq}\cdot\nabla u_{1,\neq}\right)_0(s)
		\right\|_{L^2}^2 ds.
	\end{align*}
	Due to $\frac{1}{A}\int_0^t {\rm e}^{-\frac{c(t-s)}{A}}ds\leq C$, we obtain
	\begin{align*}
		\|\partial_z\mathbf{U}_1(t)\|_{H^1}^2
		\leq C\Big(
		\|\partial_z(u_{1,\rm in})_0\|_{H^1}^2
		+\|\partial_z n_0\|_{L^\infty L^2}^2
		+\epsilon^2
		\|\mathbf{U}_1\|_{L^\infty H^1}^2+\frac{1}{A}
		\left\|
		\partial_z\left(u_{\neq}\cdot\nabla u_{1,\neq}\right)_0
		\right\|_{L^2L^2}^2
		\Big).
	\end{align*}
	Substituting the estimate of $\|\mathbf{U}_1\|_{L^\infty H^1}$ in Lemma \ref{lemma:u1:1} and $(\ref{eq:non-neq0})_{6}$
	into it, we get
    \begin{equation*}
		\begin{aligned}
			\|\partial_z\mathbf{U}_1\|_{L^\infty H^1}
			\leq C\left(
			\|(u_{1,\rm in})_0\|_{H^2}
			+\|\partial_z n_0\|_{L^\infty L^2}
			+\|n_0\|_{L^\infty L^2}+A^{\frac16}E_{2}^{2}+\epsilon^{2}
			\right).
		\end{aligned}
	\end{equation*}
    
    The proof is complete.
\end{proof}
}

\begin{lemma}\label{u1_hat2}
Under the assumptions of Lemma \ref{lemma_u23_1}, there exists a positive constant $A_{2}$ independent of $A$ and $t$, such that if $A\geq A_{2}$, it holds that
\begin{equation*}
	\begin{aligned}
		\frac{\|\Delta\mathbf{U}_2\|_{L^{\infty}H^2}}{A}
		+\frac{\|\nabla\Delta\mathbf{U}_2\|_{L^{2}H^2}}{A^{\frac32}}
		+\|\partial_t\mathbf{U}_2\|_{L^{\infty}H^2}
		\leq C\epsilon.
	\end{aligned}
\end{equation*}
\end{lemma}
\begin{proof}
{\bf Estimate of $\widetilde{\mathbf{B}}_{2}(t,y,z).$}	For $(\ref{u1:decom2})_1$, the $H^4$ energy estimate shows that
\begin{equation}\label{u1:temp}
	\begin{aligned}
		\frac{\|\Delta^2\widetilde{\mathbf{B}}_{2}\|^2_{L^{\infty}L^2}}{A^{2}}
		+\frac{\|\nabla\Delta^2\widetilde{\mathbf{B}}_{2}\|^2_{L^{2}L^2}}{A^{3}}\leq\frac{C}{A}\|\nabla\Delta\widetilde{u}_{2,0}\|_{L^{2}L^{2}}^{2}+\frac{C}{A^{3}}\|\nabla\Delta\big(u_{2,0}\partial_{y}\widetilde{\mathbf{B}}_{2}+u_{3,0}\partial_{z}\widetilde{\mathbf{B}}_{2} \big)\|_{L^{2}L^{2}}^{2}.
	\end{aligned}
\end{equation}
Using Lemma \ref{sob_inf_1}, Lemma \ref{lemma_u23_1} and Poincar\'{e} inequality, one obtains
\begin{equation}\label{estimate 1}
	\begin{aligned}
		\frac{\|\nabla\Delta(u_{2,0}\partial_{y}\widetilde{\mathbf{B}}_{2})\|_{L^{2}L^{2}}^{2}}{A^{3}}	\leq~&\frac{C}{A^{3}}\left(\|\nabla u_{2,0}\|_{L^{2}H^{2}}^{2}\|\partial_{y}\widetilde{\mathbf{B}}_{2}\|_{L^{\infty}H^{2}}^{2}+\|u_{2,0}\|_{L^{\infty}H^{2}}^{2}\|\nabla\partial_{y}\widetilde{\mathbf{B}}_{2}\|_{L^{2}H^{2}}^{2} \right)\\\leq~&C\epsilon^{2}\Big(\frac{\|\Delta^{2}\widetilde{\mathbf{B}}_{2}\|_{L^{\infty}L^{2}}^{2}}{A^{2}}+\frac{\|\nabla\Delta^{2}\widetilde{\mathbf{B}}_{2}\|_{L^{2}L^{2}}^{2}}{A^{3}} \Big),
	\end{aligned}
\end{equation}
\begin{equation}\label{estimate 2}
	\begin{aligned}
		\frac{\|u_{3,0}\nabla\Delta\partial_{z}\widetilde{\mathbf{B}}_{2}+\nabla u_{3,0}\Delta\partial_{z}\widetilde{\mathbf{B}}_{2}\|_{L^{2}L^{2}}^{2}}{A^{3}}\leq\frac{C}{A^{3}}\|u_{3,0}\|_{L^{\infty}H^{1}}^{2}\|\partial_{z}\widetilde{\mathbf{B}}_{2}\|_{L^{2}H^{4}}^{2}\leq C\epsilon^{2}\frac{\|\nabla\Delta^{2}\widetilde{\mathbf{B}}_{2}\|_{L^{2}L^{2}}^{2}}{A^{3}},
	\end{aligned}
\end{equation}
and
\begin{equation}\label{estimate 3}
	\frac{\|\Delta u_{3,0}\nabla\partial_{z}\widetilde{\mathbf{B}}_{2}\|_{L^{2}L^{2}}^{2}}{A^{3}}\leq\frac{C}{A^{3}}\|\Delta u_{3,0}\|_{L^{2}L^{2}}^{2}\|\widetilde{\mathbf{B}}_{2}\|_{L^{\infty}H^{4}}^{2}\leq C\epsilon^{2}\frac{\|\Delta^{2}\widetilde{\mathbf{B}}_{2}\|_{L^{\infty}L^{2}}^{2}}{A^{2}}.
\end{equation}
Due to $\|\widetilde{\mathbf{B}}_{2}\|_{H^{3}}^{2}\leq\|\widetilde{\mathbf{B}}_{2}\|_{H^{2}}\|\widetilde{\mathbf{B}}_{2}\|_{H^{4}},$ and
\begin{equation*}
	\|\widetilde{\mathbf{B}}_{2}\|_{H^{2}}\leq\int_{0}^{t}\|\partial_{s}\widetilde{\mathbf{B}}_{2}(s)\|_{H^{2}}ds\leq t\|\partial_{t}\widetilde{\mathbf{B}}_{2}\|_{L^{\infty}H^{2}},
\end{equation*}
there holds
\begin{equation}\label{B2 t}
	\frac{\|\widetilde{\mathbf{B}}_{2}\|_{H^{3}}^{2}}{A^{-1}t}\leq\frac{\|\widetilde{\mathbf{B}}_2\|_{H^{2}}\|\widetilde{\mathbf{B}}_2\|_{H^{4}}}{A^{-1}t}\leq A\|\partial_{t}\widetilde{\mathbf{B}}_{2}\|_{L^{\infty}H^{2}}\|\widetilde{\mathbf{B}}_{2}\|_{H^{4}}.
\end{equation}
From this, along with $\|\nabla\Delta u_{3,0}\partial_{z}\widetilde{\mathbf{B}}_{2}\|_{L^{2}}^{2}\leq C\|\nabla\Delta u_{3,0}\|_{L^{2}}^{2}\|\partial_{z}\widetilde{\mathbf{B}}_{2}\|_{H^{2}}^{2},$ we get
\begin{equation*}
	\begin{aligned}
		\|\nabla\Delta u_{3,0}\partial_{z}\widetilde{\mathbf{B}}_{2}\|_{L^{2}}^{2}\leq~&C\|\min\{(A^{-\frac23}+A^{-1}t)^{\frac12}, 1\}\nabla\Delta u_{3,0}\|_{L^{2}}^{2}\bigg(\frac{\|\widetilde{\mathbf{B}}_{2}\|_{H^{3}}^{2}}{A^{-1}t}+\|\widetilde{\mathbf{B}}_{2}\|_{H^{4}}^{2} \bigg)\\\leq~&C\|\min\{(A^{-\frac23}+A^{-1}t)^{\frac12}, 1\}\nabla\Delta u_{3,0}\|_{L^{2}}^{2}\Big(A^{2}\|\partial_{t}\widetilde{\mathbf{B}}_{2}\|_{L^{\infty}H^{2}}^{2}+\|\widetilde{\mathbf{B}}_{2}\|_{H^{4}}^{2} \Big),
	\end{aligned}
\end{equation*}
which implies that
\begin{equation}\label{estimate 4}
	\frac{\|\nabla\Delta u_{3,0}\partial_{z}\widetilde{\mathbf{B}}_{2}\|_{L^{2}L^{2}}^{2}}{A^{3}}\leq C\epsilon^{2}\Big(\|\partial_{t}\widetilde{\mathbf{B}}_{2}\|_{L^{\infty}H^{2}}^{2}+\frac{\|\Delta^{2}\widetilde{\mathbf{B}}_{2}\|_{L^{\infty}L^{2}}^{2}}{A^{2}} \Big),
\end{equation}
where we use Lemma \ref{lemma_u23_1} and Poincar$\acute{\rm e}$ inequality.

Combining (\ref{estimate 1}), (\ref{estimate 2}), (\ref{estimate 3}) and (\ref{estimate 4}), we get from (\ref{u1:temp}) that
\begin{equation}\label{estimate B2}
	\begin{aligned}
		&\frac{\|\Delta^{2}\widetilde{\mathbf{B}}_{2}\|_{L^{\infty}L^{2}}^{2}}{A^{2}}+\frac{\|\nabla\Delta^{2}\widetilde{\mathbf{B}}_{2}\|_{L^{2}L^{2}}^{2}}{A^{3}}\\\leq~&C\|\Delta u_{2,0}\|_{Y_{0}}^{2}+C\epsilon^{2}\|\partial_{t}\widetilde{\mathbf{B}}_{2}\|_{L^{\infty}H^{2}}^{2}+C\epsilon^{2}\bigg(\frac{\|\Delta^{2}\widetilde{\mathbf{B}}_{2}\|_{L^{\infty}L^{2}}^{2}}{A^{2}}+\frac{\|\nabla\Delta^{2}\widetilde{\mathbf{B}}_{2}\|_{L^{2}L^{2}}^{2}}{A^{3}} \bigg).
	\end{aligned}
\end{equation}
When $\epsilon$ is small enough satisfying $C\epsilon^{2}\leq\frac12,$ using Lemma \ref{lemma_u23_1}, (\ref{estimate B2}) yields
\begin{equation}\label{estimate B2 1}
	\begin{aligned}
		\frac{\|\Delta^{2}\widetilde{\mathbf{B}}_{2}\|_{L^{\infty}L^{2}}^{2}}{A^{2}}+\frac{\|\nabla\Delta^{2}\widetilde{\mathbf{B}}_{2}\|_{L^{2}L^{2}}^{2}}{A^{3}}\leq C\epsilon^{2}+C\epsilon^{2}\|\partial_{t}\widetilde{\mathbf{B}}_{2}\|_{L^{\infty}H^{2}}^{2}.
	\end{aligned}
\end{equation}
From $(\ref{u1:decom2})_{1}$ and (\ref{u20 u30 infty}), direct calculations indicate that
\begin{equation}\label{B2't}
	\begin{aligned}
		\|\partial_{t}\widetilde{\mathbf{B}}_{2}\|_{L^{\infty}L^{2}}\leq~&\frac{\|\Delta\widetilde{\mathbf{B}}_{2}\|_{L^{\infty}L^{2}}}{A}+\|\widetilde{u}_{2,0}\|_{L^{\infty}L^{2}}+\frac{1}{A}\left(\|u_{2,0}\|_{L^{\infty}L^{\infty}}+\|u_{3,0}\|_{L^{\infty}L^{\infty}} \right)\|\nabla\widetilde{\mathbf{B}}_{2}\|_{L^{\infty}L^{2}}\\\leq~&\frac{\|\Delta\widetilde{\mathbf{B}}_{2}\|_{L^{\infty}L^{2}}}{A}+\|\widetilde{u}_{2,0}\|_{L^{\infty}L^{2}}+C\epsilon\frac{\|\nabla\widetilde{\mathbf{B}}_{2}\|_{L^{\infty}L^{2}}}{A}.
	\end{aligned}
\end{equation}
Taking $\Delta$ for $(\ref{u1:decom2})_{1},$ there holds
\begin{equation*}
	\begin{aligned}
		\partial_{t}\Delta\widetilde{\mathbf{B}}_{2}-\frac{1}{A}\Delta^{2}\widetilde{\mathbf{B}}_{2}=-\frac{1}{A}\Delta\left(u_{2,0}\partial_{y}\widetilde{\mathbf{B}}_{2}+u_{3,0}\partial_{z}\widetilde{\mathbf{B}}_{2} \right)-\Delta\widetilde{u}_{2,0},
	\end{aligned}
\end{equation*}
which follows that
\begin{equation}\label{B2't 1}
	\|\partial_{t}\Delta\widetilde{\mathbf{B}}_{2}\|_{L^{\infty}L^{2}}^{2}\leq\frac{\|\Delta^{2}\widetilde{\mathbf{B}}_{2}\|_{L^{\infty}L^{2}}^{2}}{A^{2}}+\frac{\|\Delta(u_{2,0}\partial_{y}\widetilde{\mathbf{B}}_{2}+u_{3,0}\partial_{z}\widetilde{\mathbf{B}}_{2})\|_{L^{\infty}L^{2}}^{2}}{A^{2}}+\|\Delta\widetilde{u}_{2,0}\|_{L^{\infty}L^{2}}^{2}.
\end{equation}
Using Lemma \ref{sob_inf_1}, Lemma \ref{lemma_u23_1} and Poincar$\acute{\rm e}$ inequality, we have
\begin{equation}\label{11}
	\frac{\|\Delta(u_{2,0}\partial_{y}\widetilde{\mathbf{B}}_{2})\|_{L^{\infty}L^{2}}^{2}}{A^{2}}\leq \frac{C}{A^{2}}\|u_{2,0}\|_{L^{\infty}H^{2}}^{2}\|\partial_{y}\widetilde{\mathbf{B}}_{2}\|_{L^{\infty}H^{2}}^{2}\leq C\epsilon^{2}\frac{\|\Delta^{2}\widetilde{\mathbf{B}}_{2}\|_{L^{\infty}L^{2}}^{2}}{A^{2}}
\end{equation}
and
\begin{equation}\label{12}
	\begin{aligned}
		\frac{\|u_{3,0}\Delta\partial_{z}\widetilde{\mathbf{B}}_{2}+\nabla u_{3,0}\cdot\nabla\partial_{z}\widetilde{\mathbf{B}}_{2}\|_{L^{\infty}L^{2}}^{2}}{A^{2}}\leq \frac{C}{A^{2}}\|u_{3,0}\|_{L^{\infty}H^{1}}^{2}\|\widetilde{\mathbf{B}}_{2}\|_{L^{\infty}H^{4}}^{2}\leq C\epsilon^{2}\frac{\|\Delta^{2}\widetilde{\mathbf{B}}_{2}\|_{L^{\infty}L^{2}}^{2}}{A^{2}}.
	\end{aligned}
\end{equation}
Note that $\|\Delta u_{3,0}\partial_{z}\widetilde{\mathbf{B}}_{2}\|_{L^{2}}^{2}\leq C\|\Delta u_{3,0}\|_{L^{2}}^{2}\|\widetilde{\mathbf{B}}_{2}\|_{H^{3}}^{2},$  by (\ref{B2 t}), and we get
\begin{equation*}
	\begin{aligned}
		\|\Delta u_{3,0}\partial_{z}\widetilde{\mathbf{B}}_{2}\|_{L^{2}}^{2}\leq~& C\|\min\{(A^{-\frac23}+A^{-1}t)^{\frac12}, 1\}\Delta u_{3,0}\|_{L^{2}}^{2}\bigg(\frac{\|\widetilde{\mathbf{B}}_{2}\|_{H^{3}}^{2}}{A^{-1}t}+\|\Delta\partial_{z}\widetilde{\mathbf{B}}_{2}\|_{L^{2}}^{2} \bigg)\\\leq~&C\|\min\{(A^{-\frac23}+A^{-1}t)^{\frac12}, 1\}\Delta u_{3,0}\|_{L^{2}}^{2}\bigg(A^{2}\|\partial_{t}\widetilde{\mathbf{B}}_{2}\|_{H^{2}}^{2}+\|\Delta^{2}\widetilde{\mathbf{B}}_{2}\|_{L^{2}}^{2} \bigg),
	\end{aligned}
\end{equation*}
which along with Lemma \ref{lemma_u23_1} implies that
\begin{equation}\label{13}
	\frac{\|\Delta u_{3,0}\partial_{z}\widetilde{\mathbf{B}}_{2}\|_{L^{\infty}L^{2}}^{2}}{A^{2}}\leq C\epsilon^{2}\bigg(\|\partial_{t}\widetilde{\mathbf{B}}_{2}\|_{L^{\infty}H^{2}}^{2}+\frac{\|\Delta^{2}\widetilde{\mathbf{B}}_{2}\|_{L^{\infty}L^{2}}^{2}}{A^{2}} \bigg).
\end{equation}
Collecting (\ref{11}), (\ref{12}) and (\ref{13}),  we get from (\ref{B2't 1}) that
\begin{equation*}
	\begin{aligned}
		\|\partial_{t}\Delta\widetilde{\mathbf{B}}_{2}\|_{L^{\infty}L^{2}}^{2}\leq C\frac{\|\Delta^{2}\widetilde{\mathbf{B}}_{2}\|_{L^{\infty}L^{2}}^{2}}{A^{2}}+C\epsilon^{2}\|\partial_{t}\widetilde{\mathbf{B}}_{2}\|_{L^{\infty}H^{2}}^{2}+\|\Delta\widetilde{u}_{2,0}\|_{L^{\infty}L^{2}}^{2}.
	\end{aligned}
\end{equation*}
Combining it with (\ref{B2't}), Lemma \ref{lemma_u23_1} and Poincar$\acute{\rm e}$ inequality, one deduces
\begin{equation}\label{B2 t end}
	\begin{aligned}
		\|\partial_{t}\widetilde{\mathbf{B}}_{2}\|_{L^{\infty}H^{2}}^{2}\leq C\frac{\|\Delta^{2}\widetilde{\mathbf{B}}_{2}\|_{L^{\infty}L^{2}}^{2}}{A^{2}}+C\epsilon^{2}\|\partial_{t}\widetilde{\mathbf{B}}_{2}\|_{L^{\infty}H^{2}}^{2}+C\epsilon^{2}.
	\end{aligned}
\end{equation}
Substituting (\ref{estimate B2 1}) into (\ref{B2 t end}), we have
\begin{equation*}
	\|\partial_{t}\widetilde{\mathbf{B}}_{2}\|_{L^{\infty}H^{2}}^{2}\leq C\epsilon^{2}\|\partial_{t}\widetilde{\mathbf{B}}_{2}\|_{L^{\infty}H^{2}}^{2}+C\epsilon^{2},
\end{equation*}
which indicates that
\begin{equation}\label{B2't end}
	\|\partial_{t}\widetilde{\mathbf{B}}_{2}\|_{L^{\infty}H^{2}}^{2}\leq C\epsilon^{2}
\end{equation}
provided that $\epsilon$ is sufficiently small satisfying $C\epsilon^{2}\leq\frac12.$ By (\ref{B2't end}), we infer from (\ref{estimate B2 1}) that
\begin{equation}\label{estimate B2 end}
	\frac{\|\Delta^{2}\widetilde{\mathbf{B}}_{2}\|_{L^{\infty}L^{2}}^{2}}{A^{2}}+\frac{\|\nabla\Delta^{2}\widetilde{\mathbf{B}}_{2}\|_{L^{2}L^{2}}^{2}}{A^{3}}\leq C\epsilon^{2}.
\end{equation}
For $j\in\{2,3\}$, combining (\ref{B2't end}) with (\ref{estimate B2 end}), and  using Poincar$\acute{\rm e}$ inequality
\begin{equation*}
	\|\nabla^{j}\widetilde{\mathbf{B}}_{2}\|_{L^{2}}^{2}\leq C\|\nabla^{j+1}\widetilde{\mathbf{B}}_{2}\|_{L^{2}}^{2},
\end{equation*}
one obtains
\begin{equation}\label{estimate U2 1}
	\frac{\|\Delta\widetilde{\mathbf{B}}_{2}\|_{L^{\infty}H^{2}}}{A}+\frac{\|\nabla\Delta\widetilde{\mathbf{B}}_{2}\|_{L^{2}H^{2}}}{A^{\frac32}}+\|\partial_{t}\widetilde{\mathbf{B}}_{2}\|_{L^{\infty}H^{2}}\leq C\epsilon.
\end{equation}

{\bf Estimate of $\overline{\mathbf{B}}_{2}(t).$}
According to $\eqref{u_zero_1}_1,$ $\overline{u}_{2,0}$ satisfies
$\partial_t\overline{u}_{2,0}=0,$ which implies that
\begin{equation*}
	\overline{u}_{2,0}(t)=(\overline{u}_{2,\rm in})_0.
\end{equation*}
Therefore, we rewrite $\eqref{u1:decom2}_{2}$ into
\begin{equation*}
	\partial_t\overline{\mathbf{B}}_{2}(t)=-\overline{u}_{2,0}(t)
	=-(\overline{u}_{2,\rm in})_0.
\end{equation*}
By H\"{o}lder's inequality, there holds
\begin{equation}\label{u1:t1}
	|\partial_t\overline{\mathbf{B}}_{2}(t)|\leq C\|(u_{2,\rm in})_0\|_{H^2}.
\end{equation}
Furthermore, $\overline{\mathbf{B}}_{2}(t)$ only depends on $t$ and does not depend on any spatial variables,
and  $\partial_t\overline{\mathbf{B}}_{2}$ is a constant.  Using assumption (\ref{conditions:u20 u30}),
we infer from \eqref{u1:t1} that
\begin{equation}\label{u1:res:3}
	\frac{\|\Delta\overline{\mathbf{B}}_{2}\|_{L^{\infty}H^2}}{A}
	+\frac{\|\nabla\Delta\overline{\mathbf{B}}_{2}\|_{L^{2}H^2}}{A^{\frac32}}
	+\|\partial_t\overline{\mathbf{B}}_{2}\|_{L^{\infty}H^2}
	\leq C\|(u_{2,\rm in})_0\|_{H^2}\leq C\epsilon.	
\end{equation}

{\bf Estimate of $\overline{\mathbf{B}}_{1}(t).$}
Next, we rewrite $\eqref{u1:decom1}_{2}$ into
\begin{equation*}
	\partial_t\overline{\mathbf{B}}_{1}(t)=\frac{M}{A|\mathbb{T}|^3}.
\end{equation*}
Moreover, $\overline{\mathbf{B}}_{1}(t)$ only depends on $t$ and does not depend on any spatial variables, and  $\partial_t\overline{\mathbf{B}}_{1}$ is a constant.  When $A\geq\max\{1, \frac{M}{\epsilon}, A_{2,3}, A_{2,4}\}=:A_{2},$
we obtain that
\begin{equation}\label{u1:res:4}
	\frac{\|\Delta\overline{\mathbf{B}}_{1}\|_{L^{\infty}H^2}}{A}
	+\frac{\|\nabla\Delta\overline{\mathbf{B}}_{1}\|_{L^{2}H^2}}{A^{\frac{3}{2}}}
	+\|\partial_t\overline{\mathbf{B}}_{1}\|_{L^{\infty}H^2}
	\leq \frac{CM}{A}\leq C\epsilon.	
\end{equation}

Combining (\ref{estimate U2 1}), \eqref{u1:res:3} and \eqref{u1:res:4}, we conclude that
\begin{equation*}
	\begin{aligned}
		\frac{\|\Delta\mathbf{U}_2\|_{L^{\infty}H^2}}{A}
		+\frac{\|\nabla\Delta\mathbf{U}_2\|_{L^{2}H^2}}{A^{\frac32}}
		+\|\partial_t\mathbf{U}_2\|_{L^{\infty}H^2}
		\leq C\epsilon.
	\end{aligned}
\end{equation*}

The proof is complete.
\end{proof}

\begin{corollary}\label{cor:E1}
Under the conditions of Theorem \ref{result0} and the assumptions \eqref{conditions:u20 u30} and \eqref{assumption}, according to Lemma \ref{lemma_u23_1} and Lemma \ref{u1_hat2}, when $A\geq A_{2},$ there holds
\begin{equation*}
	E_{1}(t)\leq C\epsilon=:E_{1}.
\end{equation*}
\end{corollary}

\section{Estimates for the non-zero modes}\label{sec5}
\subsection{Energy estimates for $E_{2,1}(t)$}
\begin{lemma}\label{lem:E_21}
Under the conditions of Theorem \ref{result0}, the assumptions \eqref{conditions:u20 u30} and \eqref{assumption}, there exists a positive constant $A_{3}$ independent of $t$ and $A$, such that if $A\geq A_{3},$ there holds
\begin{equation*}
	E_{2,1}(t)\leq C\left(\|(\partial_{x}^{2}n_{\rm in})_{\neq}\|_{L^{2}}+1 \right).
\end{equation*}
\end{lemma}
\begin{proof}
According to $(\ref{ini11})_{1},$ the non-zero mode $\partial_{x}^{2}n_{\neq}$ satisfies
\begin{equation}\label{n:neq1}
	\begin{aligned}
		\partial_{t}\partial_{x}^{2}n_{\neq}+y\partial_{x}^{3}n_{\neq}-\frac{1}{A}\Delta\partial_{x}^{2}n_{\neq}=&-\frac{1}{A}\nabla\cdot\partial_{x}^{2}(un)_{\neq}-\frac{1}{A}\nabla\cdot\partial_{x}^{2}(n\nabla c)_{\neq}.
	\end{aligned}
\end{equation}
Noting that for given functions $f$ and $g$,  we have
\begin{equation}\label{fg}
	(fg)_{\neq}=f_{0}g_{\neq}+f_{\neq}g_{0}+(f_{\neq}g_{\neq})_{\neq}.
\end{equation}
Therefore, by decomposing $u_{1,0}=\mathbf{U}_1+\mathbf{U}_2$, we rewrite (\ref{n:neq1}) into
\begin{equation*}\label{n:neq2}
	\begin{aligned}
		\partial_{t}\partial_{x}^{2}n_{\neq}+\left(y+\frac{\mathbf{U}_2}{A}\right)\partial_{x}^{3}n_{\neq}-\frac{\Delta\partial_{x}^{2}n_{\neq}}{A}=&-\frac{1}{A}\nabla\cdot\partial_{x}^{2}(u_{\neq}n_{\neq})_{\neq}-\frac{1}{A}\nabla\cdot\partial_{x}^{2}(n\nabla c)_{\neq}\\
		&-\frac{1}{A}\nabla\cdot(\partial_{x}^{2}u_{\neq}n_{0})-\frac{1}{A}\nabla\cdot\partial_{x}^{2}(U_{0}n_{\neq}),
	\end{aligned}
\end{equation*}
where $U_{0}=(\mathbf{U}_1,u_{2,0},u_{3,0}).$

Moreover, a careful deformation of the nonlinear interaction term $\partial_x^3(u_{1,\neq}n_{\neq})_{\neq}$ plays a crucial role.
Specifically, it can be expanded as follows:
\begin{equation*}
	\begin{aligned}
		\partial_x^3(u_{1,\neq}n_{\neq})_{\neq}=(\partial_x^3u_{1,\neq}n_{\neq})_{\neq}
		+(u_{1,\neq}\partial_x^3n_{\neq})_{\neq}+3\partial_x(\partial_xu_{1,\neq}\partial_xn_{\neq})_{\neq}.
	\end{aligned}
\end{equation*}
When $A>Cc^{-1},$
it follows from (\ref{conditions:u20 u30}) and Lemma \ref{u1_hat2} that $\frac{\|\Delta\mathbf{U}_2\|_{L^{\infty}H^{2}}}{A}+\|\partial_{t}\mathbf{U}_2\|_{L^{\infty}L^{\infty}}<C\epsilon$. By applying Proposition \ref{Lvf 0}, we get
\begin{equation*}\label{n1}
	\begin{aligned}
		\|\partial_{x}^{2}n_{\neq}\|_{X_{b}}^{2}\leq~& C\Big(\|(\partial_{x}^{2}n_{\rm in})_{\neq}\|_{L^{2}}^{2}
		+\frac{1}{A}\|{\rm e}^{bA^{-\frac13}t}U_{0}\partial_{x}^{2}n_{\neq}\|_{L^{2}L^{2}}^{2}
		+\frac{1}{A^{\frac53}}\|{\rm e}^{bA^{-\frac13}t}n_0\partial_{x}^{3}u_{1,\neq}\|_{L^{2}L^{2}}^{2}\\
		&+\frac{1}{A}\|{\rm e}^{bA^{-\frac13}t}n_{0}\partial_{x}^{2}(u_{2}, u_{3})_{\neq}\|_{L^{2}L^{2}}^{2}
		+\frac{1}{A^{\frac53}}\|{\rm e}^{bA^{-\frac13}t}n_{\neq}\partial_{x}^{3}u_{1,\neq}\|_{L^{2}L^{2}}^{2}
		\\&+\frac{1}{A^{\frac53}}\|{\rm e}^{bA^{-\frac13}t}u_{1,\neq}\partial_{x}^{3}n_{\neq}\|_{L^{2}L^{2}}^{2}
		+\frac{1}{A}\|{\rm e}^{bA^{-\frac13}t}\partial_{x}u_{1,\neq}\partial_{x}n_{\neq}\|_{L^{2}L^{2}}^{2}
		\\&+\frac{1}{A}\|{\rm e}^{bA^{-\frac13}t}\partial_{x}^{2}\left((u_{2}, u_{3})_{\neq}n_{\neq} \right)\|_{L^{2}L^{2}}^{2}
		+\frac{1}{A}\|{\rm e}^{bA^{-\frac13}t}\partial_{x}^{2}(n\nabla c)_{\neq}\|_{L^{2}L^{2}}^{2}
		\Big)\\=:~&C\left(\|(\partial_{x}^{2}n_{\rm in})_{\neq}\|^2_{L^{2}}+T_{1,1}+T_{1,2}+\cdots+T_{1,8} \right).
	\end{aligned}
\end{equation*}

{\bf Estimate of $T_{1,1}.$}
By Lemma \ref{lem:GNS}, we have
\begin{equation}\label{GN n}
	\|\partial_{x}^{2}n_{\neq}\|_{L^{2}L^{4}}\leq C\|\partial_{x}^{2}n_{\neq}\|_{L^{2}L^{2}}^{\frac14}\|\nabla\partial_{x}^{2}n_{\neq}\|_{L^{2}L^{2}}^{\frac34}
\end{equation}
and
\begin{equation*}\label{GN n1}
	\begin{aligned}
		\|U_0\|_{L^{4}}\leq C\|U_0\|_{H^{1}}\leq C(\|\mathbf{U}_1\|_{H^{1}}+\|u_{2,0}\|_{H^{1}}+\|u_{3,0}\|_{H^{1}}).		
	\end{aligned}
\end{equation*}
Then using Lemma \ref{lemma_u23_1}, Lemma \ref{lemma:u1:1} and $\|n_0\|^2_{L^{\infty}L^2}\leq CME_3$, we obtain
\begin{equation*}
	\begin{aligned}
		\|{\rm e}^{bA^{-\frac13}t}U_{0}\partial_{x}^{2}n_{\neq}\|_{L^{2}L^{2}}^{2}
		\leq& \|U_{0}\|_{L^{\infty}L^4}^{2}\|{\rm e}^{bA^{-\frac13}t}\partial_{x}^{2}n_{\neq}\|_{L^{2}L^{4}}^{2}\\
		\leq& C\|U_{0}\|_{L^{\infty}H^{1}}^{2}\|{\rm e}^{bA^{-\frac13}t}\partial_{x}^{2}n_{\neq}\|_{L^{2}L^{2}}^{\frac12}
		\|{\rm e}^{bA^{-\frac13}t}\nabla\partial_{x}^{2}n_{\neq}\|_{L^{2}L^{2}}^{\frac32}
		\leq CA^{\frac{5}{6}}H_{2}^{2} E_2^2,
	\end{aligned}
\end{equation*}
where $H_2=\|(u_{1,\rm in})_{0}\|_{H^{1}}
+E_3+M+1.$

{\bf Estimates of $T_{1,2}$ and $T_{1,3}.$}
Using $\|n_{0}\|_{L^{\infty}L^{\infty}}\leq \|n\|_{L^{\infty}L^{\infty}}\leq 2E_{3}$
and $\partial_{x}u_{1,\neq}+\partial_{y}u_{2,\neq}+\partial_{z}u_{3,\neq}=0,$ we arrive at
\begin{equation}\label{n0 nneq}
	\begin{aligned}
		&\|{\rm e}^{bA^{-\frac13}t}n_{0}\partial_{x}^{3}u_{1,\neq}\|_{L^{2}L^{2}}^{2}\leq CE_{3}^{2}\|{\rm e}^{bA^{-\frac13}t}\partial_{x}^{3}u_{1,\neq}\|_{L^{2}L^{2}}^{2}\leq CAE_3^2E_{4}^{2},\\&
		\|{\rm e}^{bA^{-\frac13}t}n_{0}\partial_{x}^{2}(u_{2}, u_{3})_{\neq}\|_{L^{2}L^{2}}^{2}\leq CE_{3}^{2}\|{\rm e}^{bA^{-\frac13}t}\partial_{x}^{2}(u_{2}, u_{3})_{\neq}\|_{L^{2}L^{2}}^{2}\leq CA^{\frac13}E_3^2E_{4}^{2}.
	\end{aligned}
\end{equation}

{\bf Estimate of $T_{1,4}.$}
Since $\|n_{\neq}\|_{L^{\infty}L^{\infty}}\leq 2\|n\|_{L^{\infty}L^{\infty}}\leq 4E_{3}$, by an argument similar to  $(\ref{n0 nneq})_{1},$ one obtains
\begin{equation*}\label{u1neq1}
	\begin{aligned}
		\|{\rm e}^{bA^{-\frac13}t}n_{\neq}\partial_{x}^{3}u_{1,\neq}\|_{L^{2}L^{2}}^{2}&\leq C\|n\|_{L^{\infty}L^{\infty}}^2\|{\rm e}^{bA^{-\frac13}t}\partial_{x}^{3}u_{1,\neq}\|_{L^{2}L^{2}}^{2}\\
		&\leq CAE_{3}^{2}\|\partial_{x}^{2}(u_{2},u_{3})_{\neq}\|_{X_{b}}^{2}\leq CAE_{3}^{2}E_4^2.				
	\end{aligned}
\end{equation*}

{\bf Estimate of $T_{1,5}.$}
By $\eqref{ap:u11}_1,$ we have
\begin{equation*}
	\begin{aligned}
		\|{\rm e}^{bA^{-\frac13}t}u_{1,\neq}\partial_{x}^{3}n_{\neq}\|_{L^{2}L^{2}}^{2}
		\leq \|u_{1,\neq}\|_{L^{\infty}L^{\infty}}^2
		\|{\rm e}^{bA^{-\frac13}t}\partial_{x}^{3}n_{\neq}\|_{L^{2}L^{2}}^{2}
		\leq CA^{\frac43}E_{2}^{4}.				
	\end{aligned}
\end{equation*}

{\bf Estimate of $T_{1,6}.$}
By the periodic Sobolev inequalities and Lemma \ref{lemma_u}, we obtain
\begin{align*}
\|e^{bA^{-\frac13}t}\partial_xu_{1,\neq}\partial_xn_{\neq}\|_{L^2L^2}^2
&\le C\|e^{aA^{-\frac13}t}(\partial_x,\partial_z)
                   \partial_xu_{1,\neq}\|_{L^\infty L^2}^2
     \|e^{aA^{-\frac13}t}\partial_x^2n_{\neq}\|_{L^2L^2}\\
&\qquad\times
\bigl(\|e^{aA^{-\frac13}t}\partial_y\partial_x^2n_{\neq}\|_{L^2L^2}
      +\|e^{aA^{-\frac13}t}\partial_x^2n_{\neq}\|_{L^2L^2}\bigr)\\
&\quad\le C E_2^4(A^{\frac23}+A^{\frac13})\le CA^{\frac23}E_2^4.
\end{align*}

{\bf Estimate of $T_{1,7}.$}
According to \eqref{ap:u11}, there holds
\begin{equation}\label{422_1}
	\begin{aligned}
		\|{\rm e}^{bA^{-\frac{1}{3}}t}(u_{2},u_{3})_{\neq} \partial_x^{2}n_{\neq}\|_{L^2L^2}^{2}
		\leq \|u_{\neq}\|_{L^{\infty}L^{\infty}}^{2} \|{\rm e}^{bA^{-\frac{1}{3}}t}\partial_x^{2} n_{\neq}\|_{L^{2}L^{2}}^{2}
		\leq CA^{\frac{2}{3}}E_2^{4}.
	\end{aligned}
\end{equation}
In addition, there also holds
\begin{equation}\label{422_2}
	\begin{aligned}
		\|{\rm e}^{bA^{-\frac{1}{3}}t}\partial_x^{2}(u_{2},u_{3})_{\neq} n_{\neq}\|_{L^2L^2}^{2}
		\leq \| n_{\neq}\|_{L^{\infty}L^{\infty}}^{2}\|{\rm e}^{bA^{-\frac{1}{3}}t}\partial_x^{2}(u_{2},u_{3})_{\neq}\|_{L^2L^2}^{2}
		\leq CA^{\frac{1}{3}}E_3^2E_4^{2}.
	\end{aligned}
\end{equation}
Using Lemma \ref{sob_inf_2}, we get
\begin{equation*}
	\begin{aligned}
		\|\partial_{x}(u_{2},u_{3})_{\neq}\partial_{x}n_{\neq}\|_{L^{2}}^{2}\leq~&\|\partial_{x}u_{\neq}\|_{L^{\infty}_{y,z}L^{2}_{x}}^{2}\|\partial_{x}n_{\neq}\|_{L^{2}_{y,z}L^{\infty}_{x}}^{2}\\\leq~& C\|(\partial_{x}, \partial_z)\partial_{x}u_{\neq}\|_{L^{2}}\|(\partial_x,\partial_{z})\partial_{x}u_{\neq}\|_{H^{1}}\|\partial_{x}^{2}n_{\neq}\|_{L^{2}}^{2},
	\end{aligned}
\end{equation*}
which along with  Lemma \ref{lemma_u} shows that
\begin{equation}\label{422_3}
	\begin{aligned}
		\|{\rm e}^{bA^{-\frac13}t}\partial_{x}u_{\neq}\partial_{x}n_{\neq}\|_{L^{2}L^{2}}^{2}
		\leq CA^{\frac23}E_{2}^{4}.
	\end{aligned}
\end{equation}
It follows from (\ref{422_1}), (\ref{422_2}) and (\ref{422_3}) that
\begin{equation*}
	\begin{aligned}
		\|{\rm e}^{bA^{-\frac{1}{3}}t}\partial_x^{2}((u_2,u_3)_{\neq} n_{\neq})_{\neq}\|_{L^2L^2}^2
		\leq CA^{\frac{2}{3}}(E_2^4+E_3^2E_4^2).
	\end{aligned}
\end{equation*}

{\bf Estimate of $T_{1,8}.$}
Due to (\ref{fg}), there holds
\begin{equation}\label{425_1}
	\begin{aligned}
		&\|{\rm e}^{bA^{-\frac{1}{3}}t}\partial_x^{2}(n\nabla c)_{\neq}\|_{L^2L^2}^{2}\\\leq~&C\left(
		\|{\rm e}^{bA^{-\frac{1}{3}}t}n_0\partial_x^{2}\nabla c_{\neq}\|_{L^2L^2}^{2}
		+\|{\rm e}^{bA^{-\frac{1}{3}}t}\partial_x^{2}n_{\neq}\nabla c_{0}\|_{L^2L^2}^{2}
		+\|{\rm e}^{bA^{-\frac{1}{3}}t}\partial_x^{2}(n_{\neq} \nabla c_{\neq})_{\neq}\|_{L^2L^2}^{2}\right).
	\end{aligned}
\end{equation}
Applying Lemma \ref{lem:ellip_0}, Lemma \ref{lem:ellip_2} and $\|n_{0}\|_{L^{\infty}L^{\infty}}\leq \|n\|_{L^{\infty}L^{\infty}},$ we get
\begin{equation*}
	\begin{aligned}
		\|{\rm e}^{bA^{-\frac{1}{3}}t}n_0\partial_x^{2}\nabla c_{\neq}\|_{L^2L^2}^{2}
		&\leq \|n_0\|_{L^{\infty}L^{\infty}}^{2}\|{\rm e}^{bA^{-\frac{1}{3}}t}\partial_x^{2}\nabla c_{\neq}\|_{L^2L^2}^{2}\leq CE_3^{2}\|{\rm e}^{bA^{-\frac{1}{3}}t}\partial_x^{2}n_{\neq}\|_{L^2L^2}^{2} \leq CA^{\frac{1}{3}}E_2^{2}E_3^{2}
	\end{aligned}
\end{equation*}
and
\begin{equation*}
	\begin{aligned}
		\|{\rm e}^{bA^{-\frac{1}{3}}t}\partial_x^{2}n_{\neq}\nabla c_0\|_{L^2L^2}^{2}
		\leq~& \|\nabla c_0\|_{L^{\infty}L^{\infty}}^{2}\|{\rm e}^{bA^{-\frac{1}{3}}t}\partial_x^{2}n_{\neq}\|_{L^2L^2}^{2}\leq CA^{\frac13}\|n_{0}-\bar{n}\|_{L^{\infty}L^{3}}^{2}\|\partial_{x}^{2}n_{\neq}\|_{X_{b}}^{2}\\\leq~& CA^{\frac13}\|n_{0}\|_{L^{\infty}L^{\infty}}^{\frac43}\|n_{0}\|_{L^{\infty}L^{1}}^{\frac23}\|\partial_{x}^{2}n_{\neq}\|_{X_{b}}^{2}\leq CA^{\frac{1}{3}}E_2^{2}(E_3^{2}+M^{2}).
	\end{aligned}
\end{equation*}
Using  (\ref{GN n}),  Lemma \ref{sob_inf_2} and Lemma \ref{lem:ellip_2}, one obtains that
	\begin{align*}
		&	\|{\rm e}^{bA^{-\frac{1}{3}}t}\partial_x^{2}(n_{\neq} \nabla c_{\neq})_{\neq}\|_{L^2L^2}^{2}
		\\\leq~& C\|\nabla c_{\neq}\|_{L^{\infty}L^{4}}^{2}\|{\rm e}^{bA^{-\frac13}t}\partial_{x}^{2} n_{\neq}\|_{L^{2}L^{4}}^{2}
		+\|n\|_{L^{\infty}L^{\infty}}^{2}\|{\rm e}^{bA^{-\frac{1}{3}}t}\partial_x^{2}\nabla c_{\neq}\|_{L^2L^2}^{2}
		\\&+\|\partial_{x}n_{\neq}\|_{L^{\infty}_{t,x}L^{2}_{y,z}}^{2}\|{\rm e}^{bA^{-\frac{1}{3}}t}\partial_{x}\nabla c_{\neq}\|_{L^{2}_{t,x}L^{\infty}_{y,z}}^{2}\\\leq~&C\|n_{\neq}\|_{L^{\infty}L^{2}}^{2}\|{\rm e}^{bA^{-\frac13}t}\partial_{x}^{2}n_{\neq}\|_{L^{2}L^{2}}^{\frac12}\|{\rm e}^{bA^{-\frac13}t}\nabla\partial_{x}^{2}n_{\neq}\|_{L^{2}L^{2}}^{\frac32}
		+C\|n\|_{L^{\infty}L^{\infty}}^{2}\|{\rm e}^{bA^{-\frac13}t}\partial_{x}^{2}n_{\neq}\|_{L^{2}L^{2}}^{2}
		\\&+C\|\partial_{x}^{2}n_{\neq}\|_{L^{\infty}L^{2}}^{2}\|{\rm e}^{bA^{-\frac{1}{3}}t}\partial_{x}^{2}n_{\neq}\|_{L^{2}L^{2}}\|{\rm e}^{bA^{-\frac{1}{3}}t}\nabla\partial_{x}^{2}n_{\neq}\|_{L^{2}L^{2}}\\\leq~&  CA^{\frac56}E_2^2(E_2^{2}+E_{3}^{2}),
		\nonumber
	\end{align*}
where we use $\|n_{\neq}\|_{L^{\infty}L^{\infty}}\leq 2\|n\|_{L^{\infty}L^{\infty}}.$
Thus, we infer from  (\ref{425_1}) that
\begin{equation*}\label{425_2}
	\begin{aligned}
		&\|{\rm e}^{bA^{-\frac{1}{3}}t}\partial_x^{2}(n\nabla c)_{\neq}\|_{L^2L^2}^{2}
		\leq CA^{\frac{5}{6}}(E_2^{2}+E_{3}^{2}+M^{2})E_2^2.
	\end{aligned}
\end{equation*}
In conclusion, we obtain that
\begin{equation*}\label{partial_x_end}
	\begin{aligned}
		\|\partial_x^{2}n_{\neq}\|_{X_b}^{2}
		\leq
		C\left(\|(\partial_x^{2}n_{\rm in})_{\neq}\|_{L^2}^{2}+\frac{E_2^4+E_3^4+E_{4}^{4}+M^4+H_{2}^{4}}{A^{\frac{1}{6}}}\right).		
	\end{aligned}
\end{equation*}
When $A\geq\max\{A_{2}, \left(E_{2}^{4}+E_{3}^{4}+M^{4}+H_{2}^{4}+E_{4}^{4}\right)^{6} \}=:A_{3},$ one deduces
\begin{equation*}
	E_{2,1}(t)=\|\partial_{x}^{2}n_{\neq}\|_{X_{b}}\leq C\left(\|(\partial_{x}^{2}n_{\rm in})_{\neq}\|_{L^{2}}+1 \right).
\end{equation*}
\end{proof}

\subsection{Energy estimates for $E_{2,2}(t)$}
\begin{lemma}\label{lem:E2(t)}
Under the conditions of Theorem \ref{result0} and the assumptions \eqref{conditions:u20 u30} and \eqref{assumption}, there exists a positive constant $A_{4}$ independent of $t$ and $A,$ such that if $A\geq A_{4},$ there holds
\begin{equation*}
	E_{2,2}(t)\leq C\left(\|(u_{\rm in})_{\neq}\|_{H^2}+E_4+1\right).
\end{equation*}
\end{lemma}
\begin{proof}
For convenience, we denote
\begin{equation}\label{ope1}
	\begin{aligned}
		\mathcal{L}=\partial_t+y\partial_x-\frac{1}{A}\Delta, \qquad 
		\mathcal{L}_V=\partial_t+(y+\frac{\mathbf{U}_2}{A})\partial_x-\frac{1}{A}\Delta.
	\end{aligned}
\end{equation}
As previously mentioned, we have decomposed the zero mode $u_{1,0}$
into two parts $u_{1,0}=\mathbf{U}_1+\mathbf{U}_2$.
According to the nonlinear interaction, we will reformulate the coupled system $\{\partial_x\omega_{2,\neq},\Delta u_{2,\neq}\}$ and $(\partial_y,\partial_z)\omega_{2,\neq}$.
Using $u_{1,0}=\mathbf{U}_1+\mathbf{U}_2$, $(\ref{ini11})$ and \eqref{fg}, the velocity $\omega_{2,\neq}$ satisfies
\begin{equation*}
	\begin{aligned}
		\mathcal{L}_V\omega_{2,\neq}+\partial_{z}u_{2,\neq}
		=~&\frac{\partial_x\left[(u_{\neq}\cdot\nabla u_{3,\neq})_{\neq}
			+{U}_0\cdot\nabla u_{3,\neq}
			+u_{\neq}\cdot\nabla u_{3,0}\right]}{A}+\frac{\partial_{z}n_{\neq}}{A}
		-\frac{\partial_{z}\mathbf{U}_2\partial_xu_{1,\neq}}{A}\\
		&-\frac{\partial_z\left[(u_{\neq}\cdot\nabla u_{1,\neq})_{\neq}
			+{U}_0\cdot\nabla u_{1,\neq}
			+u_{\neq}\cdot\nabla (\mathbf{U}_1+\mathbf{U}_2)\right]}{A},\\
	\end{aligned}
\end{equation*}
where $U_{0}=(\mathbf{U}_1,u_{2,0},u_{3,0}).$
Precise calculations show that
\begin{equation}\label{u2_1}
	\begin{aligned}
		(\partial_x^2+\partial_z^2)(\mathbf{U}_2\partial_xu_{2,\neq})
		=\mathbf{U}_2\partial_x(\partial_x^2+\partial_z^2)u_{2,\neq}+\partial_z(\partial_z\mathbf{U}_2\partial_xu_{2,\neq})
		+\partial_z\mathbf{U}_2\partial_x\partial_zu_{2,\neq}.
	\end{aligned}
\end{equation}
Due to ${\rm div}~u_{\neq}=0,$ we have
\begin{equation}\label{u2_2}
	\begin{aligned}
		&\quad\partial_y\left[\partial_x(\mathbf{U}_2\partial_xu_{1,\neq})
		+\partial_z(\mathbf{U}_2\partial_xu_{3,\neq})\right]
		=\partial_y\left(-\mathbf{U}_2\partial_x\partial_yu_{2,\neq}+\partial_z\mathbf{U}_2\partial_xu_{3,\neq}\right)\\
		&=-\mathbf{U}_2\partial_x\partial_y^2u_{2,\neq}-\partial_y\mathbf{U}_2\partial_x\partial_yu_{2,\neq}
		+\partial_y(\partial_z\mathbf{U}_2\partial_xu_{3,\neq}).
	\end{aligned}
\end{equation}
Using \eqref{u2_1} and \eqref{u2_2}, $\Delta u_{2,\neq}$  satisfies
	\begin{align*}
		\mathcal{L}_V\Delta u_{2,\neq}=&-\frac{\partial_x\partial_yn_{\neq}}{A}-\frac{(\partial_x^2+\partial_z^2)
			\left[(u_{\neq}\cdot\nabla u_{2,\neq})_{\neq}+{U}_0\cdot\nabla u_{2,\neq}
			+u_{\neq}\cdot\nabla u_{2,0}\right]}{A}\\
		&+\frac{\partial_y\partial_x\left[(u_{\neq}\cdot\nabla u_{1,\neq})_{\neq}
			+{U}_0\cdot\nabla u_{1,\neq}
			+u_{\neq}\cdot\nabla (\mathbf{U}_1+\mathbf{U}_2)\right]}{A}
		-\frac{\partial_j\mathbf{U}_2\partial_x\partial_ju_{2,\neq}}{A}\\
		&+\frac{\partial_y\partial_z\left[(u_{\neq}\cdot\nabla u_{3,\neq})_{\neq}
			+{U}_0\cdot\nabla u_{3,\neq}+u_{\neq}\cdot\nabla u_{3,0}\right]}{A}
		+\frac{\partial_y(\partial_z\mathbf{U}_2\partial_xu_{3,\neq})
			-\partial_z(\partial_z\mathbf{U}_2\partial_xu_{2,\neq})}{A}.
	\end{align*}
Therefore, one obtains that 	
\begin{equation}\label{ini4}
	\left\{
	\begin{array}{lr}
		\mathcal{L}_V\omega_{2,\neq}+\partial_{z}u_{2,\neq}
		=\frac{\partial_x\left[(u_{\neq}\cdot\nabla u_{3,\neq})_{\neq}
			+{U}_0\cdot\nabla u_{3,\neq}
			+u_{\neq}\cdot\nabla u_{3,0}\right]}{A}+\frac{\partial_{z}n_{\neq}}{A}
		-\frac{\partial_{z}\mathbf{U}_2\partial_xu_{1,\neq}}{A}\\
		\qquad\qquad\qquad\qquad-\frac{\partial_z\left[(u_{\neq}\cdot\nabla u_{1,\neq})_{\neq}
			+{U}_0\cdot\nabla u_{1,\neq}
			+u_{\neq}\cdot\nabla (\mathbf{U}_1+\mathbf{U}_2)\right]}{A},\\
		\mathcal{L}_V\Delta u_{2,\neq}=-\frac{\partial_x\partial_yn_{\neq}}{A}-\frac{(\partial_x^2+\partial_z^2)
			\left[(u_{\neq}\cdot\nabla u_{2,\neq})_{\neq}+{U}_0\cdot\nabla u_{2,\neq}
			+u_{\neq}\cdot\nabla u_{2,0}\right]}{A}\\
		\qquad\qquad\quad+\frac{\partial_y\partial_x\left[(u_{\neq}\cdot\nabla u_{1,\neq})_{\neq}
			+{U}_0\cdot\nabla u_{1,\neq}
			+u_{\neq}\cdot\nabla (\mathbf{U}_1+\mathbf{U}_2)\right]}{A}
		-\frac{\partial_j\mathbf{U}_2\partial_x\partial_ju_{2,\neq}}{A}\\
		\qquad\qquad\quad+\frac{\partial_y\partial_z\left[(u_{\neq}\cdot\nabla u_{3,\neq})_{\neq}
			+{U}_0\cdot\nabla u_{3,\neq}+u_{\neq}\cdot\nabla u_{3,0}\right]}{A}
		+\frac{\partial_y(\partial_z\mathbf{U}_2\partial_xu_{3,\neq})-\partial_z(\partial_z\mathbf{U}_2\partial_xu_{2,\neq})}{A}.
	\end{array}
	\right.
\end{equation}
Taking $\partial_x$ in $\eqref{ini4}_1$, we apply
Proposition~\ref{timespace1} with
$h_1=\Delta u_{2,\neq}$ and $h_2=\partial_x\omega_{2,\neq}$.
The construction of $\mathbf U_2$ gives $\mathbf U_2(0)=0$, and
Lemma~\ref{u1_hat2} supplies \eqref{condition U2 high} with
$\delta\le C\epsilon$. The two $L^2$ source terms in that proposition are
\[
r_1=-A^{-1}\partial_j\mathbf U_2\,
                 \partial_x\partial_j u_{2,\neq},\qquad
r_2=-A^{-1}\partial_z\mathbf U_2\,
                 \partial_x^2u_{1,\neq};
\]
the remaining terms are assigned to $\nabla\cdot g_1$ and
$\nabla\cdot g_2$. Since
$F_b=\|\partial_x^2u_{2,\neq}\|_{X_b}\le E_4(t)$, we obtain
	\begin{align}
		&\quad\|\partial_x\omega_{2,\neq}\|_{X_a}^2+
		\|\Delta u_{2,\neq}\|_{X_a}^2
		\leq C\Big(	\|(u_{\rm in})_{\neq}\|_{H^2}^2+\frac{\|\partial_xn_{\neq}\|_{X_b}^2}{A^{\frac23}}
		+\frac{\|{\rm e}^{aA^{-\frac{1}{3}}t}\partial_{z}\mathbf{U}_2\partial_x(u_{2,\neq},u_{3,\neq})\|_{L^2L^2}^2}{A}\Big)\nonumber\\
		&+CA^{-1}\Big(\|{\rm e}^{aA^{-\frac{1}{3}}t}\partial_x\left[(u_{\neq}\cdot\nabla u_{1,\neq})_{\neq}
		+{U}_0\cdot\nabla u_{1,\neq}
		+u_{\neq}\cdot\nabla (\mathbf{U}_1+\mathbf{U}_2)\right]\|_{L^2L^2}^2\nonumber\\
		&+\|{\rm e}^{aA^{-\frac{1}{3}}t}(\partial_x,\partial_z)
		\left[(u_{\neq}\cdot\nabla u_{2,\neq})_{\neq}
		+{U}_0\cdot\nabla u_{2,\neq}+u_{\neq}\cdot\nabla u_{2,0}\right]\|_{L^2L^2}^2 \label{omega:1}\\
		&+\|{\rm e}^{aA^{-\frac{1}{3}}t}(\partial_x,\partial_z)
		\left[(u_{\neq}\cdot\nabla u_{3,\neq})_{\neq}
		+{U}_0\cdot\nabla u_{3,\neq}+u_{\neq}\cdot\nabla u_{3,0}\right]\|_{L^2L^2}^2\Big)\nonumber\\
		&+CA^{-\frac53}\Big(\|{\rm e}^{aA^{-\frac{1}{3}}t}\partial_{z}\mathbf{U}_2\partial_x^2u_{1,\neq}\|_{L^2L^2}^2
		+\|{\rm e}^{aA^{-\frac{1}{3}}t}\partial_{j}\mathbf{U}_2\partial_x\partial_ju_{2,\neq}\|_{L^2L^2}^2\Big)\nonumber\\
		&+C\epsilon^2A^{-\frac13}\|\partial_x^2u_{2,\neq}\|_{X_b}^2=: C\left(\|(u_{\rm in})_{\neq}\|_{H^2}^2+T_{2,1}+T_{2,2}+\cdots+T_{2,8}\right).\nonumber
	\end{align}
The additional term satisfies
\[
T_{2,8}:=\epsilon^2A^{-\frac13}
             \|\partial_x^2u_{2,\neq}\|_{X_b}^2
\le \epsilon^2A^{-\frac13}E_4(t)^2
\le \epsilon^2E_4(t)^2.
\]
Only the bootstrap bound \eqref{assumption} is used here; the improved
bound for $E_4(t)$ proved later is not needed. Thus $T_{2,8}$ is included
in the $\epsilon^2E_4(t)^2$ contribution in
\eqref{w_result1}.

{\bf Estimates of $T_{2,2},$ $T_{2,6}$ and $T_{2,7}.$}
By Lemma \ref{lemma_neq22}, Lemma \ref{u1_hat2} and Young's inequality, we have
\begin{equation*}
	\begin{aligned}
		T_{2,2}+T_{2,6}+T_{2,7}\leq C\frac{E_2^2+E_4^2}{A^{\frac14}}+C\epsilon^2(E_2^2(t)+E_4^2(t)).
	\end{aligned}
\end{equation*}

{\bf Estimate of $T_{2,3}.$}
Using $\eqref{eq:non-neq0}_3$, $\eqref{lemma_neq2_2}_{3,5}$, $\eqref{lemneq2_2}_1,$
Lemma \ref{lemma_neq3}, Lemma \ref{lemma_u23_1}, Lemma \ref{lemma:u1:1}, Lemma \ref{lem:n0 L2} and  Lemma \ref{u1_hat2},
there holds
\begin{equation*}
	T_{2,3}\leq C\frac{E_2^4+E_4^4+H_1^4+1}{A^{\frac{5}{6}-\alpha}}
	+C\epsilon^2(E_2^2(t)+E_4^2(t)).
\end{equation*}

{\bf Estimate of  $T_{2,4}.$}
Using $\eqref{eq:non-neq0}_4$, $\eqref{lemma_neq2_2}_{1,5}$,
Lemma \ref{lemma_neq3},
Lemma \ref{lemma:u1:1} and Lemma \ref{lem:n0 L2},
we get
\begin{equation*}
	T_{2,4}\leq C\frac{E_2^4+H_1^4+1}{A^{\frac12-\frac{2}{3}\alpha}}
	+C\epsilon^2 E_2^2(t).
\end{equation*}

{\bf Estimate of  $T_{2,5}.$}
Using $\eqref{eq:non-neq0}_{3,5}$, $\eqref{lemma_neq2_2}_{2,5}$,
Lemma \ref{lemma_neq3},
Lemma \ref{lemma:u1:1} and Lemma \ref{lem:n0 L2},
we arrive
\begin{equation*}
	T_{2,5}\leq C\frac{E_2^4+H_1^4+1}{A^{\frac12-\frac{2}{3}\alpha}}
	+C\epsilon^2 E_2^2(t).
\end{equation*}
Therefore, we infer from \eqref{omega:1} that
\begin{equation}\label{w_result1}
	\begin{aligned}
		\|\partial_x\omega_{2,\neq}\|_{X_a}^2+
		\|\Delta u_{2,\neq}\|_{X_a}^2
		\leq C\Big(	\|(u_{\rm in})_{\neq}\|_{H^2}^2+\frac{E_2^4+E_4^4+H_1^4+1}{A^{\frac12-\frac{2}{3}\alpha}}
		+\epsilon^2 (E_2^2(t)+E_4^2(t)) \Big).
	\end{aligned}
\end{equation} 	
For $j\in\{2,3\},$ taking $\partial_j$ to $\eqref{ini4}_1,$ one gets
\begin{equation*}
	\begin{aligned}
		\mathcal{L}_V\partial_j\omega_{2,\neq}+\partial_j\partial_{z}u_{2,\neq}
		=&\frac{\partial_j\partial_{z}n_{\neq}}{A}
		-\frac{\partial_j(\partial_{z}\mathbf{U}_2\partial_xu_{1,\neq})}{A}
		-\left(\partial_jy+\frac{\partial_j\mathbf{U}_2}{A}\right)
		\partial_x\omega_{2,\neq}\\
		&+\frac{\partial_j\partial_x\left[(u_{\neq}\cdot\nabla u_{3,\neq})_{\neq}
			+{U}_0\cdot\nabla u_{3,\neq}
			+u_{\neq}\cdot\nabla u_{3,0}\right]}{A}\\
		&-\frac{\partial_j\partial_z\left[(u_{\neq}\cdot\nabla u_{1,\neq})_{\neq}
			+{U}_0\cdot\nabla u_{1,\neq}
			+u_{\neq}\cdot\nabla (\mathbf{U}_1+\mathbf{U}_2)\right]}{A}.
	\end{aligned}
\end{equation*}
Applying Proposition \ref{Lvf 0} to it,
due to $\big\|\partial_jy+\frac{\partial_j\mathbf{U}_2}{A}\big\|_{L^{\infty}L^{\infty}}\leq C,$
we obtain that
\begin{equation}\label{omega:2}
	\begin{aligned}
		\frac{\|\partial_j\omega_{2,\neq}\|_{X_a}^2}{A^{\frac23}}\leq& C
		\left(\frac{\|(u_{\rm in})_{\neq}\|_{H^2}^2}{A^{\frac23}}+
		\|\Delta u_{2,\neq}\|_{X_a}^2
		+\|\partial_x\omega_{2,\neq}\|_{X_a}^2
		+\frac{\|\partial_xn_{\neq}\|_{X_b}^2}{A^{\frac23}}+T_{3,1}+T_{3,2}\right)\\
		&+CA^{-\frac53}\|{\rm e}^{aA^{-\frac{1}{3}}t}\partial_{z}\mathbf{U}_2\partial_xu_{1,\neq}\|_{L^2L^2}^2,
	\end{aligned}
\end{equation}
where
\begin{equation*}
	\begin{aligned}
		&T_{3,1}=A^{-\frac{5}{3}}
		\|{\rm e}^{aA^{-\frac{1}{3}}t}\partial_x
		\left[(u_{\neq}\cdot\nabla u_{3,\neq})_{\neq}
		+{U}_0\cdot\nabla u_{3,\neq}+u_{\neq}\cdot\nabla u_{3,0}\right]\|_{L^2L^2}^2,\\
		&T_{3,2}=A^{-\frac{5}{3}}\|{\rm e}^{aA^{-\frac{1}{3}}t}\partial_z\left[(u_{\neq}\cdot\nabla u_{1,\neq})_{\neq}
		+{U}_0\cdot\nabla u_{1,\neq}
		+u_{\neq}\cdot\nabla (\mathbf{U}_1+\mathbf{U}_2)\right]\|_{L^2L^2}^2.
	\end{aligned}
\end{equation*}
It is obvious that
$$T_{3,1}\leq T_{2,5}\leq C\frac{E_2^4+H_1^4+1}{A^{\frac12-\frac{2}{3}\alpha}}
+C\epsilon^2 E_2^2(t)$$
Using $\eqref{eq:non-neq0}_6$, $\eqref{lemma_neq2_2}_{5}$, $\eqref{lem:uneq3}$,
Lemma \ref{lemma_neq3},
Lemma \ref{lemma:u1:1} and Lemma \ref{lem:n0 L2},
we get
\begin{equation*}
	T_{3,2}\leq C\frac{E_2^4+E_4^4+H_1^4+1}{A^{\frac12-\frac{2}{3}\alpha}}
	+C\epsilon^2 E_2^2(t).
\end{equation*}
By using \eqref{w_result1}, it follows from \eqref{omega:2} that
\begin{equation}\label{w_result2}
	\begin{aligned}
		\frac{\|(\partial_y,\partial_z)\omega_{2,\neq}\|_{X_a}^2}{A^{\frac23}}
		\leq C\left(\|(u_{\rm in})_{\neq}\|_{H^2}^2+\frac{E_2^4+E_4^4+H_1^4+1}{A^{\frac12-\frac{2}{3}\alpha}}
		+\epsilon^2 E_2^2(t)+\epsilon^2 E_4^2(t)\right).
	\end{aligned}
\end{equation}
When $\epsilon$ is small satisfying $C\epsilon^2\leq \frac12$ and
$ A\geq \left(E_{2}^{4}+E_{4}^{4}+H_{1}^{4}+1 \right)^{\frac{6}{3-4\alpha}}=:A_{4}$, we infer from
\eqref{w_result1} and \eqref{w_result2} that
\begin{equation*}
	E_{2,2}^2(t)\leq C\left(\|(u_{\rm in})_{\neq}\|_{H^2}^2
	+ E_4^2(t)+1\right).
\end{equation*}
\end{proof}

\section{Estimates for $L^{2}$-norm of the density}\label{sec4}
%
If $M=0$, nonnegativity and mass conservation give $n\equiv0$,
so the density estimates below are immediate.
We henceforth assume $0<m=\frac{M}{2\pi}<8\pi$. For nonnegative initial data,
the strong maximum principle gives positivity for positive times.
Motivated by \cite{Bedro2}, we next consider the following 2D free energy of $ n_{0} $ on $ \mathbb{T}^{2}: $
\begin{equation*}\label{L}
\mathcal{L}[n_{0}]=\int_{\mathbb{T}^{2}}\left[n_{0}\log n_{0}-\frac12(n_{0}-\overline{n})c_{0} \right]dydz.
\end{equation*}
{
\begin{lemma}\label{lem: L}
Under the conditions of Theorem \ref{result0} and the assumption \eqref{assumption}, there exists a positive constant $A_{5}$ independent of $t$ and $A$, such that if $ A\geq A_{5},$ there holds
\begin{equation}\label{result L}
	\mathcal{L}[n_{0}(t)]\leq \mathcal{L}[(n_{\rm in})_{0}]+C,\quad{\rm for}\quad t\in[0,T].
\end{equation}
\end{lemma}
\begin{proof}
Using (\ref{ini11}), direct calculation shows that
\begin{equation}\label{dt L}
	\begin{aligned}
		\frac{d}{dt}\mathcal{L}[n_{0}]=&-\frac{1}{A}\int_{\mathbb{T}^{2}}n_{0}|\nabla\log n_{0}-\nabla c_{0}|^{2}dydz-\frac{1}{A}\int_{\mathbb{T}^{2}}n_{0}u_{0}\cdot\nabla c_{0}dydz\\&+\frac{1}{A}\int_{\mathbb{T}^{2}}\left[(n_{\neq}\nabla c_{\neq})_{0}+(n_{\neq}u_{\neq})_{0} \right]\cdot\left(\nabla\log n_{0}-\nabla c_{0} \right) dydz\\=:&-\frac{1}{A}\int_{\mathbb{T}^{2}}n_{0}|\nabla\log n_{0}-\nabla c_{0}|^{2}dydz+J_{1}+J_{2}.
	\end{aligned}
\end{equation}
For $j\in\{2,3\}$, by (\ref{u23:decompose}), we decompose $u_{j,0}(t,y,z)=\overline{u}_{j,0}(t)+\widetilde{u}_{j,0}(t,y,z)$, therefore
\begin{equation}\label{eq:I1}
	\begin{aligned}
		J_{1}=&-\frac{1}{A}\int_{\mathbb{T}^{2}}n_{0}u_{j,0}\partial_{j}c_{0}dydz=-\frac{1}{A}\int_{\mathbb{T}^{2}}n_{0}\overline{u}_{j,0}\partial_{j}c_{0}dydz-\frac{1}{A}\int_{\mathbb{T}^{2}}n_{0}\widetilde{u}_{j,0}\partial_{j}c_{0}dydz=:J_{11}+J_{12}.
	\end{aligned}
\end{equation}
For $J_{11}$, using $(\ref{ini11})_{2}$ and integration by parts, we have
$$
		J_{11}=\frac{1}{A}\overline{u}_{j,0}(t)\int_{\mathbb{T}^{2}}\left(\partial_{y}^{2}+\partial_{z}^{2} \right)c_{0}\partial_{j}c_{0}dydz-\frac{1}{A}\overline{u}_{j,0}(t)\overline{n}(t)\int_{\mathbb{T}^{2}}\partial_{j}c_{0}dydz=0.
$$
Since
$\partial_y\widetilde{u}_{2,0}
+\partial_z\widetilde{u}_{3,0}=0$, integration by parts gives
\begin{align*}
	\sum_{j=2}^3
	\int_{\mathbb T^2}
	n_0\widetilde{u}_{j,0}\partial_j\log n_0 dydz
	&=
	\sum_{j=2}^3
	\int_{\mathbb T^2}
	\widetilde{u}_{j,0}\partial_jn_0 dydz=
	-\int_{\mathbb T^2}
	n_0
	\left(
	\partial_y\widetilde{u}_{2,0}
	+\partial_z\widetilde{u}_{3,0}
	\right)dydz
	=0.
\end{align*}
Consequently,
\begin{align}\label{eq:J12-new}
	J_{12}
	&=
	\frac1A
	\sum_{j=2}^3
	\int_{\mathbb T^2}
	n_0\widetilde{u}_{j,0}
	\left(
	\partial_j\log n_0-\partial_jc_0
	\right)dydz
	\nonumber\\
	&\leq
	\frac{1}{4A}
	\int_{\mathbb T^2}
	n_0
	\left|
	\nabla\log n_0-\nabla c_0
	\right|^2 dydz+\frac{C}{A}
	\int_{\mathbb T^2}
	n_0
	\left(
	|\widetilde{u}_{2,0}|^2
	+|\widetilde{u}_{3,0}|^2
	\right)dydz\\
	&\leq
	\frac{1}{4A}
	\int_{\mathbb T^2}
	n_0
	\left|
	\nabla\log n_0-\nabla c_0
	\right|^2 dydz
	+\frac{Cm}{A}
	\sum_{j=2}^3
	\|\widetilde{u}_{j,0}\|_{L^\infty}^2.\nonumber
\end{align}

Next, we estimate $J_2$. For $j\in\{2,3\}$, since
$
\int_{\mathbb T}\partial_jc_{\neq}dx=
\int_{\mathbb T}u_{j,\neq}dx=0,$
we have
\[
\int_{\mathbb T}
n_{\neq}\partial_jc_{\neq}dx
=
\int_{\mathbb T}
n\partial_jc_{\neq}dx,
\qquad
\int_{\mathbb T}
n_{\neq}u_{j,\neq}dx
=
\int_{\mathbb T}
nu_{j,\neq}dx.
\]
Therefore,
\begin{align*}
	J_2
	=
	\frac{1}{A|\mathbb T|}
	\sum_{j=2}^3
	\int_{\mathbb T^3}
	n\left(
	\partial_jc_{\neq}+u_{j,\neq}
	\right)
	\left(
	\partial_j\log n_0-\partial_jc_0
	\right)
	dxdydz.
\end{align*}
By the Cauchy--Schwarz and Young's inequalities, we obtain
\begin{align}
	J_2
	\leq{}&
	\frac{1}{4A}
	\int_{\mathbb T^2}
	n_0
	\left|
	\nabla\log n_0-\nabla c_0
	\right|^2 dydz+
	\frac{C}{A}
	\int_{\mathbb T^3}
	n\left(
	|\nabla c_{\neq}|^2+|u_{\neq}|^2
	\right)dxdydz.
	\label{eq:J2-new}
\end{align}

Combining $J_{11}=0$, \eqref{eq:J12-new} and
\eqref{eq:J2-new}, we infer from \eqref{dt L} that
\begin{equation*}
	\begin{aligned}
		\frac{d}{dt}\mathcal L[n_0]
		\leq{}&
		-\frac{1}{2A}
		\int_{\mathbb T^2}
		n_0
		\left|
		\nabla\log n_0-\nabla c_0
		\right|^2 dydz\\&+
		\frac{Cm}{A}
		\sum_{j=2}^3
		\|\widetilde{u}_{j,0}\|_{L^\infty}^2
		+
		\frac{C}{A}
		\int_{\mathbb T^3}
		n\left(
		|\nabla c_{\neq}|^2+|u_{\neq}|^2
		\right)dxdydz.
		\label{eq:free-energy-new}
	\end{aligned}
\end{equation*}
Integrating it over $[0,t]$, we arrive at
\begin{align}\label{eq:free-energy-integrated-new}
	&\mathcal L[n_0(t)]-\mathcal L[(n_{\rm in})_0]
	\nonumber\\
	\leq{}&
	-\frac{1}{2A}
	\int_0^t\int_{\mathbb T^2}
	n_0
	\left|
	\nabla\log n_0-\nabla c_0
	\right|^2 dydzds\\
	&+
	\frac{Cm}{A}
	\sum_{j=2}^3
	\int_0^t
	\|\widetilde{u}_{j,0}\|_{L^\infty}^2 ds+
	\frac{C}{A}
	\int_0^t\int_{\mathbb T^3}
	n\left(
	|\nabla c_{\neq}|^2+|u_{\neq}|^2
	\right)dxdydzds.\nonumber
\end{align}
By the Sobolev embedding and the Poincar\'e inequality,
\[
\|\widetilde{u}_{j,0}\|_{L^\infty}
\leq
C\|\widetilde{u}_{j,0}\|_{H^2}
\leq
C\|\nabla\widetilde{u}_{j,0}\|_{H^1},
\]
it follows from Lemma \ref{lem:H1-decay-zero-mode} that
\begin{align}
	\frac1A
	\sum_{j=2}^3
	\int_0^t
	\|\widetilde{u}_{j,0}\|_{L^\infty}^2\,ds
	\leq
	C\left[
	\sum_{j=2}^3
	\left\|
	\widetilde{(u_{j,\rm in})}_0
	\right\|_{H^1}^2
	+1
	\right]\leq C.
	\label{eq:u-zero-L2Linfty}
\end{align}
Moreover, by using elliptic estimates, Lemma \ref{lem: poincare}, Lemma \ref{lemma_u} and (\ref{assumption}), we obtain
\begin{equation}
	\begin{aligned}
		&\frac{C}{A}
		\int_0^t\int_{\mathbb T^3}
		n\left(
		|\nabla c_{\neq}|^2+|u_{\neq}|^2
		\right)dxdydzds\\
		\leq&~
		\frac{CE_3}{A}
		\int_0^t
		\left(
		\|n_{\neq}\|_{L^2}^2
		+\|u_{\neq}\|_{L^2}^2
		\right)ds\leq
		\frac{CE_2^2E_3}{A}
		\int_0^\infty
		{\rm e}^{-2aA^{-\frac13}s}\,ds
		\leq
		\frac{CE_2^2E_3}{aA^{\frac23}}.
		\label{eq:J2-time-new}
	\end{aligned}
\end{equation}
Substituting \eqref{eq:u-zero-L2Linfty} and
\eqref{eq:J2-time-new} into
\eqref{eq:free-energy-integrated-new}, we conclude that
\begin{align*}
	&\mathcal L[n_0(t)]-\mathcal L[(n_{\rm in})_0]
	\leq
	-\frac{1}{2A}
	\int_0^t\int_{\mathbb T^2}
	n_0
	\left|
	\nabla\log n_0-\nabla c_0
	\right|^2 dydzds+Cm+
	\frac{CE_2^2E_3}{aA^{\frac23}}.
	\label{eq:free-energy-final-new}
\end{align*}
Hence, when
$$A\geq \max\Big\{A_{2}, A_{3}, A_{4}, \Big(\frac{E_{2}^{2}E_{3}^{2}}{a} \Big)^{\frac32}\Big\}=:A_{5},$$
we obtain that
$\mathcal{L}[n_{0}]-\mathcal{L}[(n_{\rm in})_{0}]\leq C.$

\end{proof}
}

Next, motivated by  \cite{Bedro2} or \cite{he24-2}, we use (\ref{result L}) and (\ref{ineq:hls}) to get a bound on $\|n_{0}\log^{+}n_{0}\|_{L^{1}}$.
\begin{lemma}
Under the assumptions of Lemma \ref{lem: L} and
$ m=\|(n_{\rm in})_{0}\|_{L^{1}}<8\pi,$ there exists a constant $ C_{L\log L}(n_{\rm in}) $ such that
\begin{equation}\label{eq:n0logn0}
	\int_{\mathbb{T}^{2}}n_{0}\log^{+}n_{0}dydz\leq C_{L\log L}(n_{\rm in}).
\end{equation}
\begin{proof}
	Let $ Y_\perp=(y,z)\in\mathbb{T}^{2} $ be fixed. Define the cut-off function $ \varphi(\tau)\in C^{\infty} $ such that
	\begin{equation}\label{varphi}
		\begin{aligned}
			{\rm supp} (\varphi)&=B(Y_\perp, 1/4),
			\\ \varphi(\tau)&=1,\quad\forall \tau\in B(Y_\perp, 1/8),
			\\ {\rm supp} \left(\nabla\varphi(\tau) \right)&\subset\overline{B}(Y_\perp, 1/4)\backslash B(Y_\perp, 1/8).
		\end{aligned}
	\end{equation}
	By periodically extending $ n_{0}(\tau) $ and $ c_{0}(\tau) $ to $ \mathbb{R}^{2},$ we can rewrite the equation $ -\Delta c_{0}=n_{0}-\overline{n} $ that holds on $ \mathbb{T}^{2} $ as the following equation that holds on $ \mathbb{R}^{2}: $
		\begin{align*}\label{rewrire eq:c0}
			-\Delta_{\tau}\left(\varphi(\tau)c_{0}(\tau) \right)=&-2\nabla_{\tau}\varphi(\tau)\cdot\nabla_{\tau}c_{0}(\tau)-c_{0}(\tau)\Delta_{\tau}\varphi(\tau)-\varphi(\tau)\Delta_{\tau}c_{0}(\tau)\nonumber\\=&-2\nabla_{\tau}\varphi(\tau)\cdot\nabla_{\tau}c_{0}(\tau)-c_{0}(\tau)\Delta_{\tau}\varphi(\tau)+\left(n_{0}(\tau)-\overline{n} \right)\varphi(\tau).
		\end{align*}
	Since $ \varphi(Y_\perp)=1 $ as $ Y_\perp\in B(Y_\perp, 1/8) $ and $ {\rm supp} (\varphi)=B(Y_\perp, 1/4), $ using the above equality and the fundamental solution of the Laplacian on $ \mathbb{R}^{2}, $ we get
	\begin{equation}\label{eq:c0(y)}
		\begin{aligned}
			&c_{0}(Y_\perp)=c_{0}(Y_\perp)\varphi(Y_\perp)\\=&-\frac{1}{2\pi}\int_{\mathbb{R}^{2}}\log(|Y_\perp-\tau|)\left[(n_{0}(\tau)-\overline{n})\varphi(\tau)-2\nabla_{\tau}\varphi(\tau)\cdot\nabla_{\tau}c_{0}(\tau)-c_{0}(\tau)\Delta_{\tau}\varphi(\tau) \right]d\tau\\=&-\frac{1}{2\pi}\int_{|Y_\perp-\tau|\leq\frac14}\log(|Y_\perp-\tau|)(n_{0}(\tau)-\overline{n})\varphi(\tau)d\tau-\frac{1}{\pi}\int_{|Y_\perp-\tau|\leq\frac14}\nabla_{\tau}\cdot\left[\log(|Y_\perp-\tau|)\nabla_{\tau}\varphi(\tau) \right]c_{0}(\tau)d\tau\\&+\frac{1}{2\pi}\int_{|Y_\perp-\tau|\leq\frac14}\log(|Y_\perp-\tau|)\Delta_{\tau}\varphi(\tau)c_{0}(\tau)d\tau.
		\end{aligned}
	\end{equation}
	Due to the support of $ \varphi, $ we can identify the above with an analogous integral on $ \mathbb{T}^{2} $ with $ |Y_\perp-\tau| $ replaced by $ d(Y_\perp,\tau). $ Multiplying (\ref{eq:c0(y)}) by $ -\frac12(n_{0}(Y_\perp)-\overline{n}) $ and integrating with respect to $ Y_\perp $ over $ \mathbb{T}^{2}, $ we have
	\begin{equation}\label{eq:n0c0}
		\begin{aligned}
			&-\frac12\int_{\mathbb{T}^{2}}\left(n_{0}(Y_\perp)-\overline{n} \right)c_{0}(Y_\perp)dY_\perp\\=&\frac{1}{4\pi}\int\int_{\mathbb{T}^{2}\times\mathbb{T}^{2}, d(Y_\perp,\tau)\leq\frac14}\log d(Y_\perp,\tau)\left(n_{0}(Y_\perp)-\overline{n} \right)\left(n_{0}(\tau)-\overline{n} \right)\varphi(\tau)d\tau dY_\perp\\&+\frac{1}{2\pi}\int\int_{\mathbb{T}^{2}\times\mathbb{T}^{2}, \frac18\leq d(Y_\perp,\tau)\leq\frac14}\left(n_{0}(Y_\perp)-\overline{n} \right)\nabla_{\tau}\cdot\left(\log d(Y_\perp,\tau)\nabla_{\tau}\varphi(\tau) \right)c_{0}(\tau)d\tau dY_\perp
			\\&-\frac{1}{4\pi}\int\int_{\mathbb{T}^{2}\times\mathbb{T}^{2}, \frac18\leq d(Y_\perp,\tau)\leq\frac14}\left(n_{0}(Y_\perp)-\overline{n} \right)\log d(Y_\perp,\tau)\Delta_{\tau}\varphi(\tau)c_{0}(\tau)d\tau dY_\perp.
		\end{aligned}
	\end{equation}
	By using $(\ref{varphi})_{2},$ one deduces that
	\begin{align*}
		&\frac{1}{4\pi}\int\int_{d(Y_\perp,\tau)\leq\frac14}\log d(Y_\perp,\tau)\left(n_{0}(Y_\perp)-\overline{n} \right)\left(n_{0}(\tau)-\overline{n} \right)\varphi(\tau)d\tau dY_\perp\\=&\frac{1}{4\pi}\int\int_{\mathbb{T}^{2}\times\mathbb{T}^{2}}\log d(Y_\perp,\tau)n_{0}(Y_\perp)n_{0}(\tau)d\tau dY_\perp-\frac{1}{4\pi}\int\int_{d(Y_\perp,\tau)>\frac18}\log d(Y_\perp,\tau)n_{0}(Y_\perp)n_{0}(\tau)d\tau dY_\perp\\&-\frac{1}{2\pi}\overline{n}\int\int_{d(Y_\perp,\tau)\leq\frac18}\log d(Y_\perp,\tau)n_{0}(Y_\perp)d\tau dY_\perp+\frac{1}{4\pi}(\overline{n})^{2}\int\int_{d(Y_\perp,\tau)\leq\frac18}\log d(Y_\perp,\tau)d\tau dY_\perp\\&+\frac{1}{4\pi}\int\int_{\frac18\leq d(Y_\perp,\tau)\leq\frac14}\log d(Y_\perp,\tau)\left(n_{0}(Y_\perp)-\overline{n} \right)\left(n_{0}(\tau)-\overline{n}\right)\varphi(\tau)d\tau dY_\perp.
	\end{align*}
	Combining it with (\ref{eq:n0c0}), we have
	\begin{align*}
		&-\frac12\int_{\mathbb{T}^{2}}\left(n_{0}(Y_\perp)-\overline{n} \right)c_{0}(Y_\perp)dY_\perp\\=&\frac{1}{4\pi}\int\int_{\mathbb{T}^{2}\times\mathbb{T}^{2}}\log d(Y_\perp,\tau)n_{0}(Y_\perp)n_{0}(\tau)d\tau dY_\perp-\frac{1}{4\pi}\int\int_{d(Y_\perp,\tau)>\frac18}\log d(Y_\perp,\tau)n_{0}(Y_\perp)n_{0}(\tau)d\tau dY_\perp\\&-\frac{1}{2\pi}\overline{n}\int\int_{d(Y_\perp,\tau)\leq\frac18}\log d(Y_\perp,\tau)n_{0}(Y_\perp)d\tau dY_\perp+\frac{1}{4\pi}(\overline{n})^{2}\int\int_{d(Y_\perp,\tau)\leq\frac18}\log d(Y_\perp,\tau)d\tau dY_\perp\\&+\frac{1}{4\pi}\int\int_{\frac18\leq d(Y_\perp,\tau)\leq\frac14}\log d(Y_\perp,\tau)\left(n_{0}(Y_\perp)-\overline{n} \right)\left(n_{0}(\tau)-\overline{n}\right)\varphi(\tau)d\tau dY_\perp\\&+\frac{1}{2\pi}\int\int_{\frac18\leq d(Y_\perp,\tau)\leq\frac14}\left(n_{0}(Y_\perp)-\overline{n} \right)\nabla_{\tau}\cdot\left(\log d(Y_\perp,\tau)\nabla_{\tau}\varphi(\tau) \right)c_{0}(\tau)d\tau dY_\perp\\&-\frac{1}{4\pi}\int\int_{\frac18\leq d(Y_\perp,\tau)\leq\frac14}\left(n_{0}(Y_\perp)-\overline{n} \right)\log d(Y_\perp,\tau)\Delta_{\tau}\varphi(\tau)c_{0}(\tau)d\tau dY_\perp=:I_{1,1}+I_{1,2}+\cdots+I_{1,7}.
	\end{align*}
First of all, direct calculations show that
	\begin{equation*}
		\begin{aligned}
			I_{1,2}+I_{1,3}+I_{1,4}+I_{1,5}\geq -Cm^{2}.
		\end{aligned}
	\end{equation*}		
	Moreover, in the region $ \frac18\leq|Y_\perp-\tau|\leq\frac14, $ note that
	\begin{equation*}
		\begin{aligned}
			&\left|\nabla_{\tau}\cdot\left[\log(|Y_\perp-\tau|)\nabla_{\tau}\varphi(\tau) \right]\right|
			\leq8|\nabla_{\tau}\varphi(\tau)|+\log8|\Delta_{\tau}\varphi(\tau)|\leq C,\\
			&|\log(|Y_\perp-\tau|)\Delta_{\tau}\varphi(\tau)|\leq \log8|\Delta_{\tau}\varphi(\tau)|\leq C,
		\end{aligned}
	\end{equation*}
	we have
	\begin{equation*}
		\begin{aligned}
			|I_{1,6}|+|I_{1,7}|\leq C\int_{\mathbb{T}^{2}}n_{0}(Y_\perp)dY_\perp\int_{\mathbb{T}^{2}}|c_{0}(\tau)|d\tau+C\overline{n}\int_{\mathbb{T}^{2}}|c_{0}(\tau)|d\tau\leq Cm\|c_{0}\|_{L^{1}(\mathbb{T}^{2})}.
		\end{aligned}
	\end{equation*}
Let $G$ be the zero-average Green function of $-\Delta_{Y_\perp}$ on
$\mathbb{T}^{2}$:
\begin{equation*}
	-\Delta_{Y_\perp} G
	=
	\delta_{0}
	-\frac{1}{|\mathbb{T}|^{2}},
	\qquad
	\int_{\mathbb{T}^{2}}G(Y_\perp) dY_\perp=0.
\end{equation*}
There exists $r_G>0$ such that, in the local Euclidean coordinate ball
$0<|Y_\perp|<r_G$,
\begin{equation*}
    G(Y_\perp)=-\frac{1}{2\pi}\log|Y_\perp|+H(Y_\perp),
\end{equation*}
where $H$ is smooth and even.
By Young's inequality, we obtain
	\begin{equation*}
		\|c_{0}\|_{L^{1}(\mathbb{T}^{2})}=\|G\ast(n_{0}-\overline{n})\|_{L^{1}(\mathbb{T}^{2})}\leq\|G\|_{L^{1}(\mathbb{T}^{2})}\|n_{0}-\overline{n}\|_{L^{1}(\mathbb{T}^{2})}\leq Cm.
	\end{equation*}
	This follows that
	\begin{equation*}
		I_{1,6}+I_{1,7}\geq -Cm^{2}.
	\end{equation*}

	Combining the estimates of $ I_{1,2}-I_{1,7}, $ we conclude that
	\begin{equation*}
		\begin{aligned}
			-\frac12\int_{\mathbb{T}^{2}}(n_{0}-\overline{n})c_{0}dY_\perp\geq\frac{1}{4\pi}\int\int_{\mathbb{T}^{2}\times\mathbb{T}^{2}}
			\log d(Y_\perp,\tau)n_{0}(Y_\perp)n_{0}(\tau)d\tau dY_\perp-Cm^{2}.
		\end{aligned}
	\end{equation*}
	From which along with (\ref{result L}), we arrive at
	\begin{align*}
		&\mathcal{L}[(n_{\rm in})_{0}]\geq\int_{\mathbb{T}^{2}}n_{0}\log n_{0}dY_\perp-\int_{\mathbb{T}^{2}}\frac12(n_{0}-\overline{n})c_{0}dY_\perp-C\\\geq&\int_{\mathbb{T}^{2}}n_{0}\log n_{0}dY_\perp+\frac{1}{4\pi}\int_{\mathbb{T}^{2}\times\mathbb{T}^{2}}\log d(Y_\perp,\tau)n_{0}(Y_\perp)n_{0}(\tau)d\tau dY_\perp-Cm^{2}-C\\=&\left(1-\frac{m}{8\pi} \right)\int_{\mathbb{T}^{2}}n_{0}\log n_{0}dY_\perp\\&+\frac{m}{8\pi}\left(\int_{\mathbb{T}^{2}}n_{0}\log n_{0}dY_\perp+\frac{2}{m}\int\int_{\mathbb{T}^{2}\times\mathbb{T}^{2}}\log d(Y_\perp,\tau)n_{0}(Y_\perp)n_{0}(\tau)d\tau dY_\perp \right)-Cm^{2}-C.
	\end{align*}
	Applying (\ref{ineq:hls}) to it, one obtains
	\begin{equation*}
		\mathcal{L}[(n_{\rm in})_{0}]\geq\left(1-\frac{m}{8\pi}\right)\int_{\mathbb{T}^{2}}n_{0}\log n_{0}dY_\perp-C(m)-Cm^{2},
	\end{equation*}
	which implies that
	\begin{equation*}
		\int_{\mathbb{T}^{2}}n_{0}\log n_{0}dY_\perp\leq\frac{\mathcal{L}[(n_{\rm in})_{0}]+C(m)+Cm^{2}}{1-\frac{m}{8\pi}}\leq C_{L\log L}(n_{\rm in})<\infty
	\end{equation*}
under the condition $m<8\pi$. 
Due to $-r\log r\le {\rm e}^{-1}$ for $0\le r\le1$, we have
\[
\int_{\mathbb T^2} n_0\log^+n_0 dY_\perp
=\int_{\mathbb T^2}n_0\log n_0 dY_\perp
 +\int_{\{n_0<1\}}(-n_0\log n_0)dY_\perp
\le \int_{\mathbb T^2}n_0\log n_0 dY_\perp+\frac{|\mathbb T|^2}{{\rm e}}.
\]
Thus, after enlarging $C_{L\log L}$,
$\sup_{0\le t\le T}\int n_0\log^+n_0 dY_\perp\le C_{L\log L}(n_{\rm in},m)$.
\end{proof}
\end{lemma}

The following lemma gives a uniform in time $ L^{2} $ bound of $ n_{0}. $
\begin{lemma}\label{lem:n0 L2}
Under the assumptions of Lemma~\ref{lem: L}, choose
$Q=2+\frac{2m}{|\mathbb T|^2}
+\exp\bigl(C_*C_{L\log L}(n_{\rm in},m)\bigr),$
where $C_*$ is a sufficiently large constant. Then
\begin{equation}\label{eq:n0 L2}
	\begin{aligned}
		\sup_{0\le t\le T}\|n_0(t)\|_{L^2}
\le \|(n_{\rm in})_0\|_{L^2}+C(m)(1+\sqrt Q)=:H_1.
	\end{aligned}
\end{equation}
\end{lemma}
\begin{proof}
The above choice gives
$|\{n_0>Q\}|\le \frac{m}{Q}\le \frac{|\mathbb T|^2}{2}$.
Hence the zero set of $(n_0-Q)_+$ occupies at least half the torus,
and the homogeneous inequalities in Lemma~\ref{lem:GNS} apply with
uniform constants. Noting that $ (n_{0}-Q)_{+}=\max\{0, n_{0}-Q \}, $ and using (\ref{eq:n0logn0}), we have
\begin{equation}\label{n0-Q}
	\begin{aligned}
		&\int_{\mathbb{T}^{2}}(n_{0}-Q)_{+}dydz=\int_{n_{0}> Q}(n_{0}-Q)dydz\leq\int_{n_{0}>Q}n_{0}dydz\\=&\int_{n_{0}>Q}\frac{1}{\log^{+}n_{0}}n_{0}\log^{+}n_{0}dydz\leq\frac{1}{\log Q}\int_{n_{0}>Q}n_{0}\log^{+}n_{0}dydz\leq\frac{C_{L\log L}}{\log Q}.
	\end{aligned}
\end{equation}
Recall that $ n_{0} $ satisfies
\begin{equation}\label{eq:n0}
	\begin{aligned}
		\partial_{t}n_{0}=&\frac{1}{A}\Delta n_{0}-\frac{1}{A}\nabla\cdot(n_{0}\nabla c_{0})-\frac{1}{A}\nabla\cdot(n_{\neq}\nabla c_{\neq})_{0}-\frac{1}{A}(u_{0}\cdot\nabla n_{0})-\frac{1}{A}(u_{\neq}\cdot\nabla n_{\neq})_{0}.
	\end{aligned}
\end{equation}
As $ (n_{0}-Q)_{+}(n_{0}-Q)_{-}=0,$ multiplying (\ref{eq:n0}) by $ (n_{0}-Q)_{+}$ and integrating with respect to $ (y,z) $ over $ \mathbb{T}^{2}, $ we get
\begin{equation}\label{I1-I5}
	\begin{aligned}
		\frac12\frac{d}{dt}\|(n_{0}-Q)_{+}\|_{L^2(\mathbb{T}^{2})}^2&=\frac{1}{A}\Big(\int_{\mathbb{T}^{2}}(n_{0}-Q)_{+}\Delta n_{0}dydz-\int_{\mathbb{T}^{2}}(n_{0}-Q)_{+}\nabla\cdot(n_{0}\nabla c_{0})dydz\\&-\int_{\mathbb{T}^{2}}(n_{0}-Q)_{+}\nabla\cdot(n_{\neq}\nabla c_{\neq})_{0}dydz-\int_{\mathbb{T}^{2}}(n_{0}-Q)_{+}(u_{0}\cdot\nabla n_{0})dydz\\&-\int_{\mathbb{T}^{2}}(n_{0}-Q)_{+}(u_{\neq}\cdot\nabla n_{\neq})_{0}dydz\Big)=:J_{1}+\cdots+J_{5}.
	\end{aligned}
\end{equation}
For $ J_{1}, $ integration by parts shows that
\begin{equation*}
	J_{1}=-\frac{1}{A}\int_{\mathbb{T}^{2}}|\nabla((n_{0}-Q)_{+})|^{2}dydz.
\end{equation*}
For $ J_{2}, $ using $ (\ref{ini11}) $ and integration by parts, we have
	\begin{align*}
		J_{2}=~&\frac{1}{2A}\int_{\mathbb{T}^{2}}\left((n_{0}-Q)_{+}\right)^{2}(n_{0}-\overline{n})dydz+\frac{Q}{A}\int_{\mathbb{T}^{2}}(n_{0}-Q)_{+}(n_{0}-\overline{n})dydz
		\\=~&\frac{1}{2A}\int_{\mathbb{T}^{2}}\left((n_{0}-Q)_{+} \right)^{3}dydz+\frac{3Q-\overline{n}}{2A}\int_{\mathbb{T}^{2}}\left((n_{0}-Q)_{+}\right)^{2}dydz+\frac{Q^{2}-Q\overline{n}}{A}\int_{\mathbb{T}^{2}}(n_{0}-Q)_{+}dydz\\\leq~&\frac{1}{2A}\int_{\mathbb{T}^{2}}\left((n_{0}-Q)_{+}\right)^{3}dydz+\frac{3Q}{2A}\int_{\mathbb{T}^{2}}\left((n_{0}-Q)_{+} \right)^{2}dydz+\frac{Q^{2}m}{A}.
	\end{align*}
For $ J_{3}$ and $J_5,$ by integration by parts, one deduces
\begin{equation*}
	J_{3}+J_5\leq \frac{1}{16A}\int_{\mathbb{T}^{2}}|\nabla\left((n_{0}-Q)_{+} \right)|^{2}dydz+\frac{C(\|(n_{\neq}\nabla c_{\neq})_{0}\|_{L^{2}}^{2}+\|(n_{\neq}u_{\neq})_{0}\|_{L^{2}}^{2})}{A}.
\end{equation*}
For $ J_{4}, $ using $ \nabla\cdot u_{0}=0$, for $j\in\{2,3\},$ we have
\begin{equation*}
	\begin{aligned}
		J_{4}=&-\frac{1}{A}\int_{\mathbb{T}^{2}}(n_{0}-Q)_{+}u_{j,0}\partial_{j}n_{0}dydz=-\frac{1}{A}\int_{\mathbb{T}^{2}}(n_{0}-Q)_{+}u_{j,0}\partial_{j}(n_{0}-Q)_{+}dydz\\=&-\frac{1}{2A}\int_{\mathbb{T}^{2}}u_{j,0}\partial_{j}((n_{0}-Q)_{+})^{2}dydz=0.
	\end{aligned}
\end{equation*}
Collecting $ J_{1}-J_{5}, $ we get by (\ref{I1-I5}) that
\begin{equation}\label{I1-I5 1}
	\begin{aligned}
		\frac12\frac{d}{dt}\|(n_{0}-Q)_{+}\|_{L^{2}}^{2}\leq&-\frac{7\|\nabla(n_{0}-Q)_{+}\|_{L^{2}}^{2}}{8A}
		+\frac{\|(n_{0}-Q)_{+}\|_{L^{3}}^{3}}{2A}+\frac{3Q\|(n_{0}-Q)_{+}\|_{L^{2}}^{2}}{2A}\\&+\frac{ Q^{2}m}{A}+\frac{C}{A}\left(\|(n_{\neq}\nabla c_{\neq})_{0}\|_{L^{2}}^{2}+\|(u_{\neq}n_{\neq})_{0}\|_{L^{2}}^{2} \right).
	\end{aligned}
\end{equation}
Using Lemma \ref{lem:GNS} and (\ref{n0-Q}), one obtains
\begin{equation*}
	\begin{aligned}
		\|(n_{0}-Q)_{+}\|_{L^{3}(\mathbb{T}^{2})}^{3}\leq C\|(n_{0}-Q)_{+}\|_{L^{1}(\mathbb{T}^{2})}\|\nabla\left((n_{0}-Q)_{+}\right)\|_{L^{2}(\mathbb{T}^{2})}^{2}\leq\frac{C_{L\log L}}{\log Q}\|\nabla\left( (n_{0}-Q)_{+}\right)\|_{L^{2}(\mathbb{T}^{2})}^{2}.
	\end{aligned}
\end{equation*}
The chosen value of $Q$, with $C_*$ sufficiently large, ensures that
\begin{equation}\label{n0-Q L3}
	\begin{aligned}
		-\frac{7}{8A}\|\nabla\left((n_{0}-Q)_{+} \right)\|_{L^{2}(\mathbb{T}^{2})}^{2}+\frac{1}{2A}\|(n_{0}-Q)_{+}\|_{L^{3}(\mathbb{T}^{2})}^{3}\leq -\frac{1}{2A}\|\nabla\left((n_{0}-Q)_{+}\right)\|_{L^{2}(\mathbb{T}^{2})}^{2}.
	\end{aligned}
\end{equation}
Similarly, it follows from Lemma \ref{lem:GNS} and (\ref{n0-Q}) that
\begin{equation}\label{GN}
	\begin{aligned}
		-\|\nabla\left((n_{0}-Q)_{+} \right)\|_{L^{2}(\mathbb{T}^{2})}^{2}\leq-\frac{\|(n_{0}-Q)_{+}\|_{L^{2}(\mathbb{T}^{2})}^{4}}{C\|(n_{0}-Q)_{+}\|_{L^{1}(\mathbb{T}^{2})}^{2}}\leq-\frac{1}{Cm^{2}}\|(n_{0}-Q)_{+}\|_{L^{2}(\mathbb{T}^{2})}^{4}.
	\end{aligned}
\end{equation}
Substituting (\ref{n0-Q L3}) and (\ref{GN}) into (\ref{I1-I5 1}), we get
\begin{equation}
    \begin{aligned}\label{n0-Q 2}
		\frac{d}{dt}\|(n_{0}-Q)_{+}\|_{L^{2}(\mathbb{T}^{2})}^{2}\leq&-\frac{1}{ACm^{2}}\|(n_{0}-Q)_{+}\|_{L^{2}(\mathbb{T}^{2})}^{4}+\frac{3Q}{A}\|(n_{0}-Q)_{+}\|_{L^{2}(\mathbb{T}^{2})}^{2}\\&+\frac{2Q^{2}m}{A}+\frac{C}{A}\left(\|(n_{\neq}\nabla c_{\neq})_{0}\|_{L^{2}(\mathbb{T}^{2})}^{2}+\|(u_{\neq}n_{\neq})_{0}\|_{L^{2}(\mathbb{T}^{2})}^{2} \right).
	\end{aligned}
\end{equation}
We denote $S(t)$ by
\begin{equation*}
	S(t):=\frac{C}{A}\int_{0}^{t}\left(\|(n_{\neq}\nabla c_{\neq})_{0}\|_{L^{2}(\mathbb{T}^{2})}^{2}+\|(u_{\neq}n_{\neq})_{0}\|_{L^{2}(\mathbb{T}^{2})}^{2} \right)ds,\quad{\rm for}\quad t\geq 0.
\end{equation*}
Then using Lemma \ref{lem: poincare}, Lemma \ref{lem:ellip_2}, Lemma \ref{lemma_u} and assumption (\ref{assumption}), direct calculations indicate that
\begin{equation}\label{G(t) bound}
	\begin{aligned}
		S(t)\leq~&\frac{C}{A}\left(\|n_{\neq}\|_{L^{\infty}L^{\infty}}^{2}\|\nabla c_{\neq}\|_{L^{2}L^{2}}^{2}+\|n_{\neq}\|_{L^{\infty}L^{\infty}}^{2}\|u_{\neq}\|_{L^{2}L^{2}}^{2} \right)\\\leq~&\frac{C}{A^{\frac23}}\|n\|_{L^{\infty}L^{\infty}}^{2}\left(\|\partial_{x}n_{\neq}\|_{X_{a}}^{2}+\|\partial_{x}\omega_{2,\neq}\|_{X_{a}}^{2}+\|\Delta u_{2,\neq}\|_{X_{a}}^{2} \right)\leq\frac{CE_{2}^{2}E_{3}^{2}}{A^{\frac23}}\leq C
	\end{aligned}
\end{equation}
provided that $A\geq A_5$. Set $Z(t)=\|(n_0-Q)_+\|_{L^2}^2$.
Young's inequality in \eqref{n0-Q 2} gives
\[
\frac{d}{dt}(Z-S)
\le -\frac{Z^2}{2ACm^2}
       +\frac{9Cm^2Q^2}{2A}+\frac{2Q^2m}{A}.
\]
Let $D_Q=Q(9C^2m^4+4Cm^3)^{\frac12}$. Since $S\ge0$, the derivative
above is nonpositive whenever $Z-S\ge D_Q$. Consequently,
\[
Z(t)\le S(t)+\max\{Z(0),D_Q\},
\]
and therefore
\begin{equation}\label{end:n0-Q}
\|(n_0-Q)_+\|_{L^2}
\le \|(n_{\rm in})_0\|_{L^2}+C+C(m)\sqrt Q.
\end{equation}
Finally, $n_0=(n_0-Q)_++\min\{n_0,Q\}$ and
$\|\min\{n_0,Q\}\|_{L^2}\le\sqrt{Qm}$ prove \eqref{eq:n0 L2}.
\end{proof}
\begin{lemma}\label{lemma:partial-z-n0}
	Under the assumptions of Lemma \ref{lem: L}, there holds
	\begin{equation}\label{estimate:partial-z-n0}
		\begin{aligned}
			\|\partial_z n_0\|_{L^\infty L^2}
			\leq C\left(\|(\partial_z n_{\rm in})_0\|_{L^2}
			+H_1E_3+E_2E_3+A^{\frac16}E_2^2+\epsilon\right),
		\end{aligned}
	\end{equation}
	where $H_1=\|(n_{\rm in})_0\|_2+C(m)(1+\sqrt Q)$
	is the constant obtained in Lemma \ref{lem:n0 L2}. 
\end{lemma}

\begin{proof}
Applying $\partial_z$ to \eqref{eq:n0} and by integrating over $\mathbb T^2$,  we get
	\begin{equation}\label{energy:partial-z-n0}
		\begin{aligned}
			&\frac{1}{2}\frac{d}{dt}
			\|\partial_z n_0\|_{L^2}^2
			+\frac{1}{A}\|\nabla\partial_z n_0\|_{L^2}^2\\
			=~&\frac{1}{A}
			\int_{\mathbb T^2}
			\left(
			\partial_z n_0\nabla c_0
			+n_0\nabla\partial_z c_0
			\right)\cdot\nabla\partial_z n_0 dydz+\frac{1}{A}
			\int_{\mathbb T^2}
			\left(
			u_0\partial_z n_0
			+n_0\partial_z u_0
			\right)\cdot\nabla\partial_z n_0 dydz\\&+\frac{1}{A}\int_{\mathbb T^2}\partial_z\left[
			\left(n_{\neq}\nabla c_{\neq}\right)_0
			+\left(u_{\neq}n_{\neq}\right)_0\right]\cdot\nabla\partial_z n_0 dydz.
		\end{aligned}
	\end{equation}

	We estimate the terms on the right-hand side separately. First, 
	by Lemma \ref{lem:ellip_0} and assumptions (\ref{assumption}), we have
	\begin{equation*}
		\|\nabla c_0\|_{L^\infty}
		\leq C\|n_0-\overline n\|_{L^3}
		\leq C\|n_0\|_{L^\infty}
		\leq CE_3.
	\end{equation*}
	Moreover, applying $\partial_z$ to $(\ref{ini11})_{2}$ gives
	$-\Delta\partial_z c_0=\partial_z n_0,$
	which implies
	\begin{equation*}
		\|\nabla\partial_z c_0\|_{L^2}
		\leq C\|\partial_z n_0\|_{L^2}.
	\end{equation*}
	Combining these with $\|n_0\|_{L^\infty}\leq\|n\|_{L^\infty}\leq2E_3$ and Young's inequality, we obtain
	\begin{equation}\label{estimate:zero-chemotaxis-partial-z-n0}
		\begin{aligned}
			\frac{1}{A}
			\left|
			\int_{\mathbb T^2}
			\left(
			\partial_z n_0\nabla c_0
			+n_0\nabla\partial_z c_0
			\right)\cdot\nabla\partial_z n_0 dydz
			\right|\leq\frac{1}{8A}
			\|\nabla\partial_z n_0\|_{L^2}^2
			+\frac{CE_3^2}{A}
			\|\partial_z n_0\|_{L^2}^2.
		\end{aligned}
	\end{equation}
	Next, since $\partial_yu_{2,0}+\partial_zu_{3,0}=0$,
	we get
	\begin{equation}\label{cancel:transport-partial-z-n0}
		\int_{\mathbb T^2}
		u_0\partial_z n_0\cdot\nabla\partial_z n_0 dydz
		=\frac12\int_{\mathbb T^2}
		u_0\cdot\nabla|\partial_z n_0|^2 dydz=0.
	\end{equation}
	For the remaining zero-mode velocity term, set
	$v_0=(u_{2,0},u_{3,0})$ and $\nabla_\perp=(\partial_y,\partial_z)$.
	Since $\partial_xn_0=0$, only $v_0$ contributes to this term.
	Using Lemma~\ref{lemma_u23_1}, we have
	\[
		\|\partial_z v_0\|_{L^\infty L^2}
		\leq
		\|\nabla u_{2,0}\|_{L^\infty L^2}
		+\|\nabla u_{3,0}\|_{L^\infty L^2}
		\leq C\epsilon.
	\]
	Therefore,
	\begin{equation}\label{estimate:velocity-partial-z-n0}
		\begin{aligned}
			\frac{1}{A}
			\left|
			\int_{\mathbb T^2}
			n_0\partial_z v_0\cdot\nabla_\perp\partial_z n_0 dydz
			\right|\leq
			\frac{1}{8A}
			\|\nabla\partial_z n_0\|_{L^2}^2
			+\frac{CE_3^2\epsilon^2}{A}.
		\end{aligned}
	\end{equation}

	We now estimate the interactions between the non-zero modes. Note that
	\begin{equation}\label{decomposition:nonzero-partial-z-n0}
	     \begin{aligned}
		&\left\|
		\partial_z
		\left[
		\left(n_{\neq}\nabla c_{\neq}\right)_0
		+\left(u_{\neq}n_{\neq}\right)_0
		\right]
		\right\|_{L^2}^2\\
		\leq~&
		C\left\|
		\partial_z n_{\neq}\nabla c_{\neq}
		+n_{\neq}\nabla\partial_z c_{\neq}
		\right\|_{L^2(\mathbb T^3)}^2+
		C\left\|
		\partial_z u_{\neq}n_{\neq}
		+u_{\neq}\partial_z n_{\neq}
		\right\|_{L^2(\mathbb T^3)}^2.
	\end{aligned}
	\end{equation}
Choose a fixed $p>3$. By the elliptic estimate and the Sobolev
	embedding $W^{2,p}(\mathbb T^3)\hookrightarrow W^{1,\infty}
	(\mathbb T^3)$, we get
	\begin{equation*}
		\|\nabla c_{\neq}\|_{L^\infty}
		\leq C\|c_{\neq}\|_{W^{2,p}}
		\leq C\|n_{\neq}\|_{L^p}
		\leq C\|n\|_{L^\infty}
		\leq CE_3,
	\end{equation*}
	and $\|\nabla\partial_zc_{\neq}\|_{L^2}
		\leq C\|n_{\neq}\|_{L^2}$.
	Hence,
	\begin{equation*}
		\begin{aligned}
			\left\|
			\partial_z n_{\neq}\nabla c_{\neq}
			+n_{\neq}\nabla\partial_zc_{\neq}
			\right\|_{L^2}^2\leq CE_3^2\left(
			\|\partial_z n_{\neq}\|_{L^2}^2
			+\|n_{\neq}\|_{L^2}^2
			\right).
		\end{aligned}
	\end{equation*}
Similarly, using
	$\|n_{\neq}\|_{L^\infty}\leq2\|n\|_{L^\infty}\leq4E_3$,
	we obtain
	\begin{equation*}
		\begin{aligned}
			\left\|
			\partial_z u_{\neq}n_{\neq}
			+u_{\neq}\partial_z n_{\neq}
			\right\|_{L^2}^2\leq
			CE_3^2\|\partial_z u_{\neq}\|_{L^2}^2
			+C\|u_{\neq}\|_{L^\infty}^2
			\|\partial_z n_{\neq}\|_{L^2}^2.
		\end{aligned}
	\end{equation*}
It follows from Lemma \ref{lemma_u}, assumptions \eqref{assumption}
	and Lemma \ref{lem: poincare} that
	\begin{equation*}
		\begin{aligned}
			\frac{1}{A}
			\|\partial_z n_{\neq}\|_{L^2L^2}^2\leq \|\partial_x^{2} n_{\neq}\|_{X_{b}}^{2}
			&\leq CE_2^2,\\
			\frac{1}{A}
			\|\partial_z u_{\neq}\|_{L^2L^2}^2\leq\frac{C}{A}\left(\|\partial_x\omega_{2,\neq}\|_{L^2L^2}^{2}+\|\Delta u_{2,\neq}\|_{L^2L^2}^2  \right)
			&\leq \frac{CE_2^2}{A^{\frac23}},\\
			\frac{1}{A}
			\|n_{\neq}\|_{L^2L^2}^2\leq\frac{1}{A}(A^{\frac13}\|\partial_x^2n_{\neq}\|_{X_{b}}^2)
			&\leq \frac{CE_2^2}{A^{\frac23}}.
		\end{aligned}
	\end{equation*}
	Moreover, $(\ref{ap:u11})_1$ implies
		$\left\|
		u_{\neq}
		\right\|_{L^\infty L^\infty}^2
		\leq CA^{\frac13}E_2^2.$
	Substituting the above estimates into \eqref{decomposition:nonzero-partial-z-n0}, one deduces
	\begin{equation}\label{estimate:nonzero-forcing-partial-z-n0}
		\begin{aligned}
			&\frac{1}{A}
			\left\|
			\partial_z
			\left[
			\left(n_{\neq}\nabla c_{\neq}\right)_0
			+\left(u_{\neq}n_{\neq}\right)_0
			\right]
			\right\|_{L^2L^2}^2\\
			\leq{}&\frac{C}{A}E_3^2
\left(\|\partial_z n_{\neq}\|_{L^2L^2}^2+\|n_{\neq}\|_{L^2L^2}^2 \right)+\frac{C}{A}E_3^2\|\partial_z u_{\neq}\|_{L^2L^2}^2+\frac{C}{A}\|u_{\neq}\|_{L^{\infty}L^{\infty}}^2\|\partial_z n_{\neq}\|_{L^2L^2}^2\\\leq{}&			C\left(
			E_3^2E_2^2
			+A^{\frac13}E_2^4
			\right).
		\end{aligned}
	\end{equation}
	Using Young's inequality, the last term in
	\eqref{energy:partial-z-n0} satisfies
	\begin{equation}\label{estimate:forcing-energy-partial-z-n0}
		\begin{aligned}
			&\frac{1}{A}\int_{\mathbb T^2}\partial_z\left[
			\left(n_{\neq}\nabla c_{\neq}\right)_0
			+\left(u_{\neq}n_{\neq}\right)_0\right]\cdot\nabla\partial_z n_0 dydz\\
			\leq{}&
			\frac{1}{8A}
			\|\nabla\partial_z n_0\|_{L^2}^2+\frac{C}{A}
			\left\|
			\partial_z
			\left[
			\left(n_{\neq}\nabla c_{\neq}\right)_0
			+\left(u_{\neq}n_{\neq}\right)_0
			\right]
			\right\|_{L^2}^2.
		\end{aligned}
	\end{equation}

	Combining
	\eqref{energy:partial-z-n0}--\eqref{estimate:velocity-partial-z-n0} and
	\eqref{estimate:forcing-energy-partial-z-n0}, we arrive at
	\begin{equation}\label{differential:partial-z-n0}
		\begin{aligned}
			\frac{d}{dt}
			\|\partial_z n_0\|_{L^2}^2
			+\frac{\|\nabla\partial_z n_0\|_{L^2}^2}{A}
			\leq 
			\frac{CE_3^2}{A}
			\|\partial_z n_0\|_{L^2}^2
			+\frac{CE_3^2\epsilon^2}{A}
			+\frac{C}{A}
			\left\|
			\partial_z
			\left[
			\left(n_{\neq}\nabla c_{\neq}\right)_0
			+\left(u_{\neq}n_{\neq}\right)_0
			\right]
			\right\|_{L^2}^2.
		\end{aligned}
	\end{equation}
	Integration by parts gives
	\begin{align*}
		\|\partial_z n_0\|_{L^2}^2=-\int_{\mathbb T^2}
		n_0\partial_z^2n_0 dydz\leq
		\|n_0\|_{L^2}
		\|\partial_z^2n_0\|_{L^2}\leq
		H_1\|\nabla\partial_z n_0\|_{L^2},
	\end{align*}
    where the definition of $H_1$ here can be found in \eqref{eq:n0 L2}.
	Therefore,
	\begin{equation*}
		\|\nabla\partial_z n_0\|_{L^2}^2
		\geq
		\frac{\|\partial_z n_0\|_{L^2}^4}{H_1^2}.
	\end{equation*}
	Substituting this into
	\eqref{differential:partial-z-n0}, and using Young's inequality,
	we obtain
	\begin{equation}\label{ode:partial-z-n0}
		\begin{aligned}
			\frac{d}{dt}
			\|\partial_z n_0\|_{L^2}^2
			\leq
			-\frac{1}{2AH_1^2}
			\|\partial_z n_0\|_{L^2}^4
			+\frac{C}{A}
			\left(
			H_1^2E_3^4+E_3^2\epsilon^2
			\right)
			+\frac{C}{A}
			\left\|
			\partial_z
			\left[
			\left(n_{\neq}\nabla c_{\neq}\right)_0
			+\left(u_{\neq}n_{\neq}\right)_0
			\right]
			\right\|_{L^2}^2.
		\end{aligned}
	\end{equation}
	Define
		$F(t):=
		\frac{C}{A}\int_0^t
		\left\|
		\partial_z
		\left[
		\left(n_{\neq}\nabla c_{\neq}\right)_0
		+\left(u_{\neq}n_{\neq}\right)_0
		\right]
		\right\|_{L^2}^2 ds.$
	By \eqref{estimate:nonzero-forcing-partial-z-n0},
	\begin{equation}\label{estimate:G-partial-z-n0}
		0\leq F(t)
		\leq C\left(
		E_3^2E_2^2+A^{\frac13}E_2^4
		\right).
	\end{equation}
	Set $Y(t)=\|\partial_z n_0(t)\|_{L^2}^2-F(t).$
	Whenever $Y(t)\geq0$, we have
		$\|\partial_z n_0(t)\|_{L^2}^4\geq Y(t)^2.$
	Hence, \eqref{ode:partial-z-n0} yields
	\[
		Y'(t)
		\leq
		-\frac{Y(t)^2}{2AH_1^2}
		+\frac{C}{A}
		\left(
		H_1^2E_3^4+E_3^2\epsilon^2
		\right).
	\]
	This gives
	\begin{align*}
		Y(t)
		\leq C\Big(
		&\|\partial_z(n_{\rm in})_0\|_{L^2}^2
		+H_1^2E_3^2
		+H_1E_3\epsilon
		\Big).
	\end{align*}
	Combining this estimate with
	\eqref{estimate:G-partial-z-n0}, we conclude that
	\begin{equation*}
		\begin{aligned}
			\|\partial_z n_0\|_{L^\infty L^2}^2
			\leq C\left(\|\partial_z(n_{\rm in})_0\|_{L^2}^2
			+H_1^2E_3^2
			+H_1E_3\epsilon+E_3^2E_2^2
			+A^{\frac13}E_2^4\right).
		\end{aligned}
	\end{equation*}
	As $(H_1E_3\epsilon)^{1/2}
		\leq C(H_1E_3+\epsilon)$, there holds
	\[
		\|\partial_z n_0\|_{L^\infty L^2}
		\leq C\left(
		\|\partial_z(n_{\rm in})_0\|_{L^2}
		+H_1E_3+E_2E_3
		+A^{\frac16}E_2^2+\epsilon
		\right).
	\]
\end{proof}

\section{Estimates for $L^{\infty}$-norm of the density $E_3(t)$: Proof of Proposition \ref{prop:E3}}\label{sec6}
\begin{proof}[Proof of Proposition \ref{prop:E3}]
For $ p=2^{j} $ with $ j\geq 1, $ multiplying $ (\ref{ini11})_{1} $ by $2pn^{2p-1}$, and integrating by parts the resulting equation over $\mathbb{T}^{3}$, one deduces
\begin{equation}
	\begin{aligned}
		&\frac{d}{dt}\|n^p\|^2_{L^2}+\frac{2(2p-1)}{Ap}\|\nabla n^{p}\|_{L^2}^2
		=\frac{2(2p-1)}{A}\int_{\mathbb{T}^{3}}n^{p}\nabla c\cdot\nabla n^{p}dxdydz \\
		\leq~&\frac{2(2p-1)}{A}\|n^p\nabla c\|_{L^2}\|\nabla n^p\|_{L^2}
		\leq\frac{2p-1}{Ap}\|\nabla n^p\|_{L^2}^2
		+\frac{(2p-1)p}{A}\|n^p\nabla c\|^2_{L^2}.\nonumber
	\end{aligned}
\end{equation}
Using H\"{o}lder's inequality and Gagliardo--Nirenberg inequality, we get
\begin{equation}
	\begin{aligned}
		\|n^p\nabla c\|_{L^2}^2
		\leq \|n^p\|_{L^4}^2\|\nabla c\|_{L^4}^2
		\leq C\|n^p\|_{L^2}^{\frac{1}{2}}\|\nabla n^p\|_{L^2}^\frac{3}{2}\|\nabla c\|_{L^4}^2+C\|n^{p}\|_{L^{2}}^{2}\|\nabla c\|_{L^{4}}^{2}, \nonumber
	\end{aligned}
\end{equation}
which follows that
\begin{equation}
	\begin{aligned}
		&\frac{d}{dt}\|n^p\|^2_{L^2}+\frac{2(2p-1)}{Ap}\|\nabla n^{p}\|_{L^2}^2 \\
		\leq~&\frac{2p-1}{Ap}\|\nabla n^p\|_{L^2}^2
		+\frac{C(2p-1)p}{A}\|n^p\|_{L^2}^{\frac{1}{2}}\|\nabla n^p\|_{L^2}^\frac{3}{2}\|\nabla c\|_{L^4}^2+\frac{C(2p-1)p}{A}\|n^{p}\|_{L^{2}}^{2}\|\nabla c\|_{L^{4}}^{2} \\
		\leq~&\frac{5(2p-1)}{4Ap}\|\nabla n^p\|_{L^2}^2
		+\frac{C(2p-1)p^7}{A}\|n^p\|_{L^2}^2\|\nabla c\|_{L^4}^8+\frac{C(2p-1)p}{A}\|n^{p}\|_{L^{2}}^{2}\|\nabla c\|_{L^{4}}^{2}.\nonumber
	\end{aligned}
\end{equation}
Consequently, there holds
\begin{equation}
	\begin{aligned}
		\frac{d}{dt}\|n^p\|^2_{L^2}+\frac{1}{2A}\|\nabla n^{p}\|_{L^2}^2\leq
		\frac{Cp^8}{A}\|n^p\|_{L^2}^2\left(1+\|\nabla c\|_{L^4}^8 \right).
		\label{58}
	\end{aligned}
\end{equation}
Using Gagliardo-Nirenberg inequality
\begin{equation}
	\begin{aligned}
		\|n^p\|_{L^2}
		\leq C\left(\|n^p\|_{L^1}^{\frac{2}{5}}\|\nabla n^p\|_{L^2}^{\frac{3}{5}}+\|n^{p}\|_{L^{1}} \right),
		\nonumber
	\end{aligned}
\end{equation}
and the inequality
$\|n^p\|_1^2\le |\mathbb T^3|\|n^p\|_2^2$, we infer from
\eqref{58}, after absorbing the lower-order correction into the
$Cp^8\|n^p\|_2^2/A$ term, that
\begin{equation}
	\begin{aligned}
		\frac{d}{dt}\|n^p\|^2_{L^2}
		\leq-\frac{\|n^p\|_{L^2}^{\frac{10}{3}}}{2AC\|n^p\|_{L^1}^{\frac{4}{3}}}
		+\frac{Cp^8}{A}\|n^p\|_{L^2}^2\left(1+\|\nabla c\|_{L^{\infty}L^4}^8 \right).
		\nonumber
	\end{aligned}
\end{equation}
Applying Lemma \ref{lem:ellip_0}, Lemma \ref{lem:ellip_2}, Lemma \ref{lem:n0 L2} and the assumption (\ref{conditions:u20 u30}), there holds
\begin{equation*}
	\begin{aligned}
		\|\nabla c\|_{L^{\infty}L^4}\leq \|\nabla c_{\neq}\|_{L^{\infty}L^4}+\|\nabla c_0\|_{L^{\infty}L^4}
		\leq C\left(\|\partial_x^2n_{\neq}\|_{L^{\infty}L^{2}}+\|n_{0}\|_{L^{\infty}L^{2}} \right) \leq C\left(E_{2}+H_{1} \right).
	\end{aligned}
\end{equation*}
Therefore
\begin{equation*}\label{np}
	\begin{aligned}
		\frac{d}{dt}\|n^p\|^2_{L^2}
		\leq-\frac{\|n^p\|^{\frac{10}{3}}_{L^2}}{2CA\|n^p\|_{L^1}^{\frac{4}{3}}}
		+\frac{Cp^8}{A}\|n^p\|^2_{L^2}(1+E_{2}^8+H_{1}^8),
	\end{aligned}
\end{equation*}
which indicates that
\begin{equation}
	\begin{aligned}
		\sup_{0\le t\le T}\|n^p\|_{L^2}^2\leq \max\left\{8C^3(1+E_{2}^{8}+H_{1}^{8})^{\frac32}p^{12}\sup_{0\le t\le T}\|n^p\|^2_{L^{1}}, 2\|n_{\rm in}^p\|_{L^2}^2\right\}\label{n1_1}.
	\end{aligned}
\end{equation}

Next, the Moser-Alikakos iteration is used to determine $E_{3}$. Recall $ p=2^{j} $ with $ j\geq 1, $
and we rewrite (\ref{n1_1}) into
\begin{equation}\label{n1_2}
	\begin{aligned}
		\sup_{0\le t\le T}\int_{\mathbb{T}^{3}}|n(t)|^{2^{j+1}}dxdydz
		\leq \max\Big\{C_{1}p^{12}\Big(\sup_{0\le t\le T}\int_{\mathbb{T}^{3}}|n(t)|^{2^{j}}dxdydz\Big)^2, 2\int_{\mathbb{T}^{3}}|n_{\rm in}|^{2^{j+1}}dxdydz\Big\},
	\end{aligned}
\end{equation}
where $C_1=8C^3(1+E_{2}^{8}+H_{1}^{8})^{\frac32}.$
From Lemma \ref{lem:n0 L2}, we note that
	$\|n_{0}\|_{L^{\infty}L^{2}}\leq H_{1}.$
Hence
$$\sup_{0\le t\le T}\|n(t)\|_{L^2}\leq|\mathbb{T}| \|n_{0}\|_{L^\infty L^2}+\|n_{\neq}\|_{L^\infty L^2}\leq |\mathbb{T}|H_1+E_2.$$
By interpolation inequality, for $0<\theta<1$, we have
$$\|n_{\rm in}\|_{L^{2^j}}\leq\|n_{\rm in}\|^{\theta}_{L^2}
\|n_{\rm in}\|^{1-\theta}_{L^\infty}
\leq\|n_{\rm in}\|_{L^2}+\|n_{\rm in}\|_{L^\infty}\leq|\mathbb{T}|H_1+E_2+\|n_{\rm in}\|_{L^\infty}$$
for $j\geq1.$ This yields that
$$2\int_{\mathbb{T}^{3}}|n_{\rm in}|^{2^{j+1}}dxdydz
\leq2\left(|\mathbb{T}|H_1+E_2+\|n_{\rm in}\|_{L^\infty}\right)^{2^{j+1}}\leq K^{2^{j+1}},$$
where $K=2(|\mathbb{T}|H_1+E_2+\|n_{\rm in}\|_{L^\infty}).$

Now, we rewrite (\ref{n1_2}) as
\begin{equation}
	\begin{aligned}
		\sup_{0\le t\le T}\int_{\mathbb{T}^{3}}|n(t)|^{2^{j+1}}dxdydz\leq \max\left\{C_1	4096^{j}\left(\sup_{0\le t\le T}\int_{\mathbb{T}^{3}}|n(t)|^{2^{j}}dxdydz\right)^2, K^{2^{j+1}} \right\}.\nonumber
	\end{aligned}
\end{equation}
For $j=k$, we get
\begin{equation}
	\begin{aligned}
		\sup_{0\le t\le T}\int_{\mathbb{T}^{3}}|n(t)|^{2^{k+1}}dxdydz\leq C_1^{a_k}	4096^{b_k}K^{2^{k+1}},\nonumber
	\end{aligned}
\end{equation}
where $a_k=1+2a_{k-1}$ and $b_k=k+2b_{k-1}$.

Generally, one can obtain the following formulas
$$a_k=2^k-1,\ {\rm and}\ \ b_k=2^{k+1}-k-2.$$
Therefore, one deduces
\begin{equation}
	\begin{aligned}
		\sup_{0\le t\le T}\left(\int_{\mathbb{T}^{3}}|n(t)|^{2^{k+1}}dxdydz\right)^{\frac{1}{2^{k+1}}}\leq C_1^{\frac{2^k-1}{2^{k+1}}}	4096^{\frac{2^{k+1}-k-2}{2^{k+1}}}K.\nonumber
	\end{aligned}
\end{equation}
Letting $k\rightarrow\infty$, there holds
\begin{equation}\label{eq:E3}
	\begin{aligned}
		\sup_{0\le t\le T}\|n(t)\|_{L^\infty}\leq C(1+E_{2}^{8}+H_{1}^{8})^{\frac{3}{4}}(|\mathbb{T}|H_1+E_2+||n_{\rm in}||_{L^\infty})=:E_3.
	\end{aligned}
\end{equation}
\end{proof}

\section{Energy estimates for $E_{4}(t)$: Proof of Proposition \ref{prop:E5}}\label{sec 8}
To estimate $\|\partial^2_{x}u_{2,\neq}\|_{X_{b}}$ and
$\|\partial_{x}^2u_{3,\neq}\|_{X_{b}},$ it is important to
introduce the new quantity $W$ defined by
$$W=u_{2,\neq}+\kappa u_{3,\neq},$$
where
\begin{equation}\label{V kappa}
V=y+\frac{\mathbf{U}_{2}}{A},\quad{\rm and}~~ \kappa=\frac{\partial_zV}{\partial_yV}.
\end{equation}
A similar quality was first proposed by Wei-Zhang in \cite{wei2} and further applied in \cite{Chen1} and \cite{CWW2025}.

For $j\in\{2,3\},$ there holds
$(u\cdot\nabla u_{j})_{\neq}=u_0\cdot\nabla u_{j,\neq}
+u_{\neq}\cdot\nabla u_{j,0}+(u_{\neq}\cdot\nabla u_{j,\neq})_{\neq}.$
Then we infer from (\ref{ini1}) that
\begin{equation}\label{ini_5}
\left\{
\begin{array}{lr}
	\mathcal{L}_Vu_{2,\neq}
	+\frac{\mathbf{U}_{1}\partial_xu_{2,\neq}}{A}
	+\frac{g_{2,1}+g_{2,2}+{G_{2,3}}}{A}
	+\frac{\partial_y(P^{N_1}_{\neq}+P^{N_2}_{\neq})}{A}=0,\\
	\mathcal{L}_Vu_{3,\neq}
	+\frac{\mathbf{U}_{1}\partial_xu_{3,\neq}}{A}
	+\frac{g_{3,1}+g_{3,2}+{G_{3,3}}}{A}
	+\frac{\partial_z(P^{N_1}_{\neq}+P^{N_2}_{\neq})}{A}=0,
\end{array}
\right.
\end{equation}
where $\mathcal{L}_V$ can be found in \eqref{ope1} and
\begin{equation}\label{g G}
\begin{aligned}
	g_{j,1}={u_{2,0}\partial_yu_{j,\neq}+u_{3,0}\partial_zu_{j,\neq}},
	\quad g_{j,2}={u_{\neq}\cdot\nabla u_{j,0}},
	\quad G_{j,3}={(u_{\neq}\cdot\nabla u_{j,\neq})_{\neq}}.
\end{aligned}
\end{equation}
Due to ${\rm div}~u=0,$ we have
\begin{equation*}
\begin{aligned}
	{\rm div}~(u\cdot\nabla u)_{\neq}
	&	=\partial_x(u\cdot\nabla u_1)_{\neq}
	+\partial_y(u\cdot\nabla u_2)_{\neq}
	+\partial_z(u\cdot\nabla u_3)_{\neq}\\
	&={\rm div}~(u_{\neq}\cdot\nabla u_{\neq})_{\neq}+
	2(\partial_yu_{1,0}\partial_xu_{2,\neq}+
	\partial_zu_{1,0}\partial_xu_{3,\neq})+2\partial_yg_{2,2}
	+2\partial_zg_{3,2},
\end{aligned}
\end{equation*}
which along with $\partial_yV=1+\frac{\partial_y\mathbf{U}_{2}}{A}$ implies that
\begin{align*}
	\frac{P^{N_1}_{\neq}+P^{N_2}_{\neq}}{A}
	=&-2\Delta^{-1}\bigg(\left(1+\frac{\partial_y\mathbf{U}_{2}}{A}\right)
	\partial_xu_{2,\neq}
	+\frac{\partial_z\mathbf{U}_{2}}{A}
	\partial_xu_{3,\neq}+\frac{{\rm div}~(u_{\neq}\cdot\nabla u_{\neq})_{\neq}}{2A}\\
	&
	+\frac{\partial_y\mathbf{U}_{1}\partial_xu_{2,\neq}+
		\partial_z\mathbf{U}_{1}\partial_xu_{3,\neq}}{A}
	+\frac{\partial_yg_{2,2}+\partial_zg_{3,2}}{A}-\frac{\partial_{x}n_{\neq}}{2A}\bigg)\\
	=&-2\Delta^{-1}\left(\partial_yV\partial_xW
	+\frac{\partial_yg_{2,2}+\partial_zg_{3,2}}{A}
	+\frac{P_{1,1}+P_{1,2}+P_{1,3}}{A}\right),
\end{align*}
where
\begin{equation}\label{P11 12 13}
P_{1,1}=\frac{{\rm div}~(u_{\neq}\cdot\nabla u_{\neq})_{\neq}}{2},
\quad P_{1,2}=\partial_y\mathbf{U}_{1}\partial_xu_{2,\neq}+
\partial_z\mathbf{U}_{1}\partial_xu_{3,\neq},\quad P_{1,3}=-\frac{\partial_{x}n_{\neq}}{2}.
\end{equation}
Using the above decomposition, we rewrite (\ref{ini_5}) into
\begin{equation}\label{ini_6}
\left\{
\begin{array}{lr}
	&\mathcal{L}_Vu_{2,\neq}
	-2\partial_y\Delta^{-1}(\partial_yV\partial_xW)
	+\frac{\mathbf{U}_{1}\partial_xu_{2,\neq}}{A}
	+\frac{g_{2,1}+g_{2,2}+{G_{2,3}}}{A}
	=\frac{2\partial_y\Delta^{-1}
		(\partial_yg_{2,2}+\partial_zg_{3,2})}{A}
	\\
	&\qquad+\frac{2\partial_y\Delta^{-1}(P_{1,1}+P_{1,2}+P_{1,3})}{A},\\
	&\mathcal{L}_Vu_{3,\neq}-2\partial_z\Delta^{-1}(\partial_yV\partial_xW)
	+\frac{\mathbf{U}_{1}\partial_xu_{3,\neq}}{A}
	+\frac{g_{3,1}+g_{3,2}+{G_{3,3}}}{A}
	=\frac{2\partial_z\Delta^{-1}
		(\partial_yg_{2,2}+\partial_zg_{3,2})}{A}\\
	&\qquad+\frac{2\partial_z\Delta^{-1}(P_{1,1}+P_{1,2}+P_{1,3})}{A}.
\end{array}
\right.
\end{equation}
Therefore $W=u_{2,\neq}+\kappa u_{3,\neq}$ satisfies
\begin{equation}\label{eq:Lv}
\begin{aligned}
	\widetilde{\mathcal{L}_{V}}W+\frac{\mathbf{U}_{1}\partial_x W}{A}
	+\frac{G^{(1)}+G^{(2)}}{A}=\left(\partial_t\kappa-\frac{\Delta\kappa}{A}\right)u_{3,\neq}-\frac{2\nabla\kappa\cdot\nabla u_{3,\neq}}{A},
\end{aligned}
\end{equation}
where
\begin{equation*}\label{LV (1)}
\widetilde{\mathcal{L}_{V}}W=\mathcal{L}_VW
-2(\partial_y+\kappa\partial_z)\Delta^{-1}(\partial_yV\partial_xW)
\end{equation*}
and
\begin{equation*}\label{G11}
\begin{aligned}
	&G^{(1)}=G_{2,3}+\kappa G_{3,3}-2(\partial_y+\kappa\partial_z)
	\Delta^{-1}(P_{1,1}+P_{1,2}+P_{1,3}),\\
	&G^{(2)}=g_{2,1}+g_{2,2}+\kappa(g_{3,1}+g_{3,2})
	-2(\partial_y+\kappa\partial_z)
	\Delta^{-1}(\partial_yg_{2,2}+\partial_zg_{3,2}).
\end{aligned}
\end{equation*}
In addition, $\Delta W$ satisfies
\begin{equation*}
\mathcal{L}_{V}\Delta W=\Delta\left(-\frac{2\nabla\kappa\cdot\nabla u_{3,\neq}}{A} \right)+{\rm good}~{\rm terms}.
\end{equation*}
To remove the singular term $\Delta\left(-\frac{2\nabla\kappa\cdot\nabla u_{3,\neq}}{A} \right)$,
motivated by the quasi-linear method in \cite{wei2}, we introduce the following decomposition
\begin{equation}\label{kappa u3}
\nabla\kappa\cdot\nabla u_{3,\neq}=\rho_{1}\nabla V\cdot\nabla u_{3,\neq}+\rho_{2}(\partial_{z}-\kappa\partial_{y})u_{3,\neq},
\end{equation}
where
\begin{equation}\label{def rho}
\rho_{1}=\frac{\partial_{y}\kappa+\kappa\partial_{z}\kappa}{\partial_{y}V(1+\kappa^{2})},\quad \rho_{2}=\frac{\partial_{z}\kappa-\kappa\partial_{y}\kappa}{1+\kappa^{2}}.
\end{equation}
As $(\partial_{z}-\kappa\partial_{y})$ has a good commutative relation with $\mathcal{L}_{V},$ it is a good derivative. Thus the second term in (\ref{kappa u3}) is good. To handle the first term in (\ref{kappa u3}) and obtain a sharp threshold of velocity, we need to make a further decomposition for $W$ as follows
\begin{equation*}
W=W^{(1)}+\frac{1}{A}W^{(2)},
\end{equation*}
where $W^{(1)}$ and $W^{(2)}$ solve
\begin{equation*}
\left\{
\begin{array}{lr}
	\mathcal{L}_{V}\Delta W^{(1)}={\rm good}~{\rm terms},
	\\
	\mathcal{L}_{V}W^{(2)}=-\rho_{1}\nabla V\cdot\nabla u_{3,\neq},\\
	W^{(1)}(0)=W(0),~W^{(2)}(0)=0.
\end{array}
\right.
\end{equation*}
Furthermore, an additional quantity $W^{(3)}$ is also needed satisfying
\begin{equation*}
\mathcal{L}_{V}W^{(3)}=-\nabla V\cdot\nabla u_{3,\neq},\quad W^{(3)}(0)=0.
\end{equation*}
We denote $$\Delta u_{3,\neq}=\left(\Delta u_{3,\neq}-2\partial_{x}W^{(3)} \right)+2\partial_{x}W^{(3)},$$
which satisfies
\begin{equation}\label{u3''-W}
\begin{aligned}
	\mathcal{L}_{V}\left(\Delta u_{3,\neq}-2\partial_{x}W^{(3)} \right)=&2\partial_{z}\left(\partial_{y}V\partial_{x}W \right)-\frac{\Delta(\mathbf{U}_{1}\partial_{x}u_{3,\neq})}{A}-\frac{\Delta(g_{3,1}+g_{3,2}+G_{3,3})}{A}\\&+\frac{2\partial_{z}\left(\partial_{y}g_{2,2}+\partial_{z}g_{3,2} \right)}{A}+\frac{2\partial_{z}\left(P_{1,1}+P_{1,2}+P_{1,3} \right)}{A}-\Delta V\partial_{x}u_{3,\neq}
\end{aligned}
\end{equation}
and
\begin{equation}\label{Lv W3}
\mathcal{L}_{V}W^{(3)}=-\nabla V\cdot\nabla u_{3,\neq}.
\end{equation}
In this way, by space-time estimates, we will prove that
\begin{equation*}
E_{4}^{2}(t)\leq E_{5,1}^{2}(t)+E_{5,2}^{2}(t)\leq C\left(\|u_{\rm in}\|_{H^{2}}^{2}+1 \right),
\end{equation*}
where $E_{5,1}^{2}(t)$ and $E_{5,2}^{2}(t)$ are the auxiliary norms defined by
\begin{equation*}
\begin{aligned}
	&	E_{5,1}(t)=A^{-\frac23}\|\Delta u_{3,\neq}\|_{X_{b}},\\
	&E_{5,2}(t)=\sum_{j=2}^{3}\left(\|\partial_{x}^{2}u_{j,\neq}\|_{X_{b}}+\|\partial_{x}(\partial_{z}-\kappa\partial_{y})u_{j,\neq}\|_{X_{b}} \right)+\|\partial_{x}\nabla W\|_{X_{b}}.
\end{aligned}
\end{equation*}

\begin{lemma}\label{lem:kappa}
Under the result of Lemma \ref{u1_hat2}, it holds that
$$ \|\kappa\|_{H^{1}}\leq CA^{-1}\|\Delta\mathbf{U}_{2}\|_{L^{2}}\leq C\epsilon, \quad \|\kappa\|_{H^{3}}\leq CA^{-1}\|\Delta \mathbf{U}_{2}\|_{H^{2}}\leq C\epsilon,$$
$$ \|\partial_{t}\kappa\|_{H^{1}}\leq CA^{-1}\|\partial_{t}\mathbf{U}_{2}\|_{H^{2}}\leq CA^{-1}\epsilon,\quad \|\partial_{t}\rho_{1}\|_{L^{2}}\leq CA^{-1}\|\partial_{t}\mathbf{U}_{2}\|_{H^{2}}\leq CA^{-1}\epsilon, $$
$$ \|\rho_{1}\|_{H^{2}}+\|\rho_{2}\|_{H^{2}}\leq CA^{-1}\|\Delta\mathbf{U}_{2}\|_{H^{2}}\leq C\epsilon, $$
where the definitions of $\kappa, \rho_{1}$ and $\rho_{2}$ are given by \eqref{V kappa} and \eqref{def rho}.
\end{lemma}

The following lemma gives the estimates of $ W^{(2)} $ and $ W^{(3)} $ associated with good derivatives.
\begin{lemma}\label{lem W}
Under the conditions of Theorem \ref{result0} and the assumptions \eqref{assumption}, it holds that
\begin{itemize}
	\item[(i)] \quad$\|\partial_{x}^{2}W^{(3)}\|_{X_{b}}^{2}+\|\partial_{x}(\partial_{z}-\kappa\partial_{y})W^{(3)}\|_{X_{b}}^{2}\leq CA^{\frac43}E_{5,2}^{2}(t),$
	\item[(ii)] \quad
	$\|\partial_{x}^{2}W^{(2)}\|_{X_{b}}^{2}\leq C\epsilon^{2}A^{\frac43}E_{5,2}^{2}(t), \quad \|\partial_{x}\nabla W^{(2)}\|_{X_{b}}^{2}\leq C\epsilon^{2}A^{2}E_{5,2}^{2}(t),$
	\item[(iii)] \quad
	$\|\partial_{x}(W^{(2)}-\rho_{1}W^{(3)})\|_{X_{b}}^{2}\leq CA^{\frac23}\epsilon^{2}E_{5,2}^{2}(t).$
\end{itemize}
\end{lemma}
\begin{proof}
{\textbf{Estimate (i).}}
Applying Proposition \ref{Lvf} to (\ref{Lv W3}), we obtain
\begin{equation}\label{W23}
	\begin{aligned}
		&\|\partial_{x}^{2}W^{(3)}\|_{X_{b}}^{2}+\|\partial_{x}(\partial_{z}-\kappa\partial_{y})W^{(3)}\|_{X_{b}}^{2}
		\leq CA^{\frac13}\big(\|{\rm e}^{bA^{-\frac13}t}\partial_{x}^{2}(\nabla V\cdot\nabla u_{3,\neq})\|_{L^{2}L^{2}}^{2}\\
		&+\|{\rm e}^{bA^{-\frac13}t}\partial_{x}(\partial_{z}-\kappa\partial_{y})
		(\nabla V\cdot\nabla u_{3,\neq})\|_{L^{2}L^{2}}^{2} \big)=:CA^{\frac13}\left(I_{1}+I_{2} \right).
	\end{aligned}
\end{equation}
Recalling that $ V=y+\frac{\mathbf{U}_{2}(t,y,z)}{A},$ by Lemma \ref{u1_hat2} and Lemma \ref{lem:kappa},
we have
\begin{align}
	&\|\nabla V\|_{L^{\infty}}\leq 1+A^{-1}\|\nabla\mathbf{U}_{2}\|_{L^{\infty}} \leq C\left(1+A^{-1}\|\Delta\mathbf{U}_{2}\|_{H^{2}} \right)\leq C,\label{V1' infty}\\
	&\|(\partial_{z}-\kappa\partial_{y})\nabla V\|_{L^{\infty}}\leq CA^{-1}\left(1+\|\kappa\|_{H^{3}} \right)\|\Delta\mathbf{U}_{2}\|_{H^{2}}\leq C.\label{V1'' infty}
\end{align}
Then for $ I_{1}, $ using (\ref{V1' infty}), we have
\begin{equation*}
	\begin{aligned}
		I_{1}\leq \|\nabla V\|_{L^{\infty}L^{\infty}}^2\|{\rm e}^{bA^{-\frac13}t}\nabla\partial_{x}^{2}u_{3,\neq}\|_{L^{2}L^{2}}^{2}
		\leq CA\|\partial_{x}^{2}u_{3,\neq}\|_{X_{b}}^{2}\leq CAE_{5,2}^{2}(t).
	\end{aligned}
\end{equation*}
For $ I_{2}, $  using (\ref{V1' infty}) and (\ref{V1'' infty}), we deduce that
\begin{equation*}
	\begin{aligned}
		&\|\partial_{x}(\partial_{z}-\kappa\partial_{y})\left(\nabla V\cdot\nabla u_{3,\neq} \right)\|_{L^{2}}
		\leq C\left(\|\nabla\partial_{x}(\partial_{z}-\kappa\partial_{y})u_{3,\neq}\|_{L^{2}}+\|\nabla\partial_{x}^{2}u_{3,\neq}\|_{L^{2}} \right),
	\end{aligned}
\end{equation*}
which implies
\begin{equation*}
	\begin{aligned}
		I_{2}\leq&CA\left(\|\partial_{x}(\partial_{z}-\kappa\partial_{y})u_{3,\neq}\|_{X_{b}}^{2}+\|\partial_{x}^{2}u_{3,\neq}\|_{X_{b}}^{2} \right)\leq CAE_{5,2}^{2}(t).
	\end{aligned}
\end{equation*}
Combining the estimates of $ I_{1} $ and $ I_{2}, $ (\ref{W23}) gives the result of (i).

{\textbf{Estimate (ii).}} Notice that
\begin{equation*}
	\mathcal{L}_{V}\partial_{x}^{2}W^{(2)}=\partial_{x}^{2}\mathcal{L}_{V}W^{(2)}=-\rho_{1}\nabla V\cdot\nabla\partial_{x}^{2}u_{3,\neq}
\end{equation*}	
and $ \partial_{x}^{2}W^{(2)}(0)=0. $ Then
applying Proposition \ref{Lvf 0}, there holds
\begin{equation*}
	\|\partial_{x}^{2}W^{(2)}\|_{X_{b}}^{2}\leq CA^{\frac13}\|{\rm e}^{bA^{-\frac13}t}\rho_{1}\nabla V\cdot\nabla\partial_{x}^{2}u_{3,\neq}\|_{L^{2}L^{2}}^{2}.
\end{equation*}
Using Lemma \ref{lem:kappa} and (\ref{V1' infty}), we obtain
\begin{equation*}
	\|\rho_{1}\nabla V\cdot\nabla\partial_{x}^{2}u_{3,\neq}\|_{L^{2}}\leq C\|\rho_{1}\|_{H^{2}}\|\nabla V\|_{L^{\infty}}\|\nabla\partial_{x}^{2}u_{3,\neq}\|_{L^{2}}\leq C\epsilon\|\nabla\partial_{x}^{2}u_{3,\neq}\|_{L^{2}}.
\end{equation*}
This indicates that
\begin{equation}\label{ineq:Wxx}
	\begin{aligned}
		\|\partial_{x}^{2}W^{(2)}\|_{X_{b}}^{2}\leq&C\epsilon^{2}A^{\frac13}\|{\rm e}^{bA^{-\frac13}t}\nabla\partial_{x}^{2}u_{3,\neq}\|_{L^{2}L^{2}}^{2}\leq C\epsilon^{2}A^{\frac43}\|\partial_{x}^{2}u_{3,\neq}\|_{X_{b}}^{2}\leq C\epsilon^{2}A^{\frac43}E_{5,2}^{2}(t).
	\end{aligned}
\end{equation}
For $ j\in\{1,2,3\}, $ we get
\begin{equation*}
	\begin{aligned}
		\mathcal{L}_{V}\partial_{x}\partial_{j}W^{(2)}=&\partial_{x}\partial_{j}\mathcal{L}_{V}W^{(2)}-\partial_{j}V\partial_{x}^{2}W^{(2)}=-\partial_{j}\left(\rho_{1}\nabla V\cdot\nabla\partial_{x}u_{3,\neq} \right)-\partial_{j}V\partial_{x}^{2}W^{(2)}
	\end{aligned}
\end{equation*}
satisfying $ \partial_{x}\partial_{j}W^{(2)}(0)=0. $ Applying Proposition \ref{Lvf 0} to it, one deduces
\begin{equation}\label{tilde Q}
	\begin{aligned}
		\|\partial_{x}\partial_{j}W^{(2)}\|_{X_{b}}^{2}\leq CA^{\frac13}\|{\rm e}^{bA^{-\frac13}t}\partial_{j}V\partial_{x}^{2}W^{(2)}\|_{L^{2}L^{2}}^{2}
		+CA\|{\rm e}^{bA^{-\frac13}t}\rho_{1}\nabla V\cdot\nabla\partial_{x}u_{3,\neq}\|_{L^{2}L^{2}}^{2}.
	\end{aligned}
\end{equation}
Due to (\ref{V1' infty}) and Lemma \ref{lem:kappa}, there holds
\begin{equation*}
	\begin{aligned}
		&\|\partial_{j}V\partial_{x}^{2}W^{(2)}\|_{L^{2}}\leq \|\partial_{j}V\|_{L^{\infty}}\|\partial_{x}^{2}W^{(2)}\|_{L^{2}}\leq C\|\partial_{x}^{2}W^{(2)}\|_{L^{2}},\\
		&\|\rho_{1}\nabla V\cdot\nabla\partial_{x}u_{3,\neq}\|_{L^{2}}\leq C\|\rho_{1}\|_{H^{2}}\|\nabla V\|_{L^{\infty}}\|\nabla\partial_{x}u_{3,\neq}\|_{L^{2}}\leq C\epsilon\|\nabla\partial_{x}^{2}u_{3,\neq}\|_{L^{2}},
	\end{aligned}
\end{equation*}
which implies that
\begin{equation*}
	\begin{aligned}
		&\|{\rm e}^{bA^{-\frac13}t}\partial_{j}V\partial_{x}^{2}W^{(2)}\|_{L^{2}L^{2}}^{2}\leq C\|{\rm e}^{bA^{-\frac13}t}\partial_{x}^{2}W^{(2)}\|_{L^{2}L^{2}}^{2}\leq CA^{\frac13}\|\partial_{x}^{2}W^{(2)}\|_{X_{b}}^{2},\\
		&\|{\rm e}^{bA^{-\frac13}t}\rho_{1}\nabla V\cdot\nabla\partial_{x}u_{3,\neq}\|_{L^{2}L^{2}}^{2}\leq C\epsilon^{2}\|{\rm e}^{bA^{-\frac13}t}\nabla\partial_{x}^{2}u_{3,\neq}\|_{L^{2}L^{2}}^{2}\leq C\epsilon^{2}A\|\partial_{x}^{2}u_{3,\neq}\|_{X_{b}}^{2}.
	\end{aligned}
\end{equation*}
Substituting the above estimations into (\ref{tilde Q}) and using (\ref{ineq:Wxx}), we arrive at
\begin{equation*}
	\|\partial_{x}\nabla W^{(2)}\|_{X_{b}}^{2}\leq CA^{\frac23}\|\partial_{x}^{2}W^{(2)}\|_{X_{b}}^{2}+C\epsilon^{2}A^{2}\|\partial_{x}^{2}u_{3,\neq}\|_{X_{b}}^{2}\leq C\epsilon^{2}A^{2}E_{5,2}^{2}(t).
\end{equation*}

{\textbf{Estimate (iii).}}
Due to
\begin{equation*}
	\begin{aligned}
		\mathcal{L}_{V}(\rho_{1}f)-\rho_{1}\mathcal{L}_{V}f=~&(\partial_{t}\rho_{1}-A^{-1}\Delta\rho_{1})f-2A^{-1}\nabla\rho_{1}\cdot\nabla f\\=~&(\partial_{t}\rho_{1}+A^{-1}\Delta\rho_{1})f-2A^{-1}\nabla\cdot(f\nabla\rho_{1})
	\end{aligned}
\end{equation*}
and $ \mathcal{L}_{V}W^{(2)}=\rho_{1}\mathcal{L}_{V}W^{(3)}, $ there holds
\begin{equation*}
	\begin{aligned}
		\mathcal{L}_{V}\partial_{x}\left(W^{(2)}-\rho_{1}W^{(3)} \right)=&\partial_{x}\mathcal{L}_{V}\left(W^{(2)}-\rho_{1}W^{(3)} \right)=\partial_{x}\left(\rho_{1}\mathcal{L}_{V}W^{(3)}-\mathcal{L}_{V}(\rho_{1}W^{(3)}) \right)\\=&-\partial_{x}\left[(\partial_{t}\rho_{1}+A^{-1}\Delta\rho_{1}) W^{(3)}-2A^{-1}\nabla\cdot(W^{(3)}\nabla\rho_{1})\right]\\=&-\Delta Q+2A^{-1}\nabla\cdot\left(\partial_{x}W^{(3)}\nabla\rho_{1} \right),
	\end{aligned}
\end{equation*}	
where $ \Delta Q=\left(\partial_{t}\rho_{1}+A^{-1}\Delta\rho_{1} \right)\partial_{x}W^{(3)}. $
Applying {Proposition \ref{Lvf 0}}, we get
\begin{equation}\label{Lv Q}
	\begin{aligned}
		&\|\partial_{x}(W^{(2)}-\rho_{1}W^{(3)} )\|_{X_{b}}^{2}\leq C\left(A\|{\rm e}^{bA^{-\frac13}t}\nabla Q\|_{L^{2}L^{2}}^{2}+A^{-1}\|{\rm e}^{bA^{-\frac13}t}\partial_{x}W^{(3)}\nabla\rho_{1}\|_{L^{2}L^{2}}^{2} \right).
	\end{aligned}
\end{equation}
It follows from Lemma \ref{lem:kappa} that $ \|\nabla\rho_{1}\|_{H^{1}}\leq\|\rho_{1}\|_{H^{2}}\leq C\epsilon $ and
\begin{equation*}
	\begin{aligned}
		\|\partial_{t}\rho_{1}+A^{-1}\Delta\rho_{1}\|_{L^{2}}\leq\|\partial_{t}\rho_{1}\|_{L^{2}}+A^{-1}\|\rho_{1}\|_{H^{2}}\leq CA^{-1}\epsilon.
	\end{aligned}
\end{equation*}
Combining them with Lemma \ref{sob_14}, one obtains
	\begin{align}\label{a}
		\|\nabla Q\|_{L^{2}}=~&\|\nabla\Delta^{-1}\Delta Q\|_{L^{2}}=\|\nabla\Delta^{-1}[\left(\partial_{t}\rho_{1}+A^{-1}\Delta\rho_{1} \right)\partial_{x}W^{(3)} ]\|_{L^{2}}\nonumber\\\leq~&C\|\partial_{t}\rho_{1}+A^{-1}\Delta\rho_{1}\|_{L^{2}}\left(\|\partial_{x}W^{(3)}\|_{L^{2}}+\|\partial_{x}\left(\partial_{z}-\kappa\partial_{y} \right)W^{(3)}\|_{L^{2}}\right)\\\nonumber\leq~&CA^{-1}\epsilon\left(\|\partial_{x}^{2}W^{(3)}\|_{L^{2}}+\|\partial_{x}(\partial_{z}-\kappa\partial_{y})W^{(3)}\|_{L^{2}}\right)
	\end{align}
and
\begin{equation}\label{b}
	\begin{aligned}
		\|\partial_{x}W^{(3)}\nabla\rho_{1}\|_{L^{2}}\leq~&C\|\nabla\rho_{1}\|_{H^{1}}\left(\|\partial_{x}W^{(3)}\|_{L^{2}}+\|\partial_{x}(\partial_{z}-\kappa\partial_{y}) W^{(3)}\|_{L^{2}}\right)\\\leq~&C\epsilon\left(\|\partial_{x}^{2}W^{(3)}\|_{L^{2}}+\|\partial_{x}(\partial_{z}-\kappa\partial_{y})W^{(3)}\|_{L^{2}} \right).
	\end{aligned}
\end{equation}
Substituting (\ref{a}) and (\ref{b}) into (\ref{Lv Q}), and using the result of (i), we have
\begin{equation*}
	\begin{aligned}
		\|\partial_{x}(W^{(2)}-\rho_{1}W^{(3)})\|_{X_{b}}^{2}\leq~&CA^{-1}\epsilon^{2}\left(\|{\rm e}^{bA^{-\frac13}t}\partial_{x}^{2}W^{(3)}\|_{L^{2}L^{2}}^{2}
		+\|{\rm e}^{bA^{-\frac13}t}\partial_{x}(\partial_{z}-\kappa\partial_{y})W^{(3)}\|_{L^{2}L^{2}}^{2} \right)\\\leq~&CA^{-1}\epsilon^{2}A^{\frac13}\left(\|\partial_{x}^{2}W^{(3)}\|_{X_{b}}^{2}+\|\partial_{x}(\partial_{z}-\kappa\partial_{y})W^{(3)}\|_{X_{b}}^{2} \right)\leq CA^{\frac23}\epsilon^{2}E_{5,2}^{2}(t).
	\end{aligned}
\end{equation*}
\end{proof}

\begin{lemma}\label{lem:U1}
Under the assumptions of Theorem \ref{result0} and  \eqref{assumption}, it holds that
\begin{equation}\label{ineq:U1}
	\begin{aligned}
		&\|{\rm e}^{bA^{-\frac13}t}\nabla(\mathbf{U}_{1}\partial_{x}W )\|_{L^{2}L^{2}}^{2}\leq C\left(\|(u_{1,\rm in})_{0}\|_{H^{1}}^{2}+H_{1}^{2}\right)\left(A^{\frac12}\|\Delta W^{(1)}\|_{X_{b}}^{2}+\epsilon^{2}A^{\frac23}E_{5,2}^{2}(t) \right), \\
		&\|{\rm e}^{bA^{-\frac13}t}\nabla(\mathbf{U}_{1}\partial_{x}u_{3,\neq})\|_{L^{2}L^{2}}^{2}\leq C\left(\|(u_{1,\rm in})_{0}\|_{H^{1}}^{2}+H_{1}^{2}\right)AE_{5,2}^{2}(t),\\
		&\|{\rm e}^{bA^{-\frac13}t}\mathbf{U}_{1}\partial_{x}^{2}(u_{2,\neq}, u_{3,\neq})\|_{L^{2}L^{2}}^{2}\leq C\left(\|(u_{1,\rm in})_{0}\|_{H^{1}}^{2}+H_{1}^{2}\right)A^{\frac13(2\alpha+1)}E_{5,2}^{2}(t),\\
		&\|{\rm e}^{bA^{-\frac13}t}P_{1,2}\|_{L^{2}L^{2}}^{2}\leq C\left(\|(u_{1,\rm in})_{0}\|_{H^{1}}^{2}+H_{1}^{2}\right)A^{\frac23}E_{5,2}^{2}(t),
	\end{aligned}
\end{equation}
where $\alpha\in(\frac12,\frac34)$ and the definition of $H_{1}$ is the same as in Lemma \ref{lem:n0 L2}.
\end{lemma}
\begin{proof}
{\bf Estimate $(\ref{ineq:U1})_{1}.$} Recalling that $W=W^{(1)}+\frac{1}{A}W^{(2)},$  we have
\begin{equation*}\label{U1 W}
	\|{\rm e}^{bA^{-\frac13}t}\nabla(\mathbf{U}_{1}\partial_{x}W)\|_{L^{2}L^{2}}^{2}\leq \|{\rm e}^{bA^{-\frac13}t}\nabla(\mathbf{U}_{1}\partial_{x}W^{(1)})\|_{L^{2}L^{2}}^{2}+\frac{1}{A^{2}}
	\|{\rm e}^{bA^{-\frac13}t}\nabla(\mathbf{U}_{1}\partial_{x}W^{(2)})\|_{L^{2}L^{2}}^{2}.
\end{equation*}
Using Lemma \ref{lemma:u1:1}, Lemma \ref{lem:n0 L2}, Lemma \ref{sob_inf_1}, Lemma \ref{sob_inf_2}
and
$$\|{\rm e}^{bA^{-\frac13}t}\partial_{x}\nabla f\|_{L^{2}L^{2}}=\|{\rm e}^{bA^{-\frac13}t}\nabla\Delta^{-1}\partial_{x}(\Delta f)\|_{L^{2}L^{2}}\leq \|\Delta f\|_{X_{b}},$$
one obtains
\begin{equation}\label{U1 W 1}
	\begin{aligned}
		&\|{\rm e}^{bA^{-\frac13}t}\nabla(\mathbf{U}_{1}\partial_{x}W^{(1)})\|_{L^{2}L^{2}}^{2}
		\\\leq~&\|\mathbf{U}_{1}\|_{L^{\infty}_{t,z}L^{2}_{y}}^{2}\|{\rm e}^{bA^{-\frac13}t}\nabla\partial_{x}W^{(1)}\|_{L^{2}_{t,x,z}L^{\infty}_{y}}^{2}
		+\|\nabla\mathbf{U}_{1}\|_{L^{\infty}_{t}L^{2}_{y,z}}^{2}\|{\rm e}^{bA^{-\frac13}t}\partial_{x}W^{(1)}\|_{L^{2}_{t,x}L^{\infty}_{y,z}}^{2}
		\\\leq~&C\|\mathbf{U}_{1}\|_{L^{\infty}H^{1}}^{2}\big(\|{\rm e}^{bA^{-\frac13}t}\partial_{x}\nabla W^{(1)}\|_{L^{2}L^{2}}
		\|{\rm e}^{bA^{-\frac13}t}\partial_{x}\Delta W^{(1)}\|_{L^{2}L^{2}}\\&+\|{\rm e}^{bA^{-\frac13}t}\partial_{x}\nabla W^{(1)}\|_{L^{2}L^{2}}^{3-2\alpha}\|{\rm e}^{bA^{-\frac13}t}\partial_{x}\Delta W^{(1)}\|_{L^{2}L^{2}}^{2\alpha-1} \big)
		\\\leq~&C\left(\|(u_{1,\rm in})_{0}\|_{H^{1}}^{2}+H_{1}^{2} \right)A^{\frac12}\|\Delta W^{(1)}\|_{X_{b}}^{2},
	\end{aligned}
\end{equation}
where $\alpha\in (\frac12,\frac34).$
Similarly, by Lemma \ref{lemma:u1:1}, Lemma \ref{lem:n0 L2}, Lemma \ref{sob_inf_1}, Lemma \ref{sob_inf_2} and (ii) of Lemma \ref{lem W}, there holds
	\begin{align*}
		&\|{\rm e}^{bA^{-\frac13}t}\nabla(\mathbf{U}_{1}\partial_{x}W^{(2)})\|_{L^{2}L^{2}}^{2}\\\leq~& C\|\mathbf{U}_{1}\|_{L^{\infty}H^{1}}^{2}\big(
		\|{\rm e}^{bA^{-\frac13}t}\partial_{x}\nabla W^{(2)}\|_{L^{2}L^{2}}\|{\rm e}^{bA^{-\frac13}t}\partial_{x}\Delta W^{(2)}\|_{L^{2}L^{2}}\\&+\|{\rm e}^{bA^{-\frac13}t}\partial_{x}\nabla W^{(2)}\|_{L^{2}L^{2}}^{3-2\alpha}\|{\rm e}^{bA^{-\frac13}t}\partial_{x}\Delta W^{(2)}\|_{L^{2}L^{2}}^{2\alpha-1}\big)\\\leq~&C(\|(u_{1,\rm in})_{0}\|_{H^{1}}^{2}+H_{1}^{2})A^{\frac23}\|\partial_{x}\nabla W^{(2)}\|_{X_{b}}^{2}\leq C(\|(u_{1,\rm in})_{0}\|_{H^{1}}^{2}+H_{1}^{2})\epsilon^{2}A^{\frac83}E_{5,2}^{2}(t).
	\end{align*}
Collecting this with (\ref{U1 W 1}), we get the first result.

\textbf{Estimate $(\ref{ineq:U1})_2$.}
Noting that $\partial_x\mathbf U_1=0$ and
$P_0(\partial_xu_{3,\neq})=0$, Lemma \ref{sob_14} gives
\[
\|\nabla(\mathbf U_1\partial_xu_{3,\neq})\|_{L^2}
\le C\|\mathbf U_1\|_{H^1}
\left(
\|\partial_xu_{3,\neq}\|_{H^1}
+\|(\partial_z-\kappa\partial_y)\partial_xu_{3,\neq}\|_{H^1}
\right).
\]
Using $[\partial_z-\kappa\partial_y,\partial_x]=0$ and Poincar\'e inequality, for $A\ge1$, we obtain
\begin{align*}
&\|{\rm e}^{bA^{-\frac13}t}\partial_xu_{3,\neq}\|_{L^2H^1}^2
+\|{\rm e}^{bA^{-\frac13}t}(\partial_z-\kappa\partial_y)\partial_xu_{3,\neq}\|_{L^2H^1}^2\\\le~& CA\left(
\|\partial_x^2u_{3,\neq}\|_{X_b}^2
+\|\partial_x(\partial_z-\kappa\partial_y)u_{3,\neq}\|_{X_b}^2
\right)
\le CAE_{5,2}^2(t).
\end{align*}
Therefore, by Lemmas \ref{lemma:u1:1}
and \ref{lem:n0 L2}, one deduces
\begin{align*}
\|{\rm e}^{bA^{-\frac13}t}\nabla(\mathbf U_1\partial_xu_{3,\neq})\|_{L^2L^2}^2\le
CA\|\mathbf U_1\|_{L^\infty H^1}^2E_{5,2}^2(t)\le
C\left(
\|(u_{1,\rm in})_0\|_{H^1}^2+H_1^2
\right)AE_{5,2}^2(t).
\end{align*}

{\bf Estimate $(\ref{ineq:U1})_{3}.$} Due to Lemma \ref{lemma:u1:1}, Lemma \ref{lem:n0 L2}, Lemma \ref{sob_inf_1} and Lemma \ref{sob_inf_2}, for $j\in\{2,3\}$ and $\alpha\in(\frac12,\frac34),$ there holds
\begin{equation*}
	\begin{aligned}
		\|{\rm e}^{bA^{-\frac13}t}\mathbf{U}_{1}\partial_{x}^{2}u_{j,\neq}\|_{L^{2}L^{2}}^{2}\leq~&\|\mathbf{U}_{1}\|_{L^{\infty}_{t}L^{2}_{z}L^{\infty}_{y}}^{2}\|{\rm e}^{bA^{-\frac13}t}\partial_{x}^{2}u_{j,\neq}\|_{L^{2}_{t}L^{\infty}_{z}L^{2}_{x,y}}^{2}\\\leq~&C\|\mathbf{U}_{1}\|_{L^{\infty}H^{1}}^{2}\|{\rm e}^{bA^{-\frac13}t}\partial_{x}^{2}u_{j,\neq}\|_{L^{2}L^{2}}^{2-2\alpha}\|{\rm e}^{bA^{-\frac13}t}\nabla\partial_{x}^{2}u_{j,\neq}\|_{L^{2}L^{2}}^{2\alpha}\\\leq~& C\left(\|(u_{1,\rm in})_{0}\|_{H^{1}}^{2}+H_{1}^{2}\right)A^{\frac13(2\alpha+1)}E_{5,2}^{2}(t).
	\end{aligned}
\end{equation*}

{\bf Estimate $(\ref{ineq:U1})_{4}.$} According to $P_{1,2}=\partial_{y}\mathbf{U}_{1}\partial_{x}u_{2,\neq}+\partial_{z}\mathbf{U}_{1}\partial_{x}u_{3,\neq}$ in (\ref{P11 12 13}), using Lemma \ref{lemma:u1:1}, Lemma \ref{lem:n0 L2} and Lemma \ref{weak kappa}, for $j\in\{2,3\},$ we get
\begin{equation*}\label{b P12}
	\begin{aligned}
		&\|{\rm e}^{bA^{-\frac13}t}P_{1,2}\|_{L^{2}L^{2}}^{2}\leq\|{\rm e}^{bA^{-\frac13}t}\partial_{j}\mathbf{U}_{1}\partial_{x}u_{j,\neq}\|_{L^{2}L^{2}}^{2}\leq\|\partial_{j}\mathbf{U}_{1}\|_{L^{\infty}_{t}L^{2}_{y,z}}^{2}\|{\rm e}^{bA^{-\frac13}t}\partial_{x}u_{j,\neq}\|_{L^{2}_{t}L^{\infty}_{y,z}L^{2}_{x}}^{2}\\\leq~& C\|\mathbf{U}_{1}\|_{L^{\infty}H^{1}}^{2}\|{\rm e}^{bA^{-\frac13}t}(\partial_{x}, \partial_{z}-\kappa\partial_{y})\partial_{x}u_{j,\neq}\|_{L^{2}L^{2}}\|{\rm e}^{bA^{-\frac13}t}\nabla(\partial_{x}, \partial_{z}-\kappa\partial_{y})\partial_{x}u_{j,\neq}\|_{L^{2}L^{2}}\\\leq~&C\left(\|(u_{1,\rm in})_{0}\|_{H^{1}}^{2}+H_{1}^{2} \right)A^{\frac23}E_{5,2}^{2}(t).
	\end{aligned}
\end{equation*}

\end{proof}

\begin{lemma}\label{lem:g}
Under the assumptions of Theorem \ref{result0} and \eqref{assumption}, it holds that
\begin{equation}\label{estimate:g}
	\begin{aligned}
		&\|{\rm e}^{bA^{-\frac13}t}\nabla g_{3,1}\|_{L^{2}L^{2}}^{2}\leq CA^{\frac53}\epsilon^{2}\left(E_{5,1}^2(t)+E_{5,2}^{2}(t) \right),
		\\
		&\|{\rm e}^{bA^{-\frac13}t}\nabla(g_{2,1}+\kappa g_{3,1})\|_{L^{2}L^{2}}^{2}+\|{\rm e}^{bA^{-\frac13}t}\partial_{x}g_{3,1}\|_{L^{2}L^{2}}^{2}\\&+\|{\rm e}^{bA^{-\frac13}t}\nabla g_{2,2}\|_{L^{2}L^{2}}^{2}+\|{\rm e}^{bA^{-\frac13}t}\nabla g_{3,2}\|_{L^{2}L^{2}}^{2}\leq C\epsilon^{2}AE_{5,2}^{2}(t).
	\end{aligned}
\end{equation}
\end{lemma}
\begin{proof}
{\bf Estimate $(\ref{estimate:g})_{1}.$}
By Lemma \ref{lemma_u23_1}, Lemma \ref{sob_inf_1} and $\partial_{z}u_{3,0}=-\partial_{y}u_{2,0},$ we get
\begin{equation}\label{u20 u30 2 infty}
	\|\nabla(u_{2,0}, u_{3,0})\|_{L^{2}L^{\infty}}^{2}\leq CA\epsilon^{2},
\end{equation}
which along with (\ref{u20 u30 infty}) and
$$\|\nabla u_{3,\neq}\|_{L^{2}}^{2}\leq\|u_{3,\neq}\|_{L^{2}}\|\Delta u_{3,\neq}\|_{L^{2}}\leq
A^{\frac23}\|\partial_{x}^{2}u_{3,\neq}\|_{L^{2}}^{2}+A^{-\frac23}\|\Delta u_{3,\neq}\|_{L^{2}}^{2}$$
shows that
\begin{equation*}\label{h31}
	\begin{aligned}
		&\|{\rm e}^{bA^{-\frac13}t}\nabla g_{3,1}\|_{L^{2}L^{2}}^{2}
		=\|{\rm e}^{bA^{-\frac13}t}\nabla\left( u_{2,0}\partial_{y}u_{3,\neq}+u_{3,0}\partial_{z}u_{3,\neq} \right)\|_{L^{2}L^{2}}^{2}
		\\\leq~&\|\nabla (u_{2,0},u_{3,0})\|_{L^{2}L^{\infty}}^{2}\|{\rm e}^{bA^{-\frac13}t}\nabla u_{3,\neq}\|_{L^{\infty}L^{2}}^{2}+\|(u_{2,0},u_{3,0})\|_{L^{\infty}L^{\infty}}^{2}\|{\rm e}^{bA^{-\frac13}t}\Delta u_{3,\neq}\|_{L^{2}L^{2}}^{2}\\\leq~&CA\epsilon^{2}\|\nabla u_{3,\neq}\|_{X_{b}}^{2}\leq CA^{\frac53}\epsilon^{2}\big(\|\partial_{x}^{2}u_{3,\neq}\|_{X_{b}}^{2}+A^{-\frac43}\|\Delta u_{3,\neq}\|_{X_{b}}^{2} \big)\leq CA^{\frac53}\epsilon^{2}\left(E_{5,2}^{2}(t)+ E_{5,1}^{2}(t)\right).
	\end{aligned}
\end{equation*}	

{\bf Estimate $(\ref{estimate:g})_{2}.$}
As
$
g_{2,1}+\kappa g_{3,1}=\left(u_{2,0}\partial_{y}+u_{3,0}\partial_{z} \right)W-u_{3,\neq}(u_{2,0}\partial_{y}+u_{3,0}\partial_{z})\kappa,
$
by (\ref{u20 u30 infty}) and (\ref{u20 u30 2 infty}), we get
\begin{equation*}
	\|{\rm e}^{bA^{-\frac13}t}\nabla(g_{2,1}+\kappa g_{3,1})\|_{L^{2}L^{2}}^{2}\leq CA\epsilon^{2}\left(\|{\rm e}^{bA^{-\frac13}t}\nabla W\|_{Y_{0}}^{2}+\|{\rm e}^{bA^{-\frac13}t}u_{3,\neq}\nabla\kappa\|_{Y_{0}}^{2} \right).
\end{equation*}
Direct calculations show that
\begin{equation*}
	\|{\rm e}^{bA^{-\frac13}t}\nabla W\|_{Y_{0}}^{2}\leq\|{\rm e}^{bA^{-\frac13}t}\partial_{x}\nabla W\|_{Y_{0}}^{2}\leq \|\partial_{x}\nabla W\|_{X_{b}}^{2}\leq E_{5,2}^{2}(t).
\end{equation*}
On the other hand, using Lemma \ref{lem:kappa}, one obtains
\begin{equation*}
	\begin{aligned}
		&\|u_{3,\neq}\nabla\kappa\|_{L^{2}}\leq\|u_{3,\neq}\|_{L^{2}}\|\nabla\kappa\|_{L^{\infty}}\leq C\|\partial_{x}^{2}u_{3,\neq}\|_{L^{2}},\\&\|\nabla(u_{3,\neq}\nabla\kappa)\|_{L^{2}}\leq\|u_{3,\neq}\nabla\kappa\|_{H^{1}}\leq C\|u_{3,\neq}\|_{H^{1}}\|\nabla\kappa\|_{H^{2}}\leq C\|\nabla\partial_{x}^{2}u_{3,\neq}\|_{L^{2}},
	\end{aligned}
\end{equation*}
which indicates that
\begin{equation*}
	\begin{aligned}
		\|{\rm e}^{bA^{-\frac13}t}u_{3,\neq}\nabla\kappa\|_{Y_{0}}^{2}\leq~&C\left(\|{\rm e}^{bA^{-\frac13}t}\partial_{x}^{2}u_{3,\neq}\|_{L^{\infty}L^{2}}^{2}+\frac{1}{A}\|{\rm e}^{bA^{-\frac13}t}\nabla\partial_{x}^{2}u_{3,\neq}\|_{L^{2}L^{2}}^{2} \right)\\\leq~& C\|\partial_{x}^{2}u_{3,\neq}\|_{X_{b}}^{2}\leq CE_{5,2}^{2}(t).
	\end{aligned}
\end{equation*}
This gives
\begin{equation}\label{g21 g31}
	\|{\rm e}^{bA^{-\frac13}t}\nabla(g_{2,1}+\kappa g_{3,1})\|_{L^{2}L^{2}}^{2}\leq C\epsilon^{2}AE_{5,2}^{2}(t).
\end{equation}
Recall that $\partial_{x}g_{3,1}=(u_{2,0}\partial_{y}+u_{3,0}\partial_{z})\partial_{x}u_{3,\neq}.$ By (\ref{u20 u30 infty}), we have
\begin{equation}\label{x g31}
	\begin{aligned}
		\|{\rm e}^{bA^{-\frac13}t}\partial_{x}g_{3,1}\|_{L^{2}L^{2}}^{2}\leq\|(u_{2,0},u_{3,0})\|_{L^{\infty}L^{\infty}}^{2}\|{\rm e}^{bA^{-\frac13}t}\nabla\partial_{x}u_{3,\neq}\|_{L^{2}L^{2}}^{2}\leq C\epsilon^{2}AE_{5,2}^{2}(t).
	\end{aligned}
\end{equation}
For $ g_{2,2}=u_{\neq}\cdot\nabla u_{2,0} $ and $ j\in\{2,3\}, $ it follows from Lemma \ref{sob_14} and  Lemma \ref{lemma_u23_1} that
\begin{equation*}
	\begin{aligned}
		\|\nabla g_{2,2}\|_{L^{2}}=~&\|\nabla\left(u_{j,\neq}\partial_{j}u_{2,0} \right)\|_{L^{2}}\leq C\|\partial_{j}u_{2,0}\|_{H^{1}}\left(\|u_{j,\neq}\|_{H^{1}}+\|(\partial_{z}-\kappa\partial_{y}) u_{j,\neq}\|_{H^{1}}\right)\\\leq~&C\|\nabla u_{2,0}\|_{H^{1}}\left(\|\nabla\partial_{x}^{2}u_{j,\neq}\|_{L^{2}}+\|\nabla\partial_{x}(\partial_{z}-\kappa\partial_{y})u_{j,\neq}\|_{L^{2}} \right)
	\end{aligned}
\end{equation*}	
and
\begin{equation}\label{h22}
	\begin{aligned}
		\|{\rm e}^{bA^{-\frac13}t}\nabla g_{2,2}\|_{L^{2}L^{2}}^{2}\leq~&C\epsilon^{2}\left(\|{\rm e}^{bA^{-\frac13}t}\nabla\partial_{x}^{2}u_{j,\neq}\|_{L^{2}L^{2}}^{2}+\|{\rm e}^{bA^{-\frac13}t}\nabla\partial_{x}(\partial_{z}-\kappa\partial_{y})u_{j,\neq}\|_{L^{2}L^{2}}^{2} \right)\\\leq~&C\epsilon^{2}A\left(\|\partial_{x}^{2}u_{j,\neq}\|_{X_{b}}^{2}+\|\partial_{x}(\partial_{z}-\kappa\partial_{y})u_{j,\neq}\|_{X_{b}}^{2} \right)\leq C\epsilon^{2}AE_{5,2}^{2}(t).
	\end{aligned}
\end{equation}
For $ g_{3,2}=u_{\neq}\cdot\nabla u_{3,0}, $ we rewrite it into
\begin{equation*}
	g_{3,2}=\left(u_{2,\neq}\partial_{y}+u_{3,\neq}\partial_{z} \right)u_{3,0}=W\partial_{y}u_{3,0}+u_{3,\neq}\left(\partial_{z}u_{3,0}-\kappa\partial_{y}u_{3,0} \right),
\end{equation*}
which implies that
\begin{equation}\label{nabla h32}
	\begin{aligned}
		\|{\rm e}^{bA^{-\frac13}t}\nabla g_{3,2}\|_{L^{2}L^{2}}^{2}\leq \|{\rm e}^{bA^{-\frac13}t}\nabla(W\partial_{y}u_{3,0})\|_{L^{2}L^{2}}^{2}+\|{\rm e}^{bA^{-\frac13}t}\nabla\left(u_{3,\neq}(\partial_{z}u_{3,0}-\kappa\partial_{y}u_{3,0}) \right)\|_{L^{2}L^{2}}^{2}.
	\end{aligned}
\end{equation}
By Lemma \ref{sob_inf_1}, Lemma \ref{sob_inf_2}, Lemma \ref{lemma_u23_1} and $\partial_{z}u_{3,0}=-\partial_{y}u_{2,0}$, we get
\begin{equation}\label{W u30}
	\begin{aligned}
		&\|{\rm e}^{bA^{-\frac13}t}\nabla(W\partial_{y}u_{3,0})\|_{L^{2}L^{2}}^{2}
		\\\leq~&\|{\rm e}^{bA^{-\frac13}t}\nabla W\|_{L^{\infty}L^{2}}^{2}\|\partial_{y}u_{3,0}\|_{L^{2}L^{\infty}}^{2}
		+\|{\rm e}^{bA^{-\frac13}t}W\|_{L^{\infty}_{t,y}L^{2}_{x,z}}^{2}\|\nabla\partial_{y}u_{3,0}\|_{L^{2}_{t,y}L^{\infty}_{z}}^{2}
		\\\leq~&C\left(\|\nabla u_{3,0}\|_{L^{2}H^{1}}^{2}+\|\nabla u_{2,0}\|_{L^{2}H^{2}}^{2} \right)
		\|{\rm e}^{bA^{-\frac13}t}\nabla\partial_{x}W\|_{L^{\infty}L^{2}}^{2}\leq C\epsilon^{2}AE_{5,2}^{2}(t).
	\end{aligned}
\end{equation}
As $\mathbf{U}_{2}(0)=0,$ using Lemma \ref{u1_hat2}, there holds
\begin{equation*}
	\|\Delta\mathbf{U}_{2}\|_{L^{2}}\leq	\|\mathbf{U}_{2}(t)\|_{H^{2}}\leq\int_{0}^{t}\|\partial_{t}\mathbf{U}_{2}(s)\|_{H^{2}}ds\leq C\epsilon t.
\end{equation*}
On the other hand, $\|\Delta \mathbf{U}_{2}(t)\|_{H^{2}}\leq CA\epsilon.$ 
Thus,
$\|\Delta\mathbf{U}_{2}\|_{L^{2}}\leq C\epsilon A\min\{A^{-1}t, 1\}$,
which implies 
\begin{equation}\label{t kappa}
	\|\kappa\|_{H^{1}}\leq CA^{-1}\|\Delta \mathbf{U}_{2}\|_{L^{2}}\leq C\epsilon\min\{A^{-1}t, 1\}.
\end{equation}
From this, along with Lemma \ref{lemma_u23_1} and Lemma \ref{lem:kappa}, one obtains
\begin{equation}\label{kappa u30'}
	\begin{aligned}
		\|\kappa\nabla u_{3,0}\|_{H^{1}}\leq~& C\|\kappa\|_{H^{2}}\left(\|\nabla u_{3,0}\|_{L^{2}}+\|\Delta u_{3,0}\|_{L^{2}} \right)\leq C\|\kappa\|_{H^{1}}^{\frac12}\|\kappa\|_{H^{3}}^{\frac12}\left(\|\nabla u_{3,0}\|_{L^{2}}+\|\Delta u_{3,0}\|_{L^{2}} \right)\\\leq~&C\epsilon\min\{A^{-1}t, 1 \}^{\frac12}\left(\|\nabla u_{3,0}\|_{L^{2}}+\|\Delta u_{3,0}\|_{L^{2}} \right)\\\leq~&C\epsilon\left(\|\nabla u_{3,0}\|_{L^{2}}+\|\min\left(A^{-\frac23}+A^{-1}t ,1 \right)^{\frac12}\Delta u_{3,0}\|_{L^{2}} \right)\leq C\epsilon^{2}.
	\end{aligned}
\end{equation}
Due to Lemma \ref{lemma_u23_1}, Lemma \ref{sob_14}, (\ref{kappa u30'}) and $\partial_{z}u_{3,0}=-\partial_{y}u_{2,0}$, we arrive
\begin{equation}\label{u3 u30}
	\begin{aligned}
		&\|{\rm e}^{bA^{-\frac13}t}\nabla\left(u_{3,\neq}(\partial_{z}u_{3,0}-\kappa\partial_{y}u_{3,0}) \right)\|_{L^{2}L^{2}}^{2}\\\leq~&C\left(\|\partial_{z}u_{3,0}\|_{L^{\infty}H^{1}}^{2}+\|\kappa\nabla u_{3,0}\|_{L^{\infty}H^{1}}^{2} \right)\left(\|{\rm e}^{bA^{-\frac13}t}u_{3,\neq}\|_{L^{2}H^{1}}^{2}+\|{\rm e}^{bA^{-\frac13}t}(\partial_{z}-\kappa\partial_{y})u_{3,\neq}\|_{L^{2}H^{1}}^{2} \right)\\\leq~&C\epsilon^{2}\left(\|{\rm e}^{bA^{-\frac13}t}\nabla\partial_{x}^{2}u_{3,\neq}\|_{L^{2}L^{2}}^{2}+\|{\rm e}^{bA^{-\frac13}t}\nabla(\partial_{z}-\kappa\partial_{y})u_{3,\neq}\|_{L^{2}L^{2}}^{2} \right)\leq C\epsilon^{2}AE_{5,2}^{2}(t).
	\end{aligned}
\end{equation}
Combining (\ref{W u30}) with (\ref{u3 u30}), we get from (\ref{nabla h32}) that
\begin{equation}\label{932}
	\|{\rm e}^{bA^{-\frac13}t}\nabla g_{3,2}\|_{L^{2}L^{2}}^{2}\leq C\epsilon^{2}AE_{5,2}^{2}(t).
\end{equation}

Then $(\ref{estimate:g})_{2}$ follows from (\ref{g21 g31}), (\ref{x g31}), (\ref{h22}) and  (\ref{932}).

\end{proof}

\begin{lemma}\label{lem:G1 G2}
Under the assumptions of Theorem \ref{result0} and \eqref{assumption}, it holds that
\begin{equation*}
	\begin{aligned}
		\|{\rm e}^{bA^{-\frac13}t}\nabla G^{(1)}\|_{L^{2}L^{2}}^{2}\leq~& CA^{\frac16+\alpha}E_{2}^{4}+CA^{\frac12+\frac23\alpha}E_{2}^{4}+CA^{\frac{\alpha}{3}+\frac12}E_{2}^{2}E_{5}^{2}\\\qquad\qquad&+C\left(\|(u_{1,\rm in})_{0}\|_{H^{1}}^{2}+H_{1}^{2}\right)A^{\frac23}E_{5,2}^{2}(t)+CA^{\frac13}E_{2}^{2},\\
		\|{\rm e}^{bA^{-\frac13}t}\nabla G^{(2)}\|_{L^{2}L^{2}}^{2}\leq~& C\epsilon^{2}AE_{5,2}^{2}(t),
	\end{aligned}
\end{equation*}
where $\alpha\in(\frac12,\frac34).$
\end{lemma}
\begin{proof}
Due to (\ref{g G}) and Lemma \ref{lemma_neq1}, there holds
\begin{equation*}
	\|{\rm e}^{bA^{-\frac13}t}\nabla G_{2,3}\|_{L^{2}L^{2}}^{2}\leq \|{\rm e}^{bA^{-\frac13}t}\nabla(u_{\neq}\cdot\nabla u_{2,\neq})_{\neq}\|_{L^{2}L^{2}}^{2}\leq CA^{\frac12+\frac23\alpha}E_{2}^{4}.
\end{equation*}
For $G_{3,3}=(u_{\neq}\cdot\nabla u_{3,\neq})_{\neq}$ in (\ref{g G}), using Lemma \ref{lemma_neq1} and Lemma \ref{lem:kappa}, we have
\begin{equation}\label{kappa G33 1}
	\begin{aligned}
		\|{\rm e}^{bA^{-\frac13}t}\nabla\kappa G_{3,3}\|_{L^{2}L^{2}}^{2}\leq C\|\kappa\|_{L^{\infty}H^{3}}^{2}\|{\rm e}^{bA^{-\frac13}t}u_{\neq}\cdot\nabla u_{3,\neq}\|_{L^{2}L^{2}}^{2}\leq CA^{\frac23}E_{2}^{4}.
	\end{aligned}
\end{equation}
On the other hand, by Lemma \ref{lem:kappa} and (\ref{t kappa}), one deduces
\begin{equation*}
	\begin{aligned}
		\|\kappa\nabla G_{3,3}\|_{L^{2}}\leq~&C\|\kappa\|_{H^{2}}\|\nabla(u_{\neq}\cdot\nabla u_{3,\neq})\|_{L^{2}}\leq C\|\kappa\|_{H^{1}}^{\frac12}\|\kappa\|_{H^{3}}^{\frac12}\|\nabla(u_{\neq}\cdot\nabla u_{3,\neq})\|_{L^{2}}\\\leq~& C (A^{-1}t)^{\frac12}\|\nabla(u_{\neq}\cdot\nabla u_{3,\neq})\|_{L^{2}},
	\end{aligned}
\end{equation*}
which along with Lemma \ref{lemma_neq1} and $(A^{-\frac13}t)^{\frac12}\leq C{\rm e}^{(2a-b)A^{-\frac13}t}$ implies that
\begin{equation}\label{kappa G33 2}
	\begin{aligned}
		\|{\rm e}^{bA^{-\frac13}t}\kappa\nabla G_{3,3}\|_{L^{2}L^{2}}^{2}\leq&CA^{-\frac23}\|(A^{-\frac13}t)^{\frac12}{\rm e}^{bA^{-\frac13}t}\nabla(u_{\neq}\cdot\nabla u_{3,\neq})\|_{L^{2}L^{2}}^{2}\\\leq&CA^{-\frac23}\|{\rm e}^{2aA^{-\frac13}t}\nabla(u_{\neq}\cdot\nabla u_{3,\neq})\|_{L^{2}L^{2}}^{2}\leq CA^{\frac{\alpha}{3}+\frac12}E_{2}^{2}E_{5}^{2}+CA^{\frac23}E_{2}^{4}.
	\end{aligned}
\end{equation}
It follows from (\ref{kappa G33 1}) and (\ref{kappa G33 2}) that
\begin{equation*}
	\begin{aligned}
		\|{\rm e}^{bA^{-\frac13}t}\nabla(\kappa G_{3,3})\|_{L^{2}L^{2}}^{2}\leq CA^{\frac23}E_{2}^{4}+CA^{\frac{\alpha}{3}+\frac12}E_{2}^{2}E_{5}^{2}.
	\end{aligned}
\end{equation*}
Using Lemma \ref{lemma_neq1}, we note that
\begin{equation*}
	\begin{aligned}
		\|{\rm e}^{bA^{-\frac13}t}P_{1,1}\|_{L^{2}L^{2}}^{2}\leq~&C\big(
		\|{\rm e}^{bA^{-\frac13}t}\partial_{x}(u_{\neq}\cdot\nabla u_{1,\neq})\|_{L^{2}L^{2}}^{2}+\|{\rm e}^{bA^{-\frac13}t}\partial_{y}(u_{\neq}\cdot\nabla u_{2,\neq})\|_{L^{2}L^{2}}^{2}\\&+\|{\rm e}^{bA^{-\frac13}t}\partial_{z}(u_{\neq}\cdot\nabla u_{3,\neq})\|_{L^{2}L^{2}}^{2}
		\big)\leq CA^{\frac16+\alpha}E_{2}^{4}+CA^{\frac12+\frac23\alpha}E_{2}^{4}.
	\end{aligned}
\end{equation*}
From this, along with (\ref{P11 12 13}), $(\ref{ineq:U1})_{4}$ and Lemma \ref{lem:kappa}, one obtains
\begin{equation}\label{estimate P}
	\begin{aligned}
		&\|{\rm e}^{bA^{-\frac13}t}\nabla\left((\partial_{y}+\kappa\partial_{z})\Delta^{-1}(P_{1,1}+P_{1,2}+P_{1,3}) \right)\|_{L^{2}L^{2}}^{2}\\\leq~&C\left(\|{\rm e}^{bA^{-\frac13}t}P_{1,1}\|_{L^{2}L^{2}}^{2}+\|{\rm e}^{bA^{-\frac13}t}P_{1,2}\|_{L^{2}L^{2}}^{2}+\|{\rm e}^{bA^{-\frac13}t}P_{1,3}\|_{L^{2}L^{2}}^{2} \right)\\\leq~&CA^{\frac16+\alpha}E_{2}^{4}+CA^{\frac12+\frac23\alpha}E_{2}^{4}+C(\|(u_{1,\rm in})_{0}\|_{H^{1}}^{2}+H_{1}^{2})A^{\frac23}E_{5,2}^{2}(t)+CA^{\frac13}E_{2}^{2}.
	\end{aligned}
\end{equation}
Based on the above estimates, we conclude that
\begin{equation}\label{estimate G1}
	\begin{aligned}
		\|{\rm e}^{bA^{-\frac13}t}\nabla G^{(1)}\|_{L^{2}L^{2}}^{2}\leq~& CA^{\frac16+\alpha}E_{2}^{4}+CA^{\frac12+\frac23\alpha}E_{2}^{4}+CA^{\frac{\alpha}{3}+\frac12}E_{2}^{2}E_{5}^{2}\\&+C\left(\|(u_{1,\rm in})_{0}\|_{H^{1}}^{2}+H_{1}^{2}\right)A^{\frac23}E_{5,2}^{2}(t)+CA^{\frac13}E_{2}^{2}.
	\end{aligned}
\end{equation}
For $G^{(2)}=g_{2,1}+g_{2,2}+\kappa(g_{3,1}+g_{3,2})-2(\partial_{y}+\kappa\partial_{z})\Delta^{-1}(\partial_{y}g_{2,2}+\partial_{z}g_{3,2}),$ by Lemma \ref{lem:kappa} and Lemma \ref{lem:g}, there holds
\begin{equation}\label{estimate G2}
	\begin{aligned}
		\|{\rm e}^{bA^{-\frac13}t}\nabla G^{(2)}\|_{L^{2}L^{2}}^{2}\leq~&C\big(
		\|{\rm e}^{bA^{-\frac13}t}\nabla(g_{2,1}+\kappa g_{3,1})\|_{L^{2}L^{2}}^{2}+\|{\rm e}^{bA^{-\frac13}t}\nabla g_{2,2}\|_{L^{2}L^{2}}^{2}\\&+\|{\rm e}^{bA^{-\frac13}t}\nabla g_{3,2}\|_{L^{2}L^{2}}^{2}
		\big)\leq C\epsilon^{2}AE_{5,2}^{2}(t).
	\end{aligned}
\end{equation}

Combining (\ref{estimate G1}) with (\ref{estimate G2}), we finish the proof.
\end{proof}

\begin{lemma}\label{E5 E6}
Under the assumptions of Theorem \ref{result0} and \eqref{assumption}, there exists a positive constant $A_{6}$ independent of $t$ and $A,$ such that if $A\geq A_{6},$ then there holds
\begin{equation*}
	\begin{aligned}
		E_{4}(t)+E_{5,1}(t)\leq C\left(\|(u_{\rm in})_{\neq}\|_{H^{2}}+1 \right)=:E_{4}+E_{5}.
	\end{aligned}
\end{equation*}
\end{lemma}
\begin{proof}
{\textbf{Step I. Estimate $ \|\Delta u_{3,\neq}-2\partial_{x}W^{(3)}\|_{X_{b}}. $}}
Applying {Proposition \ref{Lvf 0}} to (\ref{u3''-W}), we obtain
\begin{equation}\label{u3 hatQ}
	\begin{aligned}
		&\|\Delta u_{3,\neq}-2\partial_{x}W^{(3)}\|_{X_{b}}^{2}\leq C\big(\|(\Delta u_{3,{\rm in}})_{\neq}\|_{L^{2}}^{2}+A^{\frac13}\|{\rm e}^{bA^{-\frac13}t}\Delta V\partial_{x}u_{3,\neq}\|_{L^{2}L^{2}}^{2}\\&+A^{\frac13}\|{\rm e}^{bA^{-\frac13}t}\partial_{z}(\partial_{y}V\partial_{x}W)\|_{L^{2}L^{2}}^{2}+\frac{1}{A}\|{\rm e}^{bA^{-\frac13}t}\nabla(\mathbf{U}_{1}\partial_{x}u_{3,\neq})\|_{L^{2}L^{2}}^{2}\\&+\frac{1}{A}
\|{\rm e}^{bA^{-\frac13}t}
\nabla(g_{2,2},g_{3,1},g_{3,2},G_{3,3})
\|_{L^2L^2}^2+\frac{1}{A}\|{\rm e}^{bA^{-\frac13}t}(P_{1,1}+P_{1,2}+P_{1,3})\|_{L^{2}L^{2}}^{2} \big),
	\end{aligned}
\end{equation}
where we use $ \partial_{x}W^{(3)}(0)=0. $ Due to $ \|\Delta V\|_{L^{\infty}}\leq CA^{-1}\|\Delta\mathbf{U}_{2}\|_{H^{2}}\leq C, $ there holds
\begin{equation}\label{c}
	\begin{aligned}
		\|{\rm e}^{bA^{-\frac13}t}\Delta V\partial_{x}u_{3,\neq}\|_{L^{2}L^{2}}^{2}\leq~&\|\Delta V\|_{L^{\infty}L^{\infty}}^{2}\|{\rm e}^{bA^{-\frac13}t}\partial_{x}u_{3,\neq}\|_{L^{2}L^{2}}^{2}\\\leq~&C\|{\rm e}^{bA^{-\frac13}t}\partial_{x}^{2}u_{3,\neq}\|_{L^{2}L^{2}}^{2}\leq CA^{\frac13}\|\partial_{x}^{2}u_{3,\neq}\|_{X_{b}}^{2}\leq CA^{\frac13}E_{5,2}^{2}(t).
	\end{aligned}
\end{equation}
Moreover, for $ \|\partial_{y}V\|_{H^3}\leq C\left(1+A^{-1}\|\Delta\mathbf{U}_{2}\|_{H^{2}} \right)\leq C, $ one deduces
\begin{equation}\label{d}
	\begin{aligned}
		&\|\partial_{z}(\partial_{y}V\partial_{x}W)\|_{L^{2}}
		\leq\|\partial_{y}V\|_{H^3}\|\partial_{x}W\|_{H^{1}}\leq C\|\nabla \partial_{x}W\|_{L^{2}},\\
		&\|{\rm e}^{bA^{-\frac13}t}
		\partial_{z}(\partial_{y}V\partial_{x}W)\|_{L^{2}L^{2}}^{2}
		\leq C\|{\rm e}^{bA^{-\frac13}t}\nabla\partial_{x}W\|_{L^{2}L^{2}}^{2}
		\leq CA^{\frac13}E_{5,2}^{2}(t).
	\end{aligned}
\end{equation}
Recall that $G_{3,3}=(u_{\neq}\cdot\nabla u_{3,\neq})_{\neq}$ in (\ref{g G}), by Lemma \ref{lemma_neq1}, and we get
\begin{equation}\label{d1}
	\|{\rm e}^{bA^{-\frac13}t}\nabla G_{3,3}\|_{L^{2}L^{2}}^{2}\leq C\left(A^{\frac{\alpha}{3}+\frac76}E_{2}^{2}E_{5}^{2}+A^{\frac43}E_{2}^{4} \right),
\end{equation}
where $\alpha\in(\frac12, \frac34).$

Collecting (\ref{c})--(\ref{d1}), and using (\ref{estimate P}), Lemma \ref{lem:U1} and Lemma \ref{lem:g}, we get from (\ref{u3 hatQ}) that
\begin{equation}\label{u3 neq''}
	\begin{aligned}
		&	\|\Delta u_{3,\neq}-2\partial_{x}W^{(3)}\|_{X_{b}}^{2}
		\\\leq~& C\left(\|(u_{\rm in})_{\neq}\|_{H^{2}}^{2}+A^{\frac23}E_{5,2}^{2}(t)+A^{\frac{\alpha}{3}+\frac16}E_{2}^{2}E_{5}^{2}
		+A^{\frac13}E_{2}^{4}+\epsilon^{2}A^{\frac23}E_{5,1}^{2}(t)+1 \right)
	\end{aligned}
\end{equation}
provided that $A\geq \max\{A_{5}, E_{2}^{\frac{24}{5-6\alpha}}, E_{2}^{\frac{24}{3-4\alpha}}, (\|(u_{1,\rm in})_{0}\|_{H^{1}}^{2}+H_{1}^{2})^{\frac32}\}=:A_{6,1}.$

{\textbf{Step II. Estimate $ \|\Delta W^{(1)}\|_{X_{b}}. $}}
By (\ref{eq:Lv}),  there holds
\begin{equation*}
	\widetilde{\mathcal{L}_{V}}W+\frac{\mathbf{U}_{1}\partial_{x}W}{A}+\frac{G^{(1)}+G^{(2)}}{A}=\left(\partial_{t}\kappa-\frac{\Delta\kappa}{A} \right)u_{3,\neq}-\frac{2\nabla\kappa\cdot\nabla u_{3,\neq}}{A}.
\end{equation*}	
Using the following decomposition
\begin{equation*}
	\nabla\kappa\cdot\nabla u_{3,\neq}=\rho_{1}\nabla V\cdot\nabla u_{3,\neq}+\rho_{2}(\partial_{z}-\kappa\partial_{y})u_{3,\neq}
\end{equation*}
and $ W=W^{(1)}+\frac{1}{A} W^{(2)}, $ we get
	\begin{align}\label{tilde Lv Q}
		&\widetilde{\mathcal{L}_{V}}W^{(1)}+\frac{\mathbf{U}_{1}\partial_{x}W}{A}+\frac{G^{(1)}+G^{(2)}}{A}-\left(\partial_{t}\kappa-\frac{\Delta\kappa}{A} \right)u_{3,\neq}\nonumber\\=&-\frac{2}{A}\left(\rho_{1}\nabla V\cdot\nabla u_{3,\neq}+\rho_{2}(\partial_{z}-\kappa\partial_{y})u_{3,\neq} \right)-\frac{1}{A}\widetilde{\mathcal{L}_{V}}W^{(2)}=-\frac{2}{A}\rho_{2}(\partial_{z}-\kappa\partial_{y})u_{3,\neq}-\frac{1}{A}J_{11},
	\end{align}
where
\begin{equation}\label{define J11}
	J_{11}=\widetilde{\mathcal{L}_{V}}W^{(2)}+2\rho_{1}\nabla V\cdot\nabla u_{3,\neq}.
\end{equation}
Applying Proposition \ref{tilde Lv} to (\ref{tilde Lv Q}), one deduces
\begin{equation}\label{tildeQ''}
	\begin{aligned}
		\|\Delta W^{(1)}\|_{X_{b}}^{2}\leq~&C\big( \|\Delta W^{(1)}(0)\|_{L^{2}}^{2}+A^{-1}\|{\rm e}^{bA^{-\frac13}t}\nabla\left(\rho_{2}(\partial_{z}-\kappa\partial_{y})u_{3,\neq} \right)\|_{L^{2}L^{2}}^{2}\\&+A^{-1}\|{\rm e}^{bA^{-\frac13}t}\nabla J_{11}\|_{L^{2}L^{2}}^{2}+A\|{\rm e}^{bA^{-\frac13}t}\nabla((\partial_{t}\kappa-A^{-1}\Delta\kappa)u_{3,\neq})\|_{L^{2}L^{2}}^{2}\\&+A^{-1}\|{\rm e}^{bA^{-\frac13}t}\nabla(\mathbf{U}_{1}\partial_{x}W)\|_{L^{2}L^{2}}^{2}+A^{-1}\|{\rm e}^{bA^{-\frac13}t}\nabla(G^{(1)}+G^{(2)})\|_{L^{2}L^{2}}^{2}
		\big).
	\end{aligned}
\end{equation}
It follows from Lemma \ref{lem:kappa} that $ \|\rho_{2}\|_{H^{2}}\leq C\epsilon $ and
\begin{equation}\label{kappa t}
	\|\partial_{t}\kappa-A^{-1}\Delta\kappa\|_{H^{1}}\leq \|\partial_{t}\kappa\|_{H^{1}}+CA^{-1}\|\kappa\|_{H^{3}}\leq CA^{-1}\epsilon.
\end{equation}
Hence we have
\begin{equation}\label{2_11}
	\begin{aligned}
		&\|\nabla\left(\rho_{2}(\partial_{z}-\kappa\partial_{y})u_{3,\neq} \right)\|_{L^{2}}\leq C\|\rho_{2}\|_{H^{2}}\|\nabla(\partial_{z}-\kappa\partial_{y})u_{3,\neq}\|_{L^{2}}\leq C\epsilon\|\nabla\partial_{x}(\partial_{z}-\kappa\partial_{y})u_{3,\neq}\|_{L^{2}}.
	\end{aligned}
\end{equation}
Besides, using (\ref{kappa t}) and Lemma \ref{sob_14}, there holds
\begin{equation}\label{2_12}
	\begin{aligned}
		&\|\nabla\left((\partial_{t}\kappa-A^{-1}\Delta\kappa)u_{3,\neq} \right)\|_{L^{2}}
		\leq CA^{-1}\epsilon\left(\|\nabla\partial_{x}^{2}u_{3,\neq}\|_{L^{2}}+\|\nabla\partial_{x}(\partial_{z}-\kappa\partial_{y})u_{3,\neq}\|_{L^{2}} \right).
	\end{aligned}
\end{equation}
For $ J_{11} $ in (\ref{define J11}), we rewrite it as follows:
\begin{equation}\label{J11}
	\begin{aligned}
		J_{11}=~&\mathcal{L}_{V}W^{(2)}-2(\partial_{y}+\kappa\partial_{z})\Delta^{-1}(\partial_{y}V\partial_{x}W^{(2)})+2\rho_{1}\nabla V\cdot\nabla u_{3,\neq}\\=~&-2(\partial_{y}+\kappa\partial_{z})\Delta^{-1}(\partial_{y}V\partial_{x}W^{(2)})+\rho_{1}\nabla V\cdot\nabla u_{3,\neq}\\=~&-(\partial_{y}+\kappa\partial_{z})\Delta^{-1}(\partial_{y}VJ_{12})-(\partial_{y}+\kappa\partial_{z})\Delta^{-1}(\partial_{y}V\rho_{1}\Delta u_{3,\neq})+\rho_{1}\nabla V\cdot\nabla u_{3,\neq},
	\end{aligned}
\end{equation}
where $ J_{12}=2\partial_{x}W^{(2)}-\rho_{1}\Delta u_{3,\neq}=2\partial_{x}(W^{(2)}-\rho_{1}W^{(3)})-\rho_{1}\left(\Delta u_{3,\neq}-2\partial_{x}W^{(3)} \right). $
Using Lemma \ref{lem:kappa} and (\ref{V1' infty}), we arrive
\begin{equation*}
	\begin{aligned}
		\|J_{12}\|_{L^2}\leq C\left(\|\partial_{x}(W^{(2)}-\rho_{1}W^{(3)})\|_{L^{2}}+\epsilon\|\Delta u_{3,\neq}-2\partial_{x}W^{(3)}\|_{L^{2}} \right)
	\end{aligned}
\end{equation*}
and
\begin{equation}\label{J12 0}
	\begin{aligned}
		&	\|(\partial_{y}+\kappa\partial_{z})\Delta^{-1}(\partial_{y}VJ_{12})\|_{H^{1}}\leq C\left(\|\partial_{y}\Delta^{-1}(\partial_{y}VJ_{12})\|_{H^{1}}+\|\kappa\|_{H^{3}}\|\partial_{z}\Delta^{-1}(\partial_{y}VJ_{12})\|_{H^{1}} \right)\\&\leq C\|J_{12}\|_{L^{2}}\leq C\left(\|\partial_{x}(W^{(2)}-\rho_{1}W^{(3)})\|_{L^{2}}+\epsilon\|\Delta u_{3,\neq}-2\partial_{x}W^{(3)}\|_{L^{2}} \right).
	\end{aligned}
\end{equation}	
Thanks to $ \nabla V\cdot\nabla u_{3,\neq}=\partial_{y}V(\partial_{y}+\kappa\partial_{z})u_{3,\neq}, $ there holds
\begin{equation*}
	\begin{aligned}
		&-(\partial_{y}+\kappa\partial_{z})\Delta^{-1}(\partial_{y}V\rho_{1}\Delta u_{3,\neq})+\rho_{1}\nabla V\cdot\nabla u_{3,\neq}=-\left[(\partial_{y}+\kappa\partial_{z})\Delta^{-1}, \partial_{y}V\rho_{1} \right]\Delta u_{3,\neq}\\=&-\left[\partial_{y}\Delta^{-1}, \partial_{y}V\rho_{1} \right]\Delta u_{3,\neq}-\kappa\left[\partial_{z}\Delta^{-1},\partial_{y}V\rho_{1} \right]\Delta u_{3,\neq}.
	\end{aligned}
\end{equation*}
From this, along with (\ref{V1' infty}), Lemma \ref{lem:kappa} and Lemma \ref{sob_14}, we get
\begin{equation}\label{J12 1}
	\begin{aligned}
		&\|-(\partial_{y}+\kappa\partial_{z})\Delta^{-1}\left(\partial_{y}V\rho_{1}\Delta u_{3,\neq} \right)+\rho_{1}\nabla V\cdot\nabla u_{3,\neq}\|_{H^{1}}\\\leq~&\|\left[\partial_{y}\Delta^{-1},\partial_{y}V\rho_{1} \right]\Delta u_{3,\neq}\|_{H^{1}}+C\|\kappa\|_{H^3}
\|\left[\partial_z\Delta^{-1},
\partial_yV\rho_1\right]\Delta u_{3,\neq}\|_{H^1}\\\leq~&C\epsilon\left(\|\nabla\partial_{x}^{2}u_{3,\neq}\|_{L^{2}}+\|\partial_{x}(\partial_{z}-\kappa\partial_{y})\nabla u_{3,\neq}\|_{L^{2}} \right).
	\end{aligned}
\end{equation}	
Then we conclude from (\ref{J11}), (\ref{J12 0}) and (\ref{J12 1}) that
\begin{equation}\label{J11 H1}
	\begin{aligned}
		\|\nabla J_{11}\|_{L^{2}}\leq~& C\left(\|\partial_{x}(W^{(2)}-\rho_{1}W^{(3)})\|_{L^{2}}+\epsilon\|\Delta u_{3,\neq}-2\partial_{x}W^{(3)}\|_{L^{2}} \right)\\&+C\epsilon\left(\|\nabla\partial_{x}^{2}u_{3,\neq}\|_{L^{2}}+\|\nabla\partial_{x}(\partial_{z}-\kappa\partial_{y})u_{3,\neq}\|_{L^{2}} \right).
	\end{aligned}
\end{equation}

Collecting (\ref{u3 neq''}), (\ref{2_11}), (\ref{2_12}), (\ref{J11 H1})  and (iii) of {Lemma \ref{lem W}}, when $A\geq A_{6,1},$ one obtains
\begin{equation*}
	\begin{aligned}
		&A^{-1}\|{\rm e}^{bA^{-\frac13}t}\nabla J_{11}\|_{L^{2}L^{2}}^{2}+A^{-1}\|{\rm e}^{bA^{-\frac13}t}\nabla\left(\rho_{2}(\partial_{z}-\kappa\partial_{y})u_{3,\neq} \right)\|_{L^{2}L^{2}}^{2}\\&+A\|{\rm e}^{bA^{-\frac13}t}\nabla\left((\partial_{t}\kappa-A^{-1}\Delta\kappa)u_{3,\neq} \right)\|_{L^{2}L^{2}}^{2}\\\leq~&CA^{-\frac23}\left(\|\partial_{x}\left(W^{(2)}-\rho_{1}W^{(3)} \right)\|_{X_{b}}^{2}+\epsilon^{2}\|\Delta u_{3,\neq}-2\partial_{x}W^{(3)}\|_{X_{b}}^{2} \right)\\&+C\epsilon^{2}\left(\|\partial_{x}^{2}u_{3,\neq}\|_{X_{b}}^{2}+\|\partial_{x}(\partial_{z}-\kappa\partial_{y})u_{3,\neq}\|_{X_{b}}^{2} \right)\\\leq~&C\epsilon^{2}\left(\|(u_{\rm in})_{\neq}\|_{H^{2}}^{2}+E_{5,1}^{2}(t)+E_{5,2}^{2}(t)+1 \right),
	\end{aligned}
\end{equation*}
Combining above with Lemma \ref{lem:U1} and Lemma \ref{lem:G1 G2}, when $$A\geq C\max\{\epsilon^{-6}(\|(u_{1,\rm in})_{0}\|_{H^{1}}^{2}+H_{1}^{2})^{3}, (\epsilon^{-1}E_{2})^{\frac{12}{3-2\alpha}}, (E_2E_5)^{\frac{12}{3-2\alpha}}, A_{6,1}\}=:A_{6,2},$$ we get from (\ref{tildeQ''}) that
\begin{equation*}
	\begin{aligned}
		\|\Delta W^{(1)}\|_{X_{b}}^{2}\leq C\left(\|(u_{\rm in})_{\neq}\|_{H^{2}}^{2}+\epsilon^{2}E_{5,1}^{2}(t)+\epsilon^{2}E_{5,2}^{2}(t)+1 \right)+CA^{-\frac12}\left(\|(u_{1,\rm in})_{0}\|_{H^{1}}^{2}+H_{1}^{2}\right)\|\Delta W^{(1)}\|_{X_{b}}^{2}.
	\end{aligned}
\end{equation*}
This implies that
\begin{equation}\label{bar Q''}
	\|\Delta W^{(1)}\|_{X_{b}}^{2}\leq C\left(\|(u_{\rm in})_{\neq}\|_{H^{2}}^{2}+\epsilon^{2}E_{5,1}^{2}(t)+\epsilon^{2}E_{5,2}^{2}(t)+1 \right).
\end{equation}

{\textbf{Step III. Estimate $  \|\partial_{x}^{2}u_{j,\neq}\|_{X_{b}}+\|\partial_{x}(\partial_{z}-\kappa\partial_{y})u_{j,\neq}\|_{X_{b}}. $}}
As $W=W^{(1)}+\frac{1}{A}W^{(2)},$ for $j\in\{2,3\},$ we rewrite (\ref{ini_6}) into
\begin{equation*}
	\begin{aligned}
		\mathcal{L}_{V}u_{j,\neq}=~&2\partial_{j}\Delta^{-1}(\partial_{y}V\partial_{x}W^{(1)})+\frac{2}{A}\partial_{j}\Delta^{-1}(\partial_{y}V\partial_{x}W^{(2)})-\frac{\mathbf{U}_{1}\partial_{x}u_{j,\neq}}{A}-\frac{g_{j,1}+g_{j,2}+G_{j,3}}{A}\\&+\frac{2\partial_{j}\Delta^{-1}(\partial_{y}g_{2,2}+\partial_{z}g_{3,2})}{A}+\frac{2\partial_{j}\Delta^{-1}(P_{1,1}+P_{1,2}+P_{1,3})}{A}.
	\end{aligned}
\end{equation*}
Applying Proposition \ref{Lvf} to it, there holds
\begin{equation}\label{Uj0}
	\begin{aligned}
		&\|\partial_{x}^{2}u_{j,\neq}\|_{X_{b}}^{2}+\|\partial_{x}(\partial_{z}-\kappa\partial_{y})u_{j,\neq}\|_{X_{b}}^{2}\\
		\leq~&C\big(\|(u_{\rm in})_{\neq}\|_{H^{2}}^{2}+\|{\rm e}^{bA^{-\frac13}t}\partial_{j}(\partial_{y}V\partial_{x}W^{(1)})\|_{L^{2}L^{2}}^{2}+A^{-\frac53}
		\|{\rm e}^{bA^{-\frac13}t}\partial_{y}V\partial_{x}^{2}W^{(2)}\|_{L^{2}L^{2}}^{2}\\&+A^{-1}\|{\rm e}^{bA^{-\frac13}t}\mathbf{U}_{1}\partial_{x}^{2}u_{j,\neq}\|_{L^{2}L^{2}}^{2}+A^{-1}\|{\rm e}^{bA^{-\frac13}t}\partial_{x}(g_{j,1}+G_{j,3})\|_{L^{2}L^{2}}^{2}\\&+A^{-1}\|{\rm e}^{bA^{-\frac13}t}\nabla(g_{2,2}, g_{3,2})\|_{L^{2}L^{2}}^{2}+A^{-1}\|{\rm e}^{bA^{-\frac13}t}(P_{1,1}+P_{1,2}+P_{1,3})\|_{L^{2}L^{2}}^{2}
		\big),
	\end{aligned}
\end{equation}
where we use
$\|\partial_{x}^{2}f\|_{L^{2}}+\|\partial_{x}(\partial_{z}-\kappa\partial_{y})f\|_{L^{2}}\leq C(1+\|\kappa\|_{L^{\infty}})\|\partial_{x}\nabla f\|_{L^{2}}\leq C\|\partial_{x}\nabla f\|_{L^{2}}.$
By Lemma \ref{u1_hat2}, we arrive
\begin{equation*}
	\begin{aligned}
		\|{\rm e}^{bA^{-\frac13}t}\partial_{j}(\partial_{y}V\partial_{x}W^{(1)})\|_{L^{2}L^{2}}^{2}\leq&C\left(\|\nabla V\|_{L^{\infty}L^{\infty}}^{2}+\|\nabla^{2}V\|_{L^{\infty}L^{\infty}}^{2} \right)\|{\rm e}^{bA^{-\frac13}t}\nabla\partial_{x}W^{(1)}\|_{L^{2}L^{2}}^{2}\\\leq&C\left(1+\frac{\|\Delta\mathbf{U}_{2}\|_{L^{\infty}H^{2}}^{2}}{A^2} \right)\|\Delta W^{(1)}\|_{X_{b}}^{2}\leq C\|\Delta W^{(1)}\|_{X_{b}}^{2}
	\end{aligned}
\end{equation*}
and
\begin{equation*}
	A^{-\frac53}\|{\rm e}^{bA^{-\frac13}t}\partial_{y}V\partial_{x}^{2}W^{(2)}\|_{L^{2}L^{2}}^{2}\leq A^{-\frac53}\|\partial_{y}V\|_{L^{\infty}L^{\infty}}^{2}\|{\rm e}^{bA^{-\frac13}t}\partial_{x}^{2}W^{(2)}\|_{L^{2}L^{2}}^{2}\leq CA^{-\frac43}\|\partial_{x}^{2}W^{(2)}\|_{X_{b}}^{2}.
\end{equation*}
Due to Lemma \ref{lem:kappa} and Lemma \ref{lem:g}, one obtains
\begin{equation*}
	\begin{aligned}
		\|{\rm e}^{bA^{-\frac13}t}\partial_{x}g_{j,1}\|_{L^{2}L^{2}}^{2}\leq C\left(\|{\rm e}^{bA^{-\frac13}t}\nabla(g_{2,1}+\kappa g_{3,1})\|_{L^{2}L^{2}}^{2}+\|{\rm e}^{bA^{-\frac13}t}\partial_{x}g_{3,1}\|_{L^{2}L^{2}}^{2} \right)\leq C\epsilon^{2}AE_{5,2}^{2}(t).
	\end{aligned}
\end{equation*}
Moreover, it follows from Lemma \ref{lemma_neq1} that
\begin{equation*}
	\begin{aligned}
		\|{\rm e}^{bA^{-\frac13}t}\partial_{x}G_{j,3}\|_{L^{2}L^{2}}^{2}\leq C\|{\rm e}^{2aA^{-\frac13}t}\partial_{x}(u_{\neq}\cdot\nabla u_{\neq})\|_{L^{2}L^{2}}^{2}\leq CA^{\frac16+\alpha}E_{2}^{4}.
	\end{aligned}
\end{equation*}
By combining the above estimates and using Lemma \ref{lem:U1}, Lemma \ref{lem:g} and (\ref{estimate P}),
when $$A\geq C\max\{\epsilon^{\frac{-3}{1-\alpha}}\left(\|(u_{1,\rm in})_{0}\|_{H^{1}}^{2}+H_{1}^{2}\right)^{\frac{3}{2(1-\alpha)}}, A_{6,2} \}=:A_{6}.$$
we get from (\ref{Uj0}) that
\begin{equation}\label{result Uj}
	\begin{aligned}
		&\|\partial_{x}^{2}u_{j,\neq}\|_{X_{b}}^{2}+\|\partial_{x}(\partial_{z}-\kappa\partial_{y})u_{j,\neq}\|_{X_{b}}^{2}\\\leq~&C\left(
		\|(u_{\rm in})_{\neq}\|_{H^{2}}^{2}+\|\Delta W^{(1)}\|_{X_{b}}^{2}+A^{-\frac43}\|\partial_{x}^{2}W^{(2)}\|_{X_{b}}^{2}+\epsilon^{2}E_{5,2}^{2}(t)+1
		\right).
	\end{aligned}
\end{equation}

{\textbf{Step IV. Estimate $ E_{5,1}^{2}(t)+E_{5,2}^{2}(t) $ and $ E_{4}^{2}(t). $}}
For $ W=W^{(1)}+\frac{1}{A} W^{(2)}, $ there holds
\begin{equation*}
	\begin{aligned}
		\|\partial_{x}\nabla W\|_{X_{b}}\leq \frac{1}{A}\|\partial_{x}\nabla W^{(2)} \|_{X_{b}}
		+\|\partial_{x}\nabla W^{(1)}\|_{X_{b}}\leq \frac{1}{A}\|\partial_{x}\nabla W^{(2)}\|_{X_{b}}+\|\Delta W^{(1)}\|_{X_{b}},
	\end{aligned}
\end{equation*}
which along with (\ref{result Uj}) gives
\begin{equation*}
	\begin{aligned}
		E_{5,2}^{2}(t)\leq C\big(
		\|(u_{\rm in})_{\neq}\|_{H^{2}}^{2}+\|\Delta W^{(1)}\|_{X_{b}}^{2}+A^{-\frac43}\|\partial_{x}^{2}W^{(2)}\|_{X_{b}}^{2}+\epsilon^{2}E_{5,2}^{2}(t)+A^{-2}\|\partial_{x}\nabla W^{(2)}\|_{X_{b}}^{2}+1
		\big).
	\end{aligned}
\end{equation*}
By (\ref{bar Q''}) and (ii) of Lemma \ref{lem W}, the above inequality indicates that
\begin{equation*}
	E_{5,2}^{2}(t)\leq C\left(\|(u_{\rm in})_{\neq}\|_{H^{2}}^{2}+\epsilon^{2}E_{5,1}^{2}(t)+1 \right)+C\epsilon^{2}E_{5,2}^{2}(t).
\end{equation*}
Taking $\epsilon$ small enough satisfying $C\epsilon^{2}\leq\frac12,$  we have
\begin{equation}\label{E6 1}
	\begin{aligned}
		E_{5,2}^{2}(t)\leq C\left(\|(u_{\rm in})_{\neq}\|_{H^{2}}^{2}+\epsilon^{2}E_{5,1}^{2}(t)+1\right).
	\end{aligned}
\end{equation}
Using (\ref{u3 neq''}) and (i) of Lemma \ref{lem W}, one deduces
\begin{equation}\label{U3''}
	\begin{aligned}
		\|\Delta u_{3,\neq}\|_{X_{b}}^{2}\leq~&C\left(\|\Delta u_{3,\neq}-2\partial_{x}W^{(3)}\|_{X_{b}}^{2}+\|\partial_{x}W^{(3)}\|_{X_{b}}^{2} \right)
		\\\leq~&C\left(\|(u_{\rm in})_{\neq}\|_{H^{2}}^{2}+A^{\frac43}E_{5,2}^{2}(t)+A^{\frac{\alpha}{3}+\frac16}E_{2}^{2}E_{5}^{2}
		+A^{\frac13}E_{2}^{4}+\epsilon^{2}A^{\frac23}E_{5,1}^{2}(t)+1 \right).
	\end{aligned}
\end{equation}
Combining (\ref{E6 1}) and (\ref{U3''}), when $A\geq A_{6},$ there holds
\begin{equation*}
	E_{5,1}^{2}(t)+E_{5,2}^{2}(t)=A^{-\frac43}\|\Delta u_{3,\neq}\|_{X_{b}}^{2}+E_{5,2}^{2}(t)\leq C\left(\|(u_{\rm in})_{\neq}\|_{H^{2}}^{2}+1 \right)+C\epsilon^{2}E_{5,1}^{2}(t).
\end{equation*}
Choosing $\epsilon$ small enough satisfying $C\epsilon^{2}\leq\frac12$, we conclude that	
\begin{equation*}
	E_{4}^{2}(t)\leq C(E_{5,1}^{2}(t)+E_{5,2}^{2}(t))\leq C\left(\|(u_{\rm in})_{\neq}\|_{H^{2}}^{2}+1 \right).
\end{equation*}

The proof is complete.
\end{proof}

\begin{corollary}\label{cor:E2}
Under the assumptions of Theorem \ref{result0} and \eqref{assumption}, according to Lemma \ref{lem:E_21}, Lemma \ref{lem:E2(t)} and Lemma \ref{E5 E6}, there exists a positive constant $A\geq\max\{A_{3}, A_{4}, A_{6}\}=:A_{7},$ such that if $A\geq A_{7}$, there holds
\begin{equation*}
	\begin{aligned}
		&E_{2}(t)\leq C\left(\|(\partial_{x}^{2}n_{\rm in})_{\neq}\|_{L^{2}}+\|(u_{\rm in})_{\neq}\|_{H^{2}}+1 \right)=:E_{2}.
	\end{aligned}
\end{equation*}
\end{corollary}

{
\begin{corollary}[$L^\infty$ estimate for the velocity]
	\label{corollary_velocity_Linfty}
	Under the assumptions of Theorem~\ref{result0} and \eqref{assumption}, define
	\begin{equation*}
		\mathcal I_{\rm in}
		:=
		1+\|n_{\rm in}\|_{H^2}
		+\|u_{\rm in}\|_{H^2},\qquad
        \mathcal J_{\rm in}:=\mathcal I_{\rm in}+H_1.
	\end{equation*}
	There exists a sufficiently large constant
	\[
		A_8\geq \max\{A_2,A_5,A_7,1\},
	\]
	depending only on the initial data and the subcritical mass gap, such that if $A\geq A_8$, then,
	for every $0<t\le T$ on the bootstrap interval,
\begin{equation}\label{velocity-Linfty-explicit}
		\|u\|_{L^\infty(0,t;L^\infty)}
		\leq
		Ct
		+C\left(
		\mathcal J_{\rm in}^{8}
		+A^{\frac16}\mathcal I_{\rm in}^{2}
		\right).
	\end{equation}
\end{corollary}
\begin{proof}
	Recall that
	\[
		u_0=(u_{1,0},u_{2,0},u_{3,0}),
		\qquad
		u_{1,0}=\mathbf{U}_1+\mathbf{U}_2,
	\]
	and hence
	\begin{equation}\label{velocity-mode-decomposition}
		u
		=
		(\mathbf{U}_1+\mathbf{U}_2,u_{2,0},u_{3,0})+u_{\neq}.
	\end{equation}
	We estimate all the terms on the right-hand side of
	\eqref{velocity-mode-decomposition} separately.

	First, we express the constants $E_2$, $E_3$ and $H_1$ in terms of
	the initial data. By Corollary \ref{cor:E2},
	\begin{equation}\label{E2-initial-estimate}
		\begin{aligned}
			E_2\leq
			C\left(
			\|(\partial_x^2n_{\rm in})_{\neq}\|_{L^2}
			+\|(u_{\rm in})_{\neq}\|_{H^2}
			+1
			\right)\leq C\mathcal I_{\rm in}.
		\end{aligned}
	\end{equation}
Lemma \ref{lem:n0 L2} gives
\begin{equation}\label{H1-initial-estimate}
H_1=\|(n_{\rm in})_0\|_{L^2}+C(m)(1+\sqrt Q)
\le\mathcal J_{\rm in}.
\end{equation}
The Sobolev embedding at $t=0$ gives
$\|n_{\rm in}\|_{L^\infty}\le C\mathcal I_{\rm in}\le C\mathcal J_{\rm in}$.
Substitution into \eqref{eq:E3} yields
\begin{equation}\label{E3-initial-estimate}
E_3\le C(1+\mathcal J_{\rm in}^8)^{\frac34}\mathcal J_{\rm in}
\le C\mathcal J_{\rm in}^7.
\end{equation}

	\textbf{Estimate of $\mathbf{U}_1$.}
	By $(\ref{sob_result_1})_1$, we get
	\begin{equation*}
		\|\mathbf{U}_1(s)\|_{L^\infty}
		\leq
		C\left(
		\|\mathbf{U}_1(s)\|_{H^1}
		+\|\partial_z\mathbf{U}_1(s)\|_{H^1}
		\right),
		\qquad 0\leq s\leq t.
	\end{equation*}
	Combining this with Lemma \ref{lemma:u1:1}, Lemma \ref{lemma:partial-z-U1}, Lemma \ref{lem:n0 L2} and Lemma \ref{lemma:partial-z-n0}, we arrive at
	\begin{equation}\label{U1-Linfty-before-substitution}
		\begin{aligned}
			\|\mathbf{U}_1\|_{L^\infty(0,t;L^\infty)}
			\leq C\Big(
			&\|(u_{1,\rm in})_0\|_{H^2}
			+\|\partial_z(n_{\rm in})_0\|_{L^2}
			+H_1
			\\
			&+H_1E_3
			+E_2E_3
			+A^{\frac16}E_2^2
			+\epsilon+\epsilon^2
			\Big).
		\end{aligned}
	\end{equation}
	By \eqref{E2-initial-estimate},
	\eqref{H1-initial-estimate} and
	\eqref{E3-initial-estimate}, there hold
	\[
		H_1E_3+E_2E_3
		\leq C\mathcal J_{\rm in}^8,\qquad
		A^{\frac16}E_2^2
		\leq CA^{\frac16}\mathcal I_{\rm in}^2,
	\]
	and
	\[
		\|(u_{1,\rm in})_0\|_{H^2}
		+\|\partial_z(n_{\rm in})_0\|_{L^2}
		+H_1+\epsilon+\epsilon^2
		\leq C\mathcal J_{\rm in}.
	\]
	Since $\mathcal J_{\rm in}\geq1$, it follows from
	\eqref{U1-Linfty-before-substitution} that
	\begin{equation}\label{U1-Linfty-final}
		\|\mathbf{U}_1\|_{L^\infty(0,t;L^\infty)}
		\leq
		C\left(
		\mathcal J_{\rm in}^8
		+A^{\frac16}\mathcal I_{\rm in}^2
		\right).
	\end{equation}

	\textbf{Estimate of $\mathbf{U}_2$.}
	Since $\mathbf{U}_2|_{t=0}=0$, the fundamental theorem of calculus and the two-dimensional Sobolev embedding imply that, for every $0\leq s\leq t$,
	\begin{equation*}
		\begin{aligned}
			\|\mathbf{U}_2(s)\|_{L^\infty}\leq C\|\mathbf{U}_2(s)\|_{H^2}\leq
			C\int_0^s
			\|\partial_\tau \mathbf{U}_2(\tau)\|_{H^2}\,d\tau\leq
			Cs\|\partial_t\mathbf{U}_2\|_{L^\infty(0,t;H^2)}.
		\end{aligned}
	\end{equation*}
    This along with Lemma \ref{u1_hat2} yields
	\begin{equation}\label{U2-Linfty-final}
		\|\mathbf{U}_2\|_{L^\infty(0,t;L^\infty)}
		\leq C\epsilon t
		\leq Ct,
	\end{equation}
	where we have used that $\epsilon>0$ is sufficiently small.

	\textbf{Estimates of $u_{2,0}$ and $u_{3,0}$.}
	By \eqref{u20 u30 infty}, we have
	\begin{equation*}
		\begin{aligned}
			\|u_{2,0}\|_{L^\infty(0,t;L^\infty)}
			+\|u_{3,0}\|_{L^\infty(0,t;L^\infty)}\leq
			C\left(
			\|u_{2,0}\|_{L^\infty H^2}
			+\|u_{3,0}\|_{L^\infty H^1}
			\right)
			\leq C\epsilon.
		\end{aligned}
	\end{equation*}
	In particular,
    \begin{equation}\label{20 30}
    \|u_{2,0}\|_{L^\infty L^\infty}
		+\|u_{3,0}\|_{L^\infty L^\infty}
		\leq C\mathcal J_{\rm in}^8.
    \end{equation}
	
	\textbf{Estimate of $u_{\neq}$.}
	By $(\ref{ap:u11})_1$, we
	have, for every $0\leq s\leq t$,
	\begin{equation}\label{uneq-Linfty-embedding}
		\begin{aligned}
			\|u_{\neq}(s)\|_{L^\infty}^2
			\leq
			C\left(
			\|\partial_x\omega_{2,\neq}(s)\|_{L^2}
			+\|\Delta u_{2,\neq}(s)\|_{L^2}
			\right)\left(
			\|\nabla\omega_{2,\neq}(s)\|_{L^2}
			+\|\Delta u_{2,\neq}(s)\|_{L^2}
			\right).
		\end{aligned}
	\end{equation}
	According to the definition of $E_2(t)$ and Corollary \ref{cor:E2},
	\begin{equation*}\label{nonzero-first-factor}
		\|\partial_x\omega_{2,\neq}(s)\|_{L^2}
		+\|\Delta u_{2,\neq}(s)\|_{L^2}
		\leq
		CE_2e^{-aA^{-\frac13}s}.
	\end{equation*}
	Moreover, since
	$A^{-\frac13}
		\left(
		\|\partial_y\omega_{2,\neq}\|_{X_a}
		+\|\partial_z\omega_{2,\neq}\|_{X_a}
		\right)
		\leq E_2$,
	we have
	\begin{equation*}\label{nonzero-second-factor}
		\|\nabla\omega_{2,\neq}(s)\|_{L^2}
		+\|\Delta u_{2,\neq}(s)\|_{L^2}
		\leq
		CA^{\frac13}E_2e^{-aA^{-\frac13}s}.
	\end{equation*}
	Substituting these into
	\eqref{uneq-Linfty-embedding}, we obtain
	\begin{equation*}
		\|u_{\neq}(s)\|_{L^\infty}
		\leq
		CA^{\frac16}E_2e^{-aA^{-\frac13}s}.
	\end{equation*}
	Therefore, by \eqref{E2-initial-estimate},
	\begin{equation}\label{uneq-Linfty-final}
		\begin{aligned}
			\|u_{\neq}\|_{L^\infty(0,t;L^\infty)}
			\leq CA^{\frac16}E_2\leq CA^{\frac16}\mathcal I_{\rm in}
			\leq CA^{\frac16}\mathcal I_{\rm in}^2.
		\end{aligned}
	\end{equation}

	Finally, combining \eqref{velocity-mode-decomposition},
	\eqref{U1-Linfty-final}, \eqref{U2-Linfty-final},
	\eqref{20 30} and
	\eqref{uneq-Linfty-final}, we conclude that
	\[
		\|u\|_{L^\infty(0,t;L^\infty)}
		\leq
		Ct
		+C\left(
		\mathcal J_{\rm in}^8
		+A^{\frac16}\mathcal I_{\rm in}^2
		\right),
	\]
	which proves \eqref{velocity-Linfty-explicit}.
\end{proof}
}

\section{Sharpness of the critical mass threshold}\label{sharp1}

In this section, we prove Theorem \ref{thm:sharpness}. Throughout this
section, the time variable is the rescaled time used in
\eqref{ini11}. Hence the diffusion and chemotactic terms carry
the factor $A^{-1}$.

\noindent \textbf{$\bullet$ An invariant  zero-mode subspace.}
Let $Y_\perp=(y,z)\in\mathbb{T}^{2}$. We consider the special class
\begin{equation}
	\label{eq:invariant-ansatz}
	n(t,x,y,z)=n_{0}(t,Y_\perp),\qquad
	c(t,x,y,z)=c_{0}(t,Y_\perp),\qquad
	u(t,x,y,z)=\bigl(u_{1,0}(t,Y_\perp),0,0\bigr).
\end{equation}

The following result is a corollary of applying the time rescaling $t\mapsto\frac{t}{A}$ to system \eqref{ini1}.
\begin{corollary}[Invariant  zero-mode subspace]
	\label{prop:invariant-zero-mode}
	On the invariant subspace \eqref{eq:invariant-ansatz}, the rescaled system is
	\begin{equation}
		\label{eq:reduced-2d-system}
		\left\{
		\begin{aligned}
			&\partial_t n_{0}=
			\frac{1}{A}
			\left[
				\Delta_{Y_{\perp}} n_{0}
				-\nabla_{Y_{\perp}}\cdot
				\bigl(n_{0}\nabla_{Y_{\perp}}c_{0}\bigr)
			\right],
			\\
			&-\Delta_{Y_{\perp}}c_{0}=n_{0}-\overline n,
			\\
			&\partial_tu_{1,0}
			-\frac{1}{A}\Delta_{Y_{\perp}}u_{1,0}=\frac{1}{A}n_{0},
		\end{aligned}
		\right.
	\end{equation}
	where
	$\overline n=
	\frac{1}{|\mathbb{T}|^{2}}
	\int_{\mathbb{T}^{2}}n_{0}(t,Y_{\perp}) dY_{\perp}.$
\end{corollary}



Let $G$ be the zero-average Green function of $-\Delta_{Y_{\perp}}$ on
$\mathbb{T}^{2}$:
\begin{equation*}
	\label{eq:green-function}
	-\Delta_{Y_{\perp}}G
	=
	\delta_{0}
	-\frac{1}{|\mathbb{T}|^{2}},
	\qquad
	\int_{\mathbb{T}^{2}}G({Y_{\perp}}) d{Y_{\perp}}=0.
\end{equation*}
By uniqueness, $G$ is even. More precisely, $Y_{\perp}\mapsto G(-Y_{\perp})$ satisfies
the same equation and the same zero-average condition as $G$, and hence
\begin{equation}
	\label{eq:green-odd-gradient}
	G(-Y_{\perp})=G(Y_{\perp}),\qquad (\nabla_{Y_{\perp}}G)(-Y_{\perp})=-(\nabla_{Y_{\perp}}G)(Y_{\perp}).
\end{equation}
There exists $r_G>0$ such that, in the local Euclidean coordinate ball
$0<|Y_{\perp}|<r_G$,
\begin{equation}
	\label{eq:green-local-expansion}
	G(Y_{\perp})
	=
	-\frac{1}{2\pi}\log|Y_{\perp}|+H(Y_{\perp}),
\end{equation}
where $H$ is smooth and even. In particular,
$\nabla_{Y_{\perp}}H(0)=0,$
and hence
\begin{equation}
	\label{eq:H-gradient}
	|\nabla_{Y_{\perp}}H({Y_{\perp}})|
	\leq C|{Y_{\perp}}|,
	\qquad |{Y_{\perp}}|<r_G.
\end{equation}
The zero-average solution of the elliptic equation is therefore
represented by
\begin{equation}
	\label{eq:c-green-representation}
	c_{0}({Y_{\perp}})
	=
	\int_{\mathbb{T}^{2}}
	G({Y_{\perp}}-{Y_{\perp}}')n_{0}({Y_{\perp}}')d{Y_{\perp}}',
	\qquad
	\nabla_{Y_{\perp}}c_{0}({Y_{\perp}})
	=
	\int_{\mathbb{T}^{2}}
	\nabla_{Y_{\perp}}G({Y_{\perp}}-{Y_{\perp}}')n_{0}({Y_{\perp}}')d{Y_{\perp}}'.
\end{equation}

We denote by $d_{\mathbb{T}^{2}}$ the periodic distance on
$\mathbb{T}^{2}$ and set
\[
B_r
=
\left\{
{Y_{\perp}}\in\mathbb{T}^{2}:
d_{\mathbb{T}^{2}}({Y_{\perp}},0)<r
\right\}.
\]

\begin{lemma}[Truncated quadratic weight]
	\label{lem:truncated-weight}
	There exist $r_{*}>0$ and a constant $C>0$, independent of $r$, such
	that for every $0<r\leq r_{*}$ there exists a nonnegative function
	$\psi_r\in C^{\infty}(\mathbb{T}^{2})$ satisfying
	\begin{equation}
		\label{eq:truncated-weight-properties}
		\left\{
		\begin{aligned}
			\psi_r({Y_{\perp}})&=|{Y_{\perp}}|^{2},
			&&{Y_{\perp}}\in B_r,
			\\
			\psi_r({Y_{\perp}})&\geq r^{2},
			&&{Y_{\perp}}\in\mathbb{T}^{2}\setminus B_r,
			\\
			0\leq\psi_r({Y_{\perp}})&\leq Cr^{2},\quad \|\nabla_{Y_{\perp}}\psi_r\|_{L^{\infty}}\leq Cr,
			&&{Y_{\perp}}\in\mathbb{T}^{2},\\
			\|\nabla_{Y_{\perp}}^{2}\psi_r\|_{L^{\infty}}
			&\leq C,\quad \Delta_{Y_{\perp}}\psi_r\leq4, &&{Y_{\perp}}\in\mathbb{T}^{2}.
		\end{aligned}
		\right.
	\end{equation}
\end{lemma}

\begin{proof}
	Take $q\in C^{\infty}([0,\infty))$ satisfying
	\begin{equation}\label{q}
	    0\leq q\leq1,\quad
	q(s)=1~~\text{for}~~0\leq s\leq1,
	\quad
	q(s)=0~~\text{for}~~s\geq2,
	\quad
	q'(s)\leq0.
	\end{equation}
	Define
	$h(s)=\int_0^s2\sigma q(\sigma)\,d\sigma.$
	Then
	\[
	h'(s)=2sq(s),
	\qquad
	h''(s)=2q(s)+2sq'(s).
	\]
	In particular,
	\[
	h(s)=s^{2}\quad\text{for }\quad s\leq1,
	\]
	and $h$ is constant for $s\geq2$. Moreover, $h$ is nondecreasing and
	\[
	h(s)\geq h(1)=1,
	\qquad s\geq1.
	\]
Choose $r_{*}>0$ sufficiently small so that $B_{2r_{*}}$ is contained
	in a Euclidean coordinate neighborhood of the flat torus. Define
	$\psi_r({Y_{\perp}})
	=r^{2}h\left(\frac{|{Y_{\perp}}|}{r}\right)$
	in this neighborhood and extend it by the constant value
	$r^{2}h(2)$ outside $B_{2r}$. Since $h$ is constant for $s\geq2$,
	the extension is smooth.

	The first five properties in
	\eqref{eq:truncated-weight-properties} follow directly from the
	definition and scaling. To verify the last one, write
	$s=|Y|/r$. The direct calculation and \eqref{q} yield
	\[
	\Delta_{Y_{\perp}}\psi_r
	=
	h''(s)+\frac{h'(s)}{s}
	=
	4q(s)+2sq'(s)
	\leq4.
	\]

    The proof is complete.
\end{proof}

Define the truncated second moment
\begin{equation}
	\label{eq:local-second-moment}
	I_r(t)
	=\int_{\mathbb{T}^{2}}
	\psi_r({Y_{\perp}})n_{0}(t,{Y_{\perp}})d{Y_{\perp}}.
\end{equation}
We also define
\begin{equation}
	\label{eq:inside-outside-mass}
	m_{\rm out}(t)
	=
	\int_{\mathbb{T}^{2}\setminus B_r}
	n_{0}(t,{Y_{\perp}})d{Y_{\perp}},
	\qquad
	m_{\rm in}(t)
	=
	\int_{B_r}n_{0}(t,{Y_{\perp}})d{Y_{\perp}}
	=
	m-m_{\rm out}(t).
\end{equation}
Since $\psi_r\geq r^{2}$ on $\mathbb{T}^{2}\setminus B_r$,
\begin{equation}
	\label{eq:outside-mass-by-moment}
	m_{\rm out}(t)
	\leq
	\frac{I_r(t)}{r^{2}}.
\end{equation}

\begin{lemma}[Localized Virial identity]
	\label{lem:localized-virial}
	Let $(n_{0},c_{0})$ be a classical solution of the first two equations
	in \eqref{eq:reduced-2d-system}. Then
	\begin{equation}
		\label{eq:localized-virial}
		\begin{aligned}
			\frac{d}{dt}&I_r(t)
			=
			\frac{1}{A}
			\int_{\mathbb{T}^{2}}
			(\Delta_{Y_{\perp}}\psi_r)(Y)n_{0}({Y_{\perp}})d{Y_{\perp}}
			\\
			&+
			\frac{1}{2A}
			\iint_{\mathbb{T}^{2}\times\mathbb{T}^{2}}
			\bigl(
				\nabla_{Y_{\perp}}\psi_r({Y_{\perp}})
				-\nabla_{Y_{\perp}}\psi_r({Y_{\perp}}')
			\bigr)
			\cdot\nabla_{Y_{\perp}}G({Y_{\perp}}-{Y_{\perp}}')n_{0}({Y_{\perp}})n_{0}({Y_{\perp}}')d{Y_{\perp}}d{Y_{\perp}}'.
		\end{aligned}
	\end{equation}
\end{lemma}

\begin{proof}
	Using the density equation and periodic integration by parts, we have
	\[
	\begin{aligned}
		\frac{d}{dt}I_r(t)
		&=
		\int_{\mathbb{T}^{2}}\psi_r\partial_tn_{0}d{Y_{\perp}}
		=
		\frac{1}{A}
		\int_{\mathbb{T}^{2}}\psi_r\Delta_{Y_{\perp}}n_{0}d{Y_{\perp}}
		-
		\frac{1}{A}
		\int_{\mathbb{T}^{2}}
		\psi_r\nabla_{Y_{\perp}}\cdot(n_{0}\nabla_{Y_{\perp}}c_{0})d{Y_{\perp}}
		\\
		&=
		\frac{1}{A}
		\int_{\mathbb{T}^{2}}
		(\Delta_{Y_{\perp}}\psi_r)n_{0}d{Y_{\perp}}
		+
		\frac{1}{A}
		\int_{\mathbb{T}^{2}}
		\nabla_{Y_{\perp}}\psi_r\cdot\nabla_{Y_{\perp}}c_{0}\,n_{0}d{Y_{\perp}}.
	\end{aligned}
	\]
	By \eqref{eq:c-green-representation}, the last term equals
	\[
	J
	:=\frac{1}{A}
	\iint_{\mathbb{T}^{2}\times\mathbb{T}^{2}}
	\nabla_{Y_{\perp}}\psi_r({Y_{\perp}})\cdot\nabla_{Y_{\perp}}G({Y_{\perp}}-{Y_{\perp}}')
	n_{0}({Y_{\perp}})n_{0}({Y_{\perp}}')d{Y_{\perp}}d{Y_{\perp}}'.
	\]
	Interchanging ${Y_{\perp}}$ and ${Y_{\perp}}'$ and using
	\eqref{eq:green-odd-gradient}, we obtain
	\[
	J=-\frac{1}{A}
	\iint_{\mathbb{T}^{2}\times\mathbb{T}^{2}}
	\nabla_{Y_{\perp}}\psi_r({Y_{\perp}}')\cdot\nabla_{Y_{\perp}}G({Y_{\perp}}-{Y_{\perp}}')
	n_{0}({Y_{\perp}})n_{0}({Y_{\perp}}')d{Y_{\perp}}d{Y_{\perp}}'.
	\]
	Averaging the two representations of $J$ yields
	\[
	J=\frac{1}{2A}
	\iint_{\mathbb{T}^{2}\times\mathbb{T}^{2}}
	\bigl(
		\nabla_{Y_{\perp}}\psi_r({Y_{\perp}})-\nabla_{Y_{\perp}}\psi_r({Y_{\perp}}')
		\bigr)
	\cdot\nabla_{Y_{\perp}}G({Y_{\perp}}-{Y_{\perp}}')
	n_{0}({Y_{\perp}})n_{0}({Y_{\perp}}')d{Y_{\perp}}d{Y_{\perp}}',
	\]
	which proves \eqref{eq:localized-virial}.
\end{proof}

\noindent \textbf{$\bullet$ The interaction estimate and the critical mass.}
Define
\begin{equation*}
	\label{eq:Kr-definition}
	K_r({Y_{\perp}},{Y_{\perp}}')
	=
	\frac12
	\bigl(
		\nabla_{Y_{\perp}}\psi_r({Y_{\perp}})-\nabla_{Y_{\perp}}\psi_r({Y_{\perp}}')
		\bigr)
	\cdot\nabla_{Y_{\perp}}G({Y_{\perp}}-{Y_{\perp}}').
\end{equation*}

\begin{lemma}[Interaction estimates]
	\label{lem:interaction-estimates}
	There exist $r_{*}>0$ and $C_0>0$, independent of
	$0<r\leq r_{*}$, such that
	\begin{equation}
		\label{eq:inner-interaction}
		\iint_{B_r\times B_r}
		K_r({Y_{\perp}},{Y_{\perp}}')n_{0}({Y_{\perp}})n_{0}({Y_{\perp}}')d{Y_{\perp}}d{Y_{\perp}}'
		\leq
		\left(
			-\frac{1}{2\pi}+Cr^{2}
		\right)m_{\rm in}^{2},
	\end{equation}
	and
	\begin{equation}
		\label{eq:outer-interaction}
		\iint_{(\mathbb{T}^{2}\times\mathbb{T}^{2})
		\setminus(B_r\times B_r)}
		K_r({Y_{\perp}},{Y_{\perp}}')n_{0}({Y_{\perp}})n_{0}({Y_{\perp}}')d{Y_{\perp}}d{Y_{\perp}}'
		\leq
		C_0\bigl(m^{2}-m_{\rm in}^{2}\bigr),
	\end{equation}
   where $m$ and $m_{\rm in}$ are as defined in \eqref{m M} and \eqref{eq:inside-outside-mass}, respectively. 
\end{lemma}

\begin{proof}
	We first estimate the interaction inside $B_r\times B_r$. We may assume $2r_{*}<r_G$. If ${Y_{\perp}},{Y_{\perp}}'\in B_r$, then
	\[
	\nabla_{Y_{\perp}}\psi_r({Y_{\perp}})=2{Y_{\perp}},
	\qquad
	\nabla_{Y_{\perp}}\psi_r({Y_{\perp}}')=2{Y_{\perp}}'.
	\]
	Thus, for ${Y_{\perp}}\neq {Y_{\perp}}'$,
	\[
	\frac12
	\bigl(
		\nabla_{Y_{\perp}}\psi_r({Y_{\perp}})-\nabla_{Y_{\perp}}\psi_r({Y_{\perp}}')
		\bigr)
	=
	{Y_{\perp}}-{Y_{\perp}}'.
	\]
	Using \eqref{eq:green-local-expansion}, we obtain
	\[
	\begin{aligned}
		K_r({Y_{\perp}},{Y_{\perp}}')
		&=
		({Y_{\perp}}-{Y_{\perp}}')\cdot
		\left[
			-\frac{1}{2\pi}
			\frac{{Y_{\perp}}-{Y_{\perp}}'}{|{Y_{\perp}}-{Y_{\perp}}'|^{2}}
			+\nabla_{Y_{\perp}}H({Y_{\perp}}-{Y_{\perp}}')
		\right]
		\\
		&=
		-\frac{1}{2\pi}
		+
		({Y_{\perp}}-{Y_{\perp}}')\cdot\nabla_{Y_{\perp}}H({Y_{\perp}}-{Y_{\perp}}').
	\end{aligned}
	\]
	By \eqref{eq:H-gradient},
	\[
	\left|
		({Y_{\perp}}-{Y_{\perp}}')\cdot\nabla_{Y_{\perp}}H({Y_{\perp}}-{Y_{\perp}}')
		\right|
		\leq C|{Y_{\perp}}-{Y_{\perp}}'|^{2}
		\leq Cr^{2}.
	\]
	Therefore,
	\[
	K_r({Y_{\perp}},{Y_{\perp}}')
	\leq
	-\frac{1}{2\pi}+Cr^{2}
	\]
	for almost every $({Y_{\perp}},{Y_{\perp}}')\in B_r\times B_r$, which gives
	\eqref{eq:inner-interaction}.

	We next prove a uniform bound for $K_r$ on the complement of
	$B_r\times B_r$. The point that requires care is that this complement
	is not separated from the diagonal ${Y_{\perp}}={Y_{\perp}}'$. We therefore split the
	product space into a neighborhood of the diagonal and a region away
	from the diagonal.

	Choose $\rho_0>0$, independent of $r$, sufficiently small such that
whenever
\[
0<d_{\mathbb T^2}({Y_{\perp}},{Y_{\perp}}')<\rho_0,
\]
the points ${Y_{\perp}}$ and ${Y_{\perp}}'$ lie in a common Euclidean coordinate chart and
the local expansion \eqref{eq:green-local-expansion} is valid for
${Y_{\perp}}-{Y_{\perp}}'$, and define
	\[
	\mathcal D_{\rho_0}
	=
	\left\{
		({Y_{\perp}},{Y_{\perp}}')\in\mathbb{T}^{2}\times\mathbb{T}^{2}:
		0<d_{\mathbb{T}^{2}}({Y_{\perp}},{Y_{\perp}}')<\rho_0
	\right\},
	\]
	and
	\[
	\mathcal F_{\rho_0}
	=
	\left\{
		({Y_{\perp}},{Y_{\perp}}')\in\mathbb{T}^{2}\times\mathbb{T}^{2}:
		d_{\mathbb{T}^{2}}({Y_{\perp}},{Y_{\perp}}')\geq\rho_0
	\right\}.
	\]

	On $\mathcal D_{\rho_0}$, the minimizing geodesic between ${Y_{\perp}}'$ and
	${Y_{\perp}}$ is contained in a Euclidean coordinate chart. By the mean-value
	formula and Lemma \ref{lem:truncated-weight},
	\begin{equation*}
		\label{eq:gradient-difference}
		\left|
		\nabla_{Y_{\perp}}\psi_r({Y_{\perp}})-\nabla_{Y_{\perp}}\psi_r({Y_{\perp}}')
		\right|
		\leq
		\|\nabla_{Y_{\perp}}^{2}\psi_r\|_{L^\infty}
		d_{\mathbb{T}^{2}}({Y_{\perp}},{Y_{\perp}}')
		\leq
		Cd_{\mathbb{T}^{2}}({Y_{\perp}},{Y_{\perp}}').
	\end{equation*}
	On the other hand, the local Green function expansion gives
		$|\nabla_{Y_{\perp}}G({Y_{\perp}}-{Y_{\perp}}')|
		\leq
		\frac{C}{d_{\mathbb{T}^{2}}({Y_{\perp}},{Y_{\perp}}')}.$
	Hence
	\[
	|K_r({Y_{\perp}},{Y_{\perp}}')|
	\leq
	Cd_{\mathbb{T}^{2}}({Y_{\perp}},{Y_{\perp}}')
	\frac{1}{d_{\mathbb{T}^{2}}({Y_{\perp}},{Y_{\perp}}')}
	\leq C
	\]
	on $\mathcal D_{\rho_0}$. Thus the singularity of $\nabla_{Y_{\perp}}G$ is
	exactly cancelled by the difference of the gradients of $\psi_r$.

	On $\mathcal F_{\rho_0}$, the difference ${Y_{\perp}}-{Y_{\perp}}'$ stays in the compact
	set
	$\left\{
	X\in\mathbb{T}^{2}:
	d_{\mathbb{T}^{2}}(X,0)\geq\rho_0
	\right\}$,
	which does not contain the singular point of $G$. Consequently,
	\[
	M_{\rho_0}
	:=
	\sup_{d_{\mathbb{T}^{2}}(X,0)\geq\rho_0}
	|\nabla_{Y_{\perp}}G(X)|
	<\infty.
	\]
	Furthermore,
	\[
	\left|
	\nabla_{Y_{\perp}}\psi_r({Y_{\perp}})-\nabla_{Y_{\perp}}\psi_r({Y_{\perp}}')
	\right|
	\leq
	2\|\nabla_{Y_{\perp}}\psi_r\|_{L^\infty}
	\leq2Cr
	\leq2Cr_{*}.
	\]
	Therefore,
	\[
	|K_r({Y_{\perp}},{Y_{\perp}}')|
	\leq Cr_{*}M_{\rho_0},\qquad {\text on}~~ \mathcal F_{\rho_0}.
	\]
	 Combining the near-diagonal and
	far-from-diagonal estimates, there exists a constant $C_0>0$,
	independent of $r$, such that
	\begin{equation}
		\label{eq:Kr-uniform-bound}
		|K_r({Y_{\perp}},{Y_{\perp}}')|\leq C_0
	\end{equation}
	for almost every $({Y_{\perp}},{Y_{\perp}}')\in\mathbb{T}^{2}\times\mathbb{T}^{2}$.
	The diagonal itself has zero measure in
	$\mathbb{T}^{2}\times\mathbb{T}^{2}$ and therefore does not affect the
	double integral.

	Finally, \eqref{eq:outer-interaction} follows from \eqref{eq:Kr-uniform-bound} and
	\[
	\begin{aligned}
		&\iint_{(\mathbb{T}^{2}\times\mathbb{T}^{2})
		\setminus(B_r\times B_r)}
		n_{0}({Y_{\perp}})n_{0}({Y_{\perp}}')d{Y_{\perp}}d{Y_{\perp}}'=
		\left(
			\int_{\mathbb{T}^{2}}n_{0}({Y_{\perp}})d{Y_{\perp}}
		\right)^{2}
		-
		\left(
			\int_{B_r}n_{0}({Y_{\perp}})d{Y_{\perp}}
		\right)^{2}=
		m^{2}-m_{\rm in}^{2}.
	\end{aligned}
	\]
\end{proof}

Since $\Delta_{Y_{\perp}}\psi_r\leq4$,
we have $\int_{\mathbb{T}^{2}}
(\Delta_{Y_{\perp}}\psi_r)n_{0}d{Y_{\perp}}
\leq4m.$
Combining Lemmas \ref{lem:localized-virial} and
\ref{lem:interaction-estimates}, we obtain
\[
\frac{d}{dt}I_r
\leq
\frac{1}{A}
\left[
	4m
	+
	\left(
		-\frac{1}{2\pi}+Cr^{2}
	\right)m_{\rm in}^{2}
	+
	C_0(m^{2}-m_{\rm in}^{2})
\right].
\]
Since
$4m-\frac{m_{\rm in}^{2}}{2\pi}
=
4m-\frac{m^{2}}{2\pi}
+
\frac{m^{2}-m_{\rm in}^{2}}{2\pi}$
and
\[
m^{2}-m_{\rm in}^{2}
=
(m-m_{\rm in})(m+m_{\rm in})
=
m_{\rm out}(m+m_{\rm in})
\leq2m\,m_{\rm out},
\]
there exists a constant $C_m>0$, depending only on $m$ and
$\mathbb{T}^{2}$, such that
\begin{equation}
	\label{eq:moment-differential}
	\frac{d}{dt}I_r(t)
	\leq
	\frac{1}{A}
	\left[
		-\Gamma_m
		+Cr^{2}m^{2}
		+C_m m_{\rm out}(t)
	\right],
\end{equation}
where
$\Gamma_m:=
	\frac{m^{2}}{2\pi}-4m=
	\frac{m(m-8\pi)}{2\pi}$.
Using \eqref{eq:outside-mass-by-moment}, we also have
\begin{equation*}
	\label{eq:moment-closed}
	\frac{d}{dt}I_r(t)
	\leq
	\frac{1}{A}
	\left[
		-\Gamma_m
		+Cr^{2}m^{2}
		+C_m\frac{I_r(t)}{r^{2}}
	\right].
\end{equation*}

\begin{lemma}[Strictly negative slope in the supercritical case]
	\label{lem:negative-slope}
	Assume that $m>8\pi$. Then there exist $ r=r(m)>0$ and $\delta=\delta(m)>0$,
	such that whenever a classical solution satisfies
	$I_r(t)\leq\delta r^{2}$,
	one has
	\begin{equation*}
		\label{eq:strict-negative-slope}
		\frac{d}{dt}I_r(t)
		\leq
		-\frac{\Gamma_m}{2A}<0.
	\end{equation*}
\end{lemma}

\begin{proof}
	Since $m>8\pi$, one has $\Gamma_m>0$. First choose
	$r=r(m)>0$ sufficiently small so that
	$Cr^{2}m^{2}\leq\frac{\Gamma_m}{4}$.
	Next choose $\delta=\delta(m)>0$ sufficiently small so that
	$C_m\delta\leq\frac{\Gamma_m}{4}$.
	If $I_r(t)\leq\delta r^{2}$, then
	\[
	m_{\rm out}(t)
	\leq
	\frac{I_r(t)}{r^{2}}
	\leq\delta.
	\]
	Substituting this into \eqref{eq:moment-differential}, we find
	\[
	\frac{d}{dt}I_r(t)
	\leq
	\frac{1}{A}
	\left(
		-\Gamma_m
		+\frac{\Gamma_m}{4}
		+\frac{\Gamma_m}{4}
	\right)
	=
	-\frac{\Gamma_m}{2A}.
	\]
\end{proof}

\noindent \textbf{$\bullet$ Construction of concentrated initial data and finite-time blow-up.}
Let
$0\leq\varrho\in C_c^\infty(B_1(0)),$
and
$\int_{\mathbb{R}^{2}}\varrho(X_{\perp})dX_{\perp}=1.$
For $\zeta>0$, set $\eta_\zeta=\zeta^{4}$.
For $0<\zeta\ll r$, define, in a local coordinate ball and then
extend periodically,
\begin{equation}
	\label{eq:concentrated-initial-data}
	(n_{\rm in})_{0}^{\zeta}(Y_{\perp})
	=
	(m-\eta_\zeta)\zeta^{-2}
	\varrho\left(\frac{Y_{\perp}}{\zeta}\right)
	+
	\frac{\eta_\zeta}{|\mathbb{T}|^{2}}.
\end{equation}
For sufficiently small $\zeta$,
\[
(n_{\rm in})_{0}^{\zeta}
\in C^{\infty}(\mathbb{T}^{2}),
\qquad
(n_{\rm in})_{0}^{\zeta}>0,\qquad \int_{\mathbb{T}^{2}}
(n_{\rm in})_{0}^{\zeta}(Y_{\perp})dY_{\perp}=m.
\]
Since $\operatorname{supp}\varrho(Y_{\perp}/\zeta)\subset B_\zeta
\subset B_r$, one has $\psi_r(Y_{\perp})=|Y_{\perp}|^{2}$ on the support of the
concentrated part. Thus, after the change of variables $Y_{\perp}=\zeta X_{\perp}$,
\[
\begin{aligned}
	(m-\eta_\zeta)\zeta^{-2}
	\int_{\mathbb{T}^{2}}
	\psi_r(Y_{\perp})
	\varrho\left(\frac{Y_{\perp}}{\zeta}\right)dY_{\perp}=
	(m-\eta_\zeta)\zeta^{2}
	\int_{B_1}|X_{\perp}|^{2}\varrho(X_{\perp})dX_{\perp}
	\leq C_m\zeta^{2}.
\end{aligned}
\]
For the positive constant background, using
$0\leq\psi_r\leq Cr^{2}$ gives
\[
\frac{\eta_\zeta}{|\mathbb{T}|^{2}}
\int_{\mathbb{T}^{2}}\psi_r(Y_{\perp})dY_{\perp}
\leq C\eta_\zeta r^{2}
=C\zeta^{4}r^{2}.
\]
Consequently, it follows from \eqref{eq:local-second-moment} that
$I_r(0)\leq
	C_m\zeta^{2}
	+C\zeta^{4}r^{2}$.
After choosing $\zeta>0$ sufficiently small, we get
\begin{equation}
	\label{eq:initial-moment-barrier}
	I_r(0)<\frac{\delta r^{2}}{2}.
\end{equation}

\begin{proposition}[Finite-time blow-up of the two-dimensional zero mode]
	\label{prop:two-dimensional-blowup}
	Let $A>0$ and $m>8\pi$. Then there exists a smooth and strictly
	positive initial datum $(n_{\rm in})_{0}$ satisfying
	$\int_{\mathbb{T}^{2}}(n_{\rm in})_{0}dY_{\perp}=m$
	such that the corresponding classical solution of
	\begin{equation*}
		\label{eq:scaled-2d-KS}
		\left\{
		\begin{aligned}
			&\partial_tn_{0}
			=
			\frac{1}{A}
			\left[
				\Delta_{Y_{\perp}}n_{0}
				-\nabla_{Y_{\perp}}\cdot
				(n_{0}\nabla_{Y_{\perp}}c_{0})
			\right],
			\\
			&-\Delta_{Y_{\perp}}c_{0}
			=
			n_{0}-\overline n
		\end{aligned}
		\right.
	\end{equation*}
	blows up in finite time.
\end{proposition}

\begin{proof}
	Let $r$ and $\delta$ be given by
	Lemma \ref{lem:negative-slope}, and choose the initial datum
	\eqref{eq:concentrated-initial-data} so that
	\eqref{eq:initial-moment-barrier} holds. Let
	$T_{\max}\in(0,\infty]$
	be the maximal classical existence time of the corresponding
	two-dimensional solution. Define
	\[
	T_1
	=
	\sup\left\{
	T<T_{\max}:
	I_r(t)<\delta r^{2}
	\text{ for all }0\leq t\leq T
	\right\}.
	\]
	By \eqref{eq:initial-moment-barrier} and the continuity of $I_r$,
	one has $T_1>0$.
For every $0\leq t<T_1$, Lemma \ref{lem:negative-slope} gives
	$I_r'(t)
	\leq
	-\frac{\Gamma_m}{2A}.$
	Integrating from $0$ to $t$, we obtain
	\begin{equation}
		\label{eq:moment-linear-decay}
		I_r(t)
		\leq
		I_r(0)-\frac{\Gamma_m}{2A}t,
		\qquad
		0\leq t<T_1.
	\end{equation}
	In particular,
	\[
	I_r(t)
	\leq I_r(0)
	<
	\frac{\delta r^{2}}{2},
	\qquad 0\leq t<T_1.
	\]

	We claim that $T_1=T_{\max}$. Indeed, if
	$T_1<T_{\max}$, then by continuity,
	\[
	I_r(T_1)
	\leq
	I_r(0)
	<
	\frac{\delta r^{2}}{2}
	<
	\delta r^{2}.
	\]
	Therefore, $I_r(t)<\delta r^{2}$ on a slightly larger interval
	$[0,T_1+\sigma]$ for some $\sigma>0$ with
	$T_1+\sigma<T_{\max}$, contradicting the definition of $T_1$.
	Hence \eqref{eq:moment-linear-decay} holds on the entire interval
	$[0,T_{\max})$.

	Suppose now that
	$T_{\max}>
	\frac{2A I_r(0)}{\Gamma_m}$.
	Set
	$t_0
	=
	\frac{2A I_r(0)}{\Gamma_m}
	<T_{\max}.$
	Then \eqref{eq:moment-linear-decay} gives
	\[
	I_r(t_0)
	\leq
	I_r(0)
	-
	\frac{\Gamma_m}{2A}
	\frac{2A I_r(0)}{\Gamma_m}
	=0.
	\]
	On the other hand, the initial density is strictly positive. By the
	strong maximum principle,
	\[
	n_{0}(t,{Y_{\perp}})>0,
	\qquad
	0<t<T_{\max}.
	\]
	Since $\psi_r\geq0$ and $\psi_r\not\equiv0$, it follows that
	\[
	I_r(t_0)
	=
	\int_{\mathbb{T}^{2}}
	\psi_r({Y_{\perp}})n_{0}(t_0,{Y_{\perp}})d{Y_{\perp}}
	>0,
	\]
	which is a contradiction. Therefore,
	\begin{equation}
		\label{eq:two-dimensional-blowup-time}
		T_{\max}
		\leq
		\frac{2A I_r(0)}{\Gamma_m}
		<\infty.
	\end{equation}
	Finally, the standard continuation criterion for the
	parabolic-elliptic Patlak--Keller--Segel system implies that if
	\[
	\sup_{0\leq t<T_{\max}}
	\|n_{0}(t)\|_{L^\infty(\mathbb{T}^{2})}
	<\infty,
	\]
	then elliptic regularity for $c_{0}$ and parabolic regularity for
	$n_{0}$ allow the solution to be extended beyond $T_{\max}$. Hence
	\[
	\limsup_{t\to T_{\max}}
	\|n_{0}(t)\|_{L^\infty(\mathbb{T}^{2})}
	=+\infty.
	\]
\end{proof}
\begin{proof}[Proof of Theorem \ref{thm:sharpness}]
	Fix $A>0$ and $M>16\pi^{2}$, and set
	$m=\frac{M}{|\mathbb{T}|}
	=\frac{M}{2\pi}>8\pi.$
	Let $(n_{\rm in})_{0}\in C^\infty(\mathbb{T}^{2})$ be the strictly
	positive initial density constructed in
	Proposition \ref{prop:two-dimensional-blowup}, and define
	\[
	n_{\rm in}(x,y,z)
	=
	(n_{\rm in})_{0}(y,z),
	\qquad
	u_{\rm in}(x,y,z)=0.
	\]
	Then
	$n_{\rm in}>0,$
	$n_{\rm in}\in C^\infty(\mathbb{T}^{3}),$
	and by $\int_{\mathbb{T}^{3}}n_{\rm in}dxdydz
	=|\mathbb{T}|
	\int_{\mathbb{T}^{2}}(n_{\rm in})_{0}d{Y_{\perp}}=2\pi m=M$.

	Let $(n_{0},c_{0})$ be the corresponding two-dimensional solution
	and let $T^{*}<\infty$ be its maximal existence time. Define
	$u_{1,0}$ as the solution of
	\[
	\partial_tu_{1,0}
	-\frac{1}{A}\Delta_{Y_{\perp}}u_{1,0}
	=
	\frac{1}{A}n_{0},
	\qquad
	u_{1,0}(0,{Y_{\perp}})=0.
	\]
	Its Duhamel representation is
		$u_{1,0}(t)
		=
		\frac{1}{A}
		\int_0^t
		{\rm e}^{\frac{t-s}{A}\Delta_{Y_{\perp}}}n_{0}(s)ds.$
	For every $T<T^{*}$, the function $n_{0}$ is smooth on
	$[0,T]\times\mathbb{T}^{2}$, and hence $u_{1,0}$ is also smooth there.
Now set
	\[
	n(t,x,y,z)=n_{0}(t,y,z),
	\qquad
	c(t,x,y,z)=c_{0}(t,y,z),\qquad u(t,x,y,z)
	=\bigl(u_{1,0}(t,y,z),0,0\bigr).
	\]
	By Corollary \ref{prop:invariant-zero-mode}, this is an exact
	classical solution of the full three-dimensional perturbation system
	on $[0,T^{*})$. Moreover,
	$\|n(t)\|_{L^\infty(\mathbb{T}^{3})}
	=
	\|n_{0}(t)\|_{L^\infty(\mathbb{T}^{2})}.$
	Therefore,
	\[
	\limsup_{t\to T^{*}}
	\|n(t)\|_{L^\infty(\mathbb{T}^{3})}
	=
	+\infty.
	\]

	The construction holds for every finite $A>0$. The factor $A^{-1}$
	in the reduced equations changes only the time scale, as is also
	visible from the upper bound
	\eqref{eq:two-dimensional-blowup-time}; it does not change the
	critical mass $8\pi$ of the two-dimensional zero mode. Finally,
	$u_{\rm in}=0$, and hence all the velocity smallness assumptions in
	Theorem \ref{result0} are satisfied in the strongest possible form.
\end{proof}

\appendix
\section{Some useful estimates and lemmas in the proof}\label{sec7}
\subsection{Several useful lemmas}
We first prove an embedding inequality for the non-zero mode.
\begin{lemma}\label{lem: poincare}
Let $f$ be a function such that $f_{\neq}\in H^{1}(\mathbb{T}^{3}),$ there holds
\begin{equation*}
	\|f_{\neq}\|_{L^{2}(\mathbb{T}^{3})}\leq C\|\partial_{x}f_{\neq}\|_{L^{2}(\mathbb{T}^{3})}\leq C\|\nabla f_{\neq}\|_{L^{2}(\mathbb{T}^{3})}.
\end{equation*}
\end{lemma}
\begin{proof}
It follows from Poincar$\acute{\rm e}$ inequality immediately and we omit it.
\end{proof}
The following lemma can be used to estimate the $L^{\infty}$ norm for the zero mode.
\begin{lemma}[Periodic anisotropic Sobolev inequalities]\label{sob_inf_1}
For a given function $f(x,y,z)$ and $f_0=\frac{1}{|\mathbb{T}|}\int_{\mathbb{T}}{f}(t,x,y,z)dx,$ we have
\begin{equation}\label{sob_result_1}
	\begin{aligned}
		&\|f_0\|_{L^{\infty}}\leq C\left(\|\partial_yf_0\|^{\frac{1}{2}}_{L^2}\|f_0\|^{\frac{1}{2}}_{L^2}+\|\partial_z\nabla f_0\|^{\frac{1}{2}}_{L^2}
		\|\partial_zf_0\|^{\alpha-\frac{1}{2}}_{L^2}
		\|f_0\|^{1-\alpha}_{L^2}+\|f_{0}\|_{L^{2}}\right),\\
		&\|f_0\|_{L^{\infty}}\leq C\big(
\|\partial_zf_0\|_2^{\frac12}\|f_0\|_2^{\frac12}
+\|\partial_y\nabla f_0\|_2^{\frac12}
 \|\partial_yf_0\|_2^{\alpha-\frac12}\|f_0\|_2^{1-\alpha}
+\|f_0\|_2\big),\\
		&\|f_{0}\|_{L^{\infty}_{z}L^2_y}\leq C\left(\|f_0\|_{L^2}+\|\partial_zf_0\|_{L^2}^{\alpha}\|f_0\|_{L^2}^{1-\alpha}\right),\\
		&\|f_{0}\|_{L^{\infty}_yL^{2}_{z}}
		\leq C(\|\partial_yf_0\|_{L^2}^{\frac{1}{2}}\|f_0\|_{L^2}^{\frac{1}{2}}+\|f_0\|_{L^2}),
	\end{aligned}
\end{equation}
where $\alpha$ is a constant with $\alpha\in\left(\frac12,\frac34\right).$
\end{lemma}
The following lemma can be used to estimate the $L^{\infty}$ norm for the non-zero mode. We only prove $\eqref{sob_result_2}_{1,2}$, and the remaining results are similar to Lemma 3.2 in \cite{CWW2025}. The proof is omitted.		

\begin{lemma}\label{sob_inf_2}
For a given function $g=g(x,y,z)$ and $\alpha\in(\frac{1}{2},\frac{3}{4})$, if $g_0=\frac{1}{|\mathbb{T}|}\int_{\mathbb{T}}{g}(t,x,y,z)dx=0,$ then we have
\begin{equation}\label{sob_result_2}
	\begin{aligned}
		&\|g\|_{L^{\infty}}\leq
		C\big(\|\partial_x\nabla g\|^{\frac{1}{2}}_{L^2}\|\partial_x\partial_z g\|^{1-\alpha}_{L^2}
		\|\partial_z^2g\|^{\alpha-\frac{1}{2}}_{L^2}+\|\partial_x\nabla g\|_{L^2}^{\frac12}\|\partial_x g\|_{L^2}^{\frac12}\big),\\
		&\|g\|_{L^{\infty}}\leq
		C\big(\|\partial_z\nabla g\|^{\frac{1}{2}}_{L^2}\|\partial_x\partial_z g\|^{1-\alpha}_{L^2}
		\|\partial_x^2g\|^{\alpha-\frac{1}{2}}_{L^2}+\|\partial_x\nabla g\|_{L^2}^{\frac12}\|\partial_x g\|_{L^2}^{\frac12}\big),\\
		&\|g\|_{L^{\infty}_{y,z}L^2_x}\leq C\big(
\|\partial_yg\|_2^{\frac12}\|g\|_2^{\frac12}
+\|\partial_z\nabla g\|_2^{\frac12}
 \|\partial_zg\|_2^{\alpha-\frac12}\|g\|_2^{1-\alpha}
+\|g\|_2\big),\\
		&\|g\|_{L^{\infty}_{x,y}L^2_{z}}\leq C\big(
		\|\partial_xg\|^{\frac{1}{2}}_{L^2}
		\|\partial_x\partial_yg\|^{\alpha-\frac{1}{2}}_{L^2}
		\|\partial_yg\|^{1-\alpha}_{L^2}+\|\partial_xg\|_{L^2}\big),\\
		&\|g\|_{L^{\infty}_{x,z}L^2_{y}}\leq  C\big(\|\partial_xg\|^{\alpha}_{L^2}\|g\|^{1-\alpha}_{L^2}
		+\|\partial_x\partial_zg\|^{\alpha}_{L^2}\|g\|^{1-\alpha}_{L^2}\big),\\
		&\|g\|_{L^{\infty}_{x}L^2_{y,z}}\leq C\|\partial_xg\|_{L^2}^{\alpha}\|g\|_{L^2}^{1-\alpha},\\
		&\|g\|_{L^{\infty}_{z}L^2_{y,x}}\leq C(\|g\|_{L^2}+\|\partial_zg\|_{L^2}^{\alpha}\|g\|_{L^2}^{1-\alpha}),\\
		&\|g\|_{L^{\infty}_{y}L^2_{x,z}}\leq  C\big(\|\partial_yg\|_{L^2}^{\frac{1}{2}}\|g\|^{\frac{1}{2}}_{L^2}+\|g\|_{L^2}\big),\\
		&\|g\|_{L^{\infty}_{x,z}L^2_{y}}\leq  C\big(\|\partial_xg\|^{\alpha}_{L^2}\|g\|^{1-\alpha}_{L^2}
		+\|\partial_x\partial_zg\|^{\frac{1}{2}}_{L^2}\|\partial_xg\|^{\alpha-\frac{1}{2}}_{L^2}\|\partial_zg\|^{\alpha-\frac{1}{2}}_{L^2}\|g\|^{\frac{3}{2}-2\alpha}_{L^2}\big).
	\end{aligned}
\end{equation}
We also use the following strong mixed-norm estimates. If $\eta$ is
one of $x,y,z$ and $\xi$ denotes the other two coordinates, then
\[
\|g\|_{L^2_\xi L^\infty_\eta}^2
\le C\|g\|_{L^2}
       \big(\|g\|_{L^2}+\|\partial_\eta g\|_{L^2}\big).
\]
This inequality does not require $g_0=0$. Moreover,
\[
\|g\|_{L^\infty_{y,z}L^2_x}^2
\le C\|(1,\partial_z)g\|_{L^2}
\big(\|(1,\partial_z)g\|_{L^2}
    +\|(1,\partial_z)\partial_y g\|_{L^2}\big).
\]
\end{lemma}
\begin{proof}
Due to $g_0=0$, we denote $	g(x,y,z)$ by
$g=\sum_{k_1,k_3\in\mathbb{Z},k_1\neq0}\widehat{g}_{k_1,k_3}(y){\rm e}^{i(k_1x+k_3z)},$
then
\begin{equation}\label{a40}
	\begin{aligned}
		\|g\|_{L^{\infty}}\leq \sum_{k_1\neq0,k_3\in\mathbb{Z}}\|\widehat{g}_{k_1,k_3}\|_{L^{\infty}_y}
		\leq C\sum_{k_1\neq0 ,k_3\in\mathbb{Z}}(\|\widehat{g}_{k_1,k_3}\|^{\frac{1}{2}}_{L^2_y}
		\|\partial_y\widehat{g}_{k_1,k_3}\|^{\frac{1}{2}}_{L^2_y}+\|\widehat{g}_{k_1,k_3}\|_{L^2_y}).
	\end{aligned}
\end{equation}	
First, there holds
\begin{align*}
	&\quad \sum_{k_1\neq0 ,k_3\in\mathbb{Z}}\|\widehat{g}_{k_1,k_3}\|^{\frac{1}{2}}_{L^2_y}
	\|\partial_y\widehat{g}_{k_1,k_3}\|^{\frac{1}{2}}_{L^2_y}\\
	&\leq\sum_{k_1\neq0,k_3\in\mathbb{Z}}
	\frac{\|k_1\partial_y\widehat{g}_{k_1,k_3}\|^{\frac12}_{L^2_y}
		\|k_1k_3\widehat{g}_{k_1,k_3}\|^{1-\alpha}_{L^2_y}\|k_3^2\widehat{g}_{k_1,k_3}\|^{\alpha-\frac12}_{L^2_y}
		+\|k_1\partial_y\widehat{g}_{k_1,k_3}\|^{\frac{1}{2}}_{L^2_y}\|\widehat{g}_{k_1,k_3}\|^{\frac{1}{2}}_{L^2_y}
		|k_1|^{\alpha-\frac12}}
	{|k_1|^{\alpha}(1+|k_3|^{\alpha})}.
\end{align*}
Using H$\rm \ddot{o}$lder's inequality, we get
\begin{equation}\label{a41}
	\begin{aligned}
		&\sum_{k_1\neq0 ,k_3\in\mathbb{Z}}\|\widehat{g}_{k_1,k_3}\|^{\frac{1}{2}}_{L^2_y}
		\|\partial_y\widehat{g}_{k_1,k_3}\|^{\frac{1}{2}}_{L^2_y}\leq
		C\big(\|\partial_x\partial_y g\|^{\frac{1}{2}}_{L^2}\|\partial_x\partial_z g\|^{1-\alpha}_{L^2}
		\|\partial_z^2g\|^{\alpha-\frac{1}{2}}_{L^2}+\|\partial_x\partial_y g\|_{L^2}^{\frac12}\|\partial_x g\|_{L^2}^{\frac12}\big).
	\end{aligned}
\end{equation}	

Thus, one deduces
\begin{equation}\label{a42}
	\begin{aligned}
		\quad \sum_{k_1\neq0 ,k_3\in\mathbb{Z}}\|\widehat{g}_{k_1,k_3}\|_{L^2_y}
		&\leq C\big(\|\partial_x g\|^{\frac{1}{2}}_{L^2}\|\partial_x\partial_z g\|^{1-\alpha}_{L^2}
		\|\partial_z^2g\|^{\alpha-\frac{1}{2}}_{L^2}+\|\partial_x g\|_{L^2}\big).
	\end{aligned}
\end{equation}
Combining \eqref{a40}, \eqref{a41} and \eqref{a42}, we finish the proof of $\eqref{sob_result_2}_1.$

Similarly, we can prove $\eqref{sob_result_2}_2.$
For the strong mixed-norm estimate, apply the one-dimensional Sobolev
inequality in $\eta$ pointwise in $\xi$, then integrate in $\xi$ and use
Cauchy--Schwarz. The last estimate follows by applying the
Hilbert-valued one-dimensional Sobolev inequality in $y$ with values
in $H^1_z(L^2_x)$, followed by the Sobolev embedding in $z$.
The lower-order terms are retained; the shear-periodic boundary
translations preserve the relevant $L^2_x$ norms.
\end{proof}

\begin{lemma}\label{weak kappa}
Under the smallness condition \eqref{condition U2}, for a given function $g=g(x,y,z),$ if $g_{0}=\frac{1}{|\mathbb{T}|}\int_{\mathbb{T}}g(t,x,y,z)dx=0,$ there holds
\begin{equation*}
	\begin{aligned}
		&\|g\|_{L^{\infty}_{y,z}L^{2}_{x}}\leq C\|(\partial_{x}, \partial_{z}-\kappa\partial_{y})g\|_{L^{2}}^{\frac12}\|\nabla(\partial_{x}, \partial_{z}-\kappa\partial_{y})g\|_{L^{2}}^{\frac12},\\
		&\|g\|_{L^{\infty}_{z}L^{2}_{y,x}}\leq C\|(\partial_{x}, \partial_{z}-\kappa\partial_{y})g\|_{L^{2}}.
	\end{aligned}
\end{equation*}
\end{lemma}
\begin{proof}
Let $G(X,Y,Z)$ such that $G(x,V(t,y,z),z)=g(x,y,z).$ Note that $(\partial_{z}-\kappa\partial_{y})g(x,y,z)=\partial_{Z}G(x,V(t,y,z),z)$ and $\|g\|_{L^{2}}\leq\|\partial_{x}g\|_{L^{2}}.$ Then the results follow from $\eqref{sob_result_2}_{3}$ and $\eqref{sob_result_2}_{7}.$
\end{proof}

\begin{lemma}[{\bf Lemma 5.5} in \cite{wei2}]\label{sob_14}
Under the conditions of Lemma \ref{lem:kappa},
if $ \partial_{x}f_{1}=0, P_{0}f_{2}=0, $ it holds that
\begin{equation*}
	\|f_{1}f_{2}\|_{L^{2}}\leq C\|f_{1}\|_{H^{1}}\left(\|f_{2}\|_{L^{2}}+\|(\partial_{z}-\kappa\partial_{y})f_{2}\|_{L^{2}} \right),
\end{equation*}
\begin{equation*}
	\|\nabla\Delta^{-1}(f_{1}f_{2})\|_{L^{2}}\leq C\|f_{1}\|_{L^{2}}\left(\|f_{2}\|_{L^{2}}+\|(\partial_{z}-\kappa\partial_{y})f_{2}\|_{L^{2}} \right),
\end{equation*}
\begin{equation*}
	\|\nabla(f_{1}f_{2})\|_{L^{2}}\leq C\|f_{1}\|_{H^{1}}\left(\|f_{2}\|_{H^{1}}+\|(\partial_{z}-\kappa\partial_{y})f_{2}\|_{H^{1}} \right),
\end{equation*}
and for $j\in\{2,3\}, $
\begin{equation*}
	\|[\partial_{j}\Delta^{-1},f_{1}]\Delta f_{2}\|_{H^{1}}\leq C\|f_{1}\|_{H^{2}}\left(\|\nabla f_{2}\|_{L^{2}}+\|(\partial_{z}-\kappa\partial_{y})\nabla f_{2}\|_{L^{2}} \right).
\end{equation*}
\end{lemma}

Next, we introduce a logarithmic Hardy--Littlewood--Sobolev inequality that plays an important role in estimation $\|n_{0}\|_{L^{2}}$ and can be found in \cite{SW2005}.
\begin{lemma}\label{lem:hls}
Let $ \mathcal{M} $ be a 2D Riemannian compact manifold. For all $ m>0, $ there exists a constant $ C(m) $ such that for all nonnegative functions $ f\in L^{1}(\mathcal{M}) $ such that $ f\log f\in L^{1} $, if $ \int_{\mathcal{M}}fdx=m, $  then
\begin{equation}\label{ineq:hls}
	\int_{\mathcal{M}}f\log fdx+\frac{2}{m}\int_{\mathcal{M}}\int_{\mathcal{M}}f(x)f(y)\log d(x,y)\geq -C(m),
\end{equation}
where $ d(x,y) $ is the distance on the Riemannian manifold.
\end{lemma}

The following Gagliardo-Nirenberg-Sobolev inequality on $\mathbb{T}^{n}$ is frequently used in the proof.

\begin{lemma}[Periodic Gagliardo--Nirenberg inequalities]\label{lem:GNS}
Let $d=2$ or $3$, $1\le r\le2\le q<\infty$, and, if $d=3$, assume
$q\le6$. For $f\in H^1(\mathbb T^d)$,
\[
\|f\|_q\le C\|\nabla f\|_2^\theta\|f\|_r^{1-\theta}
             +C\|f\|_r,\qquad
\theta=\frac{1/r-1/q}{1/d-1/2+1/r}.
\]
The lower-order term can be omitted if $f$ has zero spatial mean,
or if $|\{f=0\}|\ge\eta|\mathbb T^d|$ for a fixed $\eta>0$;
in the latter case the constant may depend on $\eta$.
In particular, on $\mathbb T^2$, if $|\{f=0\}|\ge|\mathbb T^2|/2$,
\[
\|f\|_3^3\le C\|f\|_1\|\nabla f\|_2^2,\qquad
\|f\|_2^4\le C\|f\|_1^2\|\nabla f\|_2^2.
\]
\end{lemma}
\begin{proof}
The nonhomogeneous inequality is the standard periodic
Gagliardo--Nirenberg inequality. Under either additional hypothesis,
the corresponding Poincare inequality gives
$\|f\|_r\le C\|\nabla f\|_2$, so its lower-order term is absorbed
in the interpolation term. The last two estimates follow by taking
$(r,q)=(1,3)$ and $(1,2)$. Approximation gives the same estimates for
positive-part truncations in $H^1$.
\end{proof}

The following  elliptic estimates are necessary.
\begin{lemma}\label{lem:ellip_0}
Let $c_0$ and $n_{0}$ be the zero mode of $c$ and $n$, respectively, satisfying
$$-\Delta c_0+\bar{n}=n_{0},$$
for any $t\geq0$, then there hold
\ben\label{eq:elliptic}
&&\|\Delta c_0(t)\|_{L^2}
\leq \|n_{0}(t)\|_{L^2},\nonumber\\
&&\|\nabla c_{0}(t)\|_{L^{\infty}}\leq C\|n_{0}(t)-\bar{n}\|_{L^{3}}\nonumber\\
&&\|\nabla c_{0}(t)\|_{L^{4}}\leq C\|n_{0}(t)\|_{L^{2}}.
\een

\end{lemma}
\begin{proof}
The basic energy estimates yield
\begin{equation}
	\begin{aligned}
		\|\Delta c_0(t)\|^2_{L^2}+|\bar{n}|^{2}|\mathbb{T}|^{2}
		=\|n_{0}(t)\|^2_{L^2},
		\nonumber
	\end{aligned}
\end{equation}
which implies $\eqref{eq:elliptic}_1$.
Using Gagliardo--Nirenberg inequality and H$\ddot{\rm o}$lder's inequality, we have
\begin{equation*}
	\begin{aligned}
		\|\nabla c_{0}(t)\|_{L^{\infty}}\leq~& C\left(\|\nabla c_{0}(t)\|_{L^{2}}^{\frac14}\|\Delta c_{0}(t)\|_{L^{3}}^{\frac34}+\|\nabla c_{0}(t)\|_{L^{2}} \right)\\\leq~& C\|\Delta c_{0}(t)\|_{L^{3}}\leq C\|n_{0}(t)-\bar{n}\|_{L^{3}},
	\end{aligned}
\end{equation*}
which gives $\eqref{eq:elliptic}_2$.

Moreover, it follows from Lemma \ref{lem:GNS} and $\eqref{eq:elliptic}_1$ that
\begin{equation*}
	\|\nabla c_{0}(t)\|_{L^{4}}\leq C\|\nabla c_{0}(t)\|_{L^{2}}^{\frac12}\|\Delta c_{0}(t)\|_{L^{2}}^{\frac12}\leq C\|\Delta c_{0}(t)\|_{L^{2}}\leq C\|n_{0}(t)\|_{L^{2}},
\end{equation*}
which gives $\eqref{eq:elliptic}_3.$
\end{proof}

\begin{lemma}\label{lem:ellip_2}
For any $t\geq 0$ and $j\geq 0$, there holds
\begin{align*}
	\|\partial_x^j\Delta c_{\neq}(t)\|_{L^2}\leq \|\partial_x^jn_{\neq}(t)\|_{L^2},\quad 
	\|\partial_x^j\nabla c_{\neq}(t)\|_{L^4}
	\leq C\|\partial_x^j n_{\neq}(t)\|_{L^2}.
\end{align*}
\end{lemma}
\begin{proof}
By integration by parts, we have
		$\|\Delta c_{\neq}(t)\|^2_{L^2}=\|n_{\neq}(t)\|^2_{L^2}.$
Using Gagliardo--Nirenberg inequality and Lemma \ref{lem: poincare}, we obtain
\begin{equation*}
	\|\nabla c_{\neq}(t)\|_{L^{4}}\leq C\left(\|c_{\neq}(t)\|_{L^{2}}^{\frac18}\|\Delta c_{\neq}(t)\|_{L^{2}}^{\frac78}+\|c_{\neq}(t)\|_{L^{2}}\right)\leq C\|n_{\neq}(t)\|_{L^{2}}.
\end{equation*}
\end{proof}

\subsection{Space-time estimates}
First, we need to prove the following space-time estimate, and this result
is also an improvement on previous space-time estimates \cite{CWW1,wei2}, which allows us to estimate the non-zero mode $n_{\neq}$ in the periodic domain $\mathbb{T}^3$ under the Kelvin   coordinate \eqref{eq:kelvin-pullback-main}.
\begin{proposition}\label{Lvf 0}
Assume that $f$ satisfies
$$\partial_t f-\frac{1}{A}\Delta f+\left(y+\frac{\mathbf{U}_2(t,y,z)}{A}\right)\partial_xf=\partial_xf_{1}+f_{2}
+\nabla\cdot f_{3}$$
for $ t\in[0,T], $ where $f$, $f_{1}$, $f_{2}$ and $f_{3}$ are given functions and $ P_{0}f=P_{0}f_{1}=P_{0}f_{2}=P_{0}f_{3}=0 $.
As long as
\begin{equation}\label{condition}
	\frac{\|\Delta \mathbf{U}_2\|_{L^{\infty}H^2}}{A}
	+\|\partial_t \mathbf{U}_2\|_{L^{\infty}L^{\infty}}<c,
\end{equation}
for some small $c$ independent of $A$ and $t,$ then there holds
\begin{equation*}
	\begin{aligned}
		\|f_{\neq}\|^2_{X_a}
		\leq C\left(\|(f_{\rm in})_{\neq}\|_{L^2}^2
		+\|{\rm e}^{aA^{-\frac{1}{3}}t}\nabla f_{1,\neq}\|_{L^2L^2}^2
		+A^{\frac{1}{3}}\|{\rm e}^{aA^{-\frac{1}{3}}t}f_{2,\neq}\|_{L^2L^2}^2 +A\|{\rm e}^{aA^{-\frac{1}{3}}t}f_{3,\neq}\|_{L^2L^2}^2\right),
	\end{aligned}
\end{equation*}
where ``$a$'' is a positive constant.
\end{proposition}
\begin{proof}
Our proof was inspired the coordinate transform method of \cite{wei2}.
First, we state the following result, which is classical.
\begin{lemma}\label{lem_1}
	Let $ g:\mathbb{T}^{3}\to\mathbb{R}^{3} $ be a $ C^{1} $ map such that $ \|\nabla g\|_{L^{\infty}}<\frac12. $ Then it holds that
	\begin{itemize}
		\item[(i)] For $1\le p\le\infty$ and $f\in W^{1,p}$,
        \[
        \|f\circ({\rm Id}+g)\|_{L^p}\sim\|f\|_{L^p},\qquad
        \|\nabla(f\circ({\rm Id}+g))\|_{L^p}\sim\|\nabla f\|_{L^p}.
        \]
        Here $A\sim B$ means $C^{-1}A\le B\le CA$ for an absolute constant $C$.
		\item[(ii)] There is a unique $C^1$ solution of
        \[
        h(x,y,z)=g\bigl((x,y,z)+h(x,y,z)\bigr),\qquad
        \|\nabla h\|_{L^\infty}\le C\|\nabla g\|_{L^\infty}.
        \]
	\end{itemize}
\end{lemma}

Then, we define the map $({\rm{Id}}+g): \mathbb{T}^{3}\to\mathbb{T}^{3}$ by
\begin{equation*}
	\begin{aligned}
		(x,y,z)\mapsto (X,Y,Z)=\left(x, V(t,y,z), z \right),
		\quad V=y+\frac{\mathbf{U}_2(t,y,z)}{A}.
	\end{aligned}
\end{equation*}
Denote
\begin{equation}\label{Fi}
	\begin{aligned}
		F(t,x,V(t,y,z),z)&=f(t,x,y,z), \\ F_{j}(t,x,V(t,y,z),z)&=f_{j}(t,x,y,z),~~~{\rm for}~~~j\in\{1,2\},\\
		F_{3}(t,x,V(t,y,z),z)&={\rm{div}}~f_{3}(t,x,y,z).
	\end{aligned}
\end{equation}
According to  (\ref{condition}), by choosing $ c $ small enough, there holds
\begin{equation*}
	\frac{1}{A}\|\nabla\mathbf{U}_2\|_{L^{\infty}}\leq\frac{C}{A}\|\nabla\mathbf{U}_2\|_{H^{2}}\leq \frac{C}{A}\|\Delta\mathbf{U}_2\|_{H^{2}}<\frac12,
\end{equation*}
which along with Lemma \ref{lem_1} implies that $ F(t,X,Y,Z)\in C^{1} $ is well-defined.

Let $ w_{t}(t,Y,Z), w_{y}(t,Y,Z) $ and $ w_{z}(t,Y,Z) $ be such that
\begin{equation}\label{w}
	\begin{aligned}
		w_{t}(t,V(t,y,z),z)&=\frac{\partial_{t}\mathbf{U}_2(t,y,z)}{A}=\partial_{t}V(t,y,z), \\
		w_{y}(t,V(t,y,z),z)&=\frac{\partial_{y}\mathbf{U}_2(t,y,z)}{A}=\partial_{y}V(t,y,z)-1 , \\
		w_{z}(t,V(t,y,z),z)&=\frac{\partial_{z}\mathbf{U}_2(t,y,z)}{A}=\partial_{z}V(t,y,z).
	\end{aligned}
\end{equation}
Direct calculations indicate that
\begin{equation}\label{partial_t f}
	\begin{aligned}
		\partial_{t}f&=\left(\partial_{t}+w_{t}\partial_{Y} \right)F,\quad \partial_{x}f=\partial_{X}F, \\\partial_{y}f&=\left(1+w_{y}\right)\partial_{Y}F,\quad \partial_{z}f=\left(\partial_{Z}+w_{z}\partial_{Y} \right)F.
	\end{aligned}
\end{equation}
Therefore, we obtain that
$V\partial_{x}f=Y\partial_{X}F,$
and
\begin{equation}\label{f''}
	\begin{aligned}
		\Delta f=~&\partial_{X}^{2}F+\partial_{y}\left[\left(1+w_{y} \right)\partial_{Y}F \right]+\partial_{z}\left[\left(\partial_{Z}+w_{z}\partial_{Y} \right)F \right] \\=~&\left(\partial_{X}^{2}+\partial_{Y}^{2}+\partial_{Z}^{2} \right)F+\left[\left(1+w_{y} \right)^{2}+w_{z}^{2}-1 \right]\partial_{Y}^{2}F+2w_{z}\partial_{Z}\partial_{Y}F+
		\left(w_{yy}+w_{zz} \right)\partial_{Y}F\\=~&\Delta F+G\partial_{Y}^{2}F+2w_{z}\partial_{Z}\partial_{Y}F+H\partial_{Y}F,
	\end{aligned}
\end{equation}
where
		$G=\left(1+w_{y} \right)^{2}+w_{z}^{2}-1$ and
		$H(t, V(t,y,z),z)=\frac{\Delta\mathbf{U}_2(t,y,z)}{A}.$

Using (\ref{partial_t f})--(\ref{f''}), we write the operator $ \mathcal{L}_{V} $ into
\begin{equation}\label{Lv}
	\begin{aligned}
		\mathcal{L}_{V}f
		&=\partial_{t}f-\frac{1}{A}\Delta f+V\partial_{x}f
		=\left(\partial_{t}+w_{t}\partial_{Y} \right)F-\frac{1}{A}\left(\Delta F+G\partial_{Y}^{2}F+2w_{z}\partial_{Z}\partial_{Y}F+H\partial_{Y}F \right)+Y\partial_{X}F\\
		&=\mathcal{L}F-\frac{1}{A}\left(G\partial_{Y}^{2}F+2w_{z}
		\partial_{Z}\partial_{Y}F \right)
		+\left(w_{t}-\frac{H}{A} \right)\partial_{Y}F,
	\end{aligned}
\end{equation}
where $\mathcal{L}=\partial_{t}-\frac{1}{A}\Delta+Y\partial_{X}.$
Due to (\ref{w}) and (\ref{partial_t f}), there hold
\begin{equation*}
	\begin{aligned}
		\partial_{Y}F&=\frac{\partial_{y}f}{\partial_{y}V},\quad \partial_{Z}F=\partial_{z}f-\frac{\partial_{z}V\partial_{y}f}{\partial_{y}V}, \\
		G(t,V(t,y,z),z)&=\left(1+\frac{\partial_{y}\mathbf{U}_2}{A} \right)^{2}+\left(\frac{\partial_{z}\mathbf{U}_2}{A} \right)^{2}-1=(\partial_{y}V)^{2}+(\partial_{z}V)^{2}-1,
	\end{aligned}
\end{equation*}
which imply that
\begin{equation}\label{G'}
	\begin{aligned}
		\partial_{Y}G+2\partial_{Z}w_{z}=~&\frac{\partial_{y}\left((\partial_{y}V)^{2}+(\partial_{z}V)^{2}-1 \right)}{\partial_{y}V}+2\left[\partial_{z}-\left(\frac{\partial_{z}V}{\partial_{y}V} \right)\partial_{y} \right]\partial_{z}V\\=~&2\partial_{y}^{2}V+\frac{2\partial_{z}V\partial_{y}\partial_{z}V}{\partial_{y}V}+2\partial_{z}^{2}V-2\frac{\partial_{z}V}{\partial_{y}V}\partial_{y}\partial_{z}V=2\Delta V=\frac{2}{A}\Delta \mathbf{U}_2=2H.
	\end{aligned}
\end{equation}

Combining (\ref{Lv}) and (\ref{G'}), one deduces that
\begin{equation}\label{Lv 1}
	\begin{aligned}
		\mathcal{L}_{V}f=\mathcal{L}F-\frac{1}{A}\partial_{Y}(G\partial_{Y}F)-\frac{2}{A}\partial_{Z}\left(w_{z}\partial_{Y}F \right)+\left(w_{t}+\frac{H}{A} \right)\partial_{Y}F.
	\end{aligned}
\end{equation}
From (\ref{Fi}), we get $ \mathcal{L}_{V}f=\partial_{X}F_{1}+F_{2}+F_{3}. $ Then (\ref{Lv 1}) yields
\begin{equation*}
	\begin{aligned}
		\mathcal{L}F=\partial_{X}F_{1}+F_{2}+F_{3}+\frac{1}{A}\partial_{Y}(G\partial_{Y}F)+\frac{2}{A}\partial_{Z}\left(w_{z}\partial_{Y}F \right)-\left(w_{t}+\frac{H}{A} \right)\partial_{Y}F.
	\end{aligned}
\end{equation*}
This implies that
\begin{equation}\label{LV 2}
	\begin{aligned}
		\|F\|_{X_{a}}^{2}\leq~&C\bigg(
		\|F(0)\|_{L^{2}}^{2}+\|{\rm e}^{aA^{-\frac13}t}\nabla F_{1}\|_{L^{2}L^{2}}^{2}
		+A^{\frac13}\|{\rm e}^{aA^{-\frac13}t}F_{2}\|_{L^{2}L^{2}}^{2}
		+A\|{\rm e}^{aA^{-\frac13}t}\nabla\Delta^{-1}F_{3}\|_{L^{2}L^{2}}^{2}
		\\&+\frac{1}{A}\left\|{\rm e}^{aA^{-\frac13}t}\left(G,2w_{z} \right)\partial_{Y}F\right\|_{L^{2}L^{2}}^{2}
		+A^{\frac13}\left\|{\rm e}^{aA^{-\frac13}t}\left(w_{t}+\frac{H}{A} \right)\partial_{Y}F\right\|_{L^{2}L^{2}}^{2}\bigg)\\:
		=~&C\Big(\|F(0)\|_{L^{2}}^{2}+\|{\rm e}^{aA^{-\frac13}t}\nabla F_{1}\|_{L^{2}L^{2}}^{2}
		+A^{\frac13}\|{\rm e}^{aA^{-\frac13}t}F_{2}\|_{L^{2}L^{2}}^{2}\\&+A\|{\rm e}^{aA^{-\frac13}t}\nabla\Delta^{-1}F_{3}\|_{L^{2}L^{2}}^{2}+K_{1}+K_{2}
		\Big).
	\end{aligned}
\end{equation}
For $K_{1},$ recall that $w_{z}=\frac{\partial_{z}\mathbf{U}_2}{A}$ and
$$ G=(\partial_{y}V)^{2}+(\partial_{z}V)^{2}-1
=\left(\frac{\partial_{y}\mathbf{U}_2}{A}+1\right)^{2}
+\left(\frac{\partial_{z}\mathbf{U}_2}{A} \right)^{2}-1
=\frac{|\nabla \mathbf{U}_2|^{2}}{A^{2}}+\frac{2\partial_{y}\mathbf{U}_2}{A}.$$
Then using (\ref{condition}), one deduces
\begin{equation}\label{1}
	\begin{aligned}
		K_{1}\leq\big(\|G\|_{L^{\infty}L^{\infty}}^{2}+\left\|2w_{z}\right\|_{L^{\infty}L^{\infty}}^{2} \big)
		\frac{\|{\rm e}^{aA^{-\frac13}t}\partial_{Y}F\|_{L^{2}L^{2}}^{2}}{A}
		\leq C(c^{4}+c^{2})\|F\|_{X_{a}}^{2},
	\end{aligned}
\end{equation}
where we use the Poincar$\acute{\rm e}$ inequality $\|\nabla \mathbf{U}_2\|_{L^{2}}\leq C\|\Delta \mathbf{U}_2\|_{L^{2}}.$

For $K_{2},$ as $w_{t}=\frac{\partial_{t}\mathbf{U}_2}{A}$ and $H=\frac{\Delta \mathbf{U}_2}{A}$, using (\ref{condition}), we have
\begin{equation}\label{1'}
	\begin{aligned}
		K_{2}\leq A^{\frac13}\left(\left\|w_{t}\right\|_{L^{\infty}L^{\infty}}^{2}
		+\left\|\frac{H}{A}\right\|_{L^{\infty}L^{\infty}}^{2} \right)\|{\rm e}^{aA^{-\frac13}t}\partial_{Y}F\|_{L^{2}L^{2}}^{2}
		\leq\frac{Cc^{2}}{A^{\frac23}}\|F\|_{X_{a}}^{2}.
	\end{aligned}
\end{equation}
Combining (\ref{1}) with (\ref{1'}), then choosing $c$ small enough, we get by (\ref{LV 2}) that
\begin{equation}\label{LV 3}
	\begin{aligned}
		\|F\|_{X_{a}}^{2}\leq C\Big(\|F(0)\|_{L^{2}}^{2}+\|{\rm e}^{aA^{-\frac13}t}\nabla F_{1}\|_{L^{2}L^{2}}^{2}
		+A^{\frac13}\|{\rm e}^{aA^{-\frac13}t}F_{2}\|_{L^{2}L^{2}}^{2}
		+A\|{\rm e}^{aA^{-\frac13}t}\nabla\Delta^{-1}F_{3}\|_{L^{2}L^{2}}^{2} \Big).
	\end{aligned}
\end{equation}
It follows from Lemma \ref{lem_1} that
	\begin{align}
		&\|{\rm e}^{aA^{-\frac13}t}\nabla F\|_{L^{2}L^{2}}^{2}\leq 	C\|{\rm e}^{aA^{-\frac13}t}\nabla f\|_{L^{2}L^{2}}^{2},~~~
		\|{\rm e}^{aA^{-\frac13}t}F\|_{L^{2}L^{2}}^{2}\leq 	C\|{\rm e}^{aA^{-\frac13}t}f\|_{L^{2}L^{2}}^{2},\label{F f}\\
        &\|{\rm e}^{aA^{-\frac13}t}\nabla F_{1}\|_{L^{2}L^{2}}^{2}\leq 	C\|{\rm e}^{aA^{-\frac13}t}\nabla f_{1}\|_{L^{2}L^{2}}^{2},
		~~~\|{\rm e}^{aA^{-\frac13}t}F_{2}\|_{L^{2}L^{2}}^{2}\leq 	C\|{\rm e}^{aA^{-\frac13}t}f_{2}\|_{L^{2}L^{2}}^{2}.		\label{2}
	\end{align}
Next we claim
\begin{equation}\label{4}
	\|\nabla\Delta^{-1}F\|_{L^{2}}\sim\|\nabla\Delta^{-1}f\|_{L^{2}}.
\end{equation}
As $ V=y+\frac{\mathbf{U}_2(t,y,z)}{A} $ and (\ref{condition}), we get
\begin{equation}\label{V' V''}
	\begin{aligned}
		\|\nabla V\|_{L^{\infty}}+\|\nabla^{2} V\|_{L^{\infty}}
		\leq 1+\frac{\|\nabla \mathbf{U}_2\|_{L^{\infty}}}{A}
		+\frac{\|\nabla^{2} \mathbf{U}_2\|_{L^{\infty}}}{A}
		\leq 1+C\frac{\|\Delta \mathbf{U}_2\|_{H^{2}}}{A}\leq C.
	\end{aligned}
\end{equation}
Denote
\begin{equation*}
	\begin{aligned}
		&	\check{F}_{1}=\Delta^{-1}F,\quad \check{f}_{1}(t,x,y,z)=\check{F}_{1}(t,x,V(t,y,z),z),
		\\&\check{f}_{2}=\Delta^{-1}f,\quad \check{f}_{2}(t,x,y,z)=\check{F}_{2}(t,x,V(t,y,z),z).
	\end{aligned}
\end{equation*}
Then using (\ref{F f}) and (\ref{V' V''}), one deduces
\begin{equation}\label{one hand}
	\begin{aligned}
		\|\nabla\Delta^{-1}F\|_{L^{2}}^{2}
		=~&\|\nabla \check{F}_{1}\|_{L^{2}}^{2}=-<F, 		
		\check{F}_{1}>=-<f,\partial_{y}V \check{f}_{1}>\\
		=~&<\nabla \check{f}_{2}, \nabla(\partial_{y}V \check{f}_{1})>
		\leq\|\nabla \check{f}_{2}\|_{L^{2}}
		\|\nabla(\partial_{y}V \check{f}_{1})\|_{L^{2}}\\
		\leq~&C\|\nabla \check{f}_{2}\|_{L^{2}}
		\left(\|\nabla V\|_{L^{\infty}}\|\nabla \check{f}_{1}\|_{L^{2}}
		+\|\nabla^{2} V\|_{L^{\infty}}\|\check{f}_{1}\|_{L^{2}} \right)
		\\\leq~&C\|\nabla \check{f}_{2}\|_{L^{2}}\|\nabla \check{f}_{1}\|_{L^{2}}
		\leq C\|\nabla \check{f}_{2}\|_{L^{2}}\|\nabla \check{f}_{1}\|_{L^{2}},
	\end{aligned}
\end{equation}
which implies
\begin{equation}\label{F1<f2}
	\|\nabla \check{F}_{1}\|_{L^{2}}
	\leq C\|\nabla \check{f}_{2}\|_{L^{2}}.
\end{equation}
On the other hand, there holds
\begin{equation*}
	\begin{aligned}
		\|\nabla\Delta^{-1}f\|_{L^{2}}^{2}=~&\|\nabla \check{f}_{2}\|_{L^{2}}^{2}=-<f,\check{f}_{2}>=-<F, \check{F}_{2}/\partial_{y}V>\\=~&<\nabla \check{F}_{1}, \nabla(\check{F}_{2}/\partial_{y}V)>\leq\|\nabla \check{F}_{1}\|_{L^{2}}\|\nabla(\check{F}_{2}/\partial_{y}V)\|_{L^{2}}.
	\end{aligned}
\end{equation*}
Notice that $ \partial_{y}V=\frac{1}{A}\partial_{y}\mathbf{U}_2 +1 $ and $ \frac{1}{A}\|\nabla \mathbf{U}_2\|_{L^{\infty}}<\frac12, $ we have $ \frac12<\partial_{y}V<\frac32. $
Along this with (\ref{V' V''}), we get
$ \|\nabla(1/\partial_{y}V)\|_{L^{\infty}}\leq C, $ and
\begin{equation}\label{other hand}
	\begin{aligned}
		\|\nabla\Delta^{-1}f\|_{L^{2}}^{2}=~&\|\nabla \check{f}_{2}\|_{L^{2}}^{2}\leq C\|\nabla \check{F}_{1}\|_{L^{2}}\left(\|\nabla \check{F}_{2}\|_{L^{2}}+\|\check{F}_{2}\|_{L^{2}} \right)\\\leq~&C\|\nabla \check{F}_{1}\|_{L^{2}}\|\nabla \check{F}_{2}\|_{L^{2}}\leq C\|\nabla \check{F}_{1}\|_{L^{2}}\|\nabla \check{f}_{2}\|_{L^{2}},
	\end{aligned}
\end{equation}
which gives
\begin{equation}\label{f2<F1}
	\|\nabla \check{f}_{2}\|_{L^{2}}\leq C\|\nabla \check{F}_{1}\|_{L^{2}}.
\end{equation}

Combining with (\ref{one hand})-(\ref{f2<F1}), we conclude that
\begin{equation*}
	\|\nabla\Delta^{-1}F\|_{L^{2}}=\|\nabla \check{F}_{1}\|_{L^{2}}\sim\|\nabla \check{f}_{2}\|_{L^{2}}=\|\nabla\Delta^{-1}f\|_{L^{2}},
\end{equation*}
which indicates (\ref{4}) holds. Similarly, we obtain
\begin{equation}\label{5}
	\|{\rm e}^{aA^{-\frac13}t}\nabla\Delta^{-1}F_{3}\|_{L^{2}L^{2}}^{2}\leq C\|{\rm e}^{aA^{-\frac13}t}f_{3}\|_{L^{2}L^{2}}^{2}.
\end{equation}

From (\ref{LV 3}), (\ref{2}) and (\ref{5}), we complete the proof.
\end{proof}	
For reference, the weaker coefficient condition used in the scalar
and Sobolev estimates is
\begin{equation}\label{condition U2}
	\frac{\|\Delta\mathbf{U}_{2}\|_{L^{\infty}H^{2}}}{A}+\|\partial_{t}\mathbf{U}_{2}\|_{L^{\infty}L^{\infty}}\leq c
\end{equation}
with $c$ sufficiently small. For the coupled estimate below and the
subsequent higher-order estimates, we use
\begin{equation}\label{condition U2 high}
\frac{\|\Delta\mathbf U_2\|_{L^\infty(0,T;H^2)}}{A}
 +\|\partial_t\mathbf U_2\|_{L^\infty(0,T;H^2)}
 \le\delta\le\delta_0.
\end{equation}
The spatial norms of $\mathbf U_2$ in these conditions are on
$\mathbb T^2$. In our applications
\eqref{condition U2 high} follows from Lemma \ref{u1_hat2}.
\begin{proposition}\label{timespace1}
There exists an absolute constant $a_*>0$ with the following property.
Fix $0<a<b<\min\{2a,a_*\}$ and $A\ge1$. Suppose that
\eqref{condition U2 high} holds with
$\delta_0=\delta_0(a,b)>0$ sufficiently small, and that
$\mathbf U_2(0)=0$. Let
\[
\begin{cases}
\mathcal L_Vh_1=\nabla\cdot g_1+r_1,\\
\mathcal L_Vh_2+\partial_x\partial_z\Delta^{-1}h_1
                    =\nabla\cdot g_2+r_2
\end{cases}
\qquad\text{on}\quad[0,T],
\]
where $P_0h_j=P_0g_j=P_0r_j=0$ for $j\in\{1,2\}$. Set
\begin{align*}
\mathcal N_a^2
&:=\|h_1(0)\|_{L^2}^2+\|h_2(0)\|_{L^2}^2
  +A\sum_{j=1}^2\|{\rm e}^{aA^{-\frac13}t}g_j\|_{L^2L^2}^2
  +A^{\frac13}\sum_{j=1}^2\|{\rm e}^{aA^{-\frac13}t}r_j\|_{L^2L^2}^2,\\
F_b&:=\|\partial_x^2\Delta^{-1}h_1\|_{X_b}.
\end{align*}
If these quantities are finite, then
\begin{equation}\label{eq:a2-revised}
\|h_1\|_{X_a}^2+\|h_2\|_{X_a}^2
\le C\mathcal N_a^2+C\delta^2A^{-\frac13}F_b^2.
\end{equation}
The constants are independent of $A,T$ and the functions. 
\end{proposition}
\begin{proof}
Write $\phi=\Delta^{-1}h_1$ and
$D_V=\partial_z-\kappa\partial_y$, where
$\kappa=\frac{\partial_zV}{\partial_yV}$. The spatial smallness in
\eqref{condition U2 high} implies $\partial_yV\ge\frac12$.
The exact decomposition
\begin{equation}\label{eq:a2-split}
\partial_x\partial_z\phi=\partial_xD_V\phi+R,
\qquad R=\kappa\partial_x\partial_y\phi,
\end{equation}
separates a coupled term controlled by a Fourier energy from a
remainder that uses $\mathbf U_2(0)=0$ and $b>a$.

\medskip
\noindent\textbf{Step 1. Estimate for the good derivative coupling.}
First replace the coupling in the second equation by
$\partial_xD_V\Delta^{-1}h_1$. We prove
\begin{equation}\label{eq:a2-good-estimate}
\|h_1\|_{X_a}^2+\|h_2\|_{X_a}^2\le C\mathcal N_a^2
\end{equation}
for this auxiliary system.
The spatial mean of $\frac{\mathbf U_2}{A}$ can be removed by a
time-dependent translation in $x$, preserving all the norms and
the zero-mode conditions. Keeping the same notation for the
mean-free part, Poincare's inequality gives
$\left\|\frac{\mathbf U_2}{A}\right\|_{L^\infty H^4}\le C\delta$.
Introduce the periodic coordinates
\[
\xi=x-tV(t,y,z),\qquad Y=V(t,y,z),\qquad Z=z.
\]
The shear-periodic boundary condition becomes ordinary periodicity
in $(\xi,Y,Z)$. Denote the pullbacks of $h_j,\phi$ by $H_j,\Phi$.
Set
\[
D_Y=\partial_Y-t\partial_\xi,\qquad
\nabla_L=(\partial_\xi,D_Y,\partial_Z),\qquad
\Delta_L=\nabla_L\cdot\nabla_L.
\]
If $\alpha,\beta,\gamma,\tau$ are the pullbacks of
$\partial_yV,\partial_zV,\Delta V,\partial_tV$, respectively,
then
\[
\partial_x=\partial_\xi,\quad
\partial_y=\alpha D_Y,\quad
\partial_z=\partial_Z+\beta D_Y,\quad
D_V=\partial_Z.
\]
The pullback of the physical Laplacian is
\begin{equation}\label{eq:a2-pulled-laplacian}
\mathscr A
=\Delta_L+D_Y(GD_Y)+2\partial_Z(\beta D_Y)-\gamma D_Y,
\qquad G=\alpha^2+\beta^2-1.
\end{equation}
Indeed, $\partial_YG+2\partial_Z\beta=2\gamma$.
All coefficients are independent of $\xi$, and uniformly in time
\begin{equation}\label{eq:a2-coefficients}
\|G\|_{H^3}+\|\beta\|_{H^3}+\|\gamma\|_{H^2}
\le C\delta,\qquad \|\tau\|_{L^\infty}\le CA^{-1}\delta.
\end{equation}
Here the coefficient norms are on $\mathbb T^2_{Y,Z}$.
The transformed transport-diffusion operator is
\[
\partial_t-A^{-1}\Delta_L
-A^{-1} D_Y(GD_Y)-2A^{-1}\partial_Z(\beta D_Y)
+(\tau+A^{-1}\gamma)D_Y.
\]
We record the elliptic estimate needed for the coupling.
For a streamwise Fourier mode $k\ne0$, define
\[
\mathcal R_k=(k^2-\partial_Z^2)^{\frac12},\qquad
\Lambda_k=(-\Delta_{L,k})^{\frac12},\qquad
\nabla_{L,k}=(ik,\partial_Y-ikt,\partial_Z).
\]
Multiplication by $B\in H^2(\mathbb T^2)$ satisfies, uniformly in
$|k|\ge1$,
\[
\|\mathcal R_k^{\frac12}(Bv)\|_{L^2}
\le C\|B\|_{H^2}\|\mathcal R_k^{\frac12}v\|_{L^2}.
\]
For example, this follows from the one-dimensional fractional
Sobolev product inequality in $Z$ and
$H^2(\mathbb T^2)\hookrightarrow C^{0,\theta}$ with
$\frac12<\theta<1$. Applying
$\Lambda_k^{-1}\mathcal R_k^{\frac12}$ to
$\mathscr A_k\Phi_k=(H_1)_k$, using
\eqref{eq:a2-pulled-laplacian} and $\Lambda_k\ge|k|\ge1$, yields
\begin{align*}
\|\Lambda_k\mathcal R_k^{\frac12}\Phi_k\|_{L^2}
&\le \|\Lambda_k^{-1}\mathcal R_k^{\frac12}(H_1)_k\|_{L^2}
 +C\delta\|\Lambda_k\mathcal R_k^{\frac12}\Phi_k\|_{L^2}.
\end{align*}
Thus, for small $\delta$,
\begin{equation}\label{eq:a2-fractional-elliptic}
\|\Lambda_k\mathcal R_k^{\frac12}\Phi_k\|_{L^2}
\le C\|\Lambda_k^{-1}\mathcal R_k^{\frac12}(H_1)_k\|_{L^2}.
\end{equation}
For $(k,\eta,\ell)\in\mathbb Z^3$, $k\ne0$, set
\[
\rho=(k^2+\ell^2)^{\frac12},\qquad
p=k^2+(\eta-kt)^2+\ell^2,\qquad
d^2=\frac{|k|\rho}{p},\qquad
s_\nu=\frac{A^{-\frac13}}{1+A^{-\frac23}(t-\eta/k)^2}.
\]
Use the real positive Fourier multiplier
\begin{equation}\label{eq:a2-multiplier}
m(t,k,\eta,\ell)
=\exp\left[-\arctan\frac{|k|(t-\eta/k)}{\rho}
           -\arctan\big(A^{-\frac13}(t-\eta/k)\big)\right].
\end{equation}
It satisfies ${\rm e}^{-\pi}\le m\le {\rm e}^\pi$ and
$-\partial_t\log m=d^2+s_\nu$.
Moreover, $d^2\ge \frac{k^2}{p}$, and
$\tfrac12A^{-1} p+s_\nu\ge\tfrac12 A^{-\frac13}$.
The latter follows by considering
$A^{-\frac23}(t-\frac{\eta}{k})^2\le1$ and its complement.

The coupling is bounded in the additional dissipation norm $d$.
In fact, $m$ commutes with $\Lambda_k,\mathcal R_k,\partial_Z$, so
\eqref{eq:a2-fractional-elliptic} and
$\|\partial_Z\mathcal R_k^{-1}\|_{L^2\to L^2}\le1$ imply
\begin{align}
\big|\langle\partial_\xi\partial_Z\Phi,m^2H_2\rangle\big|
&\le C\sum_{k\ne0}|k|
 \|\Lambda_k^{-1}\mathcal R_k^{\frac12}(H_1)_k\|_{L^2}
 \|\Lambda_k^{-1}\mathcal R_k^{\frac12}(H_2)_k\|_{L^2}\le C\|dH_1\|_{L^2}\|dH_2\|_{L^2}.
\label{eq:a2-coupling-form}
\end{align}
The same inequality holds with the time weight ${\rm e}^{aA^{-\frac13}t}$ on both factors.
For completeness, the remaining terms in the Fourier energy have
uniform bounds. Writing $W_j=w_aH_j$, the two variable-coefficient
diffusion terms and the drift term in the transformed operator
contribute at most
$CA^{-1}\delta\|\nabla_LW_j\|_{L^2}^2$. This follows by integration by
parts, \eqref{eq:a2-coefficients}, and
$\|W_j\|_{L^2}\le\|\partial_\xi W_j\|_{L^2}$.
If $\widetilde g_j,\widetilde r_j$ denote the pullbacks of the sources,
then integration by parts in the transformed divergence gives
\[
\big|\langle {\rm e}^{aA^{-\frac13}t}\widetilde{\nabla\cdot g_j},m^2W_j\rangle\big|
\le C\|{\rm e}^{aA^{-\frac13}t}\widetilde g_j\|_{L^2}\|\nabla_LW_j\|_{L^2},
\quad
\big|\langle {\rm e}^{aA^{-\frac13}t}\widetilde r_j,m^2W_j\rangle\big|
\le C\|{\rm e}^{aA^{-\frac13}t}\widetilde r_j\|_{L^2}\|W_j\|_{L^2}.
\]
The derivatives of $\alpha,\beta$ arising in the first inequality
are bounded; its lower-order terms are controlled by streamwise
Poincare. Young's inequality assigns respectively the costs
$CA\|{\rm e}^{aA^{-\frac13}t}\widetilde g_j\|_{L^2}^2$ and
$CA^{\frac13}\|{\rm e}^{aA^{-\frac13}t}\widetilde r_j\|_{L^2}^2$.
The term from differentiating $w_a$ is absorbed by the enhanced
dissipation when $a<a_*$ is sufficiently small.

Multiply the first energy identity by a sufficiently large fixed
constant and add the second. Inequality
\eqref{eq:a2-coupling-form} is then absorbed by the two $d$-dissipation
terms. Reducing $\delta_0$ absorbs the coefficient errors. We obtain
\begin{equation}\label{eq:a2-fourier-energy}
\begin{aligned}
  \sum_{j=1}^2\big(
 \|{\rm e}^{aA^{-\frac13}t}H_j\|_{L^\infty L^2}^2
 &+A^{-\frac13}\|{\rm e}^{aA^{-\frac13}t}H_j\|_{L^2L^2}^2
 +A^{-1}\|{\rm e}^{aA^{-\frac13}t}\nabla_LH_j\|_{L^2L^2}^2
 \\&+\|{\rm e}^{aA^{-\frac13}t}dH_j\|_{L^2L^2}^2\big)
\le C\mathcal N_a^2.  
\end{aligned}
\end{equation}
The Jacobian of the coordinate change is bounded above and below,
and physical gradients are uniformly equivalent to $\nabla_L$.
The corresponding negative first-order norms are equivalent by
duality, as in the proof of Proposition~\ref{Lvf 0}; multiplication
by the Jacobian is bounded in the gradient norm by streamwise
Poincare and the coefficient bounds. Since $d^2\ge \frac{k^2}{p}$,
\eqref{eq:a2-fourier-energy} controls the physical $X_a$ norms.
This proves \eqref{eq:a2-good-estimate}.

\noindent\textbf{Step 2. Estimate of the remainder.}
Apply \eqref{eq:a2-good-estimate} to the original system, using
\eqref{eq:a2-split} and replacing the second $L^2$ source by $r_2-R$.
Then
\begin{equation}\label{eq:a2-before-remainder}
\|h_1\|_{X_a}^2+\|h_2\|_{X_a}^2
\le C\mathcal N_a^2+CA^{\frac13}\|{\rm e}^{aA^{-\frac13}t}R\|_{L^2L^2}^2.
\end{equation}
The initial condition and the time derivative bound give
\[
\|\nabla\mathbf U_2(t)\|_{H^1}
\le\int_0^t\|\partial_s\mathbf U_2(s)\|_{H^2}ds
\le\delta t.
\]
For $f=f(y,z)$, the two-dimensional $L^4$ embedding and the
Hilbert-valued Gagliardo--Nirenberg inequality yield
\[
\|fg\|_{L^2}^2\le C\|f\|_{H^1(\mathbb T^2)}^2\|g\|_{L^2}
                       (\|g\|_{L^2}+\|\nabla_{y,z}g\|_{L^2}).
\]
Indeed, apply the scalar inequality to
$\|g(\cdot,y,z)\|_{L_x^2}$, which is periodic in $(y,z)$ even under
the shear-periodic boundary condition. Since $P_0\phi=0$, Fourier
interpolation also gives
\[
\|\partial_x\partial_y\phi\|_{L^2}
\le\|\partial_x^2\phi\|_{L^2}^{\frac12}\|h_1\|_{L^2}^{\frac12},\qquad
\|(1,\nabla)\partial_x\partial_y\phi\|_{L^2}
\le C\|\nabla h_1\|_{L^2}.
\]
Using $\partial_yV\ge\frac12$, we therefore obtain
\begin{equation}\label{eq:a2-remainder-pointwise}
\|R(t)\|_{L^2}^2
\le C\delta^2A^{-2}t^2
 \|\partial_x^2\phi(t)\|_{L^2}^{\frac12}
 \|h_1(t)\|_{L^2}^{\frac12}\|\nabla h_1(t)\|_{L^2}.
\end{equation}
The weight gap is used only at this point:
$\sup_{t\ge0}t^2e^{-(b-a)A^{-\frac13}\frac{t}{2}}\le C_{a,b}A^{\frac23}$.
By H\"older's inequality in time, with exponents $4,4,2$, it follows that
\begin{equation}
    \begin{aligned}\label{eq:a2-remainder-time}
A^{\frac13}\|{\rm e}^{aA^{-\frac13}t}R\|_{L^2L^2}^2
&\le C\delta^2A^{-1}
 \|{\rm e}^{bA^{-\frac13}t}\partial_x^2\phi\|_{L^2L^2}^{\frac12}
 \|{\rm e}^{aA^{-\frac13}t}h_1\|_{L^2L^2}^{\frac12}
 \|{\rm e}^{aA^{-\frac13}t}\nabla h_1\|_{L^2L^2}\\&\le C\delta^2A^{-\frac13}F_b^{\frac12}\|h_1\|_{X_a}^{\frac32}.
\end{aligned}
\end{equation}
Here the three time norms contribute respectively
$A^{\frac{1}{12}}$, $A^{\frac{1}{12}}$ and $A^{\frac12}$.
Finally,
$F_b^{\frac12}\|h_1\|_{X_a}^{\frac32}
\le\tfrac14F_b^2+\tfrac34\|h_1\|_{X_a}^2$.
Choosing $\delta_0(a,b)$ sufficiently small and using $A\ge1$,
we absorb the last term in \eqref{eq:a2-before-remainder} and obtain
\eqref{eq:a2-revised}.
The argument is an a priori estimate for smooth functions; the same
bound follows by approximation whenever the displayed norms are finite.
\end{proof}

For the following two higher-order estimates, the coefficient bound
\eqref{condition U2 high}, after removing the spatial mean of
$\frac{\mathbf U_2}{A}$ by an $x$-translation, supplies the spatial and
time-derivative assumptions used in \cite{wei2}. We decrease
$\delta_0$ if necessary to satisfy those assumptions as well.

\begin{proposition}[{\bf Proposition 4.7} in \cite{wei2}]\label{Lvf}
Let $ f $ satisfy
\begin{equation*}
	\mathcal{L}_{V}f=f_{1}+f_{2}+f_{3}
\end{equation*}
for $ t\in[0,T]. $ If $\mathbf{U}_{2}$ satisfies \eqref{condition U2 high}
and $ P_{0}f=P_{0}f_{1}=P_{0}f_{2}=P_{0}f_{3}=0,$  for $A\ge1$ and $0\le a<a_*$ with $a_*\leq 4$, then it holds that
\begin{equation*}
	\begin{aligned}
		\|\partial_{x}^{2}f_{\neq}\|_{X_{a}}^{2}+\|\partial_{x}(\partial_{z}-\kappa\partial_{y})f_{\neq}\|_{X_{a}}^{2}&\leq C\big(\|(f_{\rm in})_{\neq}\|_{H^{2}}^{2}+\|{\rm e}^{aA^{-\frac13}t}\Delta f_{1,\neq}\|_{L^{2}L^{2}}^{2}+A^{\frac13}\|{\rm e}^{aA^{-\frac13}t}\partial_{x}^{2}f_{2,\neq}\|_{L^{2}L^{2}}^{2}\\&+A^{\frac13}\|{\rm e}^{aA^{-\frac13}t}\partial_{x}(\partial_{z}-\kappa\partial_{y})f_{2,\neq}\|_{L^{2}L^{2}}^{2}+A\|{\rm e}^{aA^{-\frac13}t}\partial_{x}f_{3,\neq}\|_{L^{2}L^{2}}^{2}		
		\big).
	\end{aligned}
\end{equation*}	
\end{proposition}

\begin{proposition}[{\bf Proposition 4.9} in \cite{wei2}]\label{tilde Lv}
Let $ f $ satisfy $ \widetilde{\mathcal{L}_{V}}f=f_{1} $ for $ t\in[0, T]. $ If $\mathbf{U}_{2}$ satisfies \eqref{condition U2 high}
and $ P_{0}f=P_{0}f_{1}=0, $  for $A\ge1$ and $0\le a<a_*$ with $a_*\leq 4$, then it holds that
\begin{equation*}
	\|\Delta f_{\neq}\|_{X_{a}}^{2}\leq C\left(\|(\Delta f_{\rm in})_{\neq}\|_{L^{2}}^{2}+A\|{\rm e}^{aA^{-\frac13}t}\nabla f_{1,\neq}\|_{L^{2}L^{2}}^{2} \right).
\end{equation*}
\end{proposition}

\section*{Acknowledgement}
W. Wang was supported by National Key R\&D Program of China (No.2023YFA1009200) and NSFC under grant 12471219 and 12071054.  The research of J. Wei is partially supported by GRF from RGC of Hong Kong
entitled ``New frontiers in singularity formations in nonlinear partial differential equations".

\section*{Declaration of competing interest}
The authors declare that they have no known competing financial interests
or personal relationships that could have appeared to influence the work reported in this paper.

\end{document}